\documentclass[letterpaper, oneside]{amsart}

\usepackage[margin=1.1in]{geometry}
\usepackage{amsmath,amsthm,amssymb}
\usepackage{mathtools,bbm,xspace}
\usepackage[numbers]{natbib}

\usepackage{ifxetex}
\usepackage{parskip}

\usepackage[pagebackref, linktocpage]{hyperref}

\usepackage[dvipsnames]{xcolor}
\hypersetup{
colorlinks=true,
urlcolor=red,
linkcolor=blue,
citecolor=OliveGreen,
}
\usepackage{ulem}
\usepackage{tcolorbox}
\usepackage{enumerate}
\usepackage{multirow}

\usepackage{algorithm}
\usepackage{algpseudocode}

\usepackage{bm}
\usepackage{csquotes}
\allowdisplaybreaks

\numberwithin{equation}{section}

\newtheorem{theorem}{Theorem}[section]
\newtheorem*{theorem*}{Theorem}

\newtheorem{proposition}[theorem]{Proposition}
\newtheorem*{proposition*}{Proposition}
\newtheorem{lemma}[theorem]{Lemma}
\newtheorem*{lemma*}{Lemma}
\newtheorem{corollary}[theorem]{Corollary}
\newtheorem*{conjecture*}{Conjecture}

\newtheorem*{fact*}{Fact}

\newtheorem*{exercise*}{Exercise}

\newtheorem*{hypothesis*}{Hypothesis}

\theoremstyle{definition}
\newtheorem{definition}[theorem]{Definition}
\newtheorem{notation}[theorem]{Notation}

\newtheorem{assumption}[theorem]{Assumption}

\newtheorem*{claim*}{Claim}

\newtheorem{remark}[theorem]{Remark}
\newtheorem*{remark*}{Remark}

\newtheorem*{observation*}{Observation}

\usepackage{prettyref}
\newcommand{\savehyperref}[2]{\texorpdfstring{\hyperref[#1]{#2}}{#2}}

\newrefformat{eq}{\savehyperref{#1}{\textup{(\ref*{#1})}}}
\newrefformat{ineq}{\savehyperref{#1}{\textup{(\ref*{#1})}}}
\newrefformat{eqn}{\savehyperref{#1}{\textup{(\ref*{#1})}}}
\newrefformat{lem}{\savehyperref{#1}{Lemma~\ref*{#1}}}
\newrefformat{def}{\savehyperref{#1}{Definition~\ref*{#1}}}
\newrefformat{thm}{\savehyperref{#1}{Theorem~\ref*{#1}}}
\newrefformat{cor}{\savehyperref{#1}{Corollary~\ref*{#1}}}
\newrefformat{cha}{\savehyperref{#1}{Chapter~\ref*{#1}}}
\newrefformat{sec}{\savehyperref{#1}{Section~\ref*{#1}}}
\newrefformat{app}{\savehyperref{#1}{Appendix~\ref*{#1}}}
\newrefformat{tab}{\savehyperref{#1}{Table~\ref*{#1}}}
\newrefformat{fig}{\savehyperref{#1}{Figure~\ref*{#1}}}
\newrefformat{hyp}{\savehyperref{#1}{Hypothesis~\ref*{#1}}}
\newrefformat{alg}{\savehyperref{#1}{Algorithm~\ref*{#1}}}
\newrefformat{rem}{\savehyperref{#1}{Remark~\ref*{#1}}}
\newrefformat{item}{\savehyperref{#1}{Item~\ref*{#1}}}
\newrefformat{step}{\savehyperref{#1}{step~\ref*{#1}}}
\newrefformat{line}{\savehyperref{#1}{line~\ref*{#1}}}
\newrefformat{conj}{\savehyperref{#1}{Conjecture~\ref*{#1}}}
\newrefformat{fact}{\savehyperref{#1}{Fact~\ref*{#1}}}
\newrefformat{prop}{\savehyperref{#1}{Proposition~\ref*{#1}}}
\newrefformat{prob}{\savehyperref{#1}{Problem~\ref*{#1}}}
\newrefformat{claim}{\savehyperref{#1}{Claim~\ref*{#1}}}
\newrefformat{clm}{\savehyperref{#1}{Claim~\ref*{#1}}}
\newrefformat{relax}{\savehyperref{#1}{Relaxation~\ref*{#1}}}
\newrefformat{rem}{\savehyperref{#1}{Remark~\ref*{#1}}}
\newrefformat{red}{\savehyperref{#1}{Reduction~\ref*{#1}}}
\newrefformat{part}{\savehyperref{#1}{Part~\ref*{#1}}}
\newrefformat{ex}{\savehyperref{#1}{Exercise~\ref*{#1}}}
\newrefformat{property}{\savehyperref{#1}{Property~\ref*{#1}}}
\newrefformat{type}{\savehyperref{#1}{Type~\ref*{#1}}}
\newrefformat{eg}{\savehyperref{#1}{Example~\ref*{#1}}}
\newrefformat{obs}{\savehyperref{#1}{Observation~\ref*{#1}}}
\newrefformat{que}{\savehyperref{#1}{Question~\ref*{#1}}}
\newrefformat{cond}{\savehyperref{#1}{Condition~\ref*{#1}}}
\newrefformat{ass}{\savehyperref{#1}{Assumption~\ref*{#1}}}
\newrefformat{not}{\savehyperref{#1}{Notation~\ref*{#1}}}
\newcommand{\Sref}[1]{\hyperref[#1]{\S\ref*{#1}}}

\let\pref=\prettyref

\renewcommand{\Pr}{\mathop{{}\mathbb{P}}}

\newcommand{\E}{\mathop{{}\mathbb{E}}}
\newcommand{\Var}{\mathop{{}\boldsymbol{\mathrm{Var}}}}

\newcommand{\supp}{\mathrm{supp}}
\newcommand{\tr}{\operatorname{tr}}
\newcommand{\diag}{\operatorname{diag}}

\newcommand{\la}{\langle}
\newcommand{\ra}{\rangle}

\newcommand{\ot}{\otimes}

\DeclareMathOperator{\sgn}{sgn}
\DeclareMathOperator{\emp}{emp}
\DeclareMathOperator{\dist}{dist}
\DeclareMathOperator{\obj}{obj}

\newcommand{\eps}{\varepsilon}

\newcommand{\calD}{\mathcal D}
\newcommand{\calE}{\mathcal E}
\newcommand{\calF}{\mathcal F}

\newcommand{\calN}{\mathcal N}

\newcommand{\calP}{\mathcal P}

\newcommand{\R}{\mathbb R}
\newcommand{\C}{\mathbb C}
\newcommand{\N}{\mathbb N}
\newcommand{\Z}{\mathbb Z}

\renewcommand{\le}{\leqslant}
\renewcommand{\leq}{\leqslant}
\renewcommand{\ge}{\geqslant}
\renewcommand{\geq}{\geqslant}

\newcommand{\norm}[1]{\left\lVert{#1}\right\rVert}

\DeclarePairedDelimiter\angles{\langle}{\rangle}

\renewcommand{\emph}{\textit}

\newcommand{\poly}{\mathrm{poly}}

\newcommand{\sop}{_\mathsf{op}}

\newcommand{\opnorm}[1]{\ensuremath{\left\lVert #1 \right\rVert_{\sop}}}

\newcommand{\Tr}{\mathsf{Tr}}

\DeclareMathOperator{\sym}{sym}

\newcommand{\iprod}[1]{\langle#1\rangle}
\newcommand{\Iprod}[1]{\left\langle#1\right\rangle}

\newcommand{\sT}{\mathsf{T}}

\newcommand{\Lip}{\mathsf{Lip}}

\DeclareMathOperator{\sech}{sech}

\DeclareMathOperator{\id}{id}
\DeclareMathOperator{\re}{Re}
\DeclareMathOperator{\im}{Im}
\DeclareMathOperator{\dom}{dom}
\DeclareMathOperator{\Spec}{Spec}
\DeclareMathOperator{\proj}{proj}

\begin{document}

\author{David Jekel}
\address{Department of Mathematics, University of Copenhagen, Denmark}
\email{\href{mailto:daj@math.ku.dk}{daj@math.ku.dk}}

\author{Juspreet Singh Sandhu}
\address{Department of Computer Science, UC Santa Cruz, USA}
\email{\href{mailto:jsinghsa@ucsc.edu}{jsinghsa@ucsc.edu}}

\author{Jonathan Shi}
\address{Chipletics Inc, Redmond WA, USA}
\email{\href{mailto:jshi@cs.cornell.edu}{jshi@cs.cornell.edu}}







\title{Potential Hessian Ascent: The Sherrington-Kirkpatrick Model}

\dedicatory{Dedicated to the memory of Luca Trevisan}

\begin{abstract}
\vspace*{-0.0em}
\setlength{\parindent}{0em}
\setlength{\parskip}{0.5em}
\small
We present the first iterative spectral algorithm to find near-optimal solutions for the Sherrington-Kirkpatrick model, which is a random quadratic objective over the discrete hypercube, resolving a conjecture of Subag~\cite{subag2021following}. The algorithm is a randomized \emph{Hessian ascent} in the solid cube, with the objective modified by subtracting an instance-independent potential function~\cite{chen2018generalized, subag2018free}.

Using tools from free probability theory, we construct an approximate projector into the top eigenspaces of the Hessian, which serves as the covariance matrix for the random increments. With high probability, the iterates' empirical distribution approximates the solution to the primal version of the Auffinger-Chen SDE~\cite{auffinger2015parisi}. 
The per-iterate change in the modified objective is bounded via a Taylor expansion, where the derivatives are controlled through Gaussian concentration bounds and smoothness properties of a semiconcave regularization of the Fenchel-Legendre dual to the Parisi PDE~\cite{auffinger2015properties}.

These results lay the groundwork for (possibly) demonstrating low-degree sum-of-squares certificates over high-entropy step distributions for a relaxed version of the Parisi formula~\cite[Open Question 1.8]{sandhu2024sum}.

\end{abstract}

\maketitle


\vspace{-5mm}
{
  \hypersetup{linkcolor=Red}
  \setcounter{tocdepth}{1}
  \tableofcontents
}

\newpage
\setcounter{page}{1}

\section{Introduction}

\subsection{Overview}

Subag \cite{subag2021following} introduced a marvellously simple algorithm for maximizing the Hamiltonian of a spherical spin glass: start at the origin and repeatedly take small steps along the top eigenspaces of the Hessian of the Hamiltonian until eventually reaching a near-maximizer on the boundary.  This can be understood as an analog of gradient ascent in a situation where the gradient is zero along the path we want to follow.  The algorithm finds a near-optimum configuration whenever the spherical spin glass problem satisfies full replica symmetry breaking (fRSB)\footnote{\,See \pref{ass:sk-frsb} for a definition.}, and the conjectured optimum achievable by efficient algorithms otherwise~\cite{huang2021tight, jones2022random}.  When working with the same Hamiltonian on the discrete hypercube rather than the sphere, efficient algorithms using approximate message passing (AMP) \cite{montanari2021optimization,alaoui2020optimization} have been developed, but it has since remained unclear how to provably extend Subag's conceptually straightforward approach which relied heavily on the highly symmetric geometry of the sphere; see \cite[Pg. 8, Ising Spins]{subag2021following}, \cite[Pg. 5]{montanari2021optimization} and \cite{montanari-lecture}.

To this end, Subag conjectured that it should be possible to follow a path of local maximizers of the objective corrected with a so-called generalized TAP correction under fRSB on the cube \cite[Pg. 8, Ising Spins]{subag2021following}. In this work, we positively resolve this conjecture by developing a Hessian ascent algorithm for the Sherrington-Kirkpatrick (SK) model that works by locally maximizing this TAP-corrected objective.  

The conceptual framework for the algorithm can be termed \emph{potential Hessian ascent} (PHA).  It subtracts a regularizing penalty function from the underlying objective function which encourages coordinates far from a corner to move more and those nearby to move less, and then iteratively optimizes the resulting objective function in the interior of the solution domain.  The PHA algorithm takes small steps, each in a randomized direction that follows the critical points of a modified objective.  Solving the stationary conditions introduced by this formulation yields an iterative spectral algorithm which follows the top eigenspace of a TAP-corrected Hessian.

In the special case of the SK model, the Hamiltonian is given by $H(\sigma) = \angles{\sigma, A \sigma}$ for $\sigma \in \{\pm 1\}^n$, where $A$ is a \emph{real Ginibre random matrix} (a matrix whose entries are i.i.d.\ Gaussian random variables with mean $0$ and variance $1/n$), and our goal is to find a near-maximizer of $H(\sigma)$ over $\{\pm 1\}^n$.  Given an input $n$ and an inverse temperature $\beta$ (related to the desired accuracy $\varepsilon$ of the output), our algorithm performs a Hessian ascent on the modified objective function
\begin{equation} \label{eq: first introduction of objective}
\obj_{\beta,\gamma,n}\left(t, \sigma\right) = \beta\Iprod{\sigma, A \sigma} - V_{\beta,\gamma,n}\left(t, \sigma\right),
\end{equation}
where the potential $V_{\beta,\gamma,n}: \R^n \to \R$ is a combination of per-coordinate terms and a global energy term:
\begin{equation} \label{eq: first introduction of potential}
V_{\beta,\gamma,n}\left(t, \sigma\right) := \sum_{j=1}^n \tilde{\Lambda}_{\beta,\gamma}\left(t, \sigma_j\right) + n r_{\beta,\mu_\beta}\left(t\right).
\end{equation}
Here $\tilde{\Lambda}_{\beta,\gamma}$ and $r_{\beta,\mu_\beta}$ are both defined in terms of the Parisi formula for the SK model \cite{parisi1980sequence}, which we explain in \S \ref{subsec: background}.  Briefly, there is a probability measure $\mu_\beta$ on $[0,1]$ and a function $\Phi_{\beta,\mu_\beta}: [0,1] \times \R \to \R$ solving a certain optimization problem depending on $\beta$.  Our $\tilde{\Lambda}_{\beta,\gamma}$ is the Fenchel--Legendre dual in the spatial variable of $\Phi_{\beta,\mu_\beta}(t,x) + \gamma x^2 / 2$, where the $\gamma$ term is added to prevent $\tilde{\Lambda}_{\beta,\gamma}$ from taking infinite values outside $[-1,1]$.  The energy term $r_{\beta,\mu_\beta}\left(t\right)$ is given by $\beta^2 \int_t^1 s\, \mu_\beta([0,s])\,ds$.

Assuming a strong form of full replica symmetry breaking, the total amount we expect to gain in the function $\angles{\sigma, A\sigma} -\sum_{j=1}^n \tilde{\Lambda}_{\beta,\gamma}\left(t, \sigma_j\right)$ on the time interval $[t,1]$ is described by $nr_{\beta}\left(t,\mu_\beta\right)$, and so by subtracting this term, we arrange that the modified objective function should have approximately a constant value over the entire path given by our algorithm.  The necessary fRSB assumption is widely believed to hold for the SK model at low temperature, and Auffinger, Chen \& Zeng~\cite{auffinger2020sk} have made significant partial progress toward establishing this rigorously.

The Hessian ascent algorithm produces iterates $(\sigma_k)_{k=1}^K$ in $\R^n$ (staying close to $[-1,1]^n$) whose increments $\sigma_{k+1} - \sigma_k$ are given by following the top part of the bulk of the spectrum of the Hessian of the modified objective $\obj_{\beta,\gamma,n}(t,\sigma)$, analogously to Subag's approach on the sphere.  To analyze the PHA algorithm on the cube, we control the Taylor expansion of the objective function at each iterate $\sigma_k$ using tools from random matrix theory and stochastic analysis.  
Conditioned on the values of $A$ and the preceding iterate $\sigma_k$, each increment $\sigma_{k+1} - \sigma_k$ in the algorithm is a Gaussian random vector with a covariance matrix designed to concentrate on the top part of the bulk of the spectrum of the Hessian at $\sigma_k$.  
Since the potential on the cube is a sum of functions of the individual coordinates, it is essential to analyze the averaged behavior of the individual coordinates of $\sigma_k$ (i.e.\ the empirical distribution of $\sigma_k$), which in turn requires precise control over the diagonal entries of the covariance matrix used to generate the increment $\sigma_{k+1} - \sigma_k$.  
Our argument proceeds in several stages, each of which has some independent interest:
\begin{enumerate}
    \item Analysis of the Hessian of the modified objective using random matrix theory~(\pref{thm:david-magic}). Prior work of Gufler, Schertzer and Schmidt \cite{gufler2023concavity} in the RS regime used free probability to study the limiting spectral distribution of the modified Hessian, and the use of free probability to study the diagonal entries of a resolvent was suggested in \cite{adhikari2021dynamical} as well. We characterize various properties about the spectral distribution \emph{as well as} bulk statistics of the diagonal entries, in the lower-temperature fRSB regime.
    \item A proof that the empirical distribution of the PHA algorithm's iterates converge to the (primal) Auffinger-Chen SDE~(\pref{thm:convergence-to-SDE})\,.
    \item Concentration bounds for the derivatives of an infimum-convolved Fenchel-Legendre dual of the solution to the Parisi PDE~(\S\ref{sec:energy-analysis}). These bounds are combined with the previous two results via a Taylor-expansion based approach to control the fluctuations of the modified objective function~(\pref{thm:taylor-bound})\,.
\end{enumerate} 

\subsection{Background} \label{subsec: background}

To set the stage for a more detailed statement of our results, we now recall background on the Sherrington-Kirkpatrick model and previous work on finding the associated maximum and maximizers.

\subsubsection{The Sherrington-Kirkpatrick model}

The Sherrington-Kirkpatrick model~\cite{sherrington1975solvable} was introduced as a ``mean-field" simplification of Ising spin glass models on $d$-dimensional lattices ($\Z^d$).  The configuration space is described by assigning $\pm 1$ to each vertex of a complete graph on $n$ vertices.  The graph has Gaussian (signed) weights assigned to each edge.  Thus, the configuration space is $\{\pm 1\}^n$, and the Hamiltonian is a random quadratic polynomial on the hypercube $\{\pm 1\}^n$. Specifically,
\begin{equation}\label{def:sk-model}
    H_n(\sigma) = \sum_{i,j} A_{i,j}\sigma_i\sigma_j\,,    
\end{equation}
where $A$ is a Gaussian random matrix with i.i.d.\ entries.  Moreover, we assume that $A$ is normalized so that $A_{i,j} \overset{i.i.d.}{\sim}\calN(0,1/n)$ for every $i, j \in [n]$.  Under this normalization $A_{\sym} := (A + A^{\sT})/2$ is $1/\sqrt{2}$ times a GOE random matrix, and hence the spectral distribution converges almost surely as $n \to \infty$ to the Wigner semicircle law (see \S \ref{subsec: free probability background} below).

One of the main problems of interest is to find the energy of the ground state; here the ground state would be the maximizer of $H(\sigma)$ on $\{\pm 1\}^n$ and its energy is the maximum value.  As
\[
H_n(\sigma) = \sum_{i,j} A_{i,j} - 2 \sum_{\sigma_i\sigma_j = -1} A_{i,j},
\]
maximizing $H(\sigma)$ is equivalent to minimizing $\sum_{\sigma_i\sigma_j = -1} A_{i,j}$, i.e.\ minimizing the weighted ``cut'' between the vertex sets $\{i: \sigma_i = 1\}$ and $\{j: \sigma_j = -1\}$.

In particular, we want to analyze the behavior of the ground state as the dimension $n$ tends to infinity.  Because of standard concentration of measure results for Lipschitz functions of Gaussian variables, the value of the maximum is close to its expectation with high probability~\cite[Theorem 2.1.1]{adler2007random}, and hence the asymptotic behavior is described by the ground state energy density given by
\begin{align}
    \lim_{n \to \infty} \frac{1}{n} \E_{A}\left[\max_{\sigma \in \{\pm 1\}^n}H_n(\sigma)\right] := \lim_{n \to \infty} \frac{1}{n} \E_A\left[\max_{\sigma \in \{\pm 1\}^n}  \langle \sigma, A\sigma\rangle\right].
\end{align}

\subsubsection{Parisi formula and Auffinger-Chen representation}

The free energy of the Sherrington-Kirkpatrick model is given by a variational principle first proposed by Parisi~\cite{parisi1980sequence} and then rigorously proved nearly twenty-five years later by Guerra~\cite{guerra2002thermodynamic, guerra2003broken}, Talagrand~\cite{talagrand2006parisi} and Panchenko~\cite{panchenko2013parisi,panchenko2014parisi}. The variational principle transforms the static optimization problem of the expression of the free energy into a variational optimization problem over the space of probability distributions on $[0,1]$, and simultaneously minimizes a free-entropy term that comes from solving the so-called Parisi PDE and a variance-like quantity.
\begin{theorem}[Parisi Variational Principle,~{\cite{parisi1980sequence, talagrand2006parisi}}]\label{thm:parisi-formula} \mbox{} \\
For $n \in \N$ and an inverse temperature $\beta = \frac{1}{T} > 0$, let
\[
\calF_{n, \beta} := \log \sum_{\sigma \in \{\pm 1\}^n} e^{\beta H_n(\sigma)}
\]
be the free energy of the $n$-dimensional Sherrington-Kirkpatrick Hamiltonian $H_n$.

For each probability measure $\mu \in \mathcal{P}([0,1])$, let $F_\mu(t) = \mu([0,t])$ be its cumulative distribution function, and let $\Phi_{\beta,\mu}$ be the solution to the parabolic PDE
\begin{equation} \label{eq: Parisi PDE}
     \partial_t \Phi(t,x) = -\beta^2\left(\partial_{x,x}\Phi(t, x)  + F_{\mu}(t)(\partial_x\Phi(t,x))^2\right)\, ,
    \end{equation}
with the terminal condition
\begin{equation}
    \Phi(1, x) = \log(2\cosh(x))\, .
\end{equation}
Let $\mathcal{P}_\beta$ be the optimal free entropy
\begin{equation} \label{eq:parisi-formula-optimization}
\calP_\beta := \inf_{\mu \in \mathcal{P}[0, 1]} \left[ \Phi_{\beta,\mu}(0, 0) - \beta^2\int_0^1tF_{\mu}(t)dt \right].
\end{equation}
Then
\begin{equation}\label{eq:parisi-formula}
    \lim_{n \to \infty}\frac{1}{n}\calF_{n, \beta} \overset{a.s.}{=} \mathcal{P}_\beta.
\end{equation}
\end{theorem}
The parabolic PDE \eqref{eq: Parisi PDE} is solved ``backwards" in time, and if $\mu$ is atomic then the solution can be computed explicitly using the Hopf-Cole transformations~\cite{auffinger2015parisi}; see \pref{sec:rpc}.  Crucially, note that the free-entropy term $\Phi_{\beta,\mu}(0,0)$ that comes from the PDE \emph{itself} depends on the probability distribution $\mu$ to be optimized over.

\pref{thm:parisi-formula} in particular allows one to evaluate the ground state energy of the Sherrington-Kirkpatrick model as the zero temperature limit of the underlying Parisi formula~\cite[(1.10)]{panchenko2013sherrington}.  We also mention that Auffinger and Chen have given a version of the Parisi PDE at $\beta = \infty$, although it has less regularity than the finite $\beta$ versions \cite{auffinger2017parisi}.

\begin{corollary}[Ground State Energy of the Sherrington-Kirkpatrick Model~{\cite[(1.10)]{panchenko2013sherrington}}]\label{cor:ground-state-sk-parisi}
    \begin{equation}
        \lim_{n \to \infty}\frac{1}{n}\E_A\left[\max_{\sigma \in \{\pm 1\}^n}H_n(\sigma)\right] = \lim_{\beta \to \infty}\frac{1}{\beta}\,\calP_\beta\, .
    \end{equation}
\end{corollary}

The study of the ground state energy as well as the configurations that realize the optimum depends crucially on the regularity of the solution $\Phi_\mu$ and the measure $\mu_\beta$ itself.  To obtain a unique minimizer $\mu$ in the Parisi formula~\cite{auffinger2015parisi}, Auffinger and Chen showed that the associated functional is convex, by reformulating the Parisi formula as a stochastic control problem.

\begin{theorem}[Auffinger-Chen Principle, {\cite{auffinger2015parisi, montanari2021optimization}}]\label{thm:auffinger-chen} \mbox{} \\
    For a fixed $\beta$ and $\mu$, the free entropy term $\Phi_\mu(t, x)$ at  $t = 0$ can be equivalently rewritten as a maximization over a set $\calD$ of Brownian-motion-adapted stochastic processes that are stopped when the process exits $[-1,1]$,
    \begin{equation}\label{eq:auffinger-chen}
        \Phi_{\beta,\mu}(0, x) = \max_{u \in \calD}\left(\E\left[\Phi_{\beta,\mu}\left(1, x + 2\beta^2\int_0^1F_{\mu}(s)u_s\,ds + \sqrt{2}\beta\int_0^1dW_s\right) \right] - \beta^2\int_0^1F_{\mu}(s)\E[u_s^2]ds \right)\, ,
    \end{equation}
    where $\{W_s\}$ is standard Brownian motion, $\calF_t$ is the filtration generated $\{W_t\}$, and 
    \[
        \calD := \left\{(u_t)_{0 \leq t \leq 1}\, |\, u_t\text{ is progressively measurable w.r.t. }\calF_t,\, |u_t| \leq 1\right\}\, .
    \]
    The maximizer of~\eqref{eq:auffinger-chen} is \emph{unique} and is given by,
    \begin{equation}
        u_s = \partial_x\Phi_{\beta,\mu}(s, x + X_s)\, ,
    \end{equation}
    and $X_s$ is the strong solution following the Stochastic Differential Equation,
    \begin{equation}\label{eq:auffinger-chen-sde}
        dX_s = 2\beta^2F_{\mu}(s)\partial_x\Phi_{\beta,\mu}(s,x + X_s)ds + \sqrt{2}\beta dW_s,\, X_0 = 0\, .
    \end{equation}
    with initial condition $X_0 = 0$.
\end{theorem}

\subsubsection{Algorithms to find the maximizers}

The Auffinger-Chen representation is critical not only for showing regularity of the Parisi optimal measure $\mu_\beta$, but also for describing the maximizers of $H_n(\sigma)$.  Montanari used the Auffinger-Chen SDE to design an approximate message passing (AMP) algorithm for finding near maximizers~\cite[Incremental AMP]{montanari2021optimization}. The successive iterates of the AMP algorithm for every coordinate $i \in [n]$ are given by multiplying the current value of $i$ with a non-linear function evaluated at a point chosen by discretizing the increments of~\eqref{eq:auffinger-chen-sde} and propagating the updated value via the Gaussian matrix $A$ that specifies the input.  Various results using the AMP framework have given efficient algorithms to optimize average-case problems such as mean-field spin glasses on the hypercube~\cite{montanari2021optimization, alaoui2020optimization} and sparse random Max-CSPs~\cite{alaoui2021local, chen2023local}.  The AMP framework has also been applied with great success to problems in estimation~\cite{bayati2011lasso,montanariestimation2021,celentano2022fundamental}, sampling~\cite{el2022sampling, huang2024sampling} and compressed sensing~\cite{donoho2009message}.  In the case of spin-glasses on the sphere, Subag gave a conceptually simpler algorithm called Hessian ascent~\cite{subag2021following} which we describe further in \S\ref{sec:concept-pha} as it forms the conceptual basis for our work.

For our algorithm to attain the actual maximum in the Parisi formula, one requires certain assumptions on the measure $\mu_\beta$, namely \emph{full Replica Symmetry Breaking (fRSB)} which means that the measure $\mu_\beta$ has full support asymptotically as $\beta \to \infty$.  Auffinger and Chen \cite{auffinger2015properties} showed that if $\mu$ has full support, then $\Phi_{\beta,\mu_\beta}$ is a sufficiently smooth function of $(t,x)$ for our purposes.  The fRSB condition is widely believed to hold in the SK model for sufficiently large $\beta$, but the state of the art for rigorous results on large $\beta$ is that $\mu$ has infinite support \cite{auffinger2020sk}.

Without assuming fRSB, efficient algorithms are expected to not perform better than an explicit threshold given by a relaxed Parisi formula introduced by Huang and Sellke~\cite[(1.9)]{huang2021tight} (which agrees with the original Parisi formula in the fRSB case).  The AMP algorithms produce solutions that approximately achieve this relaxed optimum in general, while there are many known hardness results for going beyond the relaxed optimum.  Variants of the so-called overlap-gap property (OGP), introduced by Gamarnik and Sudan~\cite{gamarnik2014limits}, obstruct various natural families of algorithms, including local classical \& quantum algorithms~\cite{gamarnik2014limits, chen2019suboptimality, chou2022limitations}, low-degree polynomials~\cite{gamarnik2020low} and AMP~\cite{gamarnikmessage}.
The unifying implications of this work are found in a result by Huang and Sellke~\cite{huang2021tight} via the introduction of the most general variant of the OGP, which holds for mean-field spin glasses on the hypercube with even degree $\ge 4$. 
These obstructions are transferred to the setting of sparse random Max-CSPs due to a result of Jones, Marwaha, Sandhu and Shi~\cite{jones2022random}. 

\subsection{Main Results}

In this section we state the main result of this paper, which is a polynomial time approximation scheme (PTAS) for the SK model that extends the algorithmic program of Subag~\cite{subag2021following} to the cube.  Note that here ``polynomial time'' means polynomial in $n$ with constants that can depend on $\beta$ or $\varepsilon$.  Like Subag's work, the algorithm relies on spectral analysis and does not use the AMP framework.

\subsubsection{Concept: Potential Hessian Ascent}\label{sec:concept-pha}

\newcommand{\tD}{\tilde{D}}
\newcommand{\tH}{\tilde{H}}
\newcommand{\tZ}{\tilde{Z}}

Suppose that we are trying to maximize some real-valued objective function $H$ on some domain $D$.  In particular, for some given precision $\eps$, we want to output in polynomial time some $\sigma \in D$ such that
\[
H(\sigma) \geq (1 - \eps) \sup_{\sigma' \in D} H(\sigma').
\]
Akin to the celebrated work of Subag~\cite{subag2021following} the PHA algorithm considers an extension $\tH: \tD \to \R$ of the original objective function $H$---which was originally defined on the domain $D$.  For instance, if $D$ is the cube $\{\pm 1\}^n$, then $\tD$ will be the $[-1,1]^n$, and more generally one could consider a bounded set $D$ and let $\tD$ be its convex hull.

We want to choose $\tH$ so that its maximum value is equal to that of $H$, and this maximum is attained along an entire path from a given point $\sigma_0 \in \tD$ to a maximizer of $H$ in $D$.  We then define iterates inductively that approximately follow such a path.  Since the gradient of $\tH$ is zero along the path, maintaining the optimum value requires us to move only in directions in the kernel of $\nabla^2 \tH$.  We therefore define iterates inductively starting from the given point $\sigma_0$ by $\sigma_{k+1} = \sigma_k + Q_k Z_k$, where $Z_k$ is a standard Gaussian vector and $Q_k^2$ is a certain covariance matrix, which should approximately project into the kernel of $\nabla^2 \tH$.  While Subag used matrix power iterations to approximately choose an eigenvector with largest eigenvalue, we can straightforwardly generalize this to a random vector from the top part of the spectrum~\cite[Algorithm 2]{sandhu2024sum}.

In the case of the SK model, $D = \{\pm 1\}^n$ and $\tD = [-1,1]^n$.  We will choose the modified objective function $\tilde{H} = \obj$ given by \eqref{eq: first introduction of objective} and \eqref{eq: first introduction of potential}.  The potential $V_{\beta,\gamma,n}$ in \eqref{eq: first introduction of potential} is a function $[0,1] \times [-1,1]^n \to \R$ given by $\sum_j \tilde{\Lambda}_{\beta,\gamma}(t, \sigma_j) + n r_{\beta,\mu_\beta}(t)$,
where $\tilde{\Lambda}_{\beta,\gamma}: [0,1] \times \R \to \R \cup \{\pm \infty\}$ is a regularized version of the Fenchel-Legendre conjugate in space of the solution $\Phi_{\beta,\mu_\beta}$ to the Parisi PDE, and $r_{\beta,\mu_\beta}$ is a certain correction depending only on time.  Note that $V_{\beta,\gamma,n}$ is independent of the particular instance of $A$, and thus only requires access to the Parisi solution $\Phi_{\beta,\mu_\beta}$ and the measure $\mu_\beta$, which is a $2$-dimensional rather than an $n$-dimensional problem.  Our goal is to show that Hessian ascent succeeds in finding a near optimizer using this modified objective function.

Our modified objective function is motivated by the concept of generalized TAP free energy~\cite{chen2018generalized, subag2018free} in Subag's algorithm.  We imagine that we start with a certain total budget of energy, and at each step, we balance increasing the objective $\la \sigma,A\sigma \ra$ with spending a limited amount of energy (i.e. not increasing $V_\beta$ too much).  Subag's algorithm maximized the TAP free energy on the ball $\tD$ by following the top eigenvector of the Hessian of this energy.  Furthermore, the Parisi-like formula on the sphere~\cite{chen2017parisi,subag2018free} simplifies in the fRSB regime to yield an energy that matches the contribution from the Hessian term.  Similarly, our PHA algorithm for the SK model follows the top-eigenspace of the Hessian of the generalized TAP free energy of the SK model on the cube. The top-eigenvalue of this Hessian corresponds to a term that is, roughly, the rate at which an ``entropy'' term changes in a direction orthogonal to the current iterate. This term ends up matching the energy given by the Parisi formula in the fRSB regime, giving the final value achieved by the algorithm. 

\subsubsection{Main results on the algorithm}

The PHA algorithm for the SK model has the following inputs:
\begin{enumerate}
    \item Oracle access to the solution to the Parisi PDE $\Phi_{\beta,\mu_\beta}(t,x)$ and its derivatives,
    \item Oracle access to the cumulative distribution function $F_{\mu_\beta}: [0,1] \to [0,1]$ of the Parisi measure $\mu_\beta$.
    \item An instance of the real Ginibre matrix $A$ that specifies the Sherrington--Kirkpatrick Hamiltonian.
\end{enumerate}
We also make the following assumption.

\begin{assumption}[The SK model is fRSB]\label{ass:sk-frsb}\mbox{} \\
The unique minimizer $\mu_\beta$ in the Parisi formula \pref{thm:parisi-formula} has support equal to $[0,q_\beta^*]$ where $q^*_\beta \to 1$ as $\beta \to \infty$.
\end{assumption}

Recall from \cite{auffinger2015properties}, if the support of $\mu_{\beta}$ contains $[0,q_\beta^*]$, then $\Phi_{\beta,\mu_\beta}$ will have sufficient smoothness as a function of $(t,x)$.

We provide a simplified and informal version of the PHA algorithm below~(\pref{alg:hessian-ascent-informal}); see \pref{alg:hessian-ascent} for a complete description.  The main result is that under~\pref{ass:sk-frsb} and the choice of potential function $V_{\beta,\gamma,n}(t,\sigma)$ given above, the PHA algorithm is a PTAS for the SK model.  More precisely, for $\varepsilon > 0$ and $\beta = 10 / \varepsilon$ and $\delta \in (0,1/19]$, with high probability, the algorithm achieves an error of $O(\varepsilon) + o_n(1)$ in runtime $\widetilde O\left(n^{2+4\delta} \exp(\text{constant}/\varepsilon^2)) \right)$, where $\widetilde O$ hides factors that are polylog in $n$.  Below, the parameter $\eta(\eps) = \mathsf{exp}\left(-\mathsf{poly}\left(\frac{1}{\eps}\right)\right)$ controls the step-size of the algorithm, and the number of steps is $K = O\left(\frac{q^*}{\eta}\right) = O\left(\mathsf{exp}\left(\mathsf{poly}\left(\frac{1}{\eps}\right)\right)\right)$.

\begin{theorem}[PHA is an asymptotic PTAS for the SK model]\label{thm:main-pha}\mbox{} \\
    Let $\eps > 0$ be sufficiently small and $\delta \in \left(0, \frac{1}{19}\right]$. Then, fix $\beta = \frac{10}{\eps}$, $\eta = e^{-C\beta^2}$ where $C$ is a sufficiently large absolute constant, and $\gamma = \eta^{1/8}$ and let $n \ge \eta^{-90}$. \\
    Sample an instance $\{A\}_{i,j \in [n]\times [n]}$ of the real Ginibre matrix used in the Hamiltonian as in~\pref{def:sk-model}. \\
    Then, under~\pref{ass:sk-frsb}, the PHA algorithm~(\pref{alg:hessian-ascent}) takes as input $A$ and $\eps$, as well as oracle access to $\Phi$, its derivatives and $\mu_{\beta}$, and outputs a configuration $\sigma^* \in \{\pm 1\}^n$ such that
    \[
\frac1n H_n(\sigma^*)
\ge
\frac{\mathcal P_\beta}{\beta}
-\frac{\eps}{5}-O(\eps^2)-O_{\eps}(n^{-\alpha}),
\]
    where $\alpha = \min\left(\frac{1}{24}, \frac{1}{4}\delta\right)$, with probability $\ge 1 - n\eta^{-2}e^{-\Omega(n^{2/9 - 4\delta})}$ (in the matrix $A$ and the randomness of the algorithm itself). Furthermore, the algorithm runs in time $\widetilde O\left(n^{2 + 4\delta}\exp\left(O(1/\eps^{2})\right)\right)$.
    
    Consequently, for every fixed $\eps>0$, with suitable parameters depending
only on $\eps$, PHA runs in time polynomial in $n$ and outputs
$\sigma^*\in\{\pm1\}^n$ satisfying
\[
H_n(\sigma^*)\ge
\left(1-\frac{\eps}{2}-O(\eps^2)-o_n(1)\right)
\max_{\sigma\in\{\pm1\}^n}H_n(\sigma)
\]
with probability $1-\exp(-n^{\Omega(1)})$ as $n\to\infty$.

\end{theorem}
This theorem is proven in \pref{cor:energy-lower-bound} and \pref{prop:time-complexity}.

\begin{algorithm}
\caption{Potential Hessian Ascent (Informal)}\label{alg:hessian-ascent-informal}
\begin{algorithmic}[1]
\Procedure{Potential\_Hessian\_Ascent (Informal) }{$A, \eps, V,q^*$}
\State $\beta \gets 10/\eps$
\State $\eta \gets O\left(e^{-\beta^{2}}\right)$
\State $K \gets \lceil q^*(\beta)/\eta\rceil$
\State $\obj(t,\sigma) \gets \beta\sigma^{\sT}A\sigma- V_{\beta}(t,\sigma)$
\State $\sigma_0 \gets (0,\dots, 0)$
\For{$k\in \{0, \dots, K-1 \}$}
\State $Q_k \gets$ Smoothed projector to top part of spectrum of $\nabla^2\obj(k\eta,\sigma_k)$
\State $u_k \sim \calN(0, Q_k^2)$
\State $\sigma_{k+1} \gets \sigma_k + \eta^{1/2}u_k$
\EndFor
\State \textbf{return} $\sigma_K$ rounded to $\{\pm 1\}^n$
\EndProcedure
\end{algorithmic}
\end{algorithm}

\subsubsection{Ingredients in the analysis}

The proof of convergence of our algorithm proceeds in three stages.

\begin{enumerate}
    \item \textbf{Spectral analysis of the modified Hessian:}  We use free probability theory to locate the top part of the spectrum of $2 \beta A_{\sym} - D$ where $\beta$ is the inverse temperature, $A_{\sym}$ is the symmetrization of the Ginibre random matrix, and $D$ is a diagonal matrix which is the Hessian of $\sum_j \Lambda\left(t, \sigma_j\right)$.  
    We construct a positive operator $Q$ whose weight is concentrated on the top part of the spectrum, and then we show that the diagonal of $Q^2$ is close to $D^{-2}$ using non-commutative conditional expectation formulas for free sums together with high-dimensional concentration of measure for the Gaussian matrix (see \S\ref{sec:free-prob-for-hessian}).  
    The proof works for arbitrary nonnegative diagonal matrices and the same tools could be applied to a variety of combinations of a Gaussian random matrix and another matrix depending on the choice of potential.
    \item \textbf{Convergence of empirical distribution of $\sigma_k$ to Auffinger-Chen SDE:} Using this control over the diagonal entries of $Q^2$ and concentration of measure for the Gaussian vector, we show that on average the $j$th entry of $\sigma_{k+1} - \sigma_k$ behaves like an independent Gaussian of variance $v(t, \sigma_{k,j})^2$ where $v(t,y) = \sqrt{2}\beta / \partial_{y,y} \Lambda(t,y)$.  
    We deduce that with high probability the empirical distribution of $\sigma_k$ is well approximated by the distribution of $Y_t$, where $Y_t$ is the solution to the SDE $dY_t = v(t,Y_t)\,dW_t$ driven by a Brownian motion $W_t$ (see \S\ref{sec:convergence-to-ac}).  
    This is the SDE studied by Auffinger and Chen that closely relates to Parisi PDE for $\Phi$, and thus our argument gives some hint about how PHA can generate solutions to an SDE appropriate to the geometry of the underlying region.
    \item \textbf{Taylor expansion and concentration estimates:}  Finally, we analyze the change in the modified objective function at each iteration $\sigma_k$ using a mixture of Taylor expansions and concentration estimates.
    To understand the derivatives of the objective function, we express Parisi's PDE for $\Phi$ (in the dual coordinate space) in terms of the inf-convolved Fenchel-Legendre conjugate $\Lambda$ (in the primal coordinate space) and prove various regularity properties about it (see \S\ref{sec:primal-parisi-pde}).
    The terms in the Taylor expansion can then be controlled using the ingredients from the first two steps as well as concentration bounds for norms of minorly correlated Gaussians (see \S\ref{sec:energy-analysis}).
    This argument provides a template to analyze modified objective functions under different potentials: to do so, systematically combine spectral estimates about the covariance matrix used by the PHA algorithm and moment estimates for the empirical distribution of its iterates with the regularity properties of the potential function and an appropriate set of concentration inequalities.  
    We remark that the Taylor expansion procedure has interest beyond simply showing convergence since this methodology would be critical to obtain sum-of-squares proofs of near-optimality under high-entropy step distributions (see \S\ref{sec:parisi-formula-and-sos}).
\end{enumerate}
An intriguing feature of the proof is that we effectively compartmentalize the randomness of the Gaussian matrix $A$ in the objective function and the randomness of the Gaussian vectors used in the algorithm.  Indeed, \S \ref{sec:free-prob-for-hessian} guarantees properties of the random matrix $A$ in relation to \emph{all} diagonal matrices uniformly, and the argument of \S \ref{sec:convergence-to-ac} and \S \ref{sec:energy-analysis} would apply to any matrix $A$ with these properties; this leaves open the possibility to generalize the random matrix used as input for the optimization problem.

\subsubsection{Significance}\label{sec:significance}

Our algorithm implements Subag's Hessian ascent idea \cite{subag2021following} in the case of the cube rather than the sphere.  It is consequently the first spectral algorithm for analyzing spin glasses on $\{\pm 1\}^n$ that is a PTAS.  The workings of this algorithm are very different from the AMP algorithm of Montanari~\cite{montanari2021optimization}.  
Both the PHA algorithm and the AMP algorithm require access to the function $\Phi_{\beta,\mu_\beta}(t,x): [0,1] \times \R \to \R$ which solves the Parisi PDE~(see \pref{thm:parisi-formula}), in addition to its derivatives and the Parisi order parameter $F_{\mu_\beta}: [0,1] \to [0,1]$. 
These are reasonable assumptions, as the $\Phi_{\beta,\mu_\beta}$ does not depend on $n$, and it is known how to solve the Parisi formula efficiently~\cite{auffinger2015parisi, jagannath2016dynamic} and the interested reader may consult~\cite{montanari2021optimization} for details.

While AMP is already known to optimize the SK model under fRSB, the contribution of PHA is to build from the basic concepts of quadratic optimization and spectral analysis of random matrices, and to demonstrate convergence to the Auffinger-Chen SDE---and therefore near-optimal solutions---as a consequence of these foundations.
The resulting algorithm is the first to directly leverage a geometric understanding of the solution landscape of the SK model, and in this way makes progress toward de-mystifying the algorithmic implications of the Parisi formula.

Our algorithm has the same time complexity as AMP up to a factor of $n^{o(1)}$. We believe that the dependence on $\eps$ needs to be of leading order that is $\exp\left(\poly(\frac{1}{\eps})\right)$ in our algorithm, but that it may be possible to bring the dependence down with use of better-than-worst-case Lipschitz estimates implied by the regularity bounds developed in \S \ref{sec:primal-parisi-pde}.

In future work, we hope that the PHA framework may be applied to optimize more general random polynomials on the cube, and other domains as well.  In particular, we propose as a problem for future work to investigate the applications to sums-of-squares (SoS) certificates for the optimizers.  It is well established that no low-degree SoS certificates over the standard SoS hierarchy exist for the Parisi formula~\cite{bhattiprolu2016sum,hopkins2017power,ghosh2020sum}. This is largely due to the inability of low-degree SoS to capture concentration-of-measure~\cite{barak2016proofs}. However, it has recently been proved that there \emph{do} exist low-degree SoS certificates that points generated by high entropy step (HES) distributions cannot have value larger than (some constant times) the relaxed value of the Parisi formula on the sphere~\cite{sandhu2024sum}.  Our work could allow analogous results to be proved on the cube in the fRSB regime as proposed in~\cite[Open Question 1.8]{sandhu2024sum}. For further discussion, see \S \ref{sec:discussion-open-problems}.

\subsection{Notation} \label{subsec: notation}

Here we briefly overview some notations and conventions that will be used throughout the paper, including in particular the notation for matrices and vectors.  Note that we will often write $[n]$ for the index set $\{1,\dots,n\}$ for $n \in \mathbb{N}$.

\begin{notation}[Matrices]~
\begin{itemize}
    \item $M_n(\mathbb{R})$ and $M_n(\mathbb{C})$ denote the $n \times n$ real and complex matrices respectively.
    \item We use $*$ for the adjoint on $M_n(\mathbb{C})$ and $\sT$ for the transpose on $M_n(\mathbb{R})$ or $M_n(\mathbb{C})$.
    \item We write $\Tr_n$ for the usual trace $\Tr_n(A) = \sum_{j=1}^n A_{j,j}$ and $\tr_n$ for the normalized trace $\tr_n(A) = (1/n) \Tr_n(A)$.
    \item For $\lambda \in \mathbb{C}$, we denote the multiple $\lambda I$ of the identity matrix simply by $\lambda$ (for instance, $\lambda + A$ denotes $\lambda I + A$).
\end{itemize}
\end{notation}

We generally follow the convention from free probability (see \S \ref{subsec: free probability background}) that traces and Schatten norms of operators are normalized so that they are equal to the moments of the spectral distribution of the operator; this convention makes it easy to compare the operators with the idealized limiting objects.

\begin{notation}[Norms for matrices]
The normalized Schatten norms are given for $p \in [1,\infty)$ by
\[
\norm{A}_p = (\tr_n((A^*A)^{p/2}))^{1/p}.
\]
Hence, if $A$ has singular values $a_1$, \dots, $a_n$, then
\[
\norm{A}_p = \left( \frac{1}{n}\sum_{j=1}^n a_j^p \right)^{1/p}.
\]
The Schatten-$\infty$ norm is by definition equal to the operator norm, which we denote simply by $\norm{A} := \norm{A}_{\infty} := \max_{i\in[n]} |a_i|$.
\end{notation}

By contrast, the $\ell^p$ norms for our vectors are \emph{not} normalized. To visually distinguish these from the operator norms, we use single bars $|\cdot|$ to denote vector norms.

\begin{notation}[Norms for vectors]
For an $n$-dimensional vector $v$ with components $v_1, \dots, v_n$, we have $|v|_p := \sqrt[p]{\sum_{i \in [n]} |v_i|^p}$ for $p \in [1,\infty)$, and $|v|_\infty = \max_i |v_i|$.
\end{notation}

\begin{notation}[Operations with diagonal matrices]
Given a vector $v = (v_1,\dots,v_n)$, we denote by $\diag(v)$ or $\diag(v_j)$ the diagonal matrix with diagonal entries $v_1$, \dots, $v_n$.  Meanwhile, we denote by $\mathcal{D}_n$ the subalgebra of diagonal matrices in $M_n(\mathbb{C})$, and $E_{\mathcal{D}_n}: M_n(\mathbb{C}) \to \mathcal{D}_n$ denotes the projection onto the diagonal matrices, which zeroes out the off-diagonal entries of the input matrix (this is understood as a non-commutative conditional expectation in \S \ref{subsec: free probability background}).
\end{notation}

\begin{definition} \label{def: real Ginibre}
A \emph{(normalized) real Ginibre matrix} is a random matrix $A$ in $M_n(\mathbb{R})$ such that the entries of $A$ are i.i.d.\ Gaussian random variables with mean $0$ and variance $1/n$.
\end{definition}

\begin{definition}
A \emph{(normalized) Gaussian orthogonal ensemble (GOE)} matrix is a random self-adjoint matrix $A$ in $M_n(\mathbb{R})$ such that the entries $\{A_{i,j}: i \leq j\}$ are independent, the entries $A_{i,i}$ are Gaussian with mean zero and variance $2/n$, and the entries $A_{i,j}$ for $i \neq j$ are normal with mean zero and variance $1/n$.
\end{definition}

\begin{lemma} \label{lem: symmetric part of Ginibre}
If $A$ is a real Ginibre random matrix, then $(A + A^{\sT})/\sqrt{2}$ is a GOE random matrix.
\end{lemma}

Finally, while functions such $V_{\beta,\gamma,n}$ depend on many different parameters, we will sometimes suppress some of the dependencies in the course of the technical arguments, to prevent the equations from becoming unwieldy.  Indeed, $\beta$, $\gamma$, and $n$ will be fixed for most of our analysis, especially in \S \ref{sec:free-prob-for-hessian} and \S \ref{sec:convergence-to-ac}.  The dependence on these parameters will be included only when it is relevant to the argument, such as in \S \ref{sec:energy-analysis}.

\section{Primal Parisi PDE}\label{sec:primal-parisi-pde}

The PHA algorithm considers an \emph{extension} of the original objective function into the convex hull of the original domain.  The main component in the potential is the Fenchel-Legendre dual, or convex conjugate, of the Parisi solution $\Phi = \Phi_{\beta,\mu_\beta}$, given by
\begin{equation} \label{eq: Lambda def}
\Lambda_\beta(t,y) = \sup_{x \in \R} \left( xy - \Phi_{\beta,\mu_\beta}(t,x) \right),
\end{equation}
which turns out to be continuous on $[-1,1]$ and infinite outside this interval.  Note that this sign convention for $\Lambda$ is the opposite of the one in \cite{chen2018generalized}.  Moreover, as mentioned in \S \ref{subsec: notation}, we suppress the dependence on $\beta$ in the notation.  The point $x$ where the supremum is achieved satisfies $y = \partial_x \Phi(t,x)$ and $x = \partial_y \Lambda(t,y)$, so that $\partial_x \Phi(t,\cdot)$ and $\partial_y \Lambda(t,\cdot)$ describe a change of coordinates between $(-1,1)$ and $\R$, and hence also from the cube $(-1,1)^n$ to $\R^n$.  Montanari's algorithm works with the dual coordinate $x$ (in the same way as the Parisi formula itself) and expresses the near optimizer as $\partial_x \Phi(q_\beta^*,x)$, but we want to work directly in the primal space with the coordinate $y$.

Conceptually, $\Lambda_\beta$ represents the entropy of the current point in the cube, and at each step a certain amount of entropy will be spent as we move closer to a corner.  In fact, $\Lambda_\beta(1,y)$ is equal to the Shannon entropy function of a Bernoulli distribution with probabilities $(1 + y)/2$ and $(1 - y)/2$ (compare~\pref{prop:rounding-energy}).  Our modified objective function can thus be understood as the original objective $\angles{ \sigma, A \sigma}$ minus the entropy budget that we have to spend (given by $\sum_j \Lambda_\beta(t,\sigma_j)$) and the energy gain that we hope to achieve (given by $r_{\beta,\mu_\beta}$), which will balance out to zero if the algorithm achieves its goal.  This idea is inspired by the Gibbs variational principle that the free energy is the expected internal energy plus the Shannon entropy times temperature, which has been in the background of the Parisi formula from the beginning.

Though working with the point $y$ in the original domain, rather than the point $x$ in the dual domain, is intuitive, significant challenges arise from $\partial_y \Lambda_\beta(t,\cdot)$ blowing up at $\pm 1$.  Therefore, we use a regularized version of $\Lambda_\beta$; for $\gamma > 0$, let
\begin{align}
\tilde{\Lambda}_{\beta,\gamma}(t,y) &= \sup_{x \in \R} \left( xy - \Phi_{\beta,\mu_\beta}(t,x) - \frac{\gamma}{2} x^2 \right) \label{eq: Lambda gamma def} \\
&= \inf_{y' \in \R} \left( \Lambda_\beta(t,y') + \frac{1}{2\gamma} (y' - y)^2 \right). \label{eq: Lambda gamma def 2}
\end{align}
This function $\tilde{\Lambda}_{\beta,\gamma}$ will be smooth on all of $\R^n$ with its second derivative bounded by $1/\gamma$.  Thus, $\tilde{\Lambda}_{\beta,\gamma}$ serves to modify our objective in the primal coordinate space, while also having similar smoothness properties as $\Phi$.  While using the function $\Lambda_\beta$ would theoretically prevent iterates from leaving the cube, this would require the step size to be made sufficiently small when the point is near the boundary; using the function $\tilde{\Lambda}_{\beta,\gamma}$ will allow the coordinates of the iterates to leave the cube, but with high probability they will not be too far away, as we will see later on.

The goal of this section is to establish analytic properties for $\Phi = \Phi_{\beta,\mu_\beta}$ and $\tilde{\Lambda} = \tilde{\Lambda}_{\beta,\gamma}$ as groundwork for our algorithm.  As noted in \S \ref{subsec: notation}, the dependence on $\beta$ and $\gamma$ will sometimes be suppressed in the notation.  Moreover, many of the properties hold for a general probability measure $\mu$ and not only the optimizer $\mu_\beta$.  Thus, for the bulk of the section, we will study the functions $\Phi$, $\Lambda$, and $\tilde{\Lambda}_\gamma$ associated to a fixed $\mu$ and $\beta$.  In particular, we will:
\begin{itemize}
    \item Recall regularity properties for the Parisi solution $\Phi$ from the literature.
    \item Prove regularity properties for $\tilde{\Lambda}_\gamma$ from Fenchel-Legendre duality.
    \item Write down a differential equation for $\tilde{\Lambda}_\gamma$.
    \item Translate the Auffinger-Chen SDE into the primal coordinate space, writing equivalent SDEs associated to $\tilde{\Lambda}_\gamma$.
    \item Obtain uniform continuity estimates for $\Lambda$ as a function $t$ and $y$, as well as convergence estimates for $\tilde{\Lambda}_\gamma$ as $\gamma \to 0$.
\end{itemize}

\subsection{Smoothness of the Parisi solution}

First, we gather some results on the regularity of solutions to the Parisi PDE; see also \cite{auffinger2015parisi,auffinger2015properties,jagannath2016dynamic}.

\begin{proposition}[Convexity and smoothness for $\Phi$ {\cite[Proposition 2]{auffinger2015parisi} and~\cite[Proposition 1]{auffinger2015properties}}]\label{prop:lipschitz-derivatives}
Let $\Phi$ be a strong solution to the Parisi PDE for some measure $\mu$ with initial condition $\Phi(1,x) = \log (2\cosh(x))$.  Then
\begin{enumerate}
    \item \textbf{Strict convexity:} $\Phi$ is strictly convex in $x$; see \cite[Section 2]{chen2018generalized}.
    \item \textbf{Smoothness:} The spatial derivatives $\partial_x^k \Phi$ exist for all $k \geq 0$, and each derivative is bounded and continuous on $[0,1] \times \R$ \cite[Lemma 10]{jagannath2016dynamic}.
    \item \textbf{Smoothness in time:}  The weak derivatives $\partial_t^{\pm} \partial_x^k \Phi$ exist and are bounded \cite[Proposition 5]{auffinger2015properties}, \cite[Lemma 10]{jagannath2016dynamic}.  Moreover, the left and right-hand time derivatives $\partial_t^{\pm} \partial_x^k \Phi$ exist and they agree whenever $t$ is a point of continuity of $F_\mu$ \cite[Corollary 11]{jagannath2016dynamic}.  See also \cite[Proposition 2]{auffinger2015parisi} and~\cite[Proposition 1]{auffinger2015properties}.
\end{enumerate}
\end{proposition}

As mentioned above, since the maximizer in \eqref{eq: Lambda def} satisfies $y = \partial_x \Phi(t,x)$, it will be important for us to get more refined control over $\partial_x \Phi(t,x)$.  From Hopf-Cole computations, it is not hard to see that, when $\mu$ is atomic, the derivative $\partial_x \Phi(t,x)$ goes to $1$ as $x \to \infty$ and $-1$ as $x \to -\infty$.  In fact, one can give a uniform bound as follows.

\begin{proposition}[Refined bounds for $\partial_x \Phi$ and $\partial_{x,x}\Phi$] \label{prop: Phi derivative bound}
Let $\Phi$ be a strong solution to the Parisi PDE for some measure $\mu$ with initial condition $\Phi(1,x) = \log (2\cosh(x))$.  Then
\[
1 - 2 \exp(8 \beta^2(1 - t)) \exp(-2x) \leq \sgn(x) \partial_x \Phi(t,x) \leq 1.
\]
and
\[
e^{-6 \beta^2(1 - t)} \sech(x)^2 \leq \partial_{x,x} \Phi(t,x) \leq 1.
\]
\end{proposition}

A similar bound for $\partial_{x,x} \Phi$ is given in \cite[Lemma 14.7.16]{talagrand2011mean}, \cite[Proposition 2, eq. (20)]{auffinger2015parisi}.  Rather than approximating by atomic measures as in those works, we use a stochastic expression for $\partial_x \Phi$ due to Jagannath and Tabasco \cite{jagannath2016dynamic}.

\begin{lemma}[{See \cite[Eq. (9)]{jagannath2016dynamic}}] \label{lem: Phi derivative stochastic}
Let $x_0 \in \R$ and $t_0 \in [0,1]$, and for $t \in [t_0,1]$ let $X_t$ be the solution to the Auffinger-Chen SDE
\[
    dX_t = \sqrt{2} \beta \,dW_t + 2 \beta^2 F_\mu(t)\partial_x \Phi(t,X_t)\,dt\, ,
\]
with initial condition $X_{t_0} = x_0$.
Then,
\begin{align*}
\partial_x \Phi(t_0,x_0) &= \E[\tanh(X_1)]\,, \\
\partial_{x,x} \Phi(t_0,x_0) &= \E \left[ \sech(X_1)^2 + 2 \beta^2 \int_{t_0}^1 F_\mu(s) \partial_{x,x} \Phi(s,X_s)^2\,ds \right]\,,\\
\partial_{x,x,x}\Phi(t_0,x_0) &= -2\E\left[\tanh(X_1)\sech^2(X_1)\right] + 6\beta^2\E\left[\int_{t_0}^1 F_\mu(s)\partial_{x,x,x}\Phi(s,X_s)\partial_{x,x}\Phi(s,X_s)ds\right]\,. 
\end{align*}
\end{lemma}

The other ingredient that we need is the following moment generating function bound for $X_t$ proved by It{\^o} calculus.

\begin{lemma}[MGF estimate for AC solution] \label{lem: MGF bound}
Let $X_t$ be the solution to the Auffinger-Chen SDE with initial condition $X_{t_0} = x_0$.  Then for $\lambda \in \R$,
\[
    \E[\exp(\lambda X_t)] \leq \exp(\beta^2 (2|\lambda| + \lambda^2) (t - t_0)) \exp(\lambda x_0)
\]
\end{lemma}

\begin{proof}
By It{\^o} calculus (see \pref{def:ito-formula}),
\begin{align*}
    d \exp(\lambda X_t) &= \lambda \exp(\lambda X_t) \,dX_t + \frac{\lambda^2}{2} \exp(\lambda X_t)(dX_t)^2 \\
    &= \lambda \exp(\lambda X_t) \sqrt{2} \beta \,dW_t + \lambda \exp(\lambda X_t) 2 \beta^2 F_\mu(t) \partial_x \Phi(t,X_t)\,dt + 2 \beta^2 \frac{\lambda^2}{2} \exp(\lambda X_t)\,dt.
\end{align*}
Taking expectations (which makes the $dW_t$ term vanish) and then differentiating yields,
\begin{align*}
    \frac{d}{dt} \E[\exp(\lambda X_t)] &= 2 \beta^2 |\lambda| \E[\exp(\lambda X_t) F_\mu(t) \partial_x \Phi(t,X_t)] + \beta^2 \lambda^2 \E[\exp(\lambda X_t)] \\
    &\leq \beta^2 (2 |\lambda| + \lambda^2) \E[\exp(\lambda X_t)]\, ,
\end{align*}
where we use the fact that $X_t$ is finite almost surely and that $|\partial_x \Phi(t,x)| \le 1$ (by~\pref{lem: Phi derivative stochastic}).
Then, by Gr\"onwall's inequality,
\[
    \E[\exp(\lambda X_t)] \leq \exp(\beta^2(2|\lambda| + \lambda^2)(t-t_0)) \E[\exp(\lambda X_{t_0})] = \exp(\beta^2 (2|\lambda| +\lambda^2)(t-t_0)) \exp(\lambda x_0).  \qedhere
\]
\end{proof}

\begin{proof}[Proof of \pref{prop: Phi derivative bound}]
Write $(t_0,x_0)$ for the point where we want to bound $\partial_x \Phi$. Let $X_t$ be as in \pref{lem: Phi derivative stochastic}.  The upper point on $\partial_x \Phi$ follows from the fact that $\tanh(X_1) \leq 1$.
 For the lower bound, note that
\[
1 - \tanh(X_1) = \frac{2 \exp(-X_1)}{\exp(X_1) + \exp(-X_1)} \leq 2 \exp(-2X_1)\,.
\]
Hence, taking the expectation
\[
1 - \partial_x \Phi(t_0,x_0) \leq 2 \E [\exp(-2X_1)] \leq 2 \exp(8 \beta^2(1 - t_0)) \exp(-2x_0).  \qedhere
\]
The upper bound $\partial_{x,x} \Phi \leq 1$ is given in \cite[Lemma 16]{jagannath2016dynamic}.  For the lower bound, note by \pref{lem: Phi derivative stochastic} that
\[
\partial_{x,x} \Phi(t_0,x_0) \geq \E[\sech(X_1)^2] \geq [\E \sech(X_1)]^2 \geq \frac{1}{\E[\cosh(X_1)]^2},
\]
where the last line follows because
\[
1 = \E[ \cosh(X_1)^{1/2} \sech(X_1)^{1/2} ] \leq (\E[\cosh(X_1)] \E[\sech(X_1)])^{1/2}.
\]
Note that by \pref{lem: MGF bound},
\begin{align*}
\E[\cosh(X_1)] &= \frac{1}{2} \left( \E[e^{X_1}] + \E[e^{-X_1}] \right) \\
&\leq \frac{1}{2} \exp(3 \beta^2(1-t_0))[e^{x_0} + e^{-x_0}].
\end{align*}
Hence,
\[
\partial_{x,x} \Phi(t_0,x_0) \geq [\exp(3 \beta^2(1 - t_0)) \cosh(x_0)]^{-2},
\]
which is the desired lower bound.
\end{proof}

\subsection{Smoothness of the convex conjugates}

Now we turn to the properties of the convex conjugate $\Lambda$.  Many of the claims in this proposition are standard facts about Fenchel-Legendre duality, but the proofs are short in this case so we include them for completeness.

\begin{proposition}[Basic properties of $\Lambda$] \label{prop: convex conjugate Lambda}
Let $\Phi$ be a strong solution to the Parisi PDE for some measure $\mu$ with initial condition $\Phi(1,x) = \log (2\cosh x)$.  Let
\[
\Lambda(t,y) = \sup_{x \in \R} \left( xy - \Phi(t,x) \right).
\]
\begin{enumerate}
    \item \textbf{Domain:} We have $\Lambda(t,y) = +\infty$ if and only if $|y| > 1$.
    \item \textbf{Gradient and unique maximizer:} If $|y| < 1$, then there is a unique maximizer $x$ in the formula above, and $x$ is the unique solution to $\partial_x \Phi(t,x) = y$. Furthermore, the maximizer is $x = \partial_y \Lambda(t,y)$.
    \item \textbf{Smoothness:} $\Lambda(t,y)$ is a $C^\infty$ function of $y$ on $(-1,1)$.
\end{enumerate}
\end{proposition}

\begin{proof}
(1) Since $|\partial_x \Phi| \leq 1$, the function $\Phi(t,x)$ is $1$-Lipschitz in $x$.  Hence, if $|y| > 1$, then $xy$ grows faster than $\Phi(t,x)$ as $x \to \pm \infty$, and so the supremum will be infinite.  For the case $|y| \leq 1$, note that
\[
\int_0^\infty |1 - \partial_x \Phi(t,x)| \,dx < \infty
\]
by \pref{prop: Phi derivative bound}, and therefore $\Phi(t,x) - x$ is bounded for $x > 0$.  Thus, by evenness of $\Phi(t,\cdot)$, we have that $|x| - \Phi(t,x)$ is bounded.  In particular for $y \in [-1,1]$, $xy - \Phi(t,x)$ is bounded above, so $\Lambda(t,y) < \infty$.

(2) By the strict convexity of $\Phi$, the function $\partial_x \Phi(t,\cdot)$ is strictly increasing.  Also, by \pref{prop: Phi derivative bound} the limits at $\pm \infty$ are $\pm 1$.  Hence, $\partial_x \Phi(t,\cdot)$ defines a bijection $\R \to (-1,1)$ which has a smooth inverse since $\partial_{x,x} \Phi > 0$.  Hence, for every $y \in (-1,1)$, there is a unique $x = [\partial_x \Phi(t,\cdot)]^{-1}(y)$ such that $\partial_x \Phi(t,x) = y$.  This relation implies that $\partial_x[ xy - \Phi(t,x)] = 0$.  Since the function $xy - \Phi(t,x)$ is concave in $x$, any critical point must be a maximizer. Writing $x(y) = [\partial_x \Phi(t,\cdot)]^{-1}(y)$, so by the chain rule
\[
\partial_y \Lambda(t,y) = \frac{\partial}{\partial y} \left[ xy - \Phi(t,x) \right] = x + y \frac{dx}{dy} - \partial_x \Phi(t,x) \frac{dx}{dy} = x. 
\]
(3) Therefore, since the maximizer is unique, integrating the previous equality gets
\[
\Lambda(t,y) = y [\partial_x \Phi(t,\cdot)]^{-1}(y) - \Phi(t, [\partial_x \Phi(t,\cdot)]^{-1}(y)). 
\]
By the inverse function theorem, $[\partial_x \Phi(t,\cdot)]^{-1}$ is $C^\infty$ since $\Phi \in C^\infty$ by \pref{prop:lipschitz-derivatives} and $\partial_{x,x}\Phi > 0$.
It follows that $\Lambda(t,\cdot)$ is a smooth function of $y$  because the $C^{\infty}$ property is closed under finite sums and products and compositions. \qedhere
\end{proof}

Now we estimate $\partial_y \Lambda$ and hence obtain uniform continuity of $\Lambda$ on $[-1,1]$.

\begin{lemma}[Continuity estimates for $\Lambda$] \label{lem: Lambda modulus of continuity}
Continue the same setup as the previous proposition.
\begin{enumerate}
    \item \textbf{Estimate for first derivative of $\Lambda$:}
    \[
    |\partial_y \Lambda(t,y)| \leq \frac{1}{2} \log \frac{2}{1-|y|} + 4 \beta^2(1-t).
    \]
    \item \textbf{Uniform continuity:}  For $y, y' \in [-1,1]$, we have
    \[
    |\Lambda(t,y') - \Lambda(t,y)| \leq \frac{1}{2} |y' - y| \left(\log \frac{2}{|y'-y|} + 1 + 8 \beta^2(1 - t) \right) \leq (1 + 8 \beta^2) \left| \frac{y'-y}{2} \right|^{\frac{8\beta^2}{1 + 8 \beta^2}}.
    \]
\end{enumerate}
\end{lemma}

\begin{proof}
(1) Fix $y$, and let $x = \partial_y \Lambda(t,y)$, so that $y = \partial_x \Phi(t,x)$.  Then we have
\[
    1 - y = 1 - \partial_x \Phi(t,x) \leq 2 \exp(8\beta^2(1-t)) \exp(-2x).
\]
Thus,
\[
\log(1 - y) \leq \log(2) + 8 \beta^2(1 - t) - 2x.
\]
and
\[
2x \leq -\log(1 - y) + \log(2) + 8 \beta^2(1-t)\,,
\]
and so
\[
\partial_y \Lambda(t,y) \leq \frac{1}{2} \log \frac{2}{1-y} + 4 \beta^2(1-t).
\]
This is the bound we want in the case that $y \geq 0$, and the bound follows in the negative case too since $\partial_y \Lambda$ is an odd function of $y$.

(2) Since $\Lambda$ is an even function of $y$, we have
\[
|\Lambda(t,y') - \Lambda(t,y)| = |\Lambda(t,|y'|) - \Lambda(t,|y|)|.
\]
Therefore, since also $||y'| - |y|| \leq |y' - y|$, it suffices to prove the claim when $y', y \geq 0$.  Furthermore, assume without loss of generality that $y' \geq y \geq 0$.  Thus,
\begin{align*}
0 \leq \Lambda(t,y') - \Lambda(t,y) &= \int_y^{y'} \partial_y \Lambda(t,u)\,du \\
&\leq \int_y^{y'} \left( \frac{1}{2} \log \frac{2}{1-u} + 4 \beta^2 (1 - t) \right)\,du.
\end{align*}
Now since right-hand side is an increasing function of $u$, shifting the interval of integration from $[y,y']$ to $[1-(y'-y),1]$ can only increase the value.  Hence,
\begin{align*}
\Lambda(t,y') - \Lambda(t,y) &\leq \int_{1-(y'-y)}^1 \left( \frac{1}{2} \log \frac{2}{1-u} + 4 \beta^2 (1 - t) \right)\,du \\
&= \int_0^{(y'-y)/2} \log \frac{1}{v}\,dv + 4 \beta^2 (1 - t)(y'-y),
\end{align*}
where we made the substitution $1 - u = 2v$ in the integral of the first term.  By integration by parts,
\begin{align*}
\int_0^w \log \frac{1}{v}\,dv &= -\int_0^w \log v\,dv \\
&= \left[ -v \log v \right]_0^w + \int_0^w v(1/v)\,dv \\
&= -w \log w + w \\
&= w\left(\log \frac{1}{w} + 1 \right).
\end{align*}
Thus, we get
\[
\Lambda(t,y') - \Lambda(t,y) \leq \frac{y' - y}{2} \left(\log \frac{2}{y'-y} + 1 + 8 \beta^2(1 - t) \right).
\]
This is the desired bound in the case where $y' \geq y \geq 0$, which implies the general case.

Finally, to show the last claim, observe that
\begin{align*}
\left(\log \frac{2}{|y'-y|} + 1 + 8 \beta^2(1-t) \right) &\leq \left(\log \frac{2}{|y'-y|} + 1 + 8 \beta^2 \right) \\
&= (1 + 8 \beta^2) \left( 1 + \frac{1}{1 + 8 \beta^2} \log \frac{2}{|y'-y|} 
 \right) \\
&\leq (1 + 8 \beta^2) \exp \left( \frac{1}{1 + 8 \beta^2} \log \frac{2}{|y'-y|} \right) \\
&= (1 + 8 \beta^2) \left| \frac{y' - y}{2} \right|^{-\frac{1}{1 + 8\beta^2}}.
\end{align*}
Therefore,
\[
\frac{|y' - y|}{2} \left(\log \frac{2}{|y'-y|} + 1 + 8 \beta^2(1 - t) \right) \leq (1 + 8 \beta^2) \left| \frac{y' - y}{2} \right|^{1 - \frac{1}{1 + 8\beta^2}},
\]
which is the desired estimate.
\end{proof}

The functions $\tilde{\Lambda}_{\gamma}$ have similar properties as in \pref{prop: convex conjugate Lambda} that can be described as follows.

\begin{proposition}[Properties of the regularized conjugate $\tilde{\Lambda}_{\gamma}$] \label{prop: convex conjugate Lambda gamma}
Let $\Phi$ be a strong solution to the Parisi PDE for some measure $\mu$ with initial condition $\Phi(1,x) = \log (2\cosh x)$.  Let
\[
\tilde{\Lambda}_{\gamma}(t,y) = \sup_{x \in \R} \left( xy - \Phi(t,x) - \frac{\gamma}{2} x^2 \right).
\]
Then
\begin{enumerate}
    \item \textbf{Domain:} $\tilde{\Lambda}_{\gamma}(t,y) < \infty$ for all $y \in \R$.
    \item \textbf{Unique maximizer:}  For each $y \in \R$, there is a unique maximizer $x$ in the above formula, which is the unique solution to $y = \partial_x \Phi(t,x) + \gamma x$. 
    \item \textbf{Smoothness:} $\tilde{\Lambda}_{\gamma}$ is a $C^\infty$ function of $y$.
    \item \textbf{Gradient and maximizer:} The maximizer $x$ is given by $x = \partial_y \tilde{\Lambda}_{\gamma}(t,y)$.
\end{enumerate}
\end{proposition}

\begin{proof}
(1) Note that since $|x| - \Phi(t,x)$ is bounded, the function $xy - \Phi(t,x) - \frac{\gamma}{2} x^2$ is bounded above by a quadratic with negative leading term and hence the supremum is finite.

(2) Note that $x \mapsto \partial_x \Phi(t,x) + \gamma x$ is a smooth strictly increasing function with derivative bounded below, and therefore it defines a bijection $\R \to \R$ with smooth inverse.  The rest of the argument for (2), as well as (3) and (4), is the same as \pref{prop: convex conjugate Lambda}.
\end{proof}

\begin{proposition}[Relating $\Lambda$ and $\tilde{\Lambda}_{\gamma}$] \label{prop: Lambda gamma versus Lambda}
Continue the same setup from the previous propositions.  Then
\begin{enumerate}
    \item \textbf{Inf-convolution formula:}
    \[
    \tilde{\Lambda}_{\gamma}(t,y) = \inf_{y' \in [-1,1]} \Lambda(t,y') + \frac{1}{2 \gamma} (y' - y)^2.
    \]
    \item \textbf{Unique minimizer:}  The formula above has a unique minimizer $y'$ that satisfies $y' + \gamma \partial_y \Lambda(t,y') = y$ and $y' = y - \gamma \partial_y \tilde{\Lambda}_{\gamma}(t,y)$.
    \item \textbf{Uniform convergence as $\gamma \to 0$:}  For $y \in [-1,1]$, we have
    \[
    \Lambda(t,y) \geq \tilde{\Lambda}_{\gamma}(t,y) \geq \Lambda(t,y) - (1 + 4 \beta^2) (2 \beta^2 \gamma)^{4\beta^2/(1 + 4 \beta^2)}.
    \]
\end{enumerate}
\end{proposition}

\begin{proof}
We prove (1) and (2) together.  From the proof of \pref{prop: convex conjugate Lambda}, we know that $\partial_y \Lambda(t,\cdot) = [\partial_x \Phi(t,\cdot)]^{-1}$ is an increasing diffeomorphism from $(-1,1)$ to $\R$, and hence so is $\id + \gamma \partial_y \Lambda(t, \cdot)$.  Thus, there is a unique point $y'$ with $y' + \gamma \partial_y \Lambda(t,y') = y$.  Direct computation shows that $y'$ is a critical point of $\Lambda(t,y') + \frac{1}{2 \gamma} (y' - y)^2$.  Since that function is strictly convex on $[-1,1]$, the critical point must be the minimizer.

To solve for the minimum value, first write $x = \partial_y \Lambda(t,y')$, so that $y' = \partial_x \Phi(t,x)$, and thus $y = y' + \gamma \partial_y \Lambda(t,y') = y' + \gamma x = \partial_x \Phi(t,x) + \gamma x$.  Hence, $x$ is the maximizer in the definition of $\tilde{\Lambda}_{\gamma}(t,y)$ and also the maximizer in the definition of $\Lambda(t,y')$, so that
\begin{align*}
    \tilde{\Lambda}_{\gamma}(t,y) &= xy - \Phi(t,x) - \frac{\gamma}{2} x^2 \\
    &= x(y' + \gamma x) - \Phi(t,x) - \frac{\gamma}{2} x^2 \\
    &= xy' - \Phi(t,x) + \frac{\gamma}{2} x^2 \\
    &= \Lambda(t,y') + \frac{1}{2 \gamma} (y' - y)^2 \\
    &= \inf_{y'' \in [-1,1]} \Lambda(t,y'') + \frac{1}{2 \gamma} (y'' - y)^2,
\end{align*}
which finishes the proof of (1).  Furthermore, since $x$ is the maximizer for $\tilde{\Lambda}_{\gamma}(t,y)$, we have $\partial_y \tilde{\Lambda}_{\gamma}(t,y) = x = \partial_y \Lambda(t,y')$, and so
\[
y' = y - \gamma \partial_y \Lambda(t,y') = y - \gamma \partial_y \tilde{\Lambda}_{\gamma}(t,y),
\]
which finishes the proof of (2).

(3) The upper bound on $\tilde{\Lambda}_{\gamma}$ follows by taking $y' = y$ as a candidate for the infimum in (1).  For the lower bound, note by \pref{lem: Lambda modulus of continuity} that 
\[
\Lambda(t,y') + \frac{1}{2 \gamma} (y' - y)^2 \geq \Lambda(t,y) - (1 + 8 \beta^2) \left( \frac{|y'-y|}{2} \right)^{1 - 1/(1 + 8 \beta^2)} + \frac{2}{\gamma} \left( \frac{|y'-y|}{2} \right)^2.
\]
Let us find the minimum value of the function $f: [0,\infty) \to \R$ given by
\[
f(s) = -(1 + 8 \beta^2) s^{8\beta^2/(1 + 8 \beta^2)} + \frac{2}{\gamma} s^2.
\]
Note
\[
f'(s) = -8 \beta^2 s^{-1/(1 + 8 \beta^2)} + \frac{4}{\gamma} s,
\]
so critical points occur when
\begin{align*}
2 \beta^2 s^{-1/(1+8\beta^2)} &= \frac{1}{\gamma} s \\
2 \beta^2 \gamma &= s^{1 + 1/(1 + 8 \beta^2)} = s^{(2 + 8\beta^2)/(1 + 8 \beta^2)} \\
s &= (2 \beta^2 \gamma)^{(1 + 8 \beta^2)/(2 + 8 \beta^2)},
\end{align*}
and here
\begin{align*}
f(s) &= -(1 + 8 \beta^2) (2 \beta^2 \gamma)^{8 \beta^2/(2 + 8 \beta^2)} + \frac{2}{\gamma} (2 \beta^2 \gamma)^{2(1+8\beta^2)/(2 + 8 \beta^2)} \\
&= -(1 + 8 \beta^2) (2 \beta^2 \gamma)^{8 \beta^2/(2 + 8 \beta^2)} + \frac{2}{\gamma}(2 \beta^2 \gamma) (2 \beta^2 \gamma)^{8\beta^2/(2 + 8 \beta^2)} \\
&= -(1 + 4 \beta^2) (2 \beta^2 \gamma)^{4\beta^2/(1 + 4 \beta^2)}.
\end{align*}
Since $f(0) = 0$ and $\lim_{s \to \infty} f(s) = \infty$, this is the minimum and hence gives a lower bound for $f(|y'-y|/2)$, which proves the desired lower bound for $\tilde{\Lambda}_{\gamma}(t,y)$.
\end{proof}

\subsection{Primal Parisi PDE and primal Auffinger-Chen SDE}

\begin{proposition}[Primal Parisi PDE and regularity properties]\label{prop:primal-pde-lipschitz}
Let $\mu$ be a probability measure on $[0,1]$, and let $\Phi$ be the corresponding solution to the Parisi PDE with initial condition $\Phi(1,x) = \log( 2 \cosh(x))$.  Let $\gamma \geq 0$.
\begin{enumerate}
    \item \textbf{Lipschitz estimates:} $\partial_y^k \tilde{\Lambda}_{\gamma}(t,y)$ is Lipschitz in $t$, uniformly for $y$ in each compact subinterval $[-a,a]$ of $\dom(\tilde{\Lambda}_{\gamma})$ (which is $(-1,1)$ if $\gamma = 0$ and $\R$ otherwise).
    \item \textbf{Differentiability in $t$:}  Let $t_0$ be a point where $F_\mu$ is continuous and let $y_0 \in \dom(\tilde{\Lambda}_{\gamma})$.  Then for $k \in \N$, $\partial_t \partial^k \tilde{\Lambda}_{\gamma}(t,y)$ exists at $(t_0,y_0)$.
    \item \textbf{Partial differential equation:} Let $t_0$ a continuity point for $F_\mu$, and $y_0 \in \dom(\tilde{\Lambda}_{\gamma})$, we have
    \[
    \partial_t \tilde{\Lambda}_{\gamma}(t_0,y_0) = \beta^2 \left( \frac{1}{\partial_{y,y} \tilde{\Lambda}_{\gamma}(t_0,y_0)} - \gamma + F_\mu(t)(y_0 - \gamma \partial_y \tilde{\Lambda}_{\gamma}(t_0,y_0))^2 \right).
    \]
    \item \textbf{Boundary conditions:}  We have
    \[
    \tilde{\Lambda}_{\gamma}(1,y) = \inf_{y'}\left(\Lambda(1,y') + \frac{1}{2\gamma}(y - y')^2\right) 
    \]
    where
    \[
    \Lambda(1,y') = \frac{(1-y')\log(1-y') + (1+y')\log(1+y')}{2} - \log 2\,.
    \]
\end{enumerate}
\end{proposition}

\begin{proof}
(1) Note that there is some interval $[-b,b]$ such that $\partial_y \tilde{\Lambda}_{\gamma}(t, \cdot)[-a,a] \subseteq [-b,b]$ for all $t \in [0,1]$.  This follows in the case $\gamma = 0$ from \pref{lem: Lambda modulus of continuity} and in the case $\gamma > 0$ from \pref{prop: convex conjugate Lambda gamma}.  Next, writing $\tilde{\Phi}_{\gamma}(t,x) = \Phi(t,x) + \frac{\gamma}{2} x^2$, we see that $\partial_{x,x} \tilde{\Phi}_{\gamma}$ is bounded below on $[-b,b]$ by some constant $c$ using  \pref{prop:lipschitz-derivatives}, and therefore, for $x, x' \in [-b,b]$, we have
\[
c|x - x'| \leq |\partial_x \tilde{\Phi}_{\gamma}(t,x) - \partial_x \tilde{\Phi}_{\gamma}(t,x')|.
\]
Furthermore, from \pref{prop:lipschitz-derivatives}, since $\partial_x^k \Phi$ is weakly differentiable in $t$ and the derivative is bounded, we know that $\partial_x^k \Phi$ is Lipschitz in $t$, uniformly in $y$.

Now fix $y \in [-a,a]$ and let $t, t' \in [0,1]$.  Write
\begin{align*}
0 &= y - y \\
&= \partial_x \tilde{\Phi}_{\gamma}(t, \partial_y \tilde{\Lambda}_{\gamma}(t,y)) - \partial_x \tilde{\Phi}_{\gamma}(t',\partial_y \tilde{\Lambda}_{\gamma}(t',y)) \\
&= \partial_x \tilde{\Phi}_{\gamma}(t, \partial_y \tilde{\Lambda}_{\gamma}(t,y)) - \partial_x \tilde{\Phi}_{\gamma}(t, \partial_y \tilde{\Lambda}_{\gamma}(t',y)) + \partial_x \tilde{\Phi}_{\gamma}(t, \partial_y \tilde{\Lambda}_{\gamma}(t',y)) - \partial_x \tilde{\Phi}_{\gamma}(t', \partial_y \tilde{\Lambda}_{\gamma}(t',y))
\end{align*}
We have that $\partial_x \tilde{\Phi}_{\gamma}(t,y)$ is Lipschitz in $t$ uniformly in $y$, and therefore, $|\partial_x \tilde{\Phi}_{\gamma}(t,x) - \partial_x \tilde{\Phi}_{\gamma}(t',x)| \leq L\,|t - t'|$.  Therefore
\begin{align*}
c\,|\partial_y \tilde{\Lambda}_{\gamma}(t,y) - \partial_y \tilde{\Lambda}_{\gamma}(t',y)| &\leq |\partial_x \tilde{\Phi}_{\gamma}(t, \partial_y \tilde{\Lambda}_{\gamma}(t,y)) - \partial_x \tilde{\Phi}_{\gamma}(t, \partial_y \tilde{\Lambda}_{\gamma}(t',y))| \\
&= |\partial_x \tilde{\Phi}_{\gamma}(t, \partial_y \tilde{\Lambda}_{\gamma}(t',y)) - \partial_x \tilde{\Phi}_{\gamma}(t', \partial_y \tilde{\Lambda}_{\gamma}(t',y))| \\
&\leq L|t - t'|.
\end{align*}
Hence, $|\partial_y \tilde{\Lambda}_{\gamma}(t,y) - \partial_y \tilde{\Lambda}_{\gamma}(t',y)| \leq (L/c) |t - t'|$ as desired.

Lipschitzness and boundedness of $\partial_y \tilde{\Lambda}_{\gamma}(t,y)$ in $t$ uniformly for $y \in [-a,a]$ implies Lipschitzness of $\tilde{\Lambda}_{\gamma}$ in $t$ uniformly for $y \in [-a,a]$ because $\tilde{\Lambda}_{\gamma}(t,y) = y \partial_y \tilde{\Lambda}_{\gamma}(t,y) - \tilde{\Phi}_{\gamma}(t, \partial_y \tilde{\Lambda}_{\gamma}(t,y))$.  To obtain Lipschitzness of the higher-order derivatives, we use the formula
\[
\partial_{y,y} \tilde{\Lambda}_{\gamma} = \frac{1}{\partial_{x,x} \tilde{\Phi}_{\gamma}(t,\partial_y \tilde{\Lambda}_{\gamma}(t,y))}\, ,
\]
which is obtained by the observation that $x = \partial_y\tilde{\Lambda}_{\gamma}(t,y) = \left(\partial_x \tilde{\Phi}_{\gamma}(t,x)\right)^{-1}[y]$ is the unique maximizer, followed by an invocation of the inverse function theorem as $\partial_x \tilde{\Phi}_{\gamma}(t,x)$ is continuously differentiable.
Taking spatial derivatives of this formula expresses $\partial_y^k \tilde{\Lambda}_{\gamma}(t,y)$ as a composition of spatial derivatives of $\Phi$ (which are Lipschitz in $t$ and $x$) with $\partial_y \tilde{\Lambda}_{\gamma}(t,y)$ (which is Lipschitz in $t$), where the only term in the denominator is $\partial_{x,x} \tilde{\Phi}_{\gamma}$ which we know is bounded below on $[-b,b]$.  This implies Lipschitzness of the higher derivatives of $\tilde{\Lambda}_{\gamma}$.

(2) Let $t_0$ be a point of continuity of $F_\mu$, so that $\partial_x^k\Phi$ is differentiable there, and hence so is $\partial_x^k \tilde{\Phi}_{\gamma}$ since $\tilde{\Phi}_{\gamma}(t,x) = \Phi(t,x) + (\gamma/2) x^2$.  Let $y \in [-a,a]$ as in the previous argument.  Recall $\partial^k \Phi(t,x)$ is bounded for $x \in [-b,b]$.  Using the same relations as in the previous argument, write
\begin{align*}
0 &= \partial_x \tilde{\Phi}_{\gamma}(t, \partial_y \tilde{\Lambda}_{\gamma}(t,y)) - \partial_x \tilde{\Phi}_{\gamma}(t, \partial_y \tilde{\Lambda}_{\gamma}(t_0,y)) + \partial_x \tilde{\Phi}_{\gamma}(t, \partial_y \tilde{\Lambda}_{\gamma}(t_0,y)) - \partial_x \tilde{\Phi}_{\gamma}(t_0, \partial_y \tilde{\Lambda}_{\gamma}(t_0,y)) \\
&= \partial_{x,x} \tilde{\Phi}_{\gamma}(t,\partial_y \tilde{\Lambda}_{\gamma}(t_0,y)) ( \partial_y \tilde{\Lambda}_{\gamma}(t,y) - \partial_y \tilde{\Lambda}_{\gamma}(t_0,y)) + O(| \partial_y \tilde{\Lambda}_{\gamma}(t,y) - \partial_y \tilde{\Lambda}_{\gamma}(t_0,y)|^2) \\
&\quad + \partial_t\partial_x \tilde{\Phi}_{\gamma}(t_0, \partial_y \tilde{\Lambda}_{\gamma}(t_0,y))(t- t_0) + o(t-t_0).
\end{align*}
Since $\partial_y \tilde{\Lambda}_{\gamma}$ is Lipschitz in $t$ uniformly for $y \in [-a,a]$, the term $O(| \partial_y \tilde{\Lambda}_{\gamma}(t,y) - \partial_y \tilde{\Lambda}_{\gamma}(t_0,y)|^2)$ is $O(|t - t_0|^2)$, and therefore,
\[
\partial_y \tilde{\Lambda}_{\gamma}(t,y) - \partial_y \tilde{\Lambda}_{\gamma}(t_0,y) = -\frac{1}{\partial_{x,x} \tilde{\Phi}_{\gamma}(t,\partial_y \tilde{\Lambda}_{\gamma}(t_0,y))} \partial_t \tilde{\Phi}_{\gamma}(t_0, \partial_y \tilde{\Lambda}_{\gamma}(t_0,y))(t - t_0) + o(t - t_0),
\]
which shows differentiability
\begin{equation} \label{eq: partial t y Lambda}
\partial_t \partial_y \tilde{\Lambda}_{\gamma}(t_0,y) = -\frac{\partial_t \partial_x \tilde{\Phi}_{\gamma}(t_0, \partial_y \tilde{\Lambda}_{\gamma}(t_0,y))}{\partial_{x,x} \tilde{\Phi}_{\gamma}(t_0,\partial_y \tilde{\Lambda}_{\gamma}(t_0,y))}.
\end{equation}
For differentiability of $\Lambda$ in $t$ at the point $(t_0,y)$, we again express $\Lambda$ in terms of $\partial_y \Lambda$ and $\Phi$.  Similarly, the higher order derivatives of $\Lambda$ are shown to be differentiable in $t$ at $t = t_0$ by expressing them as compositions of derivatives $\partial_x^k \Phi$ (which are differentiable in $t$ at $t = t_0$ by \pref{prop:lipschitz-derivatives} (2)) and $\partial_y \Lambda$ and using the chain rule.

(3) Let $t$ be a point of continuity for $F_\mu$.  Note that
\begin{align*}
\partial_t \tilde{\Phi}_{\gamma}(t,x) &= \partial_t \Phi(t,x) \\
&= -\beta^2 \left( \partial_{x,x} \Phi(t,x) + F_\mu(t) \partial_x \Phi(t,x)^2 \right) \\
&= -\beta^2 \left( \partial_{x,x} \tilde{\Phi}_{\gamma}(t,x) - \gamma + F_\mu(t) (\partial_x \tilde{\Phi}_{\gamma}(t,x) - \gamma x)^2 \right).
\end{align*}
Now differentiate the relation
\[
\tilde{\Lambda}_{\gamma}(t,y) = y \partial_y \tilde{\Lambda}_{\gamma}(t,y) - \tilde{\Phi}_{\gamma}(t, \partial_y \tilde{\Lambda}_{\gamma}(t,y))
\]
in $t$ to obtain
\[
\partial_t \tilde{\Lambda}_{\gamma}(t,y) = y \partial_t \partial_y \tilde{\Lambda}_{\gamma}(t,y) - \partial_t \tilde{\Phi}_{\gamma}(t,\partial_y \tilde{\Lambda}_{\gamma}(t,y)) - \partial_x \tilde{\Phi}_{\gamma}(t,\partial_y \tilde{\Lambda}_{\gamma}(t,y)) \partial_t \partial_y\tilde{\Lambda}_{\gamma}(t,y).
\]
Since $y = \partial_x \tilde{\Phi}_{\gamma}(t, \partial_y \tilde{\Lambda}_{\gamma}(t,y))$, the first and third terms on the right-hand side cancel and
\begin{align*}
\partial_t \tilde{\Lambda}_{\gamma}(t,y) &= -\partial_t \tilde{\Phi}_{\gamma}(t, \partial_y \tilde{\Lambda}_{\gamma}(t,y)) \\
&= \beta^2 \left( \partial_{x,x} \tilde{\Phi}_{\gamma}(t, \partial_y \tilde{\Lambda}_{\gamma}(t,y)) - \gamma + F_\mu(t)( \partial_x \tilde{\Phi}_{\gamma}(t,\partial_y \tilde{\Lambda}_{\gamma}(t,y)) - \gamma \partial_y \tilde{\Lambda}_{\gamma}(t,y))^2 \right) \\
&= \beta^2 \left( \frac{1}{\partial_{y,y} \tilde{\Lambda}_{\gamma}(t,y)} -\gamma + F_\mu(t)(y - \gamma \partial_y \tilde{\Lambda}_{\gamma}(t,y))^2 \right).
\end{align*}

(4) In the case $\gamma = 0$, the initial condition arises due to $\log \left(2\cosh\right)$ being the Fenchel-Legendre conjugate of the binary entropy function rescaled to the domain $[-1,1]$, see for example~\cite[Appendix B.1]{hsieh2018mirrored}.  The case for $\gamma > 0$ then follows by \pref{prop: Lambda gamma versus Lambda}.
\end{proof}

A recent result of Mourrat~\cite{mourrat2025inverting} gives another rewrite of the Parisi formula using Legendre-Fenchel duality, stated in terms of an expectation over a martingale process.  We give a related It{\^o} computation below that describes the transformation of the solution to the Auffinger-Chen SDE from the dual coordinates to the primal coordinates.  In light of the proposition below, we refer to
\begin{equation}\label{eq:primal-auffinger-chen}
    dY_t = \frac{\sqrt{2} \beta}{\partial_{y,y} \Lambda(t,Y_t)}\,dW_t\,,    
\end{equation}
as the \emph{primal Auffinger-Chen SDE}.  Later in \pref{lem:sde-closeness} below, we will study a similar SDE associated to $\tilde{\Lambda}_{\gamma}$.  To justify existence and uniqueness of solutions, note that \pref{prop: Lambda estimates} (3) and (4) below gives Lipschitz estimates for $1 / \partial_{y,y} \Lambda$, and clearly $1/\partial_{y,y} \Lambda$ vanishes at $y = \pm 1$ so that it can be extended by $0$ on the complement of $(-1,1)$.  Finally, we remark that \cite[Proposition 22]{jagannath2016dynamic} shows that It{\^o} calculus can be applied in this situation, even though the function has limited regularity especially in the time variable.

\begin{proposition}[Equivalence between primal and dual AC SDE]\label{prop:sde-equivalence-primal-dual}
Fix $\beta \in (0,\infty)$.  Let $\mu$ be a probability measure on $[0,1]$, and let $\Phi$ be the solution to the Parisi PDE.  Let $W_t$ be a standard Brownian motion. Let $X_t$ and $Y_t$ be stochastic processes with $Y_t = \partial_x \Phi(t,X_t)$, or equivalently $X_t = \partial_y \Lambda(t,Y_t)$.  Then, we have
\[
dY_t = \frac{\sqrt{2} \beta}{\partial_{y,y} \Lambda(t,Y_t)}\,dW_t \iff dX_t = \sqrt{2} \beta \,dW_t + 2 \beta^2 F_{\mu}(t) \partial_x \Phi(t,X_t)\,dt.
\]
\end{proposition}

\begin{proof}
Assuming the equation for $X_t$, we have by It{\^o} calculus that
    \begin{align*}
    dY_t &= d \partial_x \Phi(t,X_t) \\
    &= \partial_{x,x} \Phi(t,X_t) \,dX_t + \partial_{t,x} \Phi(t,X_t)\,dt + \frac{1}{2} \partial_{x,x,x} \Phi(t,X_t) (dX_t)^2 \\
    &= \sqrt{2} \beta \partial_{x,x} \Phi(t,X_t)\,dW_t + 2 \beta^2 F_\mu(t)\partial_{x,x} \Phi(t,X_t) \partial_x \Phi(t,X_t)\,dt + \partial_{t,x} \Phi(t,X_t)\,dt + \beta^2 \partial_{x,x,x} \Phi(t,X_t)\,dt \\
    &= \sqrt{2} \beta \partial_{x,x} \Phi(t,X_t)\,dW_t + \partial_x[\partial_t \Phi + \beta^2 \partial_{x,x} \Phi + \beta^2 F_{\mu}(t) (\partial_x \Phi)^2](t,X_t)\,dt \\
    &= \frac{\sqrt{2} \beta}{\partial_{y,y}\Lambda(t,Y_t)}\,dW_t + 0.
    \end{align*}
    Conversely, suppose that $Y_t$ satisfies this equation and note $X_t = \partial_y \Lambda(t,Y_t)$.  Then by It{\^o} calculus,
    \begin{align*}
    dX_t &= d \partial_y\Lambda(t,Y_t) \\
    &= \partial_{y,y} \Lambda(t,Y_t)\,dY_t + \partial_{t,y} \Lambda(t,Y_t)\,dt + \frac{1}{2} \partial_{y,y,y} \Lambda(t,Y_t) (dY_t)^2 \\
    &= \sqrt{2} \beta \,dW_t + \partial_{t,y} \Lambda(t,Y_t)\,dt + \beta^2 \frac{\partial_{y,y,y} \Lambda(t,Y_t)}{\left(\partial_{y,y} \Lambda(t,Y_t)\right)^2}\,dt \\
    &= \sqrt{2} \beta\,dW_t + \partial_y\left[\partial_t \Lambda - \frac{\beta^2}{\partial_{y,y} \Lambda} \right](t,Y_t)\,dt \\
    &= \sqrt{2} \beta\,dW_t + \partial_y\left[\beta^2 F_{\mu}(t) y^2 \right]_{y=Y_t}\,dt \\
    &= \sqrt{2} \beta\,dW_t + 2 \beta^2 F_{\mu}(t) Y_t\,dt \\
    &= \sqrt{2} \beta\,dW_t + 2 \beta^2 F_{\mu}(t) \partial_x \Phi(t,X_t)\,dt. \qedhere
    \end{align*}
\end{proof}

\subsection{Estimates for the primal solutions}

Recall $\tilde{\Lambda}_{\gamma}$ is one of the main ingredients in our objective function, which we will seek to estimate at the iterates of our algorithm by using Taylor expansion at each iterate (see \pref{sec:energy-analysis}).  Hence, to prove the validity of our algorithm, we need as good of bounds on higher derivatives of $\tilde{\Lambda}_{\gamma}$ as we can find.  Moreover, since we study the SDE driven by Brownian motion with the coefficient function $1 / \partial_{y,y} \tilde{\Lambda}_{\gamma}$, we also want Lipschitz bounds for $1 / \partial_{y,y} \tilde{\Lambda}_{\gamma}$ in both time and space; this will be crucial for our convergence argument in \pref{prop:sde-equivalence-primal-dual} and energy analysis in \pref{sec:energy-analysis}.  Specifically, we will show the following result.

\begin{proposition}[Estimates for the derivatives of $\tilde{\Lambda}_{\gamma}$] \label{prop: Lambda estimates}
Let $\Phi$ and $\tilde{\Lambda}_{\gamma}$ be as above.  Then we have the following estimates for $\gamma \geq 0$ and $y \in \dom(\tilde{\Lambda}_{\gamma})$.
\begin{enumerate}
    \item \textbf{Bounds for second derivative of $\tilde{\Lambda}_{\gamma}$:}
    \[
    \frac{1}{1 + \gamma} \leq \partial_{y,y} \tilde{\Lambda}_{\gamma}(t,y) \leq \frac{1}{\gamma}.
    \]
    \label{item:second-deriv}
    \item \textbf{Bound for third derivative of $\tilde{\Lambda}_{\gamma}$:}
    \[
    \left|\partial_{y,y,y} \tilde{\Lambda}_{\gamma}(t,y)\right| \leq \frac{2}{\gamma^2}.
    \]
    \label{item:third-deriv}
    \item \textbf{Spatial Lipschitz estimate for $1 / \partial_{y,y} \tilde{\Lambda}_{\gamma}$:}
    \[
    \left| \partial_y \left( \frac{1}{\partial_{y,y} \tilde{\Lambda}_{\gamma}(t,y)} \right) \right| \leq 2.
    \]
    \label{item:spatial-lipschitz}
    \item \textbf{Temporal Lipschitz estimate for $1 / \partial_{y,y} \tilde{\Lambda}_{\gamma}$:}
     \[
    \left| \partial_t \left( \frac{1}{\partial_{y,y} \tilde{\Lambda}_{\gamma}(t,y)} \right) \right| \leq 18 \beta^2.
    \]
    \label{item:temporal-lipschitz}
\end{enumerate}
\end{proposition}

The proof proceeds in several steps:
\begin{enumerate}
    \item For atomic $\mu$, we express the solution $\Phi(t,x)$ using Ruelle probability cascades.
    \item We obtain estimates comparing various derivatives of $\Phi$ for atomic $\mu$, which we then extend to arbitrary measures by density.
    \item We express derivatives of $\tilde{\Lambda}_{\gamma}$ in terms of derivatives of $\Phi$ and estimate them.
\end{enumerate}

The description of the Parisi formula in terms of RPCs is standard in spin-glass theory, but we provide some explanation in~\pref{sec:rpc}. In particular, we use the explicit RPC-based representation for $\Phi$ over atomic measures to estimate uniform bounds on the spatial derivatives of $\Phi$.

\begin{lemma}[Polynomial expressions for $\partial^{(j)}_{x}\Phi$ via the RPC representation]\label{lem:rpc-derivatives}
Let $0 \leq t_0 < \dots < t_r = 1$, and fix a finitely supported probability measure $\mu$ with $\supp(\mu) = [0,t_0] \cup \{t_1,\dots,t_r\}$.  Then there exists a random variable $T(x)$ depending on $x \in \R$ such that
\begin{enumerate}
    \item $|T(x)| \leq 1$.
    \item $T(x)$ is differentiable in $x$ and $T'(x) = 1 - T(x)^2$.
    \item Let $F_j(\tau)$ be the polynomial given recursively by
    \[
    F_1(\tau) = \tau, \qquad F_{j+1}(\tau) = F_j'(\tau)(1 - \tau^2).
    \]
    Then for all $x$,
    \[
    \partial_x^j \Phi(t_0,x) = \E[F_j(T(x))].
    \]
    In particular,
    \begin{align*}
        \partial_x \Phi(t,x) &= \E[T(x)] \\
        \partial_{x,x} \Phi(t,x) &= \E[1 - T(x)^2] \\
        \partial_{x,x,x} \Phi(t,x) &= \E[-2T(x)(1 - T(x)^2)] \\
        \partial_{x,x,x,x} \Phi(t,x) &= \E[-2(1 - 3 T(x)^2)(1 - T(x)^2)].
    \end{align*}
\end{enumerate}
\end{lemma}

\begin{proof}
As in the proof of \cite[Lemma 14.7.16]{talagrand2011mean}, we use the Ruelle Probability Cascade construction, which we include for the reader's convenience in \pref{sec:rpc}.  By \pref{lem: RPC model}, there exist nonnegative random variables $(v_\alpha)_{\alpha \in \N^r}$ such that $\sum_{\alpha \in \N^r} v_\alpha = 1$, and random variables $(Z_\alpha)_{\alpha \in \N^r}$ such that for all $x$,
\[
\Phi(t_0,x) = \E \log \sum_{\alpha \in \N^r} 2 v_\alpha \cosh(x + Z_\alpha).
\]
Now let
\[
T(x) = \frac{d}{dx} \log \sum_{\alpha \in \N^r} 2 v_\alpha \cosh(x + Z_\alpha) = \frac{\sum_{\alpha \in \N^r} v_\alpha \sinh(x + Z_\alpha)}{\sum_{\alpha \in \N^r} v_\alpha \cosh(x + Z_\alpha)}.
\]
Since $|\sinh| \leq |\cosh|$, we have $|T(x)| \leq 1$.  Also, by direct computation
\begin{align*}
T'(x) &= \frac{d}{dx} \frac{\sum_{\alpha \in \N^r} v_\alpha \sinh(x + Z_\alpha)}{\sum_{\alpha \in \N^r} v_\alpha \cosh(x + Z_\alpha)} \\
&= \frac{\frac{d}{dx} [\sum_{\alpha \in \N^r} v_\alpha \sinh(x + Z_\alpha)]}{\sum_{\alpha \in \N^r} v_\alpha \cosh(x + Z_\alpha)} - \frac{[\sum_{\alpha \in \N^r} v_\alpha \sinh(x + Z_\alpha)] \frac{d}{dx} [\sum_{\alpha \in \N^r} v_\alpha \cosh(x + Z_\alpha)]}{[\sum_{\alpha \in \N^r} v_\alpha \cosh(x + Z_\alpha)]^2} \\
&= \frac{\sum_{\alpha \in \N^r} v_\alpha \cosh(x + Z_\alpha)}{\sum_{\alpha \in \N^r} v_\alpha \cosh(x + Z_\alpha)} - \frac{[\sum_{\alpha \in \N^r} v_\alpha \sinh(x + Z_\alpha)] [\sum_{\alpha \in \N^r} v_\alpha \sinh(x + Z_\alpha)]}{[\sum_{\alpha \in \N^r} v_\alpha \cosh(x + Z_\alpha)]^2} \\
&= 1 - T(x)^2.
\end{align*}
Note that $\log \sum_{\alpha \in \N^r} 2 v_\alpha \cosh(x + Z_\alpha)$ is integrable over the probability space.  This random variable above is also $1$-Lipschitz in $x$ since $|T(x)| \leq 1$.  Hence, we may apply the bounded convergence theorem to the difference quotients to conclude that
\[
\partial_x \Phi(t_0,x) = \frac{d}{dx} \E \log \sum_{\alpha \in \N^r} 2 v_\alpha \cosh(x + Z_\alpha) = \E[T(x)].
\]
By similar reasoning, for any polynomial $F(\tau)$, since $(d/dx) F(T(x)) = F'(T(x)) (1 - T(x)^2)$ is bounded, we have
\[
\frac{d}{dx} \E[F(T(x))] = \E[F'(T(x))(1 - T(x)^2)].
\]
We apply this procedure inductively starting with $F_1(\tau) = \tau$ and this results in claim (3).
\end{proof}

Now we are ready to begin estimating the derivatives of $\Phi$.

\begin{lemma}[Bounds for $\partial^{(j)}_{x}\Phi$ in terms of $\partial_{x,x}\Phi$] \label{lem: bound by second derivative of Phi}
Let $\Phi$ be a solution to the Parisi equation for some measure $\mu$ on $[0,1]$.  Then for $j \geq 3$, there exists a constant $C_j$ independent of $\beta$ such that
\[
|\partial_x^j \Phi(t,x)| \leq C_j \partial_{x,x} \Phi(t,x).
\]
In particular, we can take $C_3 = 2$ and $C_4 = 4$.
\end{lemma}

\begin{proof}
Fix $t_0$ and assume first that $\mu$ is atomic.  Let $T(x)$ be as in \pref{lem:rpc-derivatives}.  Then
\[
\partial_x^j \Phi(t_0,x) = \E[F_j(T(x))] = \E[F_{j-1}'(T(x))(1 - T(x)^2)].
\]
Let $C_j = \max_{\tau \in [-1,1]} |F_{j-1}'(\tau)|$.  Then
\[
|\partial_x^j \Phi(t_0,x)| \leq  \E[|F_{j-1}'(T(x))|(1 - T(x)^2)] \leq C_j \E[1 - T(x)^2] = C_j \partial_x^2 \Phi(t_0,x).
\]
In particular, since $F_2(\tau) = 1 - \tau^2$ and $F_2'(\tau) = -2 \tau$, we have $C_3 = 2$.  Moreover, $F_3(\tau) = -2 \tau(1 - \tau^2)$ and $F_3'(\tau) = -2(1 - 3 \tau^2)$.  Clearly, $-2 \leq 1 - 3 \tau^2 \leq 1$, and so $C_4 = 4$.

It remains to extend the claim from atomic measures to general measures.  By \cite[Proposition 1]{auffinger2015properties}, if $\mu_k$ is a sequence of measures with $\mu_k \to \mu$ in the weak-$*$ topology, then $\Phi_{\mu_k} \to \Phi_\mu$ uniformly, and hence our estimates also hold for $\mu$.
\end{proof}

Finally, we deduce the asserted estimates for $\Lambda$.

\begin{proof}[Proof of \pref{prop: Lambda estimates}]
(1) Fix $t$, and let $x(y) = \partial_y \tilde{\Lambda}_{\gamma}(t,y)$, so that $y = \partial_x \Phi(t, x(y)) + \gamma x(y)$.  Recall also by the formula for derivatives of inverse function that
\[
\frac{dx}{dy} = \partial_{y,y} \tilde{\Lambda}_{\gamma}(t,y) = \frac{1}{\partial_{x,x} \Phi(t,x(y)) + \gamma}.
\]
Then because $0 \leq \partial_{x,x} \Phi(t,x) \leq 1$ by \pref{prop: Phi derivative bound}, we get
\[
\frac{1}{1 + \gamma} \leq \partial_{y,y} \tilde{\Lambda}_{\gamma}(t,y) \leq \frac{1}{\gamma},
\]
which is the asserted estimate.

(2) By the chain rule,
\[
\partial_{y,y,y} \tilde{\Lambda}_{\gamma}(t,y) = \partial_y \left( \frac{1}{\partial_{x,x} \Phi(t,x(y)) + \gamma} \right) = -\frac{\partial_{x,x,x} \Phi(t,x)}{(\partial_{x,x} \Phi(t,x) + \gamma)^2} \frac{dx}{dy} = -\frac{\partial_{x,x,x} \Phi(t,x)}{(\partial_{x,x} \Phi(t,x) + \gamma)^3}.
\]
Then by \pref{lem: bound by second derivative of Phi},
\[
\frac{|\partial_{x,x,x} \Phi(t,x)|}{(\partial_{x,x} \Phi(t,x) + \gamma)^3} \leq \frac{|\partial_{x,x,x} \Phi(t,x)|}{\gamma^2 \partial_{x,x} \Phi(t,x)} \leq \frac{2}{\gamma^2}.
\]

(3) By the chain rule,
\[
\partial_y \left( \frac{1}{\partial_{y,y} \tilde{\Lambda}_{\gamma} (t,y)} \right) = \partial_y \left( \partial_{x,x} \Phi(t,x) + \gamma \right) = \partial_{x,x,x} \Phi(t,x) \frac{dx}{dy} = \frac{\partial_{x,x,x} \Phi(t,x)}{\partial_{x,x} \Phi(t,x) + \gamma}.
\]
By \pref{lem: bound by second derivative of Phi}, this is bounded in absolute value by $2$.

(4) Note that
\begin{align*}
\partial_t \left( \frac{1}{\partial_{y,y} \tilde{\Lambda}_{\gamma} (t,y)} \right) &= \partial_t \left( \partial_{x,x} \Phi(t, \partial_y \tilde{\Lambda}_{\gamma}(t,y)) \right) \\
&= \partial_t \partial_{x,x} \Phi(t, x) + \partial_{x,x,x} \Phi(t,x) \partial_t \partial_y\tilde{\Lambda}_{\gamma}(t,y).
\end{align*}
From \eqref{eq: partial t y Lambda},
\begin{align*}
\partial_t \partial_y \tilde{\Lambda}_{\gamma}(t,y) &= -\frac{\partial_t \partial_x \tilde{\Phi}_{\gamma}(t,x)}{\partial_{x,x} \tilde{\Phi}_{\gamma}(t,x)} \\
&= -\frac{\partial_t \partial_x \Phi(t,x)}{\partial_{x,x} \Phi(t,x) + \gamma},
\end{align*}
so that
\begin{equation} \label{eq: formula for dt dy dy}
\partial_t \left( \frac{1}{\partial_{y,y} \tilde{\Lambda}_{\gamma} (t,y)} \right) = \partial_t \partial_{x,x} \Phi(t, x) -\partial_{x,x,x} \Phi(t,x)\frac{\partial_t \partial_x \Phi(t,x)}{\partial_{x,x} \Phi(t,x) + \gamma};
\end{equation}
we now proceed to estimate both terms on the right-hand side.  By the Parisi PDE~\pref{eq: Parisi PDE}, we have
\begin{align*}
\partial_t \partial_x \Phi &= -\beta^2 \partial_x \left( \partial_{x,x} \Phi + F_\mu(t) \partial_x (\partial_x \Phi)^2 \right) \\
&= -\beta^2 \left( \partial_{x,x,x} \Phi + 2 F_\mu(t) \partial_x \Phi \partial_{x,x} \Phi \right) \\
\partial_t \partial_{x,x} \Phi &= -\beta^2 \partial_x \left( \partial_{x,x,x} \Phi + 2 F_\mu(t) \partial_x \Phi \partial_{x,x} \Phi \right) \\
&= -\beta^2 \left( \partial_{x,x,x,x} \Phi + 2 F_\mu(t) (\partial_{x,x} \Phi)^2 + 2 F_\mu(t) \partial_x \Phi \partial_{x,x,x} \Phi \right).
\end{align*}
In particular, by \pref{lem: bound by second derivative of Phi},
\begin{align}
|\partial_{x,x,x} \Phi(t,x) \cdot \partial_t \partial_x \Phi(t, x)| &\leq \beta^2 |\partial_{x,x,x} \Phi|^2 + 2 \beta^2 F_\mu(t) |\partial_x \Phi\partial_{x,x,x} \Phi| \partial_{x,x} \Phi \label{eq: estimate for dt dx} \\
&\leq 4 \beta^2 (\partial_{x,x} \Phi)^2 + 4 \beta^2 (\partial_{x,x} \Phi)^2. \nonumber
\end{align}
and
\begin{align}
|\partial_t \partial_{x,x} \Phi(t, x)| &\leq \beta^2 |\partial_{x,x,x,x} \Phi| + 2 \beta^2 F_\mu(t) (\partial_{x,x} \Phi)^2 + 2 \beta^2 F_\mu(t) |\partial_x \Phi \partial_{x,x,x} \Phi| \label{eq: estimate for dt dx dx} \\
&\leq 4 \beta^2 \partial_{x,x} \Phi + 2 \beta^2 (\partial_{x,x} \Phi)^2 + 4 \beta^2 |\partial_x \Phi| \partial_{x,x} \Phi \nonumber \\
&\leq 10 \beta^2. \nonumber
\end{align}
Substituting our estimates \eqref{eq: estimate for dt dx dx} for $\partial_t \partial_{x,x} \Phi$ and \eqref{eq: estimate for dt dx} for $\partial_t \partial_x \Phi$ into \eqref{eq: formula for dt dy dy} yields
\begin{align*}
\left| \partial_t \left( \frac{1}{\partial_{y,y} \tilde{\Lambda}_{\gamma} (t,y)} \right) \right| &\leq |\partial_t \partial_{x,x} \Phi| + \left| \frac{\partial_t \partial_x \Phi(t,x)}{\partial_{x,x} \Phi(t,x) + \gamma} \right| \\
&\leq 10 \beta^2 + \frac{8 \beta^2 (\partial_{x,x} \Phi(t,x))^2}{\partial_{x,x} \Phi(t,x) + \gamma} \\
&\leq 18 \beta^2.  \qedhere
\end{align*}
\end{proof}

\begin{remark}[Differential equation for $1 / \partial_{y,y} \Lambda$]
Let $v(t,y) = 1 / \partial_{y,y} \Lambda(t,y) = \partial_{x,x} \Phi(t,\partial_y \Lambda(t,y))$ for $|y| < 1$ and set $v(t,y) = 0$ for $|y| > 1$.  Then $v$ satisfies the differential equation,
\[  
    \partial_t v = -\beta^2 v^2 \left(\partial_{y,y} v + 2 F_\mu(t)\right)\,,
\]
for $|y| < 1$ since
\[
\partial_t v = -\frac{\partial_t \partial_{y,y} \Lambda}{(\partial_{y,y} \Lambda)^2}
= -v^2 \partial_{y,y} \partial_t \Lambda
= -\beta^2 v^2 \partial_{y,y} \left( \frac{1}{\partial_{y,y} \Lambda} + F_\mu(t) y^2 \right)
= -\beta^2 v^2 \left( \partial_{y,y} v + 2 F_\mu(t) \right).
\]
Moreover, it trivially satisfies the equation when $|y| > 1$ since both sides are zero.  The terminal condition at $t=1$ can easily be computed from \pref{prop:primal-pde-lipschitz} (4) as $v(1,y) = 1 - y^2$.  This equation somewhat resembles the equation for $\Phi$ itself, but now with the worse nonlinearity $v^2 \partial_{y,y} v$ instead of the additive nonlinear $(\partial_x \Phi)^2$ term from the Parisi PDE.  This equation should be investigated further using the tools of viscosity solutions and free boundary problems.
\end{remark}

\subsection{Energy simplifications under fRSB}

Previous approaches to optimizing spin glass Hamiltonians required the fRSB regularity property for the measure $\mu$ that minimizes the Parisi formula~\cite{subag2021following,montanari2021optimization}.
Under fRSB, the Auffinger-Chen SDE has quadratic variation $\E Y_t^2 = t$ for $t$ up to a certain time $q_\beta^*$.  An equivalent condition is that
\[
\E \frac{2\beta^2}{(\partial_{y,y}\Lambda(t,Y_t))^2} = 1,
\]
which in our paper is exactly the condition needed to satisfy a self-consistency equation that arises from computing the Cauchy-Stieltjes transform of the spectrum of $\nabla^2 \obj$ while requiring that the largest eigenvalue of $\nabla^2 \obj$ is near zero.  The precise statement of this fRSB simplification is as follows.

\begin{lemma}[{fRSB Property of $Y_t$~\cite[Lemma 3.3-3.4]{montanari2021optimization}, \cite[Proposition 1]{chen2017variational}}] \mbox{} \\
\label{lem:qt-equality}
    Recall $q^*$ from \pref{ass:sk-frsb}.
    Assume $\beta > 0$ is large enough and that \pref{ass:sk-frsb} holds.
    Let $\mu$ be the optimizer in the Parisi formula for a given $\beta$, and let $\Phi$ and $\Lambda$ be the corresponding solutions.  Let $Y_t$ be the solution to the primal AC SDE.  Then, for every $t \in [0, q^*_\beta]$, we have
    \[
    \E_{Y_t} \left[\frac{2\beta^2}{(\partial_{y,y}\Lambda(t,Y_t))^2}\right] = 1\,.
    \]
    Equivalently,
    \[
    \E_{Y_t} Y_t^2 = t\,.
    \]
    Moreover,
    \[
    \E\left[\frac{1}{\partial_{y,y} \Lambda(t,Y_t)}\right] = \int_t^1\mu([0,s])\,ds = \int_t^1 F_\mu(s)\,ds.
    \]
\end{lemma}

A direct consequence is the following formula for $\E \Lambda(t,Y_t)$ at $t = q_\beta^*$ under fRSB.  Montanari \cite{montanari2021optimization} shows that it agrees with $\beta$ times a certain energy functional $\mathcal{E}_\beta$, which agrees with $\mathcal{P}_\beta$ in the large $\beta$ limit.

\begin{corollary}[Value of entropy along the AC process] \label{cor: SDE energy estimate}
Under Assumption \ref{ass:sk-frsb}, with the same setup as in the previous lemma, we have,
\[
\E \Lambda(q_\beta^*,Y_{q_\beta^*}) - \Lambda(0,0) - \beta^2 \int_0^{q^*_\beta} s F_\mu(s)\,ds
= 2 \beta^2 \int_0^{q^*_\beta} \int_s^1 F_\mu(u)\,du\,ds.
\]
\end{corollary}

\begin{proof}
To estimate this value, we use~\pref{def:ito-formula} and~\pref{prop:primal-pde-lipschitz} to simplify the expected value under the primal Auffinger-Chen dynamics \pref{eq:primal-auffinger-chen},
\begin{align*}
d \Lambda(t,Y_t) &= \partial_y \Lambda(t,Y_t) \,dY_t + \frac{1}{2} \partial_{y,y} \Lambda(t,Y_t) (dY_t)^2 + \partial_t \Lambda(t,Y_t)\,dt \\
&= \frac{\sqrt{2} \beta \partial_y \Lambda(t,Y_t)}{\partial_{y,y} \Lambda(t,Y_t)}\,dW_t + \left( \frac{\beta^2}{\partial_{y,y} \Lambda(t,Y_t)} + \partial_t \Lambda(t,Y_t) \right)\,dt \\
&= \frac{\sqrt{2} \beta \partial_y \Lambda(t,Y_t)}{\partial_{y,y} \Lambda(t,Y_t)}\,dW_t + \left( \frac{2 \beta^2}{\partial_{y,y} \Lambda(t,Y_t)} + \beta^2 F_\mu(t) Y_t^2 \right)\,dt\,.
\end{align*}
Hence, after taking expectations and using~\pref{lem:qt-equality},
\begin{align*}
\frac{d}{dt} \E \Lambda(t,Y_t) &= 2 \beta^2 \E \frac{1}{\partial_{y,y} \Lambda(t,Y_t)} + \beta^2 F_\mu(t) \E[Y_t^2] \\
&= \beta^2\left( 2 \int_t^1 F_\mu(s)\,ds + t F_\mu(t) \right).
\end{align*}
We obtain the statement asserted in the lemma by integrating this and noting that $Y_0 = 0$.
\end{proof}

Since our algorithm uses $\tilde{\Lambda}_{\gamma}$ rather than $\Lambda$ itself, we want a version of \pref{lem:qt-equality} for $\tilde{\Lambda}_{\gamma}$.  For this purpose we compare in Wasserstein distance two processes associated to $\Lambda$ and $\tilde{\Lambda}_{\gamma}$.  This lemma does not itself invoke fRSB.

\begin{lemma}[Wasserstein-$2$ distance between the non-convolved \& convolved primal SDE] \label{lem:sde-closeness}\mbox{} \\
Let $\Phi$ be the solution to the Parisi equation for some $\mu$, and let $\tilde{\Lambda}_{\gamma}$ be as above.
Let $Y_t$ and $Y_t^\gamma$ solve the equations
\begin{equation}\label{eq:primal-ac-sde}
dY_t = \frac{\sqrt{2} \beta}{\partial_{y,y} \Lambda(t, Y_t)}\,dW_t\,, 
\end{equation}
and
\begin{equation}\label{eq:regularized-primal-ac-sde}
dY_t^\gamma = \frac{\sqrt{2} \beta}{\partial_{y,y} \tilde{\Lambda}_{\gamma}(t, Y_t^\gamma)}\,dW_t\,,    
\end{equation}
with $Y^\gamma_0 = 0$ and $Y_0 = 0$.
Then
\begin{equation} \label{eq: sde-closeness 1}
\E\left|Y_t^\gamma - \gamma \partial_y \tilde{\Lambda}_{\gamma}(t,Y_t^\gamma) - Y_t\right|^2 \leq \E\left|Y_t^\gamma - (Y_t + \gamma \partial_y \Lambda(t,Y_t))\right|^2 \leq \frac{1}{5} \gamma^2 (e^{10\beta^2 t} - 1)
\end{equation}
and
\begin{equation} \label{eq: sde-closeness 2}
\norm{Y_t^\gamma - Y_t}_{L^2}^2 \leq 2 \gamma^2 (e^{10 \beta^2 t} - 1).
\end{equation}
\end{lemma}

\begin{proof}
First, let $Z_t^\gamma = Y_t + \gamma \partial_y \Lambda(t,Y_t)$.  Proposition \ref{prop: Lambda gamma versus Lambda} (2) implies that $\id - \gamma \partial_y \tilde{\Lambda}_{\gamma}(t,\cdot)$ is the inverse function of $\id + \gamma \partial_y \Lambda(t,\cdot)$.  Thus we have $Y_t = Z_t^\gamma - \gamma \partial_y \tilde{\Lambda}_{\gamma}(t,Z_t^\gamma)$.  Moreover, the function $\id - \gamma \partial_y \tilde{\Lambda}_{\gamma}(t,\cdot)$ is $1$-Lipschitz since its derivative is $1 - \gamma \partial_{y,y} \tilde{\Lambda}_{\gamma}(t,\cdot)$, which is bounded below by $1 - \gamma / \gamma = 0$ and bounded above by $1 - \gamma / (1 + \gamma) \leq 1$, using Proposition \ref{prop: Lambda estimates} (1).  Hence, we have
\[
\E |Y_t^\gamma - \gamma \partial_y \tilde{\Lambda}_{\gamma}(t,Y_t^\gamma) - Y_t|^2 \leq \E |Y_t^\gamma - (Y_t + \gamma \partial_y \Lambda(t,Y_t))|^2,
\]
which is the first inequality in \eqref{eq: sde-closeness 1}.  It remains to estimate $\E |Y_t^\gamma - (Y_t + \gamma \partial_y \Lambda(t,Y_t))|^2$.

It follows from It{\^o} calculus that
\begin{align*}
dZ_t^\gamma &= dY_t + \gamma \partial_{y,y}\Lambda(t,Y_t)\,dY_t + \frac{1}{2} \gamma \partial_{y,y,y} \Lambda(t,Y_t) (dY_t)^2 + \gamma \partial_{t,y} \Lambda(t,Y_t)\,dt \\
&= \sqrt{2} \beta \left( \frac{1}{\partial_{y,y} \Lambda(t,Y_t)} + \gamma \right)\,dW_t + \gamma \left( \beta^2 \frac{\partial_{y,y,y} \Lambda(t,Y_t)}{\partial_{y,y} \Lambda(t,Y_t)^2} + \partial_{t,y} \Lambda(t,Y_t) \right)\,dt \\
&= \frac{\sqrt{2}\beta}{\partial_{y,y} \tilde{\Lambda}_{\gamma}(t,Z_t^\gamma)}\,dW_t + 2 \gamma \beta^2 F_\mu(t) Y_t\,dt.
\end{align*}
Here we simplified the $dt$ term using the same computations as converting from the dual to primal PDE.  To simplify the $dW_t$ time, we used the following fact:  Recall from the proof of \pref{prop: Lambda gamma versus Lambda} that $\partial_y \tilde{\Lambda}_{\gamma} = \partial_y \Lambda \circ (\id + \gamma \partial_y \Lambda)^{-1}$, and hence
\[
\partial_{y,y} \tilde{\Lambda}_{\gamma} = \frac{\partial_{y,y} \Lambda}{1 + \gamma \partial_{y,y} \Lambda} \circ (\id + \gamma \partial_y \Lambda)^{-1},
\]
and hence
\[
\frac{1}{\partial_{y,y} \tilde{\Lambda}_{\gamma}} \circ (\id + \gamma \partial_y \Lambda) = \frac{1}{\partial_{y,y} \Lambda} + \gamma.
\]
Thus in particular,
\begin{equation} \label{eq: relating two drivers}
\frac{1}{\partial_{y,y} \tilde{\Lambda}_{\gamma}(t,Z_t^\gamma)} = \frac{1}{\partial_{y,y} \Lambda(t,Y_t)} + \gamma.
\end{equation}
Finally, we remark that some further justification is required for the It{\^o} computation above since $\partial_{y,y} \Lambda$ blows up at $\pm 1$.  One can proceed rigorously by expressing $Z_t^\gamma$ in terms of the dual process $X_t$ as
\[
Z_t^\gamma = \partial_x \tilde{\Phi}_{\gamma}(t,X_t) = \partial_x \Phi(t,X_t) + \gamma X_t.
\]
Computing $dZ_t^\gamma$ by the It{\^o} formula is justified by \cite[Proposition 22]{jagannath2016dynamic} and results in the same expression as above, as in the proof of \pref{prop:sde-equivalence-primal-dual}.

With $dZ_t^\gamma$ in hand, we now compute
\begin{align*}
d[(Z_t^\gamma - Y_t^\gamma)^2] &= 2(Z_t^\gamma - Y_t^\gamma)(dZ_t^\gamma - dY_t^\gamma) + (dZ_t^\gamma - dY_t^\gamma)^2 \\
&= 2(Z_t^\gamma - Y_t^\gamma) \left( \frac{\sqrt{2} \beta}{\partial_{y,y} \tilde{\Lambda}_{\gamma}(t,Z_t^\gamma)} - \frac{\sqrt{2} \beta}{\partial_{y,y} \tilde{\Lambda}_{\gamma}(t,Y_t^\gamma)} \right)\,dW_t + 2(Z_t^\gamma - Y_t^\gamma) \cdot 2 \gamma \beta^2 F_\mu(t) Y_t\,dt \\
 & \quad +  2 \beta^2 \left( \frac{1}{\partial_{y,y} \tilde{\Lambda}_{\gamma}(t,Z_t^\gamma)} - \frac{1}{\partial_{y,y} \tilde{\Lambda}_{\gamma}(t,Y_t^\gamma)} \right)^2\,dt.
\end{align*}
Taking the expectation and using the fact that $1/ \partial_{y,y} \tilde{\Lambda}_{\gamma}$ is $2$-Lipschitz by \pref{prop: Lambda estimates} (3), we get that
\begin{align*}
\frac{d}{dt} \E[(Z_t^\gamma - Y_t^\gamma)^2] &= 4 \beta^2 \E[(Z_t^\gamma - Y_t^\gamma) \cdot \gamma F_\mu(t) Y_t] + 2 \beta^2 \E \left( \frac{1}{\partial_{y,y} \tilde{\Lambda}_{\gamma}(t,Z_t^\gamma)} - \frac{1}{\partial_{y,y} \tilde{\Lambda}_{\gamma}(t,Y_t^\gamma)} \right)^2 \\
&\leq 2 \beta^2 \E[(Z_t^\gamma - Y_t^\gamma)^2] + 2 \beta^2 \gamma^2 \E[Y_t^2] + 8 \beta^2 \E[(Z_t^\gamma - Y_t^\gamma)^2] \\
&\leq 10 \beta^2 \E[(Z_t^\gamma - Y_t^\gamma)^2] + 2 \beta^2 \gamma^2.
\end{align*}
since $\E[Y_t^2] \leq 1$.  Therefore, since $Z_0^\gamma = 0 = Y_0^\gamma$, Gr{\"o}nwall's inequality implies that
\begin{equation} \label{eq: Y minus Z estimate}
\E[(Z_t^\gamma - Y_t^\gamma)^2] \leq \int_0^t 2 \beta^2 \gamma^2 e^{10\beta^2 (t-s)}  \,ds = \frac{1}{5} \gamma^2 (e^{10\beta^2 t} - 1).
\end{equation}
This proves the second inequality in \eqref{eq: sde-closeness 1}.  For the last claim \eqref{eq: sde-closeness 2}, note that $\partial_y\Lambda(t,Y_t) = X_t$ and $|\partial_x\Phi(x,t)| \le 1$ by \pref{prop: Phi derivative bound}.  Moreover,
\begin{align*}
d(X_t^2) &= 2 X_t \,dX_t + (dX_t)^2 \\
&= 2X_t \left( \sqrt{2} \beta\,dW_t + 2 \beta^2 F_\mu(t) \partial_x \Phi(t,X_t)\,dt \right) + 2 \beta^2 \,dt \\
&\leq 2\sqrt{2}\beta X_t \,dW_t + 2 \beta^2 (2 + X_t^2) \,dt,
\end{align*}
where we have used the inequality $2 X_t \leq 1 + X_t^2$.  Hence,
\[
\frac{d}{dt} \E[X_t^2] \leq 2 \beta^2(2 + \E[X_t^2]) \leq 10 \beta^2 (2/5 + \E[X_t^2])
\]
Therefore, since $X_0 = 0$, we have
\[
\E[X_t^2] \leq \frac{2}{5} (e^{10 \beta^2 t} - 1).
\]
Hence, by this inequality and \pref{eq: Y minus Z estimate},
\begin{align*}
\norm{Y_t^\gamma - Y_t}_{L^2} &\leq \norm{Y_t^\gamma - Y_t - \gamma X_t}_{L^2} + \gamma \norm{X_t}_{L^2} \\
&\leq \frac{\gamma}{\sqrt{5}} (e^{10 \beta^2 t} - 1)^{1/2} + \frac{\sqrt{2}}{\sqrt{5}} \gamma (e^{10 \beta^2 t} - 1)^{1/2} \\
&\leq \sqrt{2} \gamma (e^{10 \beta^2 t} - 1)^{1/2},
\end{align*}
which proves the asserted estimate \eqref{eq: sde-closeness 2}.
\end{proof}

\begin{corollary} \label{cor: L2 norm of gamma SDE weight}
Suppose \pref{ass:sk-frsb} holds, and consider the functions $\Phi$, $\tilde{\Lambda}_{\gamma}$ associated to the optimizing measure $\mu$ in the Parisi formula.  Let $Y_t$ and $Y_t^\gamma$ be as in the previous lemma.  Then we have for $t \in [0,q_\beta^*]$ that
\begin{equation} \label{eq: first L2 norm estimate}
\frac{1}{\sqrt{2} \beta} - e^{5 \beta^2 t} \gamma \leq \norm{\frac{1}{\partial_{y,y} \tilde{\Lambda}_{\gamma}(t,Y_t^\gamma)}}_{L^2} \leq \frac{1}{\sqrt{2} \beta} + (1 + e^{5 \beta^2 t})\gamma,
\end{equation}
and hence
\begin{equation} \label{eq: second L2 norm estimate}
t^{1/2} \left(1 - \sqrt{2} \beta e^{5 \beta^2 t} \gamma \right) \leq \norm{Y_t^\gamma}_{L^2} \leq t^{1/2} \left(1 + \sqrt{2} \beta (1 + e^{5 \beta^2 t})\gamma \right).
\end{equation}
\end{corollary}

\begin{proof}
Consider $Z_t^\gamma$ as in the previous proof.  From \eqref{eq: relating two drivers}, we observe that
\begin{equation} \label{eq: compare L2 norm for Zt and Yt}
\norm{ \frac{1}{\partial_{y,y} \Lambda(t,Y_t)} }_{L^2} \leq \norm{ \frac{1}{\partial_{y,y} \tilde{\Lambda}_{\gamma}(t,Z_t^\gamma)} }_{L^2} \leq \norm{ \frac{1}{\partial_{y,y} \Lambda(t,Y_t)} }_{L^2} + \gamma.
\end{equation}
Using the fact that $1 / \partial_{y,y} \tilde{\Lambda}_{\gamma}$ is $2$-Lipschitz~(\pref{prop: Lambda estimates} (3)), together with \eqref{eq: Y minus Z estimate},
\[
\norm{\frac{1}{\partial_{y,y} \tilde{\Lambda}_{\gamma}(t,Z_t^\gamma)} - \frac{1}{\partial_{y,y} \tilde{\Lambda}_{\gamma}(t,Y_t^\gamma)}}_{L^2} \leq 2 \norm{Z_t^\gamma - Y_t^\gamma}_{L^2} \leq 2 \cdot 5^{-1/2} \gamma (e^{10 \beta^2 t} - 1)^{1/2} \leq e^{5 \beta^2 t} \gamma.
\]
Hence, 
\begin{equation} \label{eq: compare L2 norm for Zt and Yt 2}
\norm{\frac{1}{\partial_{y,y} \tilde{\Lambda}_{\gamma}(t,Z_t^\gamma)}}_{L^2} - e^{5 \beta^2 t} \gamma \leq \norm{\frac{1}{\partial_{y,y} \tilde{\Lambda}_{\gamma}(t,Y_t^\gamma)}}_{L^2} \leq \norm{\frac{1}{\partial_{y,y} \tilde{\Lambda}_{\gamma}(t,Z_t^\gamma)}}_{L^2} + e^{5 \beta^2 t} \gamma.
\end{equation}
By \eqref{eq: compare L2 norm for Zt and Yt} and \eqref{eq: compare L2 norm for Zt and Yt 2}, we have
\[
\norm{\frac{1}{\partial_{y,y} \Lambda(t,Y_t)} }_{L^2} - e^{5 \beta^2 t}\gamma \leq \norm{\frac{1}{\partial_{y,y} \tilde{\Lambda}_{\gamma}(t,Y_t^\gamma)}}_{L^2} \leq \norm{\frac{1}{\partial_{y,y} \Lambda(t,Y_t)} }_{L^2} + ( 1 + e^{5 \beta^2 t}) \gamma
\]
We finally substitute in the fact that under fRSB $\norm{\frac{1}{\partial_{y,y} \Lambda(t,Y_t)}}_{L^2} = \frac{1}{\sqrt{2} \beta}$ (\pref{lem:qt-equality}) to obtain the first asserted estimate \eqref{eq: first L2 norm estimate}.

To prove the second estimate \eqref{eq: second L2 norm estimate}, observe that by the It{\^o} isometry,
\[
\norm{Y_t^\gamma}_{L^2}^2 = \int_0^t 2 \beta^2 \norm{\frac{1}{\partial_{y,y} \tilde{\Lambda}_{\gamma}(s,Y_s^\gamma)}}_{L^2}^2\,ds.
\]
Using \eqref{eq: first L2 norm estimate}, the latter can be upper-bounded by
\[
\int_0^t 2\beta^2\left(\frac{1}{\sqrt{2}\beta} + \left(1 + e^{5\beta^2 t}\gamma\right)\right)^2\,ds \leq t\left(1 + \sqrt{2} \beta (1 + e^{5 \beta^2 t})\gamma\right)^2,
\]
and the lower bound is proved similarly.
\end{proof}

\section{Extended Objective Function \& Quadratic Optimization}
\label{sec:algorithm}

In this section we will introduce the potential function that is derived directly from the generalized TAP free energy~\cite[Section 2]{chen2018generalized}. The generalized TAP free energy gives the extension of the original objective function into $[-1,1]^n$. While technically the modified objective is of interest only in the convex hull, as the increments are Gaussian, it will be the case that some small fraction of the coordinates of the final iterate will escape the cube. Consequently, the entropy term is regularized to deal with these outliers, ensuring that the modified objective is over $\R^n$ while continuing to be a uniformly near-faithful representation on $[-1,1]^n$ (see \pref{prop: Lambda gamma versus Lambda}).

\subsection{Extended objective via the generalized TAP representation}

We define the potential function for the PHA algorithm which consists of a sum of ``entropy-like'' functions for every coordinate, corrected by an average radial term which depends on time. 
For conceptual and interpretative reasons that make the analogy of the use of the generalized TAP free energy in~\pref{alg:hessian-ascent} similar to the structure of Subag's algorithm~\cite{subag2021following}, it would be preferable to have the potential function be \emph{purely} dependent on geometry (and not time).
While we believe it is perfectly plausible to work with the radial term evaluated at the normalized distance $\left(\frac{1}{n}\left|\sigma\right|^2_2\right)$, we substitute in an ``external clock'' $t \in [0,1]$ since it simplifies certain technical details in the analysis.

\begin{definition}[Potential function \& objective function]\label{def:potential-function}
Let the potential $V_{\beta}: [0,1] \times \R^n \to \R \cup \{\pm \infty\}$ be given by
\[ V_{\beta}(t,\sigma) := \left(\sum_{i \in [n]} \tilde{\Lambda}_{\gamma}(t, \sigma_i)\right) + \beta^2n\int_{t}^1F_{\mu_{\beta}}(s)\,s\,ds\,.  \]
Note that $\tilde{\Lambda}_{\gamma}$ carries an implicit dependence on $\beta$ and $\mu_{\beta}$, and that $\mu_{\beta}$ is from \pref{ass:sk-frsb}.

Given $H(\sigma) = \sigma^{\sT}A\sigma$, let the objective function $\mathrm{obj}: [0,1] \times \R^n \to \R \cup \{\pm \infty\}$ be
\[ \obj(t,\sigma) := \beta\, H(\sigma) - V_{\beta}(t,\sigma) \,.\]
\end{definition}

We will evaluate $t$ as given by the step-size $\eta$ scaled by the iterate number $j \in [K]$. Below, the spatial derivatives of the potential function are introduced, which are important during the Taylor expansion analysis calculation~(see~\pref{thm:taylor-bound}) to track the change in the modified objective function.
\begin{corollary}[Spatial derivatives of time-dependent $V$]\label{cor:time-potential-derivatives} \mbox{} \\
    Let $V_\beta(t,\sigma_t)$ be the time-dependent potential function defined in~\pref{def:potential-function}.
    Then, the following expressions hold for the partial derivatives (with respect to space) of $V_{\beta}(t,\sigma_t)$, where $e_i$ is the $i$th standard basis vector of $\R^n$:
    \begin{align*}
        \nabla V_{\beta}(t,\sigma_t) &= \sum_{i \in [n]}\partial_y \tilde{\Lambda}_{\gamma}(t,(\sigma_t)_i)e_i\, , \\
        \nabla^2 V_{\beta}(t,\sigma_t) &= \sum_{i \in [n]}\partial_{y,y}\tilde{\Lambda}_{\gamma}(t,(\sigma_t)_i)e_ie_i^\sT\, , \\
        \nabla^3 V_{\beta}(t,\sigma_t) &= \sum_{i \in [n]}\partial_{y,y,y}\tilde{\Lambda}_{\gamma}(t,(\sigma_t)_i) e_i^{\ot 3}\,.
    \end{align*}
\end{corollary}

\subsection{Potential Hessian ascent: quadratic optimization with entropic corrections}

The algorithm generates a sequence of iterates $\sigma_k \in \R^n$ starting at $\sigma_0 := (0, \dots, 0)$.

Let $\delta$ be a precision parameter controlling how concentrated the random step will be in the top part of the spectrum of the Hessian and let $\beta := 10/\eps$.
We will take $K := \lceil\frac{q^*_\beta}{\eta}\rceil$ iterations in total, where $\eta$ is chosen to depend on the approximation parameter $\eps$~(see~\pref{prop:time-complexity}).

At each iterate $\sigma_k$, we locally optimize a quadratic approximation to the objective function $\obj(t,\sigma)$ defined in \pref{def:potential-function}.
Rather than setting the next iterate deterministically, we sample from a distribution of near-optimizers.
Specifically, we sample a Gaussian vector $u_k \sim \calN(0, Q_k^2)$ with the following covariance\footnote{In the language of resolvents, this is equivalent to
\[ Q_k^2 := -2\beta\tilde bn^{\delta}\Pi_{\sigma_k^{\perp}}\left(\Im R(\tilde{a} - i\tilde{b};\, \nabla^2\obj(k\eta, \sigma_k))\right)\Pi_{\sigma_k^{\perp}} \,.\]}:
\[
    Q_k^2 := 2\beta\tilde bn^{\delta}\Pi_{\sigma_k^{\perp}}\left(\tilde b^2 + \left(\tilde a - \nabla^2\obj(k\eta, \sigma_k)\right)^2\right)^{-1}\Pi_{\sigma_k^{\perp}} \,.
\]
where $\tilde a$ and $\tilde b$ are specific near-zero quantities controlling the sharpness of the distribution, and $\Pi_{\sigma_k^{\perp }}$ projects to the subspace orthogonal to $\sigma_k$.
To elucidate what this does, the spectral theorem in conjunction with~\pref{thm:david-magic} shows us that $Q_k^2$ has the same eigenvectors as $\nabla^2\obj(k\eta, \sigma_k)$, with transformed eigenvalues: eigenvalues of the Hessian that are close to $0$ become large eigenvalues (of order $2\beta^2$) in $Q_k^2$, while eigenvalues much smaller than $0$ in the Hessian become vanishing in $Q_k^2$.

We then set the update to decide the next iterate as,
\[\sigma_{k+1} := \sigma_k + \eta^{1/2} u_k\,.\]

In the end, to obtain a point in $\{-1,+1\}^n$, we first truncate the coordinates to be in $[-1,1]$ individually and then sample the $i$th coordinate to be $+1$ with probability $(1+(\sigma_K)_i)/2$ and $-1$ with probability $(1-(\sigma_K)_i)/2$.

\begin{algorithm}
\caption{Potential Hessian Ascent}\label{alg:hessian-ascent}
\begin{algorithmic}[1]
\Procedure{Potential\_Hessian\_Ascent}{$A, \eps, V, q^*$}\Comment{The entries of $A$ are iid $\mathcal{N}(0,1/n)$.}
\State $\beta \gets 10/\eps$
\State $\eta \gets \Theta(e^{-\beta^2})$
\State $K \gets \lceil q^*(\beta)/\eta\rceil$ 
\State $\obj(t,\sigma) \gets \beta\sigma^{\sT}A\sigma - V_{\beta}(t,\sigma)$
\State $\sigma_0 \gets (0,\dots, 0)$
\For{$k\in \{0, \dots, K-1 \}$}
\State Compute $\tilde a \in \R$ and $\tilde b \in \R$ as in \pref{thm:david-magic}.
\State $Q^2_k \gets 2\beta\tilde bn^{\delta}\Pi_{\sigma_k^{\perp}}\left(\tilde b^2 + \left(\tilde a - \nabla^2\obj(k\eta,\sigma_k)\right)^{2}\right)^{-1}\Pi_{\sigma_k^{\perp}}$
\State $u_k \sim \calN(0, Q_k^2)$\label{line:uk-sample}
\State $\sigma_{k+1} \gets \sigma_k + \eta^{1/2}u_k$
\EndFor
\For{$i \in \{1,\dots, n\}$}
\State $z_i \sim 2\cdot \operatorname{Ber}\left(\frac{1+\operatorname{trunc}\left((\sigma_K)_i\right)}{2}\right) - 1$
\EndFor
\State \textbf{return} $z$
\EndProcedure
\end{algorithmic}
\end{algorithm}

\begin{lemma}[{Fast matrix square root~\cite[Corollary 1]{pleiss2020fast}}]
\label{lem:fast-matrix-square-root}
  For every $J, m \in \N$, given $M \in \R^{n\times n}$ satisfying $M \succeq 0$ and $v \in \R^{n}$, it is possible to compute $u_J$ satisfying
  \[ \left|u_J - M^{-1/2}v\right|_2 \le O\left(\frac{1}{\lambda_{\min}}\exp\left(-\frac{\Omega(m)}{\log \kappa + 3}\right)\right) 
  + O\left(\frac{m\kappa\log(\kappa)}{\sqrt{\lambda_{\min}}}\right)\left(\frac{\sqrt{\kappa}-1}{\sqrt{\kappa}+1}\right)^{J-1}\left|v\right|_2
  \]
  in time $O(Jn^2)$, where the runtime bottleneck comes from exactly $J$ matrix-vector product computations with the matrix $M$, $\lambda_{\min}$ is the smallest eigenvalue of $M$, and $\kappa$ is the condition number of $M$, defined as the ratio between the largest and smallest eigenvalues of $M$.
\end{lemma}

\begin{corollary}
\label{cor:fast-matrix-sqrt}
    Given $M\in\R^{n\times n}$ satisfying $M \succeq 0$, $v \in \R^{n}$, and $\eps' > 0$, it is possible to compute a vector $u$ satisfying
  \[ \left|u - M^{-1/2}v\right|_2 \le \eps'\left|v\right|_2 \]
  in time $O(\sqrt{\kappa} n^2\operatorname{polylog}(\kappa,1/\eps', 1/\left|v\right|_2,\lambda_{\min}))$, where $\kappa$ is the condition number of $M$ and $\lambda_{\min}$ is the least eigenvalue of $M$.

  As a result, if $g \sim \calN(0,I_n)$ so that $h = M^{-1/2}g$ is a sample from $\calN(0, M^{-1})$, when $g$ and $M$ are given explicitly, it is possible to compute $h'$ satisfying $|h' - h|_2 \le \eps'|h|_2$ in time $O(\sqrt{\kappa} n^2\operatorname{polylog}(\kappa,1/\eps', 1/\left|g\right|_2,\lambda_{\min}))$.
\end{corollary}
\begin{proof}
    Invoke \pref{lem:fast-matrix-square-root} with $J = \sqrt{\kappa}\log\left (\frac{m\kappa\log\kappa}{\eps'\sqrt{\lambda_{\min}}}\right)$ and with $m = (\log \kappa + 3)\log(1/(\eps'\lambda_{\min}\left|v\right|_2))
    $.
\end{proof}

\begin{proposition}[Time complexity of~\pref{alg:hessian-ascent}]\label{prop:time-complexity}\mbox{} \\
    Assume that $A$ satisfies the high probability event stated in \pref{thm:david-magic} \pref{item:4-approx-diag}, which happens with probability $1 - C_1 \exp(-C_2 n)$.  
    Then, there is an implementation of \pref{alg:hessian-ascent} that runs in time $O\left(n^{2+4\delta}e^{O(1/\eps^{2})}\operatorname{polylog}(n,1/\eps)\right)$ where \pref{line:uk-sample} generates $u_k$ satisfying $|u_k - h_k|_2 \le O_{\eps}(n^{-100})$ with overwhelming probability, where $h_k \sim \calN(0,Q_k^2)$.
    This implementation maintains the same algorithmic guarantees with only an $o(1)$ loss in the value of $H(\sigma^*)/n$, where $\sigma^*$ is the output of the algorithm.
\end{proposition}
\begin{proof}
  The runtime bottleneck will be sampling $u_k$ approximating $h_k \sim \calN(0, Q_k^2)$, which may be characterized as $h_k := L_kg_k$ for $g_k \in \R^n$ a vector with i.i.d. standard Gaussian entries, where
    \begin{align*}
        L_k := \sqrt{2\beta \tilde bn^{\delta}}\;\Pi_{\sigma_k^{\perp}}\left(\tilde{b}^2 + (\tilde a - \nabla^2\operatorname{obj}(k\eta,\sigma_k))^2\right)^{-1/2}
    \end{align*}
    so that $L_kL_k^{\sT} = Q_k^2$.
    To apply \pref{cor:fast-matrix-sqrt}, we will note that since $\nabla^2\obj(k\eta, \sigma_k) = \beta(A + A^{\sT}) - \nabla^2V(k\eta, \sigma_k)$, matrix-vector multiplication by the matrix $K_k := \tilde{b}^2 + (\tilde a - \nabla^2\obj(k\eta, \sigma_k))^2$ can be computed in $O(n^2)$ time.
    Furthermore, the spectrum of $K_k$ is lower-bounded by $\tilde b^2$ and it is upper-bounded by $\tilde b^2 + (|\tilde a| + \opnorm{\nabla^2\obj(k\eta, \sigma_k)})^2 \le \exp(O(\beta^2))$ since $\opnorm{\nabla^2\obj(k\eta, \sigma_k)} \le 3\beta + \gamma^{-1} \le 3\beta + e^{O(\beta^2)}$.

    Therefore, by \pref{cor:fast-matrix-sqrt}, there is a procedure that computes a vector $u_k\in \R^n$ satisfying $|u_k - h_k|_2 \le \sqrt{2\beta\tilde b n^\delta}\,\eps' |K_k^{-1/2}g_k|_2$ where $h_k \sim \calN(0, Q_k^2)$ and $|K_k^{-1/2}g_k|_2$ is upper-bounded by $O_{\eps}(n)$ with overwhelming probability in time $O(\tilde b^{-1}e^{O(\beta^2)}n^{2}\operatorname{polylog}(n, 1/\eps'))$.
    By \pref{lem: D dependent resolvent}, $\tilde b \ge \Omega(b^3/\beta^2) \ge \Omega(\beta n^{-3\delta})$ when $b = \beta n^{-\delta}$.
    We may take $\eps'$ to be $O(n^{-102})$ to obtain  $|u_k - h_k|_2 \le O(n^{-100})$ with high probability while incurring only an additional logarithmic dependence on $n$ in runtime.

    The value of $a$ in \pref{thm:david-magic} can be found by binary search in $O(n\operatorname{polylog}(n,1/\eps))$ time, since the step of computing the trace of $D^{-2}$ is linear time, with all matrices involved being diagonal.
    Similarly, to find $\tilde a$ and $\tilde b$, it is only necessary to compute the traces and inverses of diagonal matrices.

    The error between $u_k$ and $h_k$ accumulates linearly over the iterates in the analysis of $H(\sigma)$ and its rounding $H(\sigma^*)$ in \pref{sec:energy-analysis}.
\end{proof}

The next proposition deals with the rounding scheme, which given a point $\sigma$ in the solid cube finds a corner $\sigma^*$ of the cube with close to the same value for the function we want to maximize.  The value of each coordinate $\sigma_j^*$ is a Bernoulli distribution on $\{\pm 1\}$ with mean $\sigma_j$, and we control the error with high probability using Hoeffding's inequality.  While this proposition assumes $\sigma \in [-1,1]^n$, if the vector is merely \emph{close} to $[-1,1]^n$, one can truncate it first and then apply the proposition. This is the chosen approach in \pref{sec:energy-analysis}.

\begin{proposition}[Small fluctuations to the energy under rounding]\label{prop:rounding-energy}\mbox{} \\
Let $M$ be a real symmetric matrix.  Let $\sigma \in [-1,1]^n$ be a fixed vector.  Let $\sigma^*$ be a random vector such that the coordinates $\sigma_j^*$ are independent for $j = 1$, \dots, $n$ and
\[
\Pr_{\sigma^*}(\sigma_j^* = 1) = \frac{1 + \sigma_j}{2}, \qquad \Pr_{\sigma^*}(\sigma_j^* = -1) = \frac{1 - \sigma_j}{2}.
\]
Then there is an absolute constant $c$ such that, for $\alpha \in [0,1/2]$, we have
\[
    \Pr_{\sigma^*}\left(\angles{\sigma^*, M \sigma^*} \geq \angles{\sigma, M \sigma} - 4 \norm{M} n^{1-\alpha} - \norm{M}_2n^{1-\alpha}/\sqrt{c} - \norm{M}n^{1-2\alpha}/c - \sum_{i \in [n]}|M_{i,i}| \right) \geq 1 - 2\exp(-n^{1-2\alpha}).
\]

If $A$ and $H$ as given in \pref{def:sk-model} satisfy the high-probability conditions stated in \pref{lem: GOE operator norm bound} or \pref{thm:david-magic} \pref{item:4-approx-diag}, and if $A$ also satisfies $|A_{i,i}| < 3n^{-\alpha}/2$ for all $i$, which occurs with probability at least $1-ne^{-n^{2(1-\alpha)}}$, then
\[
    \Pr_{\sigma^*}\left(H(\sigma^*) \geq H(\sigma) - \frac{3}{2}n^{1-\alpha}\left(5 + 1/\sqrt{c} + n^{-\alpha}/c \right)\right) \geq 1 - 2\exp(-n^{1-2\alpha}).
\]

\end{proposition}

\begin{proof}
Note that
\[
\E[\sigma_j^*] = \frac{1+\sigma_j}{2} - \frac{1 - \sigma_j}{2} = \sigma_j.
\]
In other words, $\E[\sigma^*] = \sigma$.  Let $\tilde{\sigma} = \sigma^* - \E \sigma^* = \sigma^* - \sigma$.  Write
\[
\angles{\sigma^*, M \sigma^*} = \angles{\tilde{\sigma}, M \tilde{\sigma}} + 2 \angles{\tilde{\sigma}, M \sigma} + \angles{\sigma, M \sigma}.
\]
Since each component of $\tilde{\sigma}$ is independent, zero-mean, and has sub-Gaussian norm at most 1, by the Hanson-Wright inequality~\cite[Theorem 6.2.2]{vershynin2018high} there is a constant $c > 0$ such that
\[\Pr\left(\angles{\tilde{\sigma}, M \tilde{\sigma}} - \E[\angles{\tilde{\sigma}, M \tilde{\sigma}}] < -t_{\mathrm{HW}}\right) \le \exp\left(-c\min\left(
\frac{t_{\mathrm{HW}}^2}{n\norm{M}_2^2},\frac{t_{\mathrm{HW}}}{\norm{M}}\right)\right)\,.\]
So take $t_{\mathrm{HW}} = \max(\norm{M}_2n^{1-\alpha}/\sqrt{c},\norm{M} n^{1-2\alpha}/c)$.
Then we have

\[ \E[\angles{\tilde{\sigma}, M \tilde{\sigma}}] = \angles{M, \E[\tilde{\sigma}\tilde{\sigma}^{\sT}]} = \sum_{i \in [n]} M_{i,i}(1-\sigma_i^2) \ge -\sum_{i \in [n]}|M_{i,i}|\,. \]

Furthermore,
\[
\angles{\tilde{\sigma}, M \sigma} = \sum_{j=1}^n (M\sigma)_j \tilde{\sigma}_j.
\]
The random variables $(M \sigma)_j \tilde{\sigma}_j$ are independent with mean zero and each one is supported in $\{(M \sigma)_j(-1 + \sigma_j), (M \sigma)_j(1 + \sigma_j)\}$.  Therefore, by Hoeffding's inequality, for $t > 0$,
\[
\Pr(\angles{\tilde{\sigma}, M \sigma} \leq -t) \leq \exp\left( -\frac{t^2}{\sum_{j=1}^n (2(M \sigma)_j)^2} \right) = \exp\left(- \frac{t^2}{4 |M\sigma|_2^2} \right) \leq \exp\left(- \frac{t^2}{4 n\norm{M}^2} \right),
\]
since $|\sigma|_2^2 \leq n$.  Now take $t = 2 \norm{M}n^{1-\alpha}$.  Thus, $\angles{\tilde{\sigma}, M \sigma} \geq -t$ implies that
\[
\angles{\sigma^*, M \sigma^*} \geq \angles{\sigma, M \sigma} - 4 \norm{M} n^{1-\alpha} - \max\left(\norm{M}_2n^{1-\alpha}/\sqrt{c},\norm{M} n^{1-2\alpha}/c\right) -\sum_{i \in [n]}|M_{i,i}|, 
\]
and the probability of violating this bound simplifies to $\exp(-n^{1 - 2 \alpha})$.
\end{proof}

\section{Random Matrix Analysis of the Modified Hessian}\label{sec:free-prob-for-hessian}

Our goal in this section is to define an operator that will serve as an approximate spectral projection for the entropy-corrected Hessian. Recall the entropy-shifted Hessian at a point $\sigma \in [-1,1]^n$ and time $t \in [0,1]$ has the form $\beta(A + A^{\sT}) - \diag(\partial_{y,y} \Lambda(t,\sigma_j))$.  Here $A$ is a non-symmetric real Gaussian matrix where the entries are independent with mean zero and variance $1/n$, and $A + A^{\sT}$ is $\sqrt{2}$ times a standard GOE matrix.  We set $A_{\sym} = (A + A^{\sT})/2$.  In this section, the specific form of the entropy correction will not play much role, and so we proceed more generally to study $2 \beta A_{\sym} - D$ for any positive diagonal matrix $D$.  We will, however, assume that $2 \beta^2 \tr_n(D^{-2}) = 1$.  Later in \S \ref{sec:convergence-to-ac} we will show that under our fRSB assumption this will be approximately true for the Hessian at the points in our algorithm, and we will be able to make it exactly true by a multiplicative renormalization of the diagonal matrix.

The large-$n$ behavior of the spectrum of $2 \beta A_{\sym} - D$ for any deterministic matrix with operator norm $O(1)$ can be described using the tools of free probability, which have already been used for the spectral analysis of the TAP equations in \cite{gufler2023concavity}; the authors in \cite[p.~7]{adhikari2021dynamical} also suggested to use random matrix theory to describe the diagonal entries of the resolvent of $2 \beta A_{\sym} - D$.  Voiculescu's asymptotic freeness theory \cite{voiculescu1991limit-laws} tells us that the matrices $2 \beta A_{\sym}$ and $D$ will be asymptotically freely independent, and their spectral measure will be well-approximated by the free convolution of the semicircular distribution with the spectral distribution of $D$.  Furthermore, in many situations, with high probability, there are no eigenvalues outside the bulk spectrum (see e.g.\ \cite[\S 5.5]{anderson2010introduction}), and so we expect that computations done in the free limit accurately reflect the maximum of the spectrum.  Free probabilistic tools have already been used in \cite{guionnet2005longtime} for the study of Kraichnan equations related to the Sherrington-Kirkpatrick model.

Concretely, our goal is the following:
\begin{quote}
    \emph{Construct an operator $P$ to approximately project onto the maximum part of the spectrum of $2 \beta A_{\sym} - D$.}
\end{quote}   
Letting $\lambda$ (which is $\tilde{a}(D)$ below) be the maximum of the spectrum of a free semicircular minus $D$ (the idealized spectrum for large $n$), and then for small $\rho > 0$ consider the operator $P(D)$ given by
\[
P(D)^2 = -\im (\lambda + i \rho - 2 \beta A_{\sym} + D)^{-1} = \rho (\rho^2 + (\lambda - 2 \beta A_{\sym} + D)^2)^{-1}.
\]
The function $t \mapsto \rho / (\rho^2 + (\lambda - t)^2)$ will as $\rho \to 0$ be highly spiked around the point $\lambda$.  Hence, applying this function to the operator $2 \beta A_{\sym} - D$ should output approximate eigenvectors for the top part of the spectrum.  However, one challenge is that we expect the spectral density for $2 \beta A_{\sym} - D$ to vanish at the boundary.  This means that $\tr(P(D)^2)$ will approach zero as $\rho \to 0$.  It is thus essential to arrange the weight $P(D)^2$ assigns away from $\lambda$ to go to zero faster than the trace of $P(D)^2$ as $\rho \to 0$.  This requires some careful estimates.  Furthermore, $\rho$ will depend on $n$; more precisely, we will set $\rho \approx n^{-3\delta}$ for a fixed small positive $\delta$.

In order to obtain convergence of our algorithm to the primal Auffinger-Chen SDE, we will additionally need to show that the diagonal entries of $P(D)^2$ behave like a scalar multiple of $D^{-2}$.  We will study the projection of $(\lambda + i \rho - 2 \beta A_{\sym} + D)^{-1}$ onto the diagonal using the idea that $A_{\sym}$ is approximately freely independent of the diagonal subalgebra as well as a magic formula for conditional expectations of the resolvent of the sum of two freely independent operators from free probability (see \pref{thm: subordination} below).

Another technical detail is that we want all our estimates to hold with high probability in $A$, simultaneously for \emph{all} choices of diagonal matrix $D$ that we input; this is because the same sample of the  Gaussian matrix $A$ is being used through the whole algorithm, while the point $\sigma \in [-1,1]$ is being updated at each step.  Thus, we want \emph{uniform} asymptotic freeness results that hold for all diagonal matrices simultaneously.  This can be achieved using union bounds and concentration of measure for the Gaussian matrix $A$, which is a standard technique in random matrix theory and probability more generally.  Since the space of diagonal matrices is $n$-dimensional, the number of diagonal matrices in a sufficiently dense subset will only be exponential in $n$, but the concentration bounds will be exponential in $-n^2$ and hence overwhelm the number of test points that we union-bound.  The same idea of ``uniform asymptotic freeness from the diagonal'' was applied by the first author and Srivatsav Kunnawalkam Elayavalli in another context in  \cite{jekel2026upgraded}.

The following theorem is a summary of the results of this section which will be used in the rest of the paper.  Background on free independence will be given in \S \ref{subsec: free probability background}.  Recall also the notation pertaining to matrix algebras in \S \ref{subsec: notation}.

\begin{theorem}\label{thm:david-magic}
Let $\beta \geq 1$ and $\delta \in (0,1/17]$ and $n \in \N$.  Let $b = \beta n^{-\delta}$.  Write $g_{-D}(z) = \tr_n((z+D)^{-1})$, let
\[
\tilde{z} := ib + 2 \beta^2 g_{-D}(ib),
\]
and let $\tilde{a}$ and $\tilde{b}$ be the real and imaginary parts of $\tilde{z}$.  For diagonal matrices $D \geq 0$ with $2 \beta^2 \tr(D^{-2}) = 1$, let $P = P(D)$ be the matrix
\[
P = \sqrt{\tilde{b}(\tilde{b}^2 + (\tilde{a} - 2 \beta A_{\sym} + D)^2)^{-1}}.
\]
Then for some absolute constants $C_1$, $C_2$, $\dots > 0$, we have the following properties.
\begin{enumerate}
    \item \textbf{Maximum of idealized spectrum:} $2 \beta^2 \tr_n(D^{-1})$ is the maximum of the spectrum of $\sqrt{2} \beta S - D$ where $S$ is a semicircular operator free from $D$.  Moreover,
    \[
    |2 \beta^2 \tr_n(D^{-1}) - \tilde{a}| \leq 2 \beta^4 n^{-2\delta} \tr_n(D^{-3}).
    \]
    \label{item:4-maxspec}
    \item \textbf{Approximate eigenvector condition:}
    \[
    \tr_n(P^2 (\tilde{a} - 2 \beta A_{\sym} + D)^2 ) \leq 2 \beta^5 n^{-3\delta} \tr_n(D^{-4}) \leq 2 \beta^3 n^{-3\delta} \norm{D^{-1}}^2.
    \]
    \label{item:4-approx-eigenvec}
    \item \textbf{Approximation for diagonal:}  With probability at least
    \[
    1 - C_1 \exp(-C_2 n)
    \]
    in the Gaussian matrix $A$, we have that $\norm{2A_{\sym}} \leq 3$ and for all nonnegative diagonal matrices with $2 \beta^2 \tr(D^{-2}) = 1$,
    \[
    \norm{E_{\mathcal{D}_n}[P^2] - b(b^2 + D^2)^{-1}}_2 \leq C_3 \beta n^{-2\delta}
    \]
    \label{item:4-approx-diag}
    \item \textbf{Lower bound for trace:} We have
    \[
    \tr_n[b(b^2 + D^2)^{-1}] \geq \frac{1}{2} \beta^{-1} n^{-\delta} - \beta \norm{D^{-1}}^2 n^{-3\delta}
    \]
    so in particular
    \[
    \tr_n[P^2] \geq C_4 \beta^{-1} n^{-\delta}  - C_3 \beta n^{-2\delta} - \frac{1}{2} \beta \norm{D^{-1}}^2 n^{-3\delta}.
    \]
    \label{item:4-trace-lower-bound}
    \item \textbf{Upper bound for operator norm:}
    \[
    \norm{P^2} \leq C_5 \beta^{-1} n^{3\delta}.
    \]
    \label{item:4-op-upper-bound}
\end{enumerate}
\end{theorem}

The theorem is based on replacing the Gaussian matrix $\sqrt{2} A_{\sym}$ with a semicircular operator $S$ freely independent of $D$, an object that describes the idealized large-$n$ limit of a GOE matrix.  Point (1) of \pref{thm:david-magic} is based on spectral analysis for this idealized situation, carried out in \S \ref{subsec: analysis of free semicircular}, while point (2) follows quickly from the choice of $P$.

The crux of the proof is point (3).  We rely on the fact that $P^2$ is chosen as the operator imaginary part of the resolvent $(\tilde{a} + i\tilde{b} - 2 \beta A_{\sym} + D)^{-1}$ (see \pref{lem: real and imaginary parts}), and resolvents relate in a natural way with non-commutative conditional expectations (see \pref{thm: subordination}).  The proof is broken into several parts:
\begin{itemize}
    \item In \S \ref{subsec: analysis of free semicircular}, we analyze the projection onto the diagonal of the resolvent of $\sqrt{2} \beta S - D$.
    \item In \S \ref{subsec: concentration estimates}, we compare the diagonal of the resolvent of $2 \beta A_{\sym} - D$ with its expectation (over $A$) using Gaussian concentration inequalities.
    \item In \S \ref{subsec: evaluation of expectation}, we compare the expectation of the resolvent to the idealized semicircular version studied in \S \ref{subsec: analysis of free semicircular}.
\end{itemize}
These results are combined using the triangle inequality to obtain (3) of the theorem, and then (4) follows quickly from (3).

We remark that although we rely on $P^2$ being the imaginary part of some resolvent at a point on the complex plane, the matrix $P$ itself always has real entries since $\sqrt{\tilde{b}(\tilde{b}^2 + (\tilde{a} - 2 \beta A_{\sym} + D)^2)^{-1}}$ where $A_{\sym}$ and $D$ have real entries.  Thus, while we use complex matrices in the proof of Theorem \ref{thm:david-magic}, which is natural for the free probability setting, we obtain in the end a real matrix to use as the covariance matrix for the real Gaussian vectors in our algorithm.

We do not care at present about the specific values of the numerical constants in the theorem, but something like the Big $O$ notation is insufficient for the logical and algebraic manipulations.  We will therefore call the numerical constants in the statements $C_1$, $C_2$, \dots.  However, each statement will have its own ``scope'' for numbering constants, thus for instance the $C_1$ in one lemma is not the same as the $C_1$ in another lemma.  Since in the course of the proofs we will need to define various constants as well, we will call the constants in the proofs $M_1$, $M_2$, \dots, and the constants $C_j$ in the statements will be defined in terms of these $M_j$'s.  Thus, each proof will also have its own scope for numbering constants.

\subsection{Free probability background} \label{subsec: free probability background}

Before going into the proof of \pref{thm:david-magic}, we review some background in free probability.

We consider the setting of \emph{tracial non-commutative probability spaces}, a non-commutative analog of a probability space.  As a motivation, recall that for a probability measure space $(\Omega,\mathcal{F},P)$, the bounded random variables form the space $L^\infty(\Omega,\mathcal{F},P)$, and the expectation $E = \int (\cdot)\,dP$ defines a linear map $L^\infty(\Omega,\mathcal{F},P) \to \C$.  We can replace $L^\infty(\Omega,\mathcal{F},P)$ with a non-commutative algebra that has similar properties.  For instance, we could consider the $n \times n$ complex matrices $M_n(\C)$ as a space of ``random variables'' and the normalized trace $\tr_n$ as the ``expectation.''  The distribution of a self-adjoint random variable $X$ would then be interpreted as its spectral measure $(1/n) \sum_{j=1}^n \delta_{\lambda_j}$.  More generally, $M_n(\C)$ can be replaced by an infinite-dimensional object with similar properties, namely a von Neumann algebra.  Since the subtleties of the theory of von Neumann algebras are not needed in this paper, we present the definitions here and refer the reader to the introductory texts \cite{zhu1993introduction,jones1997introduction,ADP} and the reference books \cite{takesaki2002theory,kadison1983fundamentals,blackadar2006operator} for further information.

\subsubsection*{Von Neumann algebras:}  A \emph{von Neumann algebra} $\mathcal{M} \subseteq B(H)$ is a subset of bounded operators on a Hilbert space $H$ with the following properties:
\begin{enumerate}
    \item $\mathcal{M}$ is a unital $*$-subalgebra, i.e., it is closed under addition, multiplication, scaling, and adjoints, and it contains the identity.
    \item $\mathcal{M}$ is closed in the weak operator topology, i.e., if $T_i$ is a net in $\mathcal{M}$ and $\la \xi, T_i \eta \ra \to \la \xi, T\eta \ra$ for all $\xi, \eta \in H$, then $T \in \mathcal{M}$.
\end{enumerate}

The analog of the expectation (or integration with respect to a probability measure) will be \emph{tracial state} on $\mathcal{M}$.  A linear functional $\tau: \mathcal{M} \to \C$ is said to be
\begin{enumerate}
    \item \emph{normal} if it is continuous with respect to the weak operator topology,
    \item \emph{unital} if $\tau(1) = 1$,
    \item \emph{tracial} if $\tau(XY) = \tau(YX)$ for $X, Y \in \mathcal{M}$,
    \item \emph{positive} if $\tau(X^*X) \geq 0$ for $X \in \mathcal{M}$,
    \item \emph{faithful} if $\tau(X^*X) = 0$ implies that $X = 0$.
\end{enumerate}
In particular, $\tau$ is called a \emph{state} if it is positive and unital.  Note that the normalized trace $\tr_n$ is an example of a faithful normal tracial state on $M_n(\mathbb{C})$.

\subsubsection*{Non-commutative probability spaces:}  A \emph{tracial non-commutative probability space} is a pair $(\mathcal{M},\tau)$ where $\mathcal{M}$ is a von Neumann algebra and $\tau$ is a faithful normal tracial state.  This should be viewed as a non-commutative analog of $L^\infty(\Omega,\mathcal{F},P)$ with its expectation and as an infinite-dimensional analog of $M_n(\C)$ with its trace.  Just as in those examples, self-adjoint elements in $\mathcal{M}$ have a well-defined distribution.  If $X$ in $\mathcal{M}$ is self-adjoint, then there is a unique compactly supported measure $\mu$ such that
\[
\tau[f(X)] = \int_{\R} f\,d\mu,
\]
for polynomials $f$, and we call $\mu$ the \emph{(spectral) distribution of $X$}.

\subsubsection*{Non-commutative $L^p$ spaces:}  Given a tracial non-commutative probability space $(\mathcal{M},\tau)$ and $p \in [1,\infty)$, we define the $L^p$ norm $\norm{X}_p = \tau(|X|^p)^{1/p}$ where $|X| = (X^*X)^{1/2}$.  Also, set $\norm{X} = \norm{X}_\infty$ to be the operator norm.  For instance, in the case of a matrix algebra with $\tr_n$, $p = 1$ gives the normalized trace norm and $p = 2$ gives the normalized Hilbert-Schmidt norm.  These norms satisfy the non-commutative H\"older's inequality:  If $1/p = 1/p_1 + \dots + 1/p_k$, then
\[
\norm{x_1 \dots x_k}_p \leq \norm{x_1}_{p_1} \dots \norm{x_k}_{p_k}.
\]
We will frequently use this in the cases where $p_j$'s are $1$, $2$, or $\infty$.  We also have $\norm{x}_\infty = \lim_{p \to \infty} \norm{x}_p$.

We denote by $L^p(\mathcal{M},\tau)$ the completion of $\mathcal{M}$ with respect to $\norm{\cdot}_p$.  In particular, $L^2(\mathcal{M},\tau)$ is a Hilbert space with inner product $\la x,y \ra_\tau = \tau(x^*y)$.  Moreover, left multiplication by $x$ defines a representation of $\mathcal{M}$ on $L^2(\mathcal{M},\tau)$, i.e.\ a $*$-homomorphism $\mathcal{M} \to B(L^2(\mathcal{M},\tau))$.  For further background on non-commutative $L^p$ spaces, see \cite{pisier2003noncommutative,dixmier1953formes,dasilva2018lecture}.

\subsubsection*{Conditional expectation:}  Let $(\mathcal{M},\tau)$ be a tracial non-commutative probability space and let $\mathcal{A} \subseteq \mathcal{M}$ be a von Neumann subalgebra of $\mathcal{M}$.  Then we can view $L^2(\mathcal{A},\tau|_{\mathcal{A}})$ as a subspace of $L^2(\mathcal{M},\tau)$.  The orthogonal projection $E_{\mathcal{A}}$ onto $L^2(\mathcal{A},\tau)$ restricts to a mapping $\mathcal{M} \to \mathcal{A}$, and in fact it is a contraction in all of the $L^p$ norms.  We call $E_{\mathcal{A}}$ the \emph{canonical conditional expectation onto $\mathcal{A}$}.  One example we will use in the sequel is $\mathcal{M} = M_n(\mathbb{C})$ and $\mathcal{A}$ the diagonal subalgebra $\mathcal{D}_n$.  Then $E_{\mathcal{D}_n}(X)$ is simply the diagonal matrix obtained by zeroing out the off-diagonal entries of $X$.

\subsubsection*{Free independence:}  Tracial non-commutative probability space $(\mathcal{M},\tau)$ admit a non-commutative version of independence, known as \emph{free independence}, studied in \cite{voiculescu1985symmetries,voiculescu1986addition}.  For introductory textbooks on the topic, see \cite{voiculescu1992free-random}, \cite[\S 5]{anderson2010introduction}, \cite{mingo2017free}.  Let $(\mathcal{M},\tau)$ be a tracial non-commutative probability space.  Let $\mathcal{A}_1$, \dots, $\mathcal{A}_d$ be $*$-subalgebras.  We say that $\mathcal{A}_1$, \dots, $\mathcal{A}_d$ are \emph{freely independent} if whenever $i_1$, \dots, $i_k \in [d]$ with $i_1 \neq i_2 \neq i_3 \neq \dots \neq i_k$ and $X_j \in \mathcal{A}_{i_j}$, we have
\[
\tau \left[ (X_1 - \tau(X_1)) \dots (X_k - \tau(X_k)) \right] = 0.
\]
Similarly, elements, tuples, or sets in $\mathcal{M}$ are said to be freely independent if the $*$-subalgebras that they generate are freely independent.  It is always possible to construct freely independent copies of any given families of non-commutative random variables; given any tracial non-commutative probability spaces $(\mathcal{A},\tau_{\mathcal{A}})$ and $(\mathcal{B},\tau_{\mathcal{B}})$, one forms the free product $(\mathcal{A} * \mathcal{B}, \tau_{\mathcal{A}} * \tau_{\mathcal{B}})$; see e.g.\ \cite[\S 1.5]{voiculescu1992free-random}, \cite[\S 5.3.1]{anderson2010introduction}.

\subsubsection*{Asymptotic free independence:}  We next recall Voiculescu's results on asymptotic freeness \cite{voiculescu1991limit-laws,voiculescu1998strengthened}, which shows that GOE and GUE matrices are asymptotically free from deterministic matrices in the large-$n$ limit; the same applies to random matrices whose distributions are invariant under unitary or orthogonal conjugations, and to many matrices with independent entries.  The limiting spectral distribution of each GOE or GUE matrix is given by Wigner's semicircle law.  Let's state an instance of this theorem relevant to our case; for proof see e.g.\ \cite[Theorem 5.4.2]{anderson2010introduction}.

\begin{theorem}[Asymptotic freeness for GOE and deterministic matrix] \label{thm: asymptotic freeness}
Let $Z_n$ be a GOE random matrix, and let $D_n$ be a deterministic matrix with $\norm{D_n} \leq M$ for some constant $M$.  Assume that the empirical spectral distribution $\mu_{D_n}$ converges to some $\mu$ as $n \to \infty$.  Consider a tracial non-commutative probability space $\mathcal{M}$ generated by two freely independent self-adjoint elements $S$ and $D$, with spectral distributions $(1/2\pi) \mathbbm{1}_{[-2,2]}(t) \sqrt{4 - t^2}\,dt$ and $\mu$ respectively.  Then we have almost surely that for non-commutative polynomials $f$ in two variables,
\[
\lim_{n \to \infty} \tr_n[f(Z_n,D_n)] = \tau[f(S,D)].
\]
\end{theorem}

\subsubsection*{Sum of freely independent variables:} An extensive theory has been developed around the sum of two variables $X$ and $Y$ that are freely independent, which in this paper we will apply with $X$ being a semicircular variable $\sqrt{2} \beta S$ and $Y$ being $-D$ for a positive diagonal matrix.  If $X$ and $Y$ have distributions $\mu$ and $\nu$ respectively, then the distribution of $X+Y$ (which is uniquely determined by $\mu$ and $\nu$) is called the \emph{free convolution of $\mu$ and $\nu$} and is denoted $\mu \boxplus \nu$.  Free convolution is computed using several complex-analytic functions related to the measure $\mu$.

\begin{definition}[Cauchy-Stieltjes Transform]
The \emph{Cauchy-Stieltjes transform} of a probability measure $\mu$ is the function
\[
g_\mu(z) = \int_{\R} \frac{1}{z - x}\,d\mu(x),
\]
defined for $z \in \mathbb{C} \setminus \supp(\mu)$. 
 Similarly, for a self-adjoint $X$ in $\mathcal{M}$, we write
\[
g_X(z) = \tau[(z - X)^{-1}] \text{ for } z \in \mathbb{C} \setminus \Spec(X),
\]
which is easily seen to be the Cauchy-Stieltjes transform of the spectral distribution of $X$.
Here, $\Spec(X)$ is the set of eigenvalues of $X$.
\end{definition}

The following properties of the Cauchy-Stieltjes transform are well known and can be found in most textbooks that prove the spectral theorem for self-adjoint operators on Hilbert space, for instance \cite[\S 6,~Thm.~3]{akhiezer1963theory}.

\begin{proposition}[Properties of Cauchy-Stieltjes transform]
\label{prop:cs-transform}
Let $\mu \in \mathcal{P}(\R)$.  Let $\supp(\mu)$ denote the closed support of $\mu$.  For $S \subseteq \mathbb{C}$ and $z \in \mathbb{C}$, recall that $d(z,S) := \inf_{w \in S} |z - w|$.
\begin{enumerate}
    \item $g_\mu(z)$ maps the upper half-plane into the lower half-plane and vice versa.
    \item $\displaystyle |g_\mu(z)| \leq \frac{1}{d(z,\supp(\mu))} \leq \frac{1}{|\im z|}$.
    \item A point $a \in \R$ is \emph{not} in the support of $\mu$ if and only if there is an analytic function $f$ defined in a neighborhood $O$ of $a$ that agrees with $g_\mu$ on $O \setminus \R$.
    \item If $\mu$ is compactly supported, then
    \[
    \lim_{z \to \infty} z g_\mu(z) = 1.
    \]
    In particular, $g_\mu^{-1}$ is defined in a neighborhood of $0$ and $g_\mu^{-1}(z) - 1/z$ is analytic near $0$.
\end{enumerate}
\end{proposition}

\begin{definition}[$R$-transform]
Let $\mu \in \mathcal{P}(\R)$.  Then $r_\mu(z) = g_\mu^{-1}(z) - 1/z$ where defined.  In particular, if $\mu$ is compactly supported, then $r_\mu$ is defined in a neighborhood of $0$.  We also write $r_X(z)$ for the $R$-transform of the spectral measure of $X$.
\end{definition}

\begin{proposition}[Additivity of $R$-transform under free convolution]
Let $X$ and $Y$ be freely independent self-adjoint elements in a tracial non-commutative probability space.  Then $r_{X+Y}(z) = r_X(z) + r_Y(z)$ in a neighborhood of $0$.  See e.g.\ \cite[\S 3.2]{voiculescu1992free-random}, \cite[\S 5.3.3]{anderson2010introduction}.
\end{proposition}

\subsubsection*{Analytic subordination:}  The complex-analytic framework also allows us to understand the conditional expectation $E_{\mathcal{A}}[(z - X - Y)^{-1}]$ when $X \in \mathcal{A}$ and $\mathcal{A}$ is freely independent of $Y$ and $z$ is in the upper half-plane $\mathbb{H} := \{z: \im(z) > 0\}$.  In fact, $E_{\mathcal{A}}[(z - X - Y)^{-1}]$ is given by $(f(z) - X)^{-1}$ where $f$ is an analytic function $\mathbb{H} \to \mathbb{H}$.  This will be the key to arrange the desired behavior of the diagonal entries in \pref{thm:david-magic} (3).  The following theorem is due mostly to Biane \cite{biane1998processes}; for further history and generalizations, see \cite{biane1998processes,voiculescu1993entropy1,voiculescu2002analytic-subordination,Voiculescu2004free-analysis,belinschi2005thesis,belinschi2007new-approach,dykema2007multilinear}.

\begin{theorem}[{See \cite[Theorem 3.1]{biane1998processes}}] \label{thm: subordination}
Let $\mathcal{A}$ and $\mathcal{B}$ be freely independent subalgebras of $\mathcal{M}$.  Let $X \in \mathcal{A}$ and $Y \in \mathcal{B}$ be freely independent self-adjoint operators in a tracial von Neumann algebra $(\mathcal{M},\tau)$.
\begin{enumerate}
    \item There exists a unique analytic function $f: \mathbb{H} \to \mathbb{H}$ such that $f(z) = z + O(1)$ for sufficiently large $z$ and $g_{X+Y}(z) = g_X(f(z))$.
    \item Let $E_{\mathcal{A}}$ denote the canonical conditional expectation onto $\mathcal{A}$.  Then for $z \in \mathbb{H}$,
    \[
    E_{\mathcal{A}}[(z - X - Y)^{-1}] = (f(z) - X)^{-1}.
    \]
\end{enumerate}
\end{theorem}

\subsubsection*{Consequences of the resolvent identity:} In light of the central role played by resolvents in \pref{thm: subordination}, we close the section with a few estimates for resolvents that will be used many times in our analysis.  In the following lemmas, recall that the plain norm $\norm{x}$ of an element of a von Neumann algebra refers to its operator norm in $B(H)$, which by definition is the same as $\norm{x}_\infty$ in the sense of the non-commutative $L^p$ spaces.

\begin{lemma} \label{lem: resolvent derivative}
Let $z \in \C \setminus \R$.  Let $(\mathcal{M},\tau)$ be a tracial von Neumann algebra.
\begin{enumerate}
    \item For $A \in \mathcal{M}_{\operatorname{sa}}$ where $\mathcal{M}_{\operatorname{sa}}$ is the subalgebra of self-adjoint operators, we have $\norm{(z - A)^{-1}}_\infty \leq 1/|\im z|$.
    \item Let $A, B \in \mathcal{M}_{\operatorname{sa}}$ and $p \in [1,\infty]$. Then
    \[
    \norm{(z - A)^{-1} - (z - B)^{-1}}_p \leq \frac{1}{|\im z|^2} \norm{A - B}_p.
    \]
    \item Let $A(t)$ be a self-adjoint element of $\mathcal{M}$ depending in a $C^1$ manner on a parameter $t$ with respect $\norm{\cdot}_p$.  Then
\[
\frac{d}{dt} [(z - A(t))^{-1}] = (z - A(t))^{-1} \,A'(t)\, (z - A(t))^{-1}.
\]
\end{enumerate}
\end{lemma}

\begin{proof}
(1) For vectors $\xi \in H$, we have
\[
|\angles{\xi, (z - A) \xi}| \geq |\im \angles {\xi, (z - A) \xi}| = |\im z|\, |\xi|^2.
\]
Thus, $|(z - A) \xi| \geq |\im z| \,|\xi|$, which implies that $z - A$ is injective and has closed range.  Meanwhile, the same inequality holds for the adjoint $\overline{z} I - A$, and hence $z I - A$ is surjective, hence invertible.  Then $|(z - A) \xi| \geq |\im z| \,|\xi|$ implies that $\norm{(z - A)^{-1}} \leq 1 / |\im z|$.

(2) By the resolvent identity and the non-commutative H\"older inequality,
\begin{align*}
\norm{(z - A)^{-1} - (z - B)^{-1}}_p &= \norm{(z - A)^{-1}(A - B)(z - B)^{-1}}_p \\
&\leq \norm{(z - A)^{-1}} \norm{A - B}_p \norm{(z - B)^{-1}} \\
&\leq \frac{1}{|\im z|^2} \norm{A - B}_p.
\end{align*}

(3) By the resolvent identity,
\begin{align*}
&(z - A(t+\epsilon))^{-1} - (z - A(t))^{-1} = \\
&(z - A(t+\epsilon))^{-1}(z - A(t))(z - A(t))^{-1} - (z - A(t+\epsilon))^{-1} (z - A(t+\epsilon))(z - A(t))^{-1} \\
&= (z - A(t+\epsilon))^{-1}(A(t+\epsilon) - A(t))(z - A(t))^{-1}.
\end{align*}
Dividing by $\epsilon$ and taking the limit as $\epsilon \to 0$ completes the proof.
\end{proof}

We also recall the following definition:  Given an operator $T$ on a complex Hilbert space, the real and imaginary parts are defined as $\re(T) = (T+T^*)/2$ and $\im(T) = (T - T^*)/2i$.  The following facts about operator real and imaginary parts are standard and easy to check.

\begin{lemma} \label{lem: real and imaginary parts}
Let $T$ be an operator in a tracial von Neumann algebra $(\mathcal{M},\tau)$.
\begin{enumerate}
    \item $\re(T)$ and $\im(T)$ are self-adjoint and $T = \re(T) + i \im(T)$.
    \item $\tau(\re(T)) = \re(\tau(T))$ and $\tau(\im(T)) = \im(\tau(T))$.
    \item For $p \in [1,\infty]$, we have $\norm{\re(T)}_p \leq \norm{T}_p$ and $\norm{\im(T)}_p \leq \norm{T}_p$.
    \item If $T$ is invertible, then $\re(T^{-1}) = (T^*)^{-1} \re(T) T^{-1}$ and $\im(T^{-1}) = -(T^*)^{-1} \im(T) T^{-1}$.
\end{enumerate}
\end{lemma}

\subsection{Analysis of the free semicircular operator} \label{subsec: analysis of free semicircular}

Our argument is based on comparing $2 \beta A_{\sym} - D$ with $\sqrt{2} \beta S - D$ where $S$ is a semicircular operator freely independent of $D$, and in this section, we first perform the analysis with the semicircular operator, and we thus set up the following notation.

\begin{notation} \label{not: free matrix algebra and semicircular}
Let $\mathcal{B} = L^\infty[-2,2]$, and let $\tau_{\mathcal{B}}(f) = \frac{1}{2\pi} \int_{-2}^2 f(x)\sqrt{4 - x^2}\,dx$, and let $S \in \mathcal{B}$ be the identity function, so that $S$ is a standard semicircular element. Let $(\mathcal{M}_n,\tau_n)$ be the free product of $(\mathbb{M}_n(\mathbb{C}),\tr_n)$ with $(\mathcal{B},\tau_{\mathcal{B}})$.  We view $\mathbb{M}_n(\mathbb{C})$ as a subalgebra of $\mathcal{M}_n$, so that $\mathcal{M}_n$ is generated by $\mathbb{M}_n(\mathbb{C})$ and a freely independent semicircular element $S$.
\end{notation}

Recall that the Cauchy transform of $\sqrt{2} \beta S - D$ is given by
\[
g_{\sqrt{2} \beta S - D}(z) = \tau_n[(z - \sqrt{2} \beta S + D)^{-1}]
\]
and
\[
g_{-D}(z) = \tr_n[(z + D)^{-1}].
\]
Our first step is to describe the subordination function $f$ for the free sum $\sqrt{2} \beta S + (-D)$ more explicitly using the special form of the $R$-transform of a semicircular variable.

\begin{proposition} \label{prop: subordination identity}
Consider the setup of Notation \ref{not: free matrix algebra and semicircular}, let $D \in M_n(\C)$.  Let $f$ be the subordination function such that $g_{\sqrt{2} \beta S - D}(z) = g_{-D}(f(z))$ as in \pref{thm: subordination}.  Then $f(z)$ is injective and
\[
f^{-1}(z) = z + 2 \beta^2 g_{-D}(z) \text{ for } z \in \operatorname{dom}(f^{-1}).
\]
Moreover, for $z \in \mathbb{H}$, we have $z \in \operatorname{dom}(f^{-1})$ if and only if $z + 2 \beta^2 g_{-D}(z)$ has positive imaginary part if and only if $2 \beta^2 \tr[((\im z)^2 + (\re z + D)^2)^{-1}] < 1$.
\end{proposition}

\begin{proof}
Recall that the $R$-transform of a semicircular element of variance $2 \beta^2$ is $r_{\sqrt{2} \beta S}(z) = 2 \beta^2 z$ (see e.g.\ \cite[Example 5.3.26]{anderson2010introduction}).  By definition of the $R$-transform and by free independence, we have for $z$ in a neighborhood of $0$ that
\begin{align*}
z &= g_{\sqrt{2} \beta S - D}(1/z + r_{\sqrt{2} \beta S - D}(z)) \\
&= g_{\sqrt{2} \beta S - D}(1/z + r_{-D}(z) + r_{\sqrt{2} \beta S}(z)) \\
&= g_{\sqrt{2} \beta S - D}( g_{-D}^{-1}(z) + 2 \beta^2 z).
\end{align*}
Now we substitute $w = g_{-D}^{-1}(z)$ (so $w$ will range in a neighborhood of $\infty$) and obtain
\[
g_{-D}(w) = g_{\sqrt{2} \beta S - D}(w + 2 \beta^2 g_{-D}(w)).
\]
Then substituting $w = f(z)$ for $z$ large, we obtain
\[
g_{\sqrt{2} \beta S - D}(z) = g_{-D}(f(z)) = g_{\sqrt{2} \beta S - D}(f(z) + 2 \beta^2 g_{-D}(f(z))).
\]
Since $g_{\sqrt{2} \beta S - D}$ is injective on a neighborhood of $\infty$, this implies
\[
z = f(z) + 2 \beta^2 g_{-D}(f(z))
\]
holds in a neighborhood of infinity.  But note that both sides are analytic on the upper half-plane and hence the identity extends to the entire upper half-plane by the identity theorem.  Thus, $f$ is injective and its inverse is given on its domain by $z + 2 \beta^2 g_{-D}(z)$.

Next, let us prove the claim about the domain of $f^{-1}$.  First, if $z \in \dom(f^{-1})$, then $z + 2 \beta^2 g_{-D}(z)$ being equal to $f^{-1}(z)$ must have positive imaginary part.  Second, suppose that $z + 2 \beta^2 g_{-D}(z)$ has positive imaginary part.  Note that
\begin{align*}
\im[z + 2 \beta^2 g_{-D}(z)] &= \im(z) + 2 \beta^2 \im \tr[(z + D)^{-1}] \\
&= \im(z) + 2 \beta^2 \tr( \im((i \im z + \re z + D)^{-1})) \\
&= \im(z) - 2 \beta^2 (\im z) \tr( ((\im z)^2 + (\re z + D)^2)^{-1} ) \\
&= \im(z) \left( 1 - 2 \beta^2 \tr( ((\im z)^2 + (\re z + D)^2)^{-1} )\right).
\end{align*}
Hence, $z + 2 \beta^2 g_{-D}(z)$ has positive imaginary part if and only if $2 \beta^2 \tr( ((\im z)^2 + (\re z + D)^2)^{-1} ) < 1$.  Furthermore, $2 \beta^2 \tr( ((\im z)^2 + (\re z + D)^2)^{-1} )$ is a decreasing function of $\im z$, when $\re z$ is fixed.  Hence, $2 \beta^2 \tr( ((\im z)^2 + (\re z + D)^2)^{-1} ) < 1$ holds for $z$, it also holds for $z + iy$ for all $y > 0$.  Now for sufficiently large $y$, we have that $z + iy \in \dom(f^{-1})$ and so
\[
f(z + iy + 2 \beta^2 g_{-D}(z + iy)) = z + iy.
\]
Now $z + iy + 2 \beta^2 g_{-D}(z + iy)$ has positive imaginary part and hence is in the domain of $f$ for all $y \geq 0$.  Thus, by analytic continuation the above identity extends to all $y \geq 0$, and in particular, $f(z + 2 \beta^2 g_{-D}(z)) = z$, so that $z \in \dom(f^{-1})$.
\end{proof}

The next lemma locates the maximum of the spectrum of $\sqrt{2} \beta S - D$.  Note that when we assume $2 \beta^2 \tr(D^{-2}) = 1$, this means that the point $a$ in the lemma below will be zero.
The approach will be to use property 3 of \pref{prop:cs-transform} to reduce statements about the spectrum of $\sqrt{2} \beta S - D$ to complex-analytic properties of its Cauchy-Stieltjes transform.

\begin{lemma} \label{lem: D dependent shift} ~
\begin{enumerate}
    \item For each real diagonal matrix $D \in M_n(\C)$, there exists a unique real number $a > -\min \Spec(D)$ such that $2 \beta^2 \tr_n((a + D)^{-2}) = 1$.
    \item $\max \Spec(\sqrt{2} \beta S - D) = a + 2\beta^2\tr_n((a+D)^{-1})$.
\end{enumerate}
\end{lemma}

\begin{proof}
(1) Note that $2 \beta^2 \tr_n((a + D)^{-2})$ is a strictly decreasing function of $a$ on $[-\min \Spec(D),\infty)$.  As $a \to \infty$, it converges to zero and as $a \to -\min \Spec(D)$, it converges to $+\infty$.  Hence, by the intermediate value theorem and strict monotonicity, there is a unique point where it equals $1$.

(2) Let $h(z) = z + 2\beta^2 g_{-D}(z) = z + 2\beta^2 \tr((z+D)^{-1})$, which is the inverse function for $f$ on the appropriate domain.  Our first goal is to show that $h$ is invertible on $(a,\infty) + i\R$ and in particular this region is contained in the image of $f$. 
 Suppose that $x \in (a,\infty)$ and $y \in \R$.  We claim first that $\im h(x+iy)$ has the same sign as $y$.  Note that
\[
\im h(x+iy) = y - 2\beta^2\tr(y(y^2 + (x+D)^2)^{-1}) = y\left(1 - 2\beta^2\tr((y^2 + (x + D)^2)^{-1})\right).
\]
Now $2\beta^2\tr((y^2 + (x+D)^2)^{-1}) \leq 2\beta^2\tr((x+D)^{-2}) < 2\beta^2\tr((a+D)^{-2}) = 1$, and so $\im h(x+iy)$ is $y$ times a strictly positive number, and so has the same sign as $y$.  By \pref{prop: subordination identity}, we see that if $y > 0$, then $x+iy \in \dom(f)$.  Similarly, if $y < 0$, then $x+iy$ is in the domain of the mirror version of $f$ in the lower half-plane.  In particular, $h$ is injective on these two regions.  We note the identity that $h'(z) = 1 - 2\beta^2\tr((z+D)^{-2})$, and hence since $\tr \re X \le \tr |X|$ for all $X$,
\begin{align*}
\re h'(x+iy) &\geq 1 - 2 \beta^2 \tr(|(x + iy +D)^{-2}|) \\
&= 1 - 2 \beta^2 \tr((y^2 + (x+D)^2)^{-1}) \\
&> 1 - 2 \beta^2 \tr((a+D)^{-2}) \\
&= 0,
\end{align*}
In particular, we see that $h$ is strictly increasing on $(a,\infty)$ and thus injective there as well.  Overall $h$ is injective on $(a,\infty) + i\R$.

Also, since $a > \min \Spec(D)$, we know $g_{-D}$ is analytic on $(a,\infty) + i\R$.  Therefore, $g_{-D} \circ h^{-1}$ is analytic on $h((a,\infty) + i\R)$.  Furthermore, we know that $h^{-1}$ agrees with $f$ on $((a,\infty) + i\R) \cap \mathbb{H}$, so that $g_{-D} \circ h^{-1}$ agrees with $g_{-D} \circ f = g_{\sqrt{2}\beta S - D}$ on this region; the symmetrical statement holds for the lower half-plane.  Hence, $g_{\sqrt{2} \beta S - D}$ is analytic on $h((a,\infty) + i\R)$, and it follows by \pref{prop:cs-transform} that $h((a,\infty) + i\R)$ is in the complement of the spectrum of $\sqrt{2} \beta S - D$.  In particular, $h((a,\infty)) = (a + 2\beta^2 g_{-D}(a),\infty)$ is in the complement of the spectrum, and so
\[
\max \Spec(\sqrt{2} \beta S - D) \leq a + 2\beta^2 g_{-D}(a) = a + 2\beta^2\tr((a+D)^{-1}).
\]

To see the reverse inequality, consider the behavior of $h$ near the point $a$.  By construction, $h'(a) = 0$ and $h''(a) = 4 \beta^2 \tr((a+D)^{-3}) > 0$.  Thus,
\[
h(z) - h(a) = \frac{1}{2}h''(a)(z - a)^2 + O((z-a)^3),
\]
hence, $h(z) - h(a) = (z-a)^2 p(z)$ for some nonzero analytic function $p(z)$, and in particular, $h(z) - h(a) = q(z)^2$ for some analytic function $q$ with $q(a) = 0$ and $q'(a) > 0$.  Now $q^{-1}$ is defined on some neighborhood $(a-\delta,a+\delta) + i(-\delta,\delta)$.  We have $h^{-1}(z) = q^{-1}(\sqrt{z - h(a)})$ on $(a,a+\delta) + i(-\delta,\delta)$, where $\sqrt{z - h(a)}$ is defined by taking the slit in the negative real direction.  This implies also that
\[
g_{\sqrt{2} \beta S - D}(z) = g_{-D}\left(q^{-1}\left(\sqrt{z - h(a)}\right) \right).
\]
Since $g_{-D}$ is analytic on a neighborhood of $a$, this identity extends to a neighborhood of $h(a)$ say $(h(a)-\delta',h(a)+\delta') + i (-\delta',\delta')$.  Our choice of square root $\sqrt{z - h(a)}$ cuts the plane along $(-\infty,h(a)]$, mapping the ``upper side'' onto the positive imaginary axis and the ``lower side'' onto the negative imaginary axis.  Since $(q^{-1})'(0) > 0$ and
\[
g_{-D}'(a) = \frac{d}{dz} \Bigr|_{z=a} \tr_n((z + D)^{-1}) = -\tr_n((a+D)^{-2}) = -\frac{1}{2 \beta^2} < 0,
\]
we get that for $x \in (h(a)-\delta',h(a))$, we have
\[
\lim_{y \to 0^+} g_{\sqrt{2} \beta S - D}(x+iy) = \lim_{y \to 0^+} g_{-D}\left(q^{-1}\left(\sqrt{x+iy - h(a)}\right)\right) > 0,
\]
and the behavior as $y \to 0^-$ is described by the complex conjugate.  Thus, $g_{\sqrt{2} \beta S - D}$ cannot extend to be analytic in a neighborhood of $h(a)$ since the values from the upper side and the lower side disagree on $(h(a)-\delta,h(a))$.  Therefore by \pref{prop:cs-transform}, $h(a) \in \Spec(\sqrt{2} \beta S - D)$.
\end{proof}

As mentioned before, in Theorem \ref{thm:david-magic}, we assume $2 \beta^2 \tr_n(D^{-2}) = 1$, take $a=0$, and fix $b = \beta n^{-\delta}$ for some small positive $\delta$.  Letting $\tilde{z} = \tilde{a} + i \tilde{b} = ib + 2 \beta^2 g_{-D}(ib)$, we have $\tilde{a}$ is the maximum of the spectrum of $\sqrt{2} \beta S - D$.  To obtain an operator that concentrates on the upper part of the spectrum of $\sqrt{2} \beta S - D$, we will use the imaginary part of the resolvent $[\tilde{a} + i \tilde{b} - (\sqrt{2} \beta S - D)]^{-1}$.  Our goal is then to study the conditional expectation of this operator onto the diagonal subalgebra $\mathcal{D}_n \subseteq \mathbb{M}_n(\mathbb{C}) \subseteq \mathcal{M}_n$, which is described by $(f(\tilde{a} + i\tilde{b}) + D)^{-1} = (ib + D)^{-1}$, thanks to \pref{thm: subordination}.

We remark that our estimates need to take the renormalization into account. Indeed, the normalized $\tr_n[\im [\tilde{z} - (\sqrt{2} \beta S - D)]]^{-1}$ will vanish as $b \to 0$, and so this operator will later be renormalized in \S \ref{sec:convergence-to-ac}.  The reason it vanishes is that density of the spectral measure of $\sqrt{2} \beta S - D$ will vanish like a square root function at the edge of the spectrum (as one can see from the proof of \pref{lem: D dependent shift} (2) and Stieltjes inversion formula).  The operator $-\im [\tilde{z} - (\sqrt{2} \beta S - D)]^{-1}$ is obtained by applying the ``spiked'' function $x \mapsto (\tilde{a} + i\tilde{b} - x)^{-1}$ to the operator $(\sqrt{2} \beta S - D)$; the spiked function is concentrated near the edge of the spectrum at $\tilde{a}$ where the density is small, and so its integral with respect to $\mu_{\sqrt{2} \beta S - D}$ will vanish as $b \to 0$.  Thus, the error terms for our approximation of the diagonal will need to be small \emph{relative} to $\tr_n[\im [\tilde{z} - (\sqrt{2} \beta S - D)]]^{-1}$.  The trace turns out to be on the order of $b$ while the error estimates for the diagonal terms are controlled in terms of $\tilde{b}$.  Thus, a key point is that we prove in the next lemma that $\tilde{b}$ is on the order $b^3$ because we chose exactly the point $\tilde{a}$ at the edge of the spectrum (this is related to the fact that $h'(a) = 0$ in the proof of \pref{lem: D dependent shift} (2) above).  It is quite necessary for our argument that $\tilde{a}$ is at the top of the spectrum (and hence $a = 0$); indeed, if $a$ and hence $\tilde{a}$ were larger, then the ``spike'' of the function $x \mapsto -\im (\tilde{a} + i\tilde{b} - x)^{-1}$ fall outside the spectrum of $\sqrt{2} \beta S - D$.  On the other hand, if we took $a < 0$, then $a+ib$ would end up outside the domain of $f^{-1}$ when $b$ is sufficiently small, which would prevent us from using the subordination function to estimate the resolvent.

\begin{lemma} \label{lem: D dependent resolvent}
Fix $b \in (0,\beta)$.  Let $D \geq 0$ be a diagonal matrix with $2\beta^2\tr(D^{-2}) = 1$.  Let $\tilde{z} = ib + 2 \beta^2 g_{-D}(ib)$, and let $\tilde{a}$ and $\tilde{b}$ be the real and imaginary parts of $\tilde{z}$ (which also depend on $D$).  Then
\begin{enumerate}
    \item $b^3 / (3 \beta^2) \leq \tilde{b} \leq 2 \beta^2 b^3 \tr_n(D^{-4}) \leq b^3 \norm{D^{-1}}^2$.
    \item $-2 \beta^2 \im g_{-D}(ib) \geq b - 2 \beta^2 b^3 \tr_n(D^{-4}) \geq b - b^3 \norm{D^{-1}}^2$.
\end{enumerate}
\end{lemma}

\begin{proof}
(1) Note that using \pref{lem: real and imaginary parts} and resolvent identities,
\begin{align*}
\tilde{b} &= b + 2 \beta^2 \im \tr_n[(ib + D)^{-1}] \\
&= b - 2 \beta^2 b \tr_n[(b^2 + D^2)^{-1}] \\
&= b(1 - 2 \beta^2 \tr_n[D^{-2}] - 2 \beta^2 \tr_n[(b^2 + D^2)^{-1} - D^{-2}]) \\
&= b(0 + 2 \beta^2 b^2 \tr_n[(b^2 + D^2)^{-1}D^{-2}]) \\
&= 2 \beta^2 b^3 \tr_n[(b^2 + D^2)^{-1}D^{-2}]).
\end{align*}
Then we note that
\[
\tr_n[(b^2 + D^2)^{-1}D^{-2}]) \leq \tr_n[D^{-4}],
\]
and
\[
2 \beta^2 \tr_n(D^{-4}) \leq 2 \beta^2 \tr_n(D^{-2}) \norm{D^{-1}}^2 = \norm{D^{-1}}^2
\]
which implies the upper bounds from (1).

On the other hand, by Jensen's inequality applied to the convex function $t \mapsto (b^2t + 1)^{-1} t^2$ on $[0,\infty)$,
\begin{align*}
\tr_n[(b^2 + D^2)^{-1}D^{-2}]) &= \tr_n[(b^2 D^{-2}+1)^{-1} D^{-4}] \\
&\geq (b^2 \tr_n(D^{-2}) + 1)^{-1} \tr_n(D^{-2})^2 \\
&= (b^2/2\beta^2 + 1)^{-1} \frac{1}{4 \beta^4} \\
&= (b^2 + 2 \beta^2)^{-1} (2 \beta^2)^{-1}.
\end{align*}
Therefore,
\[
\tilde{b} \geq b^3(b^2 + 2 \beta^2)^{-1} \geq b^3(\beta^2 + 2 \beta^2)^{-1} = \frac{b^3}{3 \beta^2},
\]
which gives the lower bound from (1).

(2) Observe that
\[
-2 \beta^2 \im g_{-D}(ib) = b - \tilde{b} \geq b - 2 \beta^2 b^3 \tr_n(D^{-4}),
\]
and again note $2 \beta^2 \tr_n(D^{-4}) \leq \norm{D^{-1}}^2$.


\end{proof}

\subsection{Concentration estimates for the resolvent} \label{subsec: concentration estimates}

The goal of this subsection is to estimate the difference between the diagonal of the resolvent $(\tilde{z} - 2 \beta A_{\sym} + D)^{-1}$ and its expectation with respect to the randomness of $A$ for all fixed $D$.  We argue using concentration inequalities, which give estimates on the probability of some random variable deviating from its expectation by a certain amount.  In particular, we use Herbst's concentration inequality for Lipschitz functions in this section and in the next section we use the Poincar{\'e} inequality.  These inequalities are closely connected to the log-Sobolev inequality, Talagrand entropy-cost inequality, and other aspects of information geometry.  For general background on concentration inequalities, see \cite{ledoux1992heat,bobkov2000brunn,ledoux2001concentration}.

\begin{lemma}[Herbst concentration estimate for Gaussian matrix] \label{lem: Herbst}
Let $A$ be a normalized real Ginibre matrix (see Definition \ref{def: real Ginibre}).  Let $f: M_n(\R) \to \R$ be a Lipschitz function with respect to $\norm{\cdot}_2$.  Then for all $\delta \geq 0$,
\[
\Pr( |f(A) - \E[f(A)]| \geq \delta) \leq 2 e^{-n^2 \delta^2 / (2 \norm{f}_{\operatorname{Lip}}^2)}
\]
\end{lemma}

For proof of this lemma, we refer to \cite[\S 2.3, \S 4.4, esp. Theorem 2.3.5]{anderson2010introduction}.  A related inequality gives a bound on the variance of a Lipschitz function of a Gaussian variable.

\begin{lemma}[Poincar\'e inequality for Gaussian] \label{lem: Poincare}
Let $f: M_n(\R) \to \R$ be Lipschitz with respect to $\norm{\cdot}_2$.  Let $A$ be a normalized real Gaussian random matrix.  Then
\[
\Var(f(A)) \leq \frac{1}{n^2} \E \norm{\nabla f(A)}_2^2 \leq \frac{1}{n^2} \norm{f}_{\operatorname{Lip}}^2.
\]
\end{lemma}

\begin{corollary} \label{cor: vector Poincare}
Let $W$ be a real or complex inner product space, and let $f, g: M_n(\R) \to W$ be Lipschitz.
\[
\E |\angles{ f(A) - \E[f(A)], g(A) - \E[g(A)] }|  \leq \frac{\dim_{\R} W}{n^2} \norm{f}_{\operatorname{Lip}} \norm{g}_{\operatorname{Lip}}.
\]
\end{corollary}

\begin{proof}
First, consider the case where $f = g$ and the inner product space is real.  Fix an orthonormal basis $w_1$, \dots, $w_d$ for $W$.  Apply the previous lemma to $\la f(A),w_j \ra$ for each $j$, and then sum up the results over the basis.  In the complex case, we apply the previous lemma to the real and imaginary parts of $\angles{f(A), w_j}$.

In the case where $f \ne g$, the left-hand side can be estimated using the Cauchy-Schwarz inequality, and then we apply the case of $f = g$ proved above.
\end{proof}

The other ingredient we will need is an estimate on the probability of large operator norm for the GOE matrix~\cite[Theorem 3.1.5]{anderson2010introduction}.  Here recall $A + A^{\sT}$ is $\sqrt{2}$ times a GOE matrix, and so asymptotically its operator norm will converge to $2 \sqrt{2} < 3$.

\begin{lemma} \label{lem: GOE operator norm bound}
Let $A$ be a normalized real Gaussian matrix.  For some universal constants $C_1$ and $C_2$, we have
\[
\Pr(\norm{A + A^{\sT}} \geq 3) \leq C_1 e^{-C_2 n}.
\]
\end{lemma}

The expectation of the operator norm can be estimated by looking at high moments \cite[\S 2.1.6]{anderson2010introduction}, and then we can estimate the probability using concentration. A short argument for this type of bound is given by Ledoux \cite{ledoux2003remark} in the GUE case, and this can be adapted to the GOE as well using the joint distribution of eigenvalues in \cite[\S 2.5]{anderson2010introduction}.

Now we are ready to apply these concentration estimates to our particular choice of resolvent operator.  In the following lemma we do not use the precise form of $\tilde{z}$ from \pref{lem: D dependent resolvent}, and so we state it more generally for any $\tilde{z}$ with positive imaginary part.

\begin{lemma} \label{lem: concentration estimate 1}
Let $A$ be a normalized real Gaussian matrix.  Let $E_{\mathcal{D}_n}: M_n(\mathbb{C}) \to \mathcal{D}_n$ be the orthogonal projection (or non-commutative conditional expectation) onto the diagonal matrices.  Let $\tilde{z}$ with $\im \tilde{z} > 0$, and let $\delta > 0$.  Fix a real diagonal matrix $D$. Then
\[
\norm{E_{\mathcal{D}_n}[(\tilde{z} - 2 \beta A_{\sym} + D)^{-1} - \E E_{\mathcal{D}_n}[(\tilde{z} - 2 \beta A_{\sym} + D)^{-1}]]}_2 \leq \frac{4 \beta}{n^{8\delta} |\im \tilde{z}|^2}
\]
with probability at least
\[
1 - 2\cdot 3^{2n} e^{-n^{2-16\delta} / 2}.
\]
\end{lemma}

\begin{proof}
Here we use a classic $\epsilon$-net argument.  Let $\Omega = \{D': \tr(D'(D')^*) = 1\}$ be a set of diagonal matrices and let $\Omega_0$ be a maximal $1/2$-separated subset of $\Omega$ with respect to $\norm{\cdot}_2$.  Thus, every element of $\Omega$ is within a distance of $1/2$ from some element of $\Omega_0$.  The balls of radius $1/2$ centered at points in $\Omega_0$ are disjoint and contained in the ball of radius $3/2$ and hence $|\Omega_0| \leq 3^{2n}$ since we used complex entries.  Thus, for any diagonal matrix $B$, we have
\[
\norm{B}_2 = \sup_{D' \in \Omega} \re \angles{B,D'} \leq \max_{D' \in \Omega_0} |\re \tr(BD')| + \frac{1}{2} \norm{B}_2,
\]
so that
\[
\norm{B}_2 \leq 2 \max_{D' \in \Omega_0} |\re \tr(BD')|.
\]
If $B$ is not necessarily diagonal, then $\tr(BD') = \tr(E_{\mathcal{D}_n}[B]D')$ and hence
\[
\norm{E_{\mathcal{D}_n}[B]}_2 \leq 2 \max_{D' \in \Omega_0} |\re \tr(BD')|.
\]
We apply this with $B = (\tilde{z} - 2\beta A_{\sym} + D)^{-1} - \E[(\tilde{z} - 2 \beta A_{\sym} + D)^{-1}]$, so
\begin{align*}
\bigl \lVert E_{\mathcal{D}_n}[(\tilde{z} &- 2\beta A_{\sym} + D)^{-1} - \E[(\tilde{z} - 2 \beta A_{\sym} + D)^{-1}]] \bigr \rVert_2 \\
&\leq 2 \max_{D' \in \Omega_0} \bigl|\re \tr_n[(\tilde{z} - 2 \beta A_{\sym} + D)^{-1}D' - \E[(\tilde{z} - 2 \beta A_{\sym} + D)^{-1}]D'] \bigr|.
\end{align*}
Moreover, for every such $D'$, we have that $\re \tr_n[(\tilde{z} - 2 \beta A_{\sym} + D)^{-1} D']$ is $2 \beta / |\im \tilde{z}|^2$-Lipschitz as a function of $A$ with respect to $\norm{\cdot}_2$.  Therefore, by the concentration inequality of \pref{lem: Herbst},
\[
\Pr\left( |(\id - \E)\re \tr_n[(\tilde{z} - 2\beta A_{\sym} + D)^{-1} D']| \geq \frac{2 \beta}{n^{8\delta} |\im \tilde{z}|^2} \right) \leq 2 e^{-n^{2 - 16\delta} / 2}.
\]
By taking a union bound, we can arrange that $|(\id - \E)\tr_n[(\tilde{z} - 2 \beta A_{\sym} + D)^{-1} D']| \leq \frac{2 \beta}{n^{8\delta} |\im \tilde{z}|^2}$ for all $D' \in \Omega_0$ with probability at least $1 - 2\cdot 3^{2n} e^{-n^{2-16\delta} / 2}$.
\end{proof}

\begin{lemma} \label{lem: concentration estimate 2}
Let $\beta \geq \sqrt{3}$.  Let $\delta \in [0,1/17]$. Suppose $n^\delta \geq \sqrt{3}$.  Fix $b = \beta n^{-\delta}$.  For nonnegative diagonal matrices $D$ with $2 \beta^2 \tr(D^{-2}) = 1$, let $\tilde{z} = \tilde{a} + i \tilde{b} = ib + 2\beta^2g_{-D}(ib)$, which depends on $D$, $n$, and $\delta$.  Then with probability $1 - C_1 \exp(-C_2 n)$ in the Gaussian matrix $A$, we have
\[
\norm{E_{\mathcal{D}_n}[(\tilde{z} - 2 \beta A_{\sym} + D)^{-1} - \E E_{\mathcal{D}_n}[(\tilde{z} - 2 \beta A_{\sym} + D)^{-1}]]}_2 \leq C_3 \beta n^{-2\delta}
\]
uniformly for \emph{all} real diagonal $D$ with $D \geq 0$ and $2 \beta^2 \tr(D^{-2}) = 1$.
\end{lemma}

\begin{proof}
We want to view $E_{\mathcal{D}_n}[(\tilde{z} - 2 \beta A_{\sym} + D)^{-1}]$ as a function of $D^{-1}$ in the sphere of radius $1/(\sqrt{2} \beta)$ with respect to $\norm{\cdot}_2$, show that it is Lipschitz, and hence deduce an estimate for all $D^{-1}$ from an estimate on a sufficiently dense subset.

For the Lipschitz estimate, fix two nonnegative diagonal matrices $D$ and $D'$ with $\tr(D^{-2}) = \tr((D')^{-2}) = 1/(2\beta^2)$, and let $\tilde{z} = ib + 2 \beta^2 \tr((ib + D)^{-1})$ and $\tilde{z}' = ib + 2 \beta^2 \tr((ib + D')^{-1})$.  Note by the resolvent identity that
\begin{align*}
(ib + D)^{-1} - (ib + D')^{-1} &= (ib + D)^{-1} (D' - D) (ib + D')^{-1} \\
&= (ib + D)^{-1} D (D^{-1} - (D')^{-1}) D'(ib + D')^{-1} \\
&= (1 - ib(ib + D)^{-1}) (D^{-1} - (D')^{-1}) (1 - ib(ib + D')^{-1}).
\end{align*}
Furthermore, $\norm{(ib + D)^{-1}} \leq 1/b$ so that $\norm{1 - ib(ib + D)^{-1}} \leq 2$, and similarly for $D'$.  Hence, by the non-commutative H\"older's inequality,
\[
\norm{(ib + D)^{-1} - (ib + D')^{-1}}_2 \leq \norm{1 - ib(ib + D)^{-1}} \norm{D^{-1} - (D')^{-1}}_2 \norm{1 - ib(ib + D')^{-1}} \leq 4 \norm{D^{-1} - (D')^{-1}}_2.
\]
Thus,
\[
|\tilde{z} - \tilde{z}'| \leq 2 \beta^2 \norm{(ib + D)^{-1} - (ib + D')^{-1}}_2  \leq 8 \beta^2 \norm{D^{-1} - (D')^{-1}}_2.
\]
Another application of the resolvent identity yields that
\begin{align}
(\tilde{z} - 2\beta A_{\sym} + D)^{-1} & - (\tilde{z}' - 2 \beta A_{\sym} + D')^{-1} \\
&= (\tilde{z} - 2 \beta A_{\sym} + D)^{-1}(\tilde{z}' - \tilde{z} + D' - D)(\tilde{z}' - 2 \beta A_{\sym} + D')^{-1} \label{eq: resolvent to estimate} \\
&= (\tilde{z} - 2 \beta A_{\sym} + D)^{-1}(\tilde{z}' - \tilde{z})(\tilde{z}' - 2 \beta A_{\sym} + D')^{-1} \label{eq: resolvent to estimate z term} \\
&\quad + (\tilde{z} - 2 \beta A_{\sym} + D)^{-1}(D' - D)(\tilde{z}' - 2 \beta A_{\sym} + D')^{-1} \label{eq: resolvent to estimate D term}
\end{align}
Now by \pref{lem: resolvent derivative} and \pref{lem: D dependent resolvent} (1),
\[
\norm{(\tilde{z} - 2 \beta A_{\sym} + D)^{-1}} \leq \frac{1}{\tilde{b}} \leq \frac{3 \beta^2}{b^3},
\]
and similarly for $D'$.  Hence, \eqref{eq: resolvent to estimate z term} can be bounded by
\begin{align*}
\norm{(\tilde{z} - 2 \beta A_{\sym} + D)^{-1}(\tilde{z}' - \tilde{z})(\tilde{z}' - 2\beta A_{\sym} + D')^{-1}}_2 &\leq \frac{M_1}{\beta^4 b^6} |\tilde{z} - \tilde{z}'| \\
&\leq \frac{M_2 \beta^6}{b^6} \norm{D^{-1} - (D')^{-1}}_2 \\
&= M_2 n^{6\delta} \norm{D^{-1} - (D')^{-1}}_2,
\end{align*}
where $M_1$ and $M_2$ are universal constants.  Then to estimate \eqref{eq: resolvent to estimate D term}, we write
\begin{multline} \label{eq: resolvent to estimate 2}
(\tilde{z} - 2 \beta A_{\sym} + D)^{-1}(D' - D)(\tilde{z}' - 2 \beta A_{\sym} + D')^{-1} = (\tilde{z} - 2 \beta A_{\sym} + D)^{-1}D(D^{-1} - (D')^{-1})D'(\tilde{z}' - 2 \beta A_{\sym} + D')^{-1}.
\end{multline}
We write
\begin{align} 
(\tilde{z} - 2 \beta A_{\sym} + D)^{-1}D &= 1 - (\tilde{z} - 2 \beta A_{\sym} + D)^{-1}(\tilde{z} - 2 \beta A_{\sym}) \nonumber \\ 
&= 1 - (\tilde{z} - 2 \beta A_{\sym} + D)^{-1}i\tilde{b} - (\tilde{z} - 2 \beta A_{\sym} + D)^{-1} (\tilde{a} - 2 \beta A_{\sym}). \label{eq: split of terms to estimate resolvent}
\end{align}
As before,
\[
\norm{(\tilde{z} - 2 \beta A_{\sym} + D)^{-1} \tilde{b}} \leq 1.
\]
Meanwhile,
\[
\tilde{a} = \re 2\beta^2\tr((ib + D)^{-1}) = 2\beta^2\tr(D(b^2 + D^2)^{-1}) \leq 2\beta^2\tr(D^{-1}) \leq 2\beta^2\tr(D^{-2})^{1/2} = \sqrt{2}\beta.
\]
Furthermore, recall by \pref{lem: GOE operator norm bound}, we have that $\norm{2A_{\sym}} \leq 3$ with probability $1 - M_3 \exp(-M_4 n)$, and so $\norm{\tilde{a} - 2 \beta A_{\sym}} \leq (\sqrt{2} + 3) \beta$.  Thus,
\[
\norm{(\tilde{z} - 2 \beta A_{\sym} + D)^{-1}(\tilde{a} - 2 \beta A_{\sym})} \leq \frac{3 \beta^2}{b^3} (\sqrt{2} + 3) \beta \leq M_5 n^{3\delta}.
\]
Substituting all these estimates into \eqref{eq: split of terms to estimate resolvent} yields
\[
\norm{(\tilde{z} - 2 \beta A_{\sym} + D)^{-1}D} \leq M_6 n^{3 \delta},
\]
and of course a similar estimate holds for $D'(\tilde{z}' - 2 \beta A_{\sym} + D')^{-1}$.  Plugging these back into \eqref{eq: resolvent to estimate 2}, we get
\[
\norm{(\tilde{z} - 2 \beta A_{\sym} + D)^{-1}(D' - D)(\tilde{z}' - 2 \beta A_{\sym} + D')^{-1}}_2 \leq M_6^2 n^{6\delta} \norm{D^{-1} - (D')^{-1}}_2.
\]
Therefore, overall
\[
\norm{(\tilde{z} - 2 \beta A_{\sym} + D)^{-1} - (\tilde{z}' - 2 \beta A_{\sym} + D')^{-1}}_2 \leq M_7 n^{6 \delta} \norm{D^{-1} - (D')^{-1}}_2.
\]
Finally, since $E_{\mathcal{D}_n}$ is contractive in the $2$-norm, the map
\[
D \mapsto E_{\mathcal{D}_n}\big[(\tilde{z}-2\beta A_{\sym}+D)^{-1}\big]
\]
is $M_7 n^{6\delta}$-Lipschitz in $D^{-1}$ on the event $\{\norm{2A_{\sym}}\le 3\}$.
Hence, since this event does not depend on $D$, the centered truncated map
\[
D^{-1}\mapsto E_{\mathcal{D}_n}\big[(\tilde{z}-2\beta A_{\sym}+D)^{-1}\big]\mathbbm{1}_{\{\norm{2A_{\sym}}\le 3\}}-\E[E_{\mathcal{D}_n}\big[(\tilde{z}-2\beta A_{\sym}+D)^{-1}\big]\mathbbm{1}_{\{\norm{2A_{\sym}}\le 3\}}]
\]
is $2M_7 n^{6\delta}$-Lipschitz.  Moreover, uniformly in $D$,
\[
\norm{E_{\mathcal{D}_n}\big[(\tilde{z}-2\beta A_{\sym}+D)^{-1}\big]}_2 \le \frac{1}{\tilde b}\le \frac{3\beta^2}{b^3}=3 \beta^2 \beta^{-3} n^{3\delta},
\]
and therefore
\begin{equation}
    \label{eq:A-diagonal-expected-truncation-ok}
\norm{\E E_{\mathcal{D}_n}\big[(\tilde{z}-2\beta A_{\sym}+D)^{-1}\big]-\E[E_{\mathcal{D}_n}\big[(\tilde{z}-2\beta A_{\sym}+D)^{-1}\big]\mathbbm{1}_{\{\norm{2A_{\sym}}\le 3\}}]}_2
\le 3M_3\beta^{-1} n^{3\delta}e^{-M_4 n}.
\end{equation}

Next, fix a maximal subset $\Omega_0$ of $\Omega = \{D^{-1}: 2 \beta^2 \tr(D^{-2}) = 1\}$ that is $\beta n^{-8\delta}$-separated with respect to $\norm{\cdot}_2$.  Thus, every element of $\Omega$ is within a distance of $\beta/n^{8\delta}$ from $\Omega_0$.  The balls centered at the elements of $\Omega_0$ of radius $\beta/(2n^{8\delta})$ are disjoint and contained in the ball of radius $3\beta /2$ centered at $0$, and so $|\Omega_0| \leq (3 n^{8 \delta})^n$.

By the union bound and \pref{lem: concentration estimate 1}, with probability at least
\[
1 - 2(3n^{8\delta})^n 3^{2n} e^{-n^{2-16\delta}/2},
\]
over the Gaussian matrix, we have for all $D^{-1} \in \Omega_0$,
\[
\norm{E_{\mathcal{D}_n}[(\tilde{z} - 2 \beta A_{\sym} + D)^{-1} - \E E_{\mathcal{D}_n}[(\tilde{z} - 2 \beta A_{\sym} + D)^{-1}]]}_2 \leq \frac{4 \beta}{n^{8\delta} |\im \tilde{z}|^2} \leq M_9 \beta^{-1} n^{-2\delta}
\]
where again we have applied \pref{lem: D dependent resolvent} (1) to estimate $1 / |\im \tilde{z}|^2 \leq 9 \beta^4 / (\beta n^{-\delta})^6$.  By our choice of $\Omega_0$, every $D^{-1} \in \Omega$ is within a distance of $\beta / n^{8\delta}$ of some $(D')^{-1} \in \Omega_0$.  
Therefore, when $\norm{2A_{\sym}} \le 3$, by \pref{eq:A-diagonal-expected-truncation-ok},
\[
\begin{aligned}
&\norm{E_{\mathcal{D}_n}\big[(\tilde{z}-2\beta A_{\sym}+D)^{-1}\big]-\E E_{\mathcal{D}_n}\big[(\tilde{z}-2\beta A_{\sym}+D)^{-1}\big]}_2 \\
&\qquad\qquad - \norm{ E_{\mathcal{D}_n}\big[(\tilde{z}'-2\beta A_{\sym}+D')^{-1}\big]-\E E_{\mathcal{D}_n}\big[(\tilde{z}'-2\beta A_{\sym}+D')^{-1}\big]}_2\\
&\le 
   2M_7 n^{6\delta}\norm{D^{-1}-(D')^{-1}}_2
   + 3M_3\beta^{-1} n^{3\delta}e^{-M_4 n} \\
&\le
   2M_7 \beta n^{-2\delta}
   + 3M_3\beta^{-1} n^{3\delta}e^{-M_4 n}.
\end{aligned}
\]
The last term is exponentially small in $n$, so it may be absorbed into the constants.
Therefore,
\[
\forall D \in \Omega, \quad \norm{E_{\mathcal{D}_n}[(\tilde{z} - 2\beta A_{\sym} + D)^{-1} - \E E_{\mathcal{D}_n}[(\tilde{z} - 2 \beta A_{\sym} + D)^{-1}]]}_2 \leq M_{11} \beta n^{-2\delta},
\]
with probability at least
\[
1 - M_3 \exp(-M_4 n) - 2(3n^{8\delta})^n 3^{2n} e^{-n^{2-16\delta}/2}.
\]
Note since $\delta \leq 1/17$,
\[
2(3n^{8\delta})^n 3^{2n} e^{-n^{2-16\delta}/2} = 2\exp(-\tfrac{1}{2}n^{2-16\delta} + 8\delta n \log n + 3 n \log 3) \leq M_{12} e^{-M_{13} n}.
\]
so the overall probability of error is bounded by $M_{14} \exp(-M_{15} n)$.  In the statement, we take $C_1 = M_{14}$, $C_2 = M_{15}$, and $C_3 = M_{11}$.
\end{proof}

\subsection{Evaluation of the expected resolvent by interpolation} \label{subsec: evaluation of expectation}

We have now studied the behavior of the resolvent $(\tilde{z} - \sqrt{2} \beta S + D)^{-1}$, and also estimated how far $(\tilde{z} - 2 \beta A_{\sym} + D)^{-1}$ is from its expectation.  It remains to compare $(\tilde{z} - \sqrt{2} \beta S + D)^{-1}$ with the expectation of $(\tilde{z} - 2 \beta A_{\sym} + D)^{-1}$.  This we will do using the following proposition.  Here we consider an arbitrary point $z$ in the upper half-plane for ease of notation, but ultimately, we will plug in $\tilde{z} = ib + 2 \beta^2 g_{-D}(ib) = i\beta n^{-\delta} + 2 \beta^2 g_{-D}(in^{\delta})$ instead of $z$.

\begin{proposition} \label{prop: expectation estimate}
Fix $z$ with positive imaginary part, let $D$ be a real diagonal matrix, and let $f$ be the subordination function with $g_{\sqrt{2} \beta S - D} = g_{-D} \circ f$.  Then
\begin{equation} \label{eq: expectation estimate}
\norm{\E \left[ E_{\mathcal{D}_n}[(z - 2 \beta A_{\sym} + D)^{-1}]\right] - (f(z) + D)^{-1}}_2 \leq \frac{32 \beta^4}{n^{3/2} |\im z|^5} + \frac{2 \beta^2}{n |\im z|^3}.
\end{equation} 
\end{proposition}

We prove this proposition using several common techniques in random matrix theory: interpolation, integration by parts, and concentration estimates.  In particular, our argument is inspired by the work of Collins, Guionnet, and Parraud \cite{collins2022operator} in the setting of the GUE.  It is well-known that for a polynomial $p$, the expectation of $\tr_n[f(X^{(n)})]$, where $X^{(n)}$ is a GUE matrix, is equal to the trace of $f(S)$, where $S$ is standard semicircular operator, plus a correction of order $1/n^2$, and in fact the expectation has an asymptotic expansion in powers of $1/n^2$, known as the genus expansion (see \cite{zvonkin1997matrixintegrals} for an introduction and historical survey).  Since it is unclear how the genus expansion would help with smooth functions beyond polynomials, Collins, Guionnet, and Parraud \cite{collins2022operator} and Parraud \cite{parraud2023asymptotic} developed another asymptotic expansion formula where the terms are expressed through non-commutative derivatives of the input function $f$, which are amenable to analytic techniques.  The idea of the proof is to put the GUE matrix $X^{(n)}$ and the semicircular operator in the same space and study the interpolation $(1-t)^{1/2} X^{(n)} + t^{1/2} S$ inspired by the free Ornstein-Uhlenbeck process. After differentiating in $t$ the expected trace of $(1-t)^{1/2} X^{(n)} + t^{1/2} S$, they use integration by parts and other tricks to get a tractable expression for the $1/n^2$ correction.  A similar interpolation method was used in \cite[\S 5]{BBvH2023} to obtain sharp non-asymptotic bounds for the operator norms of general Gaussian random matrices.

We remark that the technique of Gaussian interpolation has formed the backbone of several rigorous results in mathematical spin-glass theory~\cite{guerra2001sum,aizenman2003extended,talagrand2010mean, talagrand2011mean}, in particular in the proofs of the existence of the limit of the free energy density~\cite{guerra2002thermodynamic} as well as the upper bounds in (both) the replica-symmetric and replica-symmetry breaking regimes~\cite{guerra2003broken}. The use of a \emph{non-commutative} version of Gaussian interpolation to bound the resolvent of the ``shifted'' Hessian, which ultimately provides a bound on the free energy achieved by the algorithmic process, is a notable parallel in technique.

We will follow a similar strategy with the GOE rather than GUE matrix, which is slightly more complicated because the expansion will have $O(1/n)$ terms.  However, since our goal is only to get a concrete bound rather than a full asymptotic expansion, we will be content to integrate by parts once and then estimate the result using the Poincar{\'e} inequality (compare \cite[\S 3.1]{parraud2023asymptotic}).  In fact, rather than directly arguing with the semicircular operator $\sqrt{2} S$, we will approximate it by a Gaussian matrix $B$ of size $nk$ much larger than $n$, which we may regard as fixed throughout this discussion (a larger GUE matrix of size $nk$ is similarly used in \cite[Proposition 3.5]{collins2022operator} and \cite[Lemma 3.3]{parraud2023asymptotic}).

Thus, fix $n$ and $k$, and let $B$ be an $nk \times nk$ Gaussian matrix where again the entries are standard normal with variance $1/(nk)$.  Let $B_{\sym} = \frac{1}{2}(B + B^{\sT})$, so that $2B_{\sym}$ is $\sqrt{2}$ times a GOE matrix.  Then we consider $A \otimes I_k$ and $B$ as elements of $M_{nk}(\R)$.  For $t \in [0,1]$ and $z \in \C \setminus \R$, let
\[
G = G(z,t) = (z I_{nk} - 2 \beta [(1-t)^{1/2} A_{\sym} \otimes I_k + t^{1/2} B_{\sym}] + D \otimes I_k)^{-1};
\]
we will often abbreviate $G(z,t)$ to $G$ for the sake of readability in our equations.  Then fix another diagonal matrix $D'$ and let
\[
h(t) = \E \tr_{nk}[G(z,t) (D' \otimes I_k)].
\]
Note that
\begin{equation} \label{eq: h of zero}
h(0) = \E \tr_n[(z I_n - 2 \beta A_{\sym} + D)^{-1} D'],
\end{equation}
while
\begin{equation} \label{eq: h of one}
h(1) = \E \tr_{nk}[(z I_{nk} - \beta(B + B^{\sT}) + D \otimes I_k)^{-1} (D' \otimes I_k)],
\end{equation}
which will involve the freely independent semicircular operator after we take $k \to \infty$ at the very end of the argument.

We will estimate $h(1) - h(0)$ by studying $h'(t)$.  Here we rely on \pref{lem: resolvent derivative} which implies that
\[
\frac{d}{dt} G(z,t) = \beta t^{-1/2} G(z,t) B_{\sym}  G(z,t) - \beta (1 - t)^{-1/2} G(z,t)(A_{\sym} \otimes I_k) G(z,t).
\]
Thus,
\begin{equation} \label{eq: formula for h prime}
h'(t) = \beta t^{-1/2} \mathbb{E} \tr_{nk}[G B_{\sym} G(D' \otimes I_k)] - \beta (1-t)^{-1/2} \mathbb{E} \tr_{nk}[G(z,t)(A_{\sym} \otimes I_k) G (D' \otimes I_k)].
\end{equation}
We will study these terms using Gaussian integration by parts.

\begin{lemma} \label{lem: IBP 1}
Consider the setup above, and abbreviate $G(z,t)$ to $G$ for convenience.  Then
\begin{equation} \label{eq: IBP 1 formula}
    \beta t^{-1/2} \mathbb{E} \tr_{nk}[G B_{\sym} G (D' \otimes I_k)] = 2 \beta^2 \E \bigg[ \tr_{nk}[G] \tr_{nk}[G(D' \otimes I_k)G]
    + \frac{1}{nk} \tr_{nk}[G^3 (D' \otimes I_k)] \bigg]. 
\end{equation}
\end{lemma}

\begin{proof}
First, by cyclic symmetry of the trace,
\[
\tr_{nk}[G (B+B^{\sT}) G (D' \otimes I_k)] = \tr_{nk}[B G (D' \otimes I_k) G] + \tr_{nk}[G (D' \otimes I_k) G B^{\sT}]
\]
Now since $G(z,t)$ is the inverse of a symmetric matrix, it is symmetric (here we mean symmetric with respect to transposition rather than Hermitian conjugation).  Also, $D'$ is symmetric.  Thus, since transposition preserves the trace,
\[
\tr_{nk}[G (D' \otimes I_k) G B^{\sT}] = \tr_{nk}[B G (D' \otimes I_k) G].
\]
Thus, we get
\[
\tr_{nk}[G B_{\sym} G (D' \otimes I_k)] = \tr_{nk}[B G (D' \otimes I_k) G] = \frac{1}{kn} \sum_{i,j=1}^{kn} B_{i,j} [G (D' \otimes I_k) G]_{j,i}.
\]
We take expectations on both sides and then use Gaussian integration by parts:
\[
\frac{1}{kn} \sum_{i,j=1}^{kn} \E \left[ B_{i,j} [G (D' \otimes I_k) G]_{j,i} \right] = \frac{1}{k^2n^2} \sum_{i,j=1}^{kn} \E \left[ \frac{\partial}{\partial B_{i,j}} \left[ G (D' \otimes I_k) G \right]_{j,i} \right]
\]
Using \pref{lem: resolvent derivative}, we get
\begin{align*}
\frac{\partial}{\partial B_{i,j}} G &= G \frac{\partial}{\partial B_{i,j}}[\beta t^{1/2}(B + B^{\sT})] G \\
&= \beta t^{1/2} G (E_{i,j} + E_{j,i}) G \\
&= \beta t^{1/2} G E_{i,j} G + \beta t^{1/2} G E_{j,i} G,
\end{align*}
where $E_{i,j}$ are the standard matrix units in $M_{nk}(\C)$.  Thus,
\begin{align}
\beta t^{-1/2} \E \tr_{nk}[G B_{\sym} G (D' \otimes I_k)] &= \frac{\beta^2}{k^2n^2} \sum_{i,j=1}^{nk} \E  [G E_{i,j}  G (D' \otimes I_k) G]_{j,i} \label{eq: matrix resolvent computation term 1} \\
&\quad + \frac{\beta^2}{k^2n^2} \sum_{i, j=1}^{nk} \E [G E_{j,i}  G (D' \otimes I_k) G]_{j,i} \label{eq: matrix resolvent computation term 2} \\
&\quad + \frac{\beta^2}{k^2n^2} \sum_{i,j=1}^{nk} \E [G(D' \otimes I_k)G E_{i,j} G]_{j,i} \label{eq: matrix resolvent computation term 3} \\
&\quad + \frac{\beta^2}{k^2n^2} \sum_{i,j=1}^{nk} \E [G (D' \otimes I_k)G E_{j,i} G]_{j,i} \label{eq: matrix resolvent computation term 4},
\end{align}
Taking first the terms \eqref{eq: matrix resolvent computation term 2} and \eqref{eq: matrix resolvent computation term 4} that contain $E_{j,i}$, note that
\[
\eqref{eq: matrix resolvent computation term 2} = \frac{\beta^2}{k^2n^2} \sum_{i, j=1}^{nk} \E [G_{j,j} [G (D' \otimes I_k) G]_{i,i}] = \beta^2 \tr_{nk}[G] \tr_{nk}[G (D' \otimes I_k) G].
\]
Likewise, \eqref{eq: matrix resolvent computation term 4} evaluates to $\tr_{nk}[G(D' \otimes I_k)G] \tr_{nk}[G]$, which also equals \eqref{eq: matrix resolvent computation term 2}.  Therefore, \eqref{eq: matrix resolvent computation term 2} and \eqref{eq: matrix resolvent computation term 4} become the term $2 \beta^2 \E \bigg[ \tr_{nk}[G] \tr_{nk}[G(D' \otimes I_k)G]$ on the right-hand side of \eqref{eq: IBP 1 formula} in the statement of the lemma.

Next, consider the terms \eqref{eq: matrix resolvent computation term 1} and \eqref{eq: matrix resolvent computation term 3} that contain $E_{i,j}$.  Recall that $G^{\sT} = G$, or $G_{i,j} = G_{j,i}$.  Hence,
\begin{align*}
\eqref{eq: matrix resolvent computation term 1} &= \frac{\beta^2}{k^2n^2} \sum_{i,j=1}^{nk} \E  [G_{j,i} [G (D' \otimes I_k) G]_{j,i}] = \frac{\beta^2}{k^2n^2} \sum_{i,j=1}^{nk} \E  [G_{i,j} [G (D' \otimes I_k) G]_{j,i}] \\
&= \frac{\beta^2}{nk} \E \tr_{nk}[G[G(D' \otimes I_k)G]] = \frac{\beta^2}{nk} \E \tr_{nk}[G^3(D' \otimes I_k)].
\end{align*}
Similarly, \eqref{eq: matrix resolvent computation term 3} also evaluates to $\frac{\beta^2}{nk} \E \tr_{nk}[G^3(D' \otimes I_k)]$.  Therefore, \eqref{eq: matrix resolvent computation term 1} and \eqref{eq: matrix resolvent computation term 3} become the term $\frac{2 \beta^2}{nk} \E \tr_{nk}[G^3(D' \otimes I_k)]$ on the right-hand side of \eqref{eq: IBP 1 formula} in the statement of the lemma.
\end{proof}

We now carry out a parallel computation for the matrix $A_{\sym} \otimes I_k$ rather than $B_{\sym}$.  In the following, recall that we identify $M_{nk}(\C)$ with $M_n(\C) \otimes M_k(\C)$.  Let $\tr_n \otimes \id: M_{nk}(\C) \to M_k(\C)$ be the partial trace map given on simple tensors by $X \otimes Y \mapsto \tr_n(X)Y$, and let $\sT \otimes \id: M_{nk}(\C) \to M_{nk}(\C)$ be the partial transpose map given by $X \otimes Y \mapsto X^{\sT} \otimes Y$.

\begin{lemma} \label{lem: IBP 2}
Continuing with the same setup as above and again abbreviating $G(z,t)$ to $G$, we have
\begin{multline} \label{eq: IBP 2 formula}
\beta (1-t)^{-1/2} \E \tr_{nk}[G (A_{\sym} \otimes I_k) G (D' \otimes I_k)]
= 2 \beta^2 \E \tr_k[(\tr_n \otimes \id)[G] (\tr_n \otimes \id)[G(D' \otimes I_k) G]] \\
 + \frac{2\beta^2}{n} \E \tr_{nk}[(\sT \otimes \id)[G] G (D' \otimes I_k) G].
\end{multline}
\end{lemma}

\begin{proof}
The argument is similar to the previous lemma.  First, we observe that $(A \otimes I_k)^{\sT} = A^{\sT} \otimes I_k$ and $G^{\sT} = G$, and apply symmetry of the trace and transposes to get
\[
\tr_{nk}[G (A^{\sT} \otimes I_k) G (D' \otimes I_k)] = \tr_{nk}[(D' \otimes I_k) G (A \otimes I_k) G] = \tr_{nk}[(A \otimes I_k)G(D' \otimes I_k) G].
\]
Therefore,
\begin{align*}
\tr_{nk}[G (A_{\sym} \otimes I_k) G (D' \otimes I_k)] &= \tr_{nk}[(A \otimes I_k) G (D' \otimes I_k) G] \\
&= \sum_{i,j=1}^n A_{i,j} \tr_{nk}[(E_{i,j} \otimes I_k) G (D' \otimes I_k) G] \\
&= \frac{1}{n} \sum_{i,j=1}^n \E \left[ \frac{\partial}{\partial A_{i,j}} \tr_{nk}[(E_{i,j} \otimes I_k) G (D' \otimes I_k) G] \right],
\end{align*}
where the last line follows from Gaussian integration by parts.  Note that
\[
\frac{\partial}{\partial A_{i,j}} G = \beta (1 - t)^{1/2} G((E_{i,j} + E_{j,i}) \otimes I_k) G,
\]
and so
\[
\frac{\partial}{\partial A_{i,j}} [(E_{i,j} \otimes I_k) G (D' \otimes I_k) G] = (E_{i,j} \otimes I_k) G ((E_{i,j} + E_{j,i}) \otimes I_k) G (D' \otimes I_k) + (E_{i,j} \otimes I_k) G (D' \otimes I_k) ((E_{i,j} + E_{j,i}) \otimes I_k) G.
\]
Thus, we get
\begin{align}
\beta (1-t)^{-1/2} \E \tr_{nk}[G (A_{\sym} \otimes I_k) G (D' \otimes I_k)] &= \frac{\beta^2}{n} \sum_{i,j=1}^{n} \E  \tr_{nk}[(E_{i,j} \otimes I_k)G(E_{i,j} \otimes I_k) G (D' \otimes I_k) G] \label{eq: matrix resolvent computation 2 term 1} \\
&\quad + \frac{\beta^2}{n} \sum_{i,j=1}^{n} \E \tr_{nk} [(E_{i,j} \otimes I_k) G (E_{j,i} \otimes I_k) G (D' \otimes I_k) G] \label{eq: matrix resolvent computation 2 term 2} \\
&\quad + \frac{\beta^2}{n} \sum_{i,j=1}^{n} \E \tr_{nk}[(E_{i,j} \otimes I_k)G(D' \otimes I_k)G (E_{i,j} \otimes I_k) G] \label{eq: matrix resolvent computation 2 term 3} \\
&\quad + \frac{\beta^2}{n} \sum_{i,j=1}^{n} \E \tr_{nk}[(E_{i,j} \otimes I_k)G (D' \otimes I_k)G (E_{j,i} \otimes I_k)G] \label{eq: matrix resolvent computation 2 term 4}.
\end{align}
Note that \eqref{eq: matrix resolvent computation 2 term 2} and \eqref{eq: matrix resolvent computation 2 term 4} are the same up to cyclic symmetry of the trace and switching the indices $i$ and $j$.  Each of these two terms evaluates to
\begin{align*}
\frac{\beta^2}{n} \sum_{i,j=1}^{n} \E  \tr_{nk}[(E_{i,j} \otimes I_k)G(E_{i,j} \otimes I_k) G (D' \otimes I_k) G] &= \beta^2 \tr_{nk}[((\tr_n \otimes \id)(G) \otimes I_k) G (D' \otimes I_k) G] \\
&= \beta^2 \mathbb{E} \tr_k[(\tr_n \otimes \id)[G] (\tr_n \otimes \id)[G(D' \otimes I_k) G].
\end{align*}
Hence, \eqref{eq: matrix resolvent computation 2 term 2} and \eqref{eq: matrix resolvent computation 2 term 4} together produce the term $2\beta^2 \mathbb{E} \tr_k[(\tr_n \otimes \id)[G] (\tr_n \otimes \id)[G(D' \otimes I_k) G]$ on the right-hand side of \eqref{eq: IBP 2 formula} in the lemma statement.  Meanwhile, \eqref{eq: matrix resolvent computation 2 term 1} and \eqref{eq: matrix resolvent computation 2 term 3} are equal to each other using cyclic symmetry of the trace, and each of these terms evaluates to
\begin{align*}
\frac{\beta^2}{n} \sum_{i,j=1}^{n} \E  \tr_{nk}[(E_{i,j} \otimes I_k)G(E_{i,j} \otimes I_k) G (D' \otimes I_k) G] &= \frac{\beta^2}{n} \sum_{i,j=1}^{n} \E  \tr_{nk}[(E_{i,i} \otimes I_k)(\sT \otimes \id)[G](E_{j,j} \otimes I_k) G (D' \otimes I_k) G] \\
&= \frac{\beta^2}{n} \tr_{nk}[(\sT \otimes \id)[G] G (D' \otimes I_k) G].
\end{align*}
Hence, \eqref{eq: matrix resolvent computation 2 term 1} and \eqref{eq: matrix resolvent computation 2 term 3} together produce the term $\frac{2\beta^2}{n} \tr_{nk}[(\sT \otimes \id)[G] G (D' \otimes I_k) G]$ on the right-hand side of \eqref{eq: IBP 2 formula} in the lemma statement.
\end{proof}

To summarize, the last two lemmas together with \eqref{eq: formula for h prime} show that $h'(t)$ is the difference of \eqref{eq: IBP 1 formula} and \eqref{eq: IBP 2 formula}, which we want to estimate in order to prove Proposition \ref{prop: expectation estimate}.  We first give an estimate for the rightmost terms in \eqref{eq: IBP 1 formula} and \eqref{eq: IBP 2 formula} respectively, which are the easier terms to deal with since they already have factors of $1/n$.

\begin{lemma} \label{lem: estimates of transpose trace terms}
With the setup above, we have
\begin{equation} \label{eq: estimate transpose trace term 1}
\left|\frac{2\beta^2}{nk} \mathbb{E} \tr_{nk}[G^3 (D' \otimes I_k)]| \right| \leq \frac{2\beta^2}{nk |\im z|^3} \norm{D'}_2
\end{equation}
and
\begin{equation} \label{eq: estimate transpose trace term 2}
\left| \frac{2\beta^2}{n} \E \tr_{nk}[(\sT \otimes \id)[G] G (D' \otimes I_k) G] \right| \leq \frac{2\beta^2}{n |\im z|^3} \norm{D'}_2.
\end{equation}
\end{lemma}

\begin{proof}
We first note that $\norm{G(z,t)} \leq \frac{1}{|\im z|}$ as a direct application of \pref{lem: resolvent derivative} (1).   Hence,
\[
|\tr_{nk}[G^3 (D' \otimes I_k)]| \leq \norm{G^3}_2 \norm{D' \otimes I_k}_2 \leq \norm{G}^3 \norm{D'}_2,
\]
where we use that $\norm{D' \otimes I_k}_2^2 = \tr_{nk}(((D')^*D' \otimes I_k)) = \tr_n((D')^*D') = \norm{D'}_2^2$.  This easily implies \eqref{eq: estimate transpose trace term 1}.  For the other estimate, note that
\begin{align*}
|\tr_{nk}[(\sT \otimes \id)[G] G (D' \otimes I_k) G]| &\leq \norm{(\sT \otimes \id)[G]}_2 \norm{G(D' \otimes I_k)G}_2 \\
&\leq \norm{G}_2 \norm{G} \norm{D' \otimes I_k}_2 \norm{G} \\
&\leq \norm{G}^3 \norm{D'}_2.
\end{align*}
Here we use the fact that $(\sT \otimes \id)$ is isometric with respect to $\norm{\cdot}_2$ since it performs a permutation of the entries of the matrix.\footnote{We need to use the $2$-norm here since $(\sT \otimes \id)$ is \emph{not} isometric with respect to operator norm.} We also use the non-commutative H{\"o}lder's inequality. 
\end{proof}

In order to estimate \eqref{eq: IBP 2 formula} minus \eqref{eq: IBP 1 formula}, it remains to estimate
\[
\E \tr_k[(\tr_n \otimes \id)[G] (\tr_n \otimes \id)[G(D' \otimes I_k) G]]
- \E [\tr_{nk}[G] \tr_{nk}[G(D' \otimes I_k) G]].
\]
This is the more challenging part, and the key idea is to understand this as a covariance and use the Poincar{\'e} inequality.

\begin{lemma} \label{lem: application of Poincare}
\begin{multline*}
\left| \E \tr_k\Bigl[(\tr_n \otimes \id)[G] (\tr_n \otimes \id)[G(D' \otimes I_k) G] \Bigr]
- \E \Bigl[\tr_{nk}[G] \tr_{nk}[G(D' \otimes I_k) G] \Bigr] \right| \\
\leq \frac{16 \beta^2 \norm{D'}_2}{n^{3/2} |\im z|^5} + \frac{16 \beta^2 \norm{D'}_2}{n^{3/2} k^2 |\im z|^5}.
\end{multline*}
\end{lemma}

\begin{proof}
Consider the mappings $M_{nk}(\R) \to M_k(\C)$ given by $f(B) = (\tr_n \otimes \id)[G^*]$ and $g(B) = (\tr_n \otimes \id)[G(D' \otimes I) G]$, and note that
\[
\E \tr_k\Bigl[(\tr_n \otimes \id)[G] (\tr_n \otimes \id)[G(D' \otimes I_k) G]\Bigr] = \E \angles{f(B),g(B)}_{\tr_k}.
\]
To apply the Poincar{\'e} inequality, we first want to verify these are Lipschitz.  Using \pref{lem: resolvent derivative} (2), we see that $G$ is $\frac{2 \beta}{|\im z|^2}$-Lipschitz as a function of $B$ with respect to $\norm{\cdot}_2$, which implies that
\begin{equation}
f \text{ is } \frac{2 \beta}{|\im z|^2} \text{-Lipschitz in } B.
\end{equation}
For $g$, we use Lipschitzness of $G$ in $B$ together with the behavior of Lipschitz functions under products; more precisely, given two inputs $B$ and $\tilde{B}$, let $G$ and $\tilde{G}$ be the corresponding values of the resolvent $G$.  Then
\begin{align*}
\norm{G(D' \otimes I_k) G - \tilde{G}(D' \otimes I_k)\tilde{G}}_2 &\leq \norm{(G - \tilde{G})(D' \otimes I_k)G}_2 + \norm{\tilde{G}(D' \otimes I_k)(G - \tilde{G})}_2 \\
&\leq \norm{G - \tilde{G}}_2 \norm{D'} \norm{G} + \norm{\tilde{G}} \norm{D'} \norm{G - \tilde{G}}_2 \\
&\leq 2 \frac{\beta}{|\im z|^2} \frac{1}{|\im z|} \norm{D'}.
\end{align*}
Note here that we have the operator norm rather than the $2$-norm of $D'$.  Then since the $(\tr_n \otimes \id)$ is a contraction with respect to $2$-norm, we see that
\begin{equation}
g \text{ is } \frac{4 \beta}{|\im z|^3} \norm{D'} \text{-Lipschitz in } B.
\end{equation}
Now let $\E_B$ denote the expectation with respect to $B$ while holding $A$ fixed.  Note that $B$ is invariant in distribution with respect to conjugation by $I_n \otimes O$ for any matrix $O$ in the $k \times k$ orthogonal group, and of course $I_n \otimes O$ leaves $A \otimes I_k$ and $D' \otimes I_k$ literally invariant.  It follows that $\E_B[(\tr_n \otimes \id)[G]]$ is invariant under conjugation by the orthogonal group and hence is equal to a multiple of the identity; thus,
\[
\E_B[(\tr_n \otimes \id)[G]] = \E_B \tr_{nk}[G].
\]
Similarly,
\[
\E_B[(\tr_n \otimes \id)[G(D' \otimes I_k) G]] = \E_B \tr_{nk}[G(D' \otimes I_k) G].
\]
By the Poincar{\'e} inequality (\pref{cor: vector Poincare}),
\begin{align}
\Bigl|\E_B \tr_k[(\tr_n \otimes \id)[G](\tr_n \otimes \id)&[G(D' \otimes I_k) G]] - \E_B \tr_{nk}[G] \E_B \tr_{nk}[G(D' \otimes I_k) G] \Bigr| \nonumber \\
&= |\E_B \angles{f(B),g(B)}_{\tr_k} - \angles{\E_B[f(B)], \E_B[g(B)]}_{\tr_k}| \nonumber \\
&= |\E_B \angles{f(B) - \E_B f(B), g(B) - \E_B g(B)}| \nonumber \\
&\leq \frac{\dim_{\R} M_k(\C)}{(nk)^2} \norm{f}_{\operatorname{Lip}} \norm{g}_{\operatorname{Lip}} \nonumber \\
&\leq \frac{2}{n^2} \frac{2 \beta}{|\im z|^2} \frac{4 \beta \norm{D'}}{|\im z|^3} \nonumber \\
&= \frac{16 \beta^2 \norm{D'}}{n^2 |\im z|^5}. \label{eq: first application of Poincare}
\end{align}
We can also apply the same reasoning to the scalar-valued functions $B \mapsto \tr_{nk}[G]$ and $B \mapsto \tr_{nk}[G(D' \otimes I_k) G]$, and thus we get
\begin{equation} \label{eq: second application of Poincare}
|\E_B[\tr_{nk}[G] \tr_{nk}[G(D' \otimes I_k) G]] - \E_B \tr_{nk}[G] \E_B \tr_{nk}[G(D' \otimes I_k) G]| \leq \frac{16 \beta^2 \norm{D'}}{(nk)^2 |\im z|^5}.
\end{equation}
Combining \eqref{eq: first application of Poincare} and \eqref{eq: second application of Poincare} with the triangle inequality shows that
\begin{multline*}
\Bigl|\E_B \tr_k[(\tr_n \otimes \id)[G(z,t)](\tr_n \otimes \id)[G(z,t)(D' \otimes I_k) G(z,t)]]  - \E_B[\tr_{nk}[G(z,t)] \tr_{nk}[G(z,t)(D' \otimes I_k) G(z,t)]] \Bigr| \\
\leq  \frac{16 \beta^2 \norm{D'}}{n^2 |\im z|^5} + \frac{16 \beta^2 \norm{D'}}{(nk)^2 |\im z|^5}.
\end{multline*}
By integrating over $A$, we can obtain the same inequality with $\E$ rather than $\E_B$.  Finally, we recall that $\norm{D'} \leq n^{1/2} \norm{D'}_2$ to complete the proof.
\end{proof}

\begin{lemma} \label{lem: estimate of h prime}
\[
|h'(t)| \leq \frac{32 \beta^4 \norm{D'}_2}{n^{3/2} |\im z|^5} + \frac{32 \beta^4 \norm{D'}_2}{n^{3/2}k^2 |\im z|^5} + \frac{2 \beta^2 \norm{D'}_2}{n |\im z|^3} + \frac{2 \beta^2 \norm{D'}_2}{nk |\im z|^3}.
\]
\end{lemma}

\begin{proof}
By substituting \eqref{eq: IBP 1 formula} and \eqref{eq: IBP 2 formula} into \eqref{eq: formula for h prime}, we have
\begin{align}
h'(t) &= 2 \beta^2 \E \tr_{nk}[G] \tr_{nk}[G(D' \otimes I_k)G] - 2 \beta^2 \E \tr_k[(\tr_n \otimes \id)[G] (\tr_n \otimes \id)[G(D' \otimes I_k) G]] \label{eq: putting it together 1} \\
&\quad + \frac{\beta^2}{nk} \E \tr_{nk}[G^3 (D' \otimes I_k)] \label{eq: putting it together 2} \\
&\quad -\frac{2\beta^2}{n} \E \tr_{nk}[(\sT \otimes \id)[G] G (D' \otimes I_k) G]. \label{eq: putting it together 3}
\end{align}
We then estimate \eqref{eq: putting it together 1} by Lemma \ref{lem: application of Poincare}, estimate \eqref{eq: putting it together 2} by \eqref{eq: estimate transpose trace term 1}, and estimate \eqref{eq: putting it together 3} by \eqref{eq: estimate transpose trace term 2}, which yields the bound asserted in this lemma.
\end{proof}

\begin{proof}[Proof of \pref{prop: expectation estimate}]
Recall the values of $h(0)$ and $h(1)$ from \eqref{eq: h of zero} and \eqref{eq: h of one}, note that $|h(0) - h(1)| \leq \sup_{t \in [0,1]} |h'(t)|$, and apply our estimate for $h'$ from Lemma \ref{lem: estimate of h prime} to obtain
\begin{multline} \label{eq: final expectation error estimate finite k}
|\E \tr_n[(z I_n - 2 \beta A_{\sym} + D)^{-1} D'] - \E \tr_{nk}[(z I_{nk} - 2\beta B_{\sym} + D \otimes I_k)^{-1} (D' \otimes I_k)]| \\
\leq \frac{32 \beta^4 \norm{D'}_2}{n^{3/2} |\im z|^5} + \frac{32 \beta^4 \norm{D'}_2}{n^{3/2} k^2 |\im z|^5} + \frac{2 \beta^2 \norm{D'}_2}{n |\im z|^3} + \frac{2 \beta^2 \norm{D'}_2}{nk |\im z|^3}.
\end{multline}
Note that $2B_{\sym}$ is $\sqrt{2}$ times a standard $nk \times nk$ GOE matrix.  Moreover, $D \otimes I_k$ and $D' \otimes I_k$ are deterministic matrices.  Thus, the asymptotic freeness theorem (\pref{thm: asymptotic freeness}) implies that, for every non-commutative polynomial $p$, the traces $\tr_n[p(D \otimes I_k, D' \otimes I_k, 2 B_{\sym})]$ converge almost surely as $k \to \infty$ to $\tau_n[p(D,D',\sqrt{2} S)]$.  The same also applies to the resolvent since this can be approximated by polynomials uniformly on compact subsets of the upper half-plane.  Hence,
\[
\lim_{k \to \infty} \E \tr_{nk}[(z I_{nk} - 2 \beta B_{\sym} + D \otimes I_k)^{-1} (D' \otimes I_k)] = \tau_n[(z - \sqrt{2} \beta S + D)^{-1} D'].
\]
Thus, when we take $k \to \infty$ in \eqref{eq: final expectation error estimate finite k}, we obtain
\[
|\E \tr_n[(z - 2 \beta A_{\sym} + D)^{-1}D'] - \tau_n[(z - \sqrt{2} \beta S + D)^{-1}D']| \leq \frac{32 \beta^4 \norm{D'}_2}{n^{3/2} |\im z|^5} + \frac{2 \beta^2 \norm{D'}_2}{n |\im z|^3}.
\]
By \pref{thm: subordination}, letting $f(z)$ be the subordination function, we obtain
\[
\tau_n[(z - \sqrt{2} \beta S + D)^{-1}D'] = \tau_n[(f(z) + D)^{-1} D'] = \tr_n[(f(z) + D)^{-1}D'],
\]
so that
\[
\left|\tr_n[(\E[(z - 2 \beta A_{\sym} + D)^{-1}] - (f(z) + D)^{-1})D'] \right| \leq \left( \frac{32 \beta^4}{n^{3/2} |\im z|^5} + \frac{2 \beta^2}{n |\im z|^3} \right) \norm{D'}_2.
\]
Taking the supremum over $\norm{D'}_2 \leq 1$, we obtain the $\norm{\cdot}_2$-norm of $\E \circ E_{\mathcal{D}_n}[(z - 2 \beta A_{\sym} + D)^{-1}] - (f(z) + D)^{-1}$, and so \eqref{eq: expectation estimate} follows.
\end{proof}

\subsection{Conclusion of the random matrix argument}

\begin{proof}[Proof of \pref{thm:david-magic}]
Let $b$, $\tilde{a}$, $\tilde{b}$, $\tilde{z}$ be as in the statement of the theorem, namely, $b = \beta n^{-\delta}$, and $\tilde{z} := ib + 2 \beta^2 g_{-D}(ib)$ where $g_{-D}(z) = \tr_n((z+D)^{-1})$, and $\tilde{a}$ and $\tilde{b}$ are the real and imaginary parts of $\tilde{z}$.

\pref{item:4-maxspec} stating that $2 \beta^2 \tr_n(D^{-1})$ is the maximum of the spectrum of $\sqrt{2} \beta S - D$ where $S$ is a semicircular operator freely independent of $D$ follows from \pref{lem: D dependent shift} (2).  Now recall
\[
\tilde{a} = 2 \beta^2 \re \tr_n((ib+D)^{-1}),
\]
and so
\[
|2 \beta^2 \tr_n(D^{-1}) - \tilde{a}| \leq 2 \beta^2 |\re \tr_n[D^{-1} - (ib + D)^{-1}]|.
\]
We have by \pref{lem: real and imaginary parts} and resolvent identities that
\begin{align*}
\re \tr_n[D^{-1} - (ib + D)^{-1}] &= \tr_n[D^{-1} - D(b^2 + D^2)^{-1}] \\
&= \tr_n[D[D^{-2} - (b^2 + D^2)^{-1}]] \\
&= \tr_n[DD^{-2}b^2(b^2+D^2)^{-1}] \\
&\leq b^2 \tr_n[D^{-3}].
\end{align*}
This proves the asserted estimate on $a + 2 \beta^2 \tr_n[D^{-1}] - \tilde{a}$ since $b^2 = \beta^2 n^{-2\delta}$.

For \pref{item:4-approx-eigenvec}, note that
\[
\tr_n[P^2(\tilde{a} - 2 \beta A_{\sym} + D)^2] = \tr_n[\tilde{b}(\tilde{b}^2+(\tilde{a} - 2 \beta A_{\sym} + D)^2)^{-1}(\tilde{a} - 2 \beta A_{\sym} + D)^2] \leq \tilde{b},
\]
and by \pref{lem: D dependent resolvent} (1), we have $\tilde{b} \leq 2 \beta^2 b^3 \tr_n(D^{-4})$ and then we substitute $b = \beta n^{-\delta}$.

For \pref{item:4-approx-diag}, note that
\[
P^2 = -\im[(\tilde{a} + i\tilde{b} - 2 \beta A_{\sym} + D)^{-1}],
\]
and observe that $\im \circ E_{\mathcal{D}_n} = E_{\mathcal{D}_n} \circ \im$.  Letting $z := ib$, we estimate
\begin{align} 
\norm{E_{\mathcal{D}_n}[(\tilde{z} - 2 \beta A_{\sym} + D)^{-1}] - (z + D)^{-1}}_2 &\leq \norm{E_{\mathcal{D}_n}[(\tilde{z} - 2 \beta A_{\sym} + D)^{-1}] - \E E_{\mathcal{D}_n}[(\tilde{z} - 2 \beta A_{\sym} + D)^{-1}]}_2 \label{eq: RMT main splitting 1} \\
&\qquad\qquad + \norm{\E E_{\mathcal{D}_n}[(\tilde{z} - 2 \beta A_{\sym} + D)^{-1}] - (z + D)^{-1}}_2 \label{eq: RMT main splitting 2}
\end{align}
By \pref{lem: concentration estimate 2}, with probability at least $1 - M_1 \exp(-M_2 n)$ in the Gaussian matrix $A$, we have for all choices of $D$ that
\begin{equation} \label{eq: RMT main splitting solution}
\norm{E_{\mathcal{D}_n}[(\tilde{z} - 2 \beta A_{\sym} + D)^{-1}] - \E E_{\mathcal{D}_n}[(\tilde{z} - 2 \beta A_{\sym} + D)^{-1}]}_2 \leq M_3 \beta n^{-2\delta},
\end{equation}
which takes care of \eqref{eq: RMT main splitting 1}.

To estimate \eqref{eq: RMT main splitting 2}, let $f$ be the subordination function as in \pref{prop: subordination identity}; by construction $\tilde{z} = z + 2 \beta^2 g_{-D}(z)$, and so $f(\tilde{z}) = z$.  Hence, by \pref{prop: expectation estimate}, we have
\[
\norm{\E[(\tilde{z} - 2 \beta A_{\sym} + D)^{-1}] - (z + D)^{-1}}_2 \leq \frac{32 \beta^4}{n^{3/2} \tilde{b}^5} + \frac{2 \beta^2}{n \tilde{b}^3}.
\]
Then recall by \pref{lem: D dependent resolvent} (1) that $\tilde{b}^{-1} \leq 3\beta^2 b^{-3} = 3 \beta^{-1} n^{3\delta}$.  Also, $E_{\mathcal{D}_n}$ is contractive in $\norm{\cdot}_2$, so that
\[
\norm{\E E_{\mathcal{D}_n}[(\tilde{z} - 2 \beta A_{\sym} + D)^{-1}] - (z + D)^{-1}}_2 \leq \frac{M_4}{\beta n^{3/2-15\delta}} + \frac{M_5}{\beta n^{1-9\delta}}.
\]
Since $\delta \leq 1/17 < 3/34 < 1/11$, we have that $3/2 - 15 \delta \geq 2 \delta$ and $1 - 9 \delta \geq 2 \delta$.  Therefore,
\begin{equation} \label{eq: RMT main splitting solution 2}
\norm{\E E_{\mathcal{D}_n}[(\tilde{z} - 2 \beta A_{\sym} + D)^{-1}] - (z + D)^{-1}}_2 \leq M_6 \beta^{-1} n^{-2\delta}.
\end{equation}
Estimating \eqref{eq: RMT main splitting 1} by \eqref{eq: RMT main splitting solution} and \eqref{eq: RMT main splitting 2} by \eqref{eq: RMT main splitting solution 2} and bounding $\beta^{-1}$ by a constant, we get
\[
\norm{E_{\mathcal{D}_n}[(\tilde{z} - 2 \beta A_{\sym} + D)^{-1}] - (z + D)^{-1}}_2 \leq M_7 \beta n^{-2\delta}.
\]
Then taking the negated imaginary parts of the operators, we get
\[
\norm{E_{\mathcal{D}_n}[\tilde{b}(\tilde{b}^2 + (\tilde{a} - 2 \beta A_{\sym} + D)^2)^{-1}] - b(b^2 + D^2)^{-1}}_2 \leq M_7 \beta n^{-2\delta},
\]
which is the desired estimate with $C_3 = M_7$.

For \pref{item:4-trace-lower-bound}, note by \pref{lem: D dependent resolvent} (2),
\[
\tr_n[b(b^2+D^2)^{-1}] = -\im g_{-D}(ib) \geq \frac{1}{2 \beta^2} (b - b^3 \norm{D^{-1}}^2) \geq \frac{1}{2} \beta^{-1} n^{-\delta} (1 - \beta^2 n^{-2\delta} \norm{D^{-1}}^2).
\]
Then using \pref{item:4-approx-diag},
\begin{align*}
\tr_n[P^2] = \tr[E_{\mathcal{D}_n}(P^2)] &\geq \tr_n[b(b^2+D^2)^{-1}] - \norm{E_{\mathcal{D}_n}(P^2) - b(b^2+D^2)^{-1}}_2 \\
&\geq \frac{1}{2} \beta^{-1} n^{-\delta} - M_7 \beta n^{-2\delta} - \frac{1}{2} \beta \norm{D^{-1}}^2 n^{-3\delta}.
\end{align*}

For \pref{item:4-op-upper-bound}, observe that
\[
\norm{P^2} = \norm{\im (\tilde{z} - 2 \beta A_{\sym} + D)^{-1}} \leq \frac{1}{\im \tilde{z}} \leq \frac{3 \beta^2}{b^3} = 3 \beta^{-1} n^{3\delta}.  \qedhere
\]
\end{proof}

\section{Convergence to the Primal Auffinger-Chen SDE under fRSB}\label{sec:convergence-to-ac}

In this section, we analyze the steps of our algorithm and describe how it converges in empirical distribution to the solution to the primal Auffinger-Chen SDE.
To this end, for a vector $v\in \R^n$, we define $\emp(v) := \frac{1}{n} \sum_{j=1}^n \delta_{v_j}$ to refer to the empirical distribution of coordinates of that vector and for $Y$ a random variable in $\R$, we let $\operatorname{dist}(Y) \in \mathcal{P}(\R)$ denote the probability distribution of $Y$.

In \pref{thm:david-magic}, with high probability over $A$, we constructed, for any diagonal matrix $D$ with $2 \beta^2 \tr(D^{-2}) = 1$, a positive semi-definite matrix $P(D)$ whose mass was concentrated on the highest part of the bulk of the spectrum of $2 \beta A_{\sym} - D$.  While the normalization of $P(D)$ was convenient for the analysis of resolvents in the last section, it will be helpful now to renormalize it so that the covariance matrix we use in the algorithm has normalized trace approximately $1$.  In addition, to make the Taylor expansion analysis easier in the next section, we will arrange that our step is exactly orthogonal to the current position.  Thus, we introduce the following quantities:  For any $\sigma \in \R^n$, let
\[
D(t,\sigma) := \left( \frac{2 \beta^2}{n} \sum_{j=1}^n \partial_{y,y} \tilde{\Lambda}_{\gamma} \left(t, \sigma_j\right)^{-2} \right)^{1/2} 
\diag\left[\partial_{y,y} \tilde{\Lambda}_{\gamma} \left(t, \sigma_j\right)\right].
\]
Note that by construction $2 \beta^2 \tr_n(D(t,\sigma)^{-2}) = 1$.

Then, let $Q(t,\sigma)$ be the positive semi-definite matrix given by
\begin{align*}
Q(t,\sigma)^2 &:= 2 \beta n^{\delta} \Pi_{\sigma^\perp} P(D(t,\sigma))^2 \Pi_{\sigma^\perp}\\
&= 2 \beta n^\delta \Pi_{\sigma^\perp} \tilde{b}(\tilde{b}^2 + (\tilde{a} - 2 \beta A_{\sym} + D(t,\sigma))^2)^{-1} \Pi_{\sigma^\perp},
\end{align*}
where we use the notation from~\pref{thm:david-magic}.  As we will see below, $\tr_n(Q(t,\sigma)^2)$ tends to $1$ as $n \to \infty$, and the diagonal entries of $Q(t,\sigma)^2$ are well-approximated by $D(t,\sigma)^{-2}$.

For our algorithm, suppose \pref{ass:sk-frsb} holds, fix a step size $\eta$ and number of steps $K$ such that $K \eta \leq q_\beta^*$.  For $k = 0$, \dots, $K$, write $t_k = k \eta$.  Define a random vector $\sigma_k$ in $\R^n$ inductively as follows:  Let $\sigma_0 := 0$, and for $k \in \{0, \dots, K-1\}$, let
\begin{equation} \label{eq: sigma step}
\sigma_{k+1} := \sigma_k + \eta^{1/2} Q(t_k,\sigma_k) Z_k,
\end{equation}
where $Z_0, Z_1$, $Z_2$, \dots, $Z_{K-1}$ are independent standard Gaussian random vectors in $\R^n$.

As in \pref{lem:sde-closeness}, let $Y_t^{\gamma}$ be the solution to the SDE:
\[
dY_t^{\gamma} = \frac{\sqrt{2} \beta}{\partial_{y,y} \tilde{\Lambda}_{\gamma}(t,Y^{\gamma}_t)} \,dW_t,\;\;  Y_0 = 0.
\]
where $W_t$ is a standard Brownian motion (see \pref{def:ito-formula}).  By \pref{prop: Lambda estimates} (3), $1 / \partial_{y,y} \tilde{\Lambda}_{\gamma}(t,y)$ is $2$-Lipschitz in space, which implies that the solution to the SDE is well-defined.  For convenience, we write 
\[v(t,y) := \frac{1}{ \partial_{y,y} \tilde{\Lambda}_{\gamma}(t,y)}\,.\]

For the remainder of this section, we use the following abbreviated notations:
\begin{align*}
    D_k &:= D(t_k, \sigma_k)
    \\
    P_k &:= P(D_k)
    \\
    Q_k &:= Q(t_k, \sigma_k)
    \\
    v_k(y) &:= v(t_k, y)
\end{align*}
We also understand $v_k(\sigma_k)$ to be the component-wise application of $v_k$ to the vector $\sigma_k$, i.e. the $i$th component of $v_k(\sigma_k)$ is $v_k((\sigma_k)_i)$.

\begin{theorem}[Convergence of empirical distribution to primal SDE]  \label{thm:convergence-to-SDE}
Assume fRSB.  Let $\delta \in (0,1/19]$.  For some universal constants $C_1$, $C_2$, \dots, the following statement holds.  Suppose that the matrix $A$ satisfies the conclusions of \pref{thm:david-magic} \pref{item:4-approx-diag}, which happens with probability $1 - C_1 \exp(-C_2 n)$.  Suppose that $\eta \in (0,1)$ and $\beta \in (1,\infty)$, $\gamma > 0$, and $n \in \N$ such that $\eta n^{1/90} \geq 1$ and
\begin{equation} \label{eq: induction hypothesis for all thm}
C_4^{-1} (\exp( C_4 \beta^2) - 1) \left[ C_5 \beta^4 \eta + C_6 n^{-\delta} + C_7 e^{10 \beta^2} \gamma^2 + C_8 \beta^{-2} \eta^{-1} n^{-1/18} \right] \leq \frac{1}{8 \beta^2}.
\end{equation}
Let $k' \in \mathbb{N}$ with $k' \le K \leq 1/\eta$, and let $\sigma_k$ be defined as in \pref{eq: sigma step} for $k = 0$, \dots, $k'$. Then with probability 
\[
1 - C_9 n \eta^{-2} \exp(-C_{10} n^{2/9-4\delta}),
\]
in the Gaussian vectors $Z_0$, \dots, $Z_{k'-1}$ from \pref{eq: sigma step}, we have that for $k = 1$, \dots, $k'$,
\[
d_{W,2}(\emp(\sigma_k), \operatorname{dist}(Y_{t_k}^\gamma))^2
\leq C_4^{-1} (\exp( C_4 \beta^2 t_k) - 1) \left[ C_5 \beta^4 \eta + C_6 n^{-\delta} + C_7 e^{10 \beta^2} \gamma^2 + C_8 \beta^{-2} \eta^{-1} n^{-1/18} \right]
\]
and
\[ 
2d_{W,2}\left(\emp(\sigma_{k}), \operatorname{dist}(Y_{t_{k}}^\gamma)\right) + (1 + e^{5 \beta^2}) \gamma \leq \frac{1}{2 \sqrt{2} \beta} .\]
\end{theorem}

We argue inductively that with high probability $d_{W,2}(\emp(\sigma_k), \operatorname{dist}(Y^{\gamma}_{t_k}))$ is small.  The bulk of the argument is for the inductive step, and during this argument we will assume that $\sigma_k$ has already been chosen, and hence we treat it as deterministic.  Then we go from $\emp(\sigma_k)$ to $\operatorname{dist}(Y_{t_k}^\gamma)$ in several steps.  We compare the random empirical distribution of the next step $\emp(\sigma_{k+1})$ with its expectation using concentration of measure (see \S \ref{subsec: SDE concentration}).  We compare the expectation with the distribution of $Y_{t_{k+1}}^\gamma$ using the inductive hypothesis and Lipschitz estimates on $v$ (see \S \ref{subsec: SDE expected distribution}), for which we also need to use the control of the diagonal entries of the matrix $Q_k$ from \pref{thm:david-magic} \pref{item:4-approx-diag} (see \S \ref{subsec: SDE diagonal approximation}).

We will continue to use $C_1$, $C_2$, \dots to denote absolute constants in each statement and $M_1$, $M_2$, \dots for constants in the proofs.  We continue to assume $\beta \geq 1$.

\subsection{Controlling the diagonal of the covariance} \label{subsec: SDE diagonal approximation}

Here we will estimate the diagonal of the covariance matrix $Q_k^2$.  From \pref{thm:david-magic} \pref{item:4-approx-diag}, we can approximate the diagonal entries of $P(D_k)^2$ by $b(b^2 + D_k^2)^{-1}$ where $b = \beta n^{-\delta}$.  We want to show this is close to $b D_k^{-2}$, the diagonal matrix whose entries are given by $v_k(\sigma_k)$, so that the increments of $\sigma_k$ will resemble those in the SDE.  We first estimate the error from the normalization we performed on $D_k$ to make $2 \beta^2 \tr_n(D_k^{-2}) = 1$.

\begin{lemma} \label{lem: SDE multiplicative normalization error}
We have
\begin{align*}
\norm{D_k^{-1} - \diag(v_k(\sigma_{k}))}_2 &= \left| \left( \frac{1}{n} \sum_{j=1}^n \frac{1}{\partial_{y,y} \tilde{\Lambda}_{\gamma}(t_k,\sigma_{k,j})^2} \right)^{1/2} - \frac{1}{\sqrt{2} \beta} \right| \\
&\leq 2 d_{W,2}(\emp(\sigma_k), \operatorname{dist}(Y_{t_k}^\gamma)) + (1 + e^{5 \beta^2}) \gamma.
\end{align*}
\end{lemma}

\begin{proof}
By definition,
\[
\sqrt{2} \beta D_k^{-1} = \frac{\sqrt{2} \beta \diag(v_k(\sigma_{k}))}{\norm{\sqrt{2} \beta \diag(v_k(\sigma_{k}))}_2}.
\]
In general, for a nonzero vector $h$ in a Hilbert space,
\[
\norm{\frac{h}{\norm{h}} - h} = \left| \frac{1}{\norm{h}} - 1 \right| \norm{h} = \left| \norm{h} - 1 \right|.
\]
Applying this to $\sqrt{2} \beta \diag(v_k(\sigma_{k,j}))$ yields the first equality.  For the second estimate, let $\Sigma$ be a random variable on the same probability space as $Y_{t_k}^\gamma$ such that the marginal of $\Sigma$ is $\emp(\sigma_k)$ and $\Sigma$ and $Y_{t_k}^\gamma$ are optimally coupled, i.e. $\norm{\Sigma - Y_{t_k}^\gamma}_{L^2} = d_{W,2}(\emp(\sigma_k),\operatorname{dist}(Y_t^\gamma))$.  Since $v$ is $2$-Lipschitz by \pref{prop: Lambda estimates}, we have
\[
\left|\frac{1}{\sqrt{n}} \left|v_k(\sigma_{k})\right|_2 - \norm{v_k(Y_{t_k}^\gamma)}_{L^2} \right| \leq \norm{v_k(\Sigma) - v_k(Y_{t_k}^\gamma)}_{L^2} \leq 2 d_{W,2}(\emp(\sigma_k), \operatorname{dist}(Y_{t_k}^\gamma)).
\]
We then recall from \pref{cor: L2 norm of gamma SDE weight} that
\[
\left| \norm{v_k(Y_{t_k}^\gamma)}_{L^2} - \frac{1}{\sqrt{2} \beta} \right| \leq (1 + e^{5 \beta^2 t_k}) \gamma \leq (1 + e^{5 \beta^2}) \gamma.
\]
Combining these estimates yields the asserted statement.
\end{proof}

Now we are ready to control the diagonal of $Q_k^2$.  At the same time, we will update the properties of \pref{thm:david-magic} for the new normalization of $Q$ for later use in \S \ref{sec:energy-analysis}.

\begin{proposition} \label{prop: diagonal matrix error and normalized magic}
Suppose that $A$ is chosen from the high probability event in \pref{thm:david-magic}.  Suppose that $\gamma < 2$ and
\[
2 d_{W,2}\left(\emp(\sigma_k), \operatorname{dist}(Y_{t_k}^\gamma)\right) + (1 + e^{5 \beta^2}) \gamma \leq \frac{1}{2 \sqrt{2} \beta}.
\]
Then
\begin{enumerate}
    \item \textbf{Approximation for diagonal:} $\norm{E_{\mathcal{D}_n}[Q_k^2] - 2 \beta^2 D_k^{-2} }_2 \leq C_1 \beta^2 n^{-\delta}$.
    \label{item:approx-diag}
    \item \textbf{Approximation for diagonal in square root:} $\norm{(E_{\mathcal{D}_n}[Q_k^2])^{1/2} - \sqrt{2} \beta D_k^{-1}}_2 \leq C_3 \beta n^{-\delta/2}$.
    \label{item:approx-diag-root}
    \item \textbf{Approximation for trace:} $|\tr_n[Q_k^2] - 1| \leq C_1 \beta^2 n^{-\delta}$.
    \label{item:approx-trace}
    \item \textbf{Hilbert-Schmidt norm of diagonal:} $\norm{E_{\mathcal{D}_n}[Q_k^2]}_2 \leq C_2 \beta^2$.
    \label{item:hs-diagonal}
    \item \textbf{Operator norm bound:}  $\norm{Q_k^2} \leq C_4 n^{4 \delta}$.
    \label{item:opnorm-bound}
    \item \textbf{Approximate eigenvector condition:}
    \[
    \norm{\left(2 \beta A_{\sym} - D_k - 2 \beta^2 \tr_n(D_k^{-1})\right) Q_k}_2 \leq C_5 \beta^2 n^{-\delta} + C_6(\beta+\gamma^{-1})n^{-1/4+2\delta}.
    \]
    \label{item:approx-eigenvec}
\end{enumerate}
\end{proposition}

\begin{proof}
(1) By \pref{thm:david-magic} \pref{item:4-approx-diag}, writing $b = \beta n^{-\delta}$,
\begin{equation}
\label{eq:5.3.1-first-bound}
\norm{E_{\mathcal{D}_n}[P_k^2] - b (b^2 + D_k^2)^{-1}}_1 \leq \norm{E_{\mathcal{D}_n}[P_k^2] - b (b^2 + D_k^2)^{-1}}_2 \leq M_1 \beta n^{-2\delta}.
\end{equation}
We define an intermediate term $\Xi$ and find a common denominator
\[
\Xi := D_k^{-2} - (b^2 + D_k^2)^{-1} = b^2 (b^2 + D_k^2)^{-1} D_k^{-2}.
\]
Hence,
\begin{align}
\norm{ \Xi }_1 &\leq b^2 \norm{(b^2 + D_k^2)^{-1}} \norm{D_k^{-2}}_1 \notag \\
&\leq b^2 \norm{D_k^{-1}}^2 \frac{1}{2 \beta^2} \label{eq:Xi-bound}
\end{align}
Now recall that $D_k^{-1} = \diag(v_k(\sigma_{k})) / \norm{\sqrt{2} \beta \diag(v_k(\sigma_{k}))}_2$, and $v \leq (1 + \gamma)$ by \pref{prop: Lambda estimates} \pref{item:second-deriv} so that
\[
\norm{D_k^{-1}} \leq \frac{1+\gamma}{\norm{\sqrt{2} \beta \diag(v_k(\sigma_{k}))}_2}
\]
By \pref{lem: SDE multiplicative normalization error}, if $2 d_{W,2}(\emp(\sigma_k), \operatorname{dist}(Y_{t_k}^\gamma)) + (1 + e^{5 \beta^2}) \gamma \leq \frac{1}{2 \sqrt{2} \beta}$, then we obtain
\[
\norm{\diag(v_k(\sigma_{k}))}_2 \geq \frac{1}{\sqrt{2} \beta} - \frac{1}{2\sqrt{2} \beta} = \frac{1}{2 \sqrt{2} \beta}.
\]
Hence, $\norm{D_k^{-1}} \leq 2(1+\gamma) \leq 6$, and so, substituting the last two inequalities into \pref{eq:Xi-bound},
\[
\norm{\Xi }_1 \leq \frac{4(1+\gamma)^2}{2 \beta^2} b^2 = 2(1+\gamma)^2n^{-2\delta}.
\]
Therefore, by the triangle inequality with the preceding bound and \pref{eq:5.3.1-first-bound} and because $\gamma < 2$,
\[
\norm{E_{\mathcal{D}_n}[P_k^2] - b D_k^{-2}}_1 \leq M_1 \beta n^{-2\delta} + 2 b n^{-2\delta} \leq M_2 \beta n^{-2\delta}.
\]
Furthermore, by \pref{thm:david-magic} \pref{item:4-op-upper-bound},
\[
\norm{P_k^2 - \Pi_{\sigma_k^\perp} P_k^2 \Pi_{\sigma_k^\perp}}_2 \leq 2 \norm{P_k^2} \, \norm{1 - \Pi_{\sigma_k^\perp}}_2 \leq \frac{C}{n^{1/2} \tilde{b}} \leq M_3 \beta^2 b^{-3} n^{-1/2} = M_4 \beta^{-1} n^{-1/2+3\delta}.
\]
We can bound $n^{-1/2+3\delta}$ by $n^{-2\delta}$ as well since $\delta \leq 1/17$.  Hence, by triangle inequality with the two preceding bounds and the contractivity of $E_{\mathcal{D}_n}$,
\[
\norm{E_{\mathcal{D}_n}[\Pi_{\sigma_k^\perp} P_k^2 \Pi_{\sigma_k^\perp}] - b D_k^{-2}}_2 \leq M_4 \beta n^{-2\delta}.
\]
Then multiplying by $2 \beta^2 / b = 2 \beta n^\delta$,
\[
\norm{E_{\mathcal{D}_n}[Q_k^2] - 2 \beta^2 D_k^{-2}}_2  \leq 2M_4 \beta^2 n^{-\delta}.
\]
This proves the first claim.

(2) By the Powers-St{\o}rmer inequality, since $E_{\mathcal{D}_n}[Q_k^2]^{1/2}$ and $D_k$ are nonnegative matrices, we have
\begin{align*}
\norm{E_{\mathcal{D}_n}[Q_k^2]^{1/2} - \sqrt{2} \beta D_k^{-1}}_2 &\leq \norm{E_{\mathcal{D}_n}[Q_k^2] - 2 \beta^2 D_k^{-2}}_1^{1/2} \\
&\leq \norm{E_{\mathcal{D}_n}[Q_k^2] - 2 \beta^2 D_k^{-2}}_2^{1/2} \\
&\leq (2M_4)^{1/2} \beta n^{-\delta/2}.
\end{align*}

(3) The third claim is immediate from the first claim since $2 \beta^2 \tr_n(D_k^{-2}) = 1$ by construction.

(4) Note that
\[
\norm{E_{\mathcal{D}_n}[Q_k^2]}_2 \leq 72 \beta^2 \norm{D_k^{-2}}_2 + 2M_4 \beta^2 n^{-\delta} \leq 8 \beta^2 + 2M_4 \beta^2
\]
since we have $\norm{D_k^{-1}} \leq 6$.

(5) Note that by \pref{thm:david-magic} \pref{item:4-op-upper-bound},
\[
\norm{Q_k^2} \leq 2 \beta n^\delta \norm{P_k^2} \leq C n^{4\delta}.
\]

(6) Using \pref{thm:david-magic} \pref{item:4-approx-eigenvec}, letting $\overline{Q}_k^2 := 2 \beta n^\delta P_k^2$, we get
\[
\norm{\left(2 \beta A_{\sym} - D_k - \tilde{a}\right) \overline{Q}_k}_2^2 = \tr_n\left(\overline{Q}_k^2 (\tilde{a} - 2 \beta A_{\sym} + D_k)^2 \right) \leq 2 \beta^4 n^{-2\delta} \norm{D_k^{-1}}^2 \leq 72 \beta^4 n^{-2 \delta}
\]
since $\norm{D_k^{-1}} \leq 6$.
By the Powers-St{\o}rmer inequality, 
\[\norm{Q_k - \overline{Q}_k}_2^2 \le \norm{Q_k^2 - \overline{Q}_k^2}_1 \le 2\beta n^{\delta}\left(2\norm{(1 - \Pi_{\sigma_k^\perp})P_k^2}_1 + \norm{(1 - \Pi_{\sigma_k^\perp})P_k^2(1 - \Pi_{\sigma_k^\perp})}_1\right),\]
which is bounded above by $6\beta n^{\delta-1/2}\norm{P_k^2}$, since $\norm{1 - \Pi_{\sigma_k^\perp}}_2 = n^{-1/2}$.
So by triangle inequality,
\begin{equation}
\label{eq:5.3.6-third-bound}
\norm{\left(2 \beta A_{\sym} - D_k - \tilde{a}\right) Q_k}_2 \le 6 \sqrt{2} \beta^2 n^{-\delta} + \sqrt{6\beta} n^{\delta/2-1/4}\norm{P_k}\norm{2 \beta A_{\sym} - D_k - \tilde{a}}.
\end{equation}
Then recall from \pref{thm:david-magic} \pref{item:4-maxspec} that
\[
|\tilde{a} - 2 \beta^2 \tr_n(D_k^{-1})| \leq 2 \beta^4 n^{-2\delta} \tr_n(D_k^{-3}) \leq n^{-2 \delta} \norm{D_k^{-1}} \cdot 2 \beta^4 \tr_n(D_k^{-2}) \leq 12 \beta^2 n^{-2\delta}.
\]
Hence, by \pref{item:approx-trace},
\[
\norm{\left(\tilde{a} - 2 \beta^2 \tr_n(D_k^{-1})\right) Q_k}_2 \leq M_5 \beta \cdot 2\beta^2 n^{-2\delta}.
\]
Therefore, applying triangle inequality with \pref{eq:5.3.6-third-bound} and the preceding bound,
\[
\norm{\left(2 \beta A_{\sym} - D_k - 2 \beta^2 \tr_n(D_k^{-1})\right) Q_k}_2 \leq 6 \sqrt{2} \beta^2 n^{-\delta} + M_6n^{-1/4+2\delta}(\beta+\gamma^{-1}) + M_5 \beta^3 \cdot 2 n^{-2\delta}
,
\]
where we used \pref{prop: Lambda estimates} (1) and the high-probability event from \pref{thm:david-magic} \pref{item:4-approx-diag}  to obtain
$\norm{\tilde{a} - 2 \beta A_{\sym} + D_k}
\leq M_7(\beta + \gamma^{-1})$ and \pref{thm:david-magic} \pref{item:4-op-upper-bound} gives $\norm{P_k} \leq M_8 \beta^{-1/2} n^{3\delta/2}$.
\end{proof}

\subsection{The expected empirical distribution of the next iterate} \label{subsec: SDE expected distribution}

Our next goal is, having fixed $\sigma_k$, to estimate the distance between the \emph{expected} empirical distribution $\E [\emp(\sigma_{k+1}) \mid A, \sigma_k]$ and $\operatorname{dist}(Y_{t_{k+1}}^\gamma \mid Y_{t_{k}}^\gamma)$.  First, to make solution of the SDE easier to compare with the result of our algorithm, we want to approximate $Y_{t_{k+1}}^\gamma$ by $Y_{t_{k}}^\gamma + \sqrt{2} \beta v_k(Y_{t_k}^\gamma) (W_{t_{k+1}} - W_{t_k})$.  For this purpose, we use the following standard estimate for solutions to SDEs.

\begin{lemma} \label{lem: SDE tangent approximation}
For all $t \in [0,1-\eta]$, we have
\[
\norm{Y_{t+\eta}^\gamma - \left[Y_t^\gamma + \sqrt{2} \beta v(t,Y_t^\gamma) (W_{t+\eta} - W_t)\right]}_{L^2}^2 \leq C \beta^6 \eta^2.
\]
\end{lemma}

\begin{proof}
Recall
\[
Y_{t + \eta}^\gamma - \left(Y_t^\gamma + \sqrt{2} \beta v(t,Y_t^\gamma)(W_{t+\eta} - W_t)\right) = \int_{0 \leq h \leq \eta} \sqrt{2} \beta \left[v(t+h,Y_{t+h}^\gamma) - v(t,Y_t^\gamma)\right] dW_{t+h}.
\]
By the It{\^o} isometry, the Lipschitz bounds for $v$ (\pref{prop: Lambda estimates}), and the arithmetic-geometric mean inequality,
\begin{align*}
\E \left|Y_{t+\eta}^\gamma - Y_t^\gamma - \sqrt{2} \beta v(t,Y_t^\gamma) (W_{t+\eta} - W_t)\right|^2 &= 2 \beta^2 \int_{0\leq h \leq \eta} \E \left|v(t+h,Y_{t+h}^\gamma) - v(t,Y_t^\gamma)\right|^2\,dh \\
&\leq 2 \beta^2 \int_0^\eta \left(M_1 \beta^2 h + 2|Y_{t+h}^\gamma - Y_t^\gamma|\right)^2 \,dh \\
&\leq \int_0^\eta \left(4M_1^2 \beta^6 h^2 + 8 \beta^2 |Y^{\gamma}_{t+h} - Y^{\gamma}_t|^2\right)\,dh.
\end{align*}
Now note that since $v$ is bounded by $1+\gamma$, we also have
\[
\E \left|Y_{t+h}^\gamma - Y_t^\gamma\right|^2 = 2 \beta^2 \int_0^h \left|v(t+h',Y_{t+h'}^\gamma)\right|^2\,dh' \leq 2 \beta^2 h(1+\gamma).
\]
Therefore, we get
\[
\E \left|Y_{t+\eta}^\gamma - Y_t^\gamma - \sqrt{2} \beta v(t,Y^{\gamma}_t) (W_{t+\eta} - W_t)\right|^2 \leq \int_0^\eta (4M_1^2 \beta^6 h^2 + 16 \beta^4 h(1+\gamma)) \,dh = M_2 \beta^6 \eta^3 + M_3 \beta^4 \eta^2 \leq M_4 \beta^6 \eta^2.
\]
\end{proof}

\begin{lemma} \label{lem: SDE expectation estimate}
Assume $A$ satisfies the conclusions of \pref{thm:david-magic}.  Let $\sigma_k$ and $\operatorname{dist}(Y_{t_k}^\gamma)$ be fixed with
\[
2 d_{W,2}\left(\emp(\sigma_k), \operatorname{dist}(Y_{t_k}^\gamma)\right) + (1 + e^{5 \beta^2 t_k}) \gamma \leq \frac{1}{2 \sqrt{2} \beta}.
\]
Then,
\begin{multline*}
d_{W,2}\left(\E \left[\emp(\sigma_{k+1}) | A, \sigma_1, \dots, \sigma_k\right], \operatorname{dist}(Y_{t_{k+1}}^\gamma \mid Y_{t_{k}}^\gamma)\right)^2 - d_{W,2}\left(\emp(\sigma_k), \operatorname{dist}(Y_{t_k}^\gamma)\right)^2 \\
\leq \eta \beta^2\cdot \left[ C_1 \beta^4 \eta  + C_2 n^{-\delta} + C_3 d_{W,2}\left(\emp(\sigma_k),\operatorname{dist}(Y_{t_k}^\gamma)\right)^2 + C_4 e^{10\beta^2} \gamma^2 \right].
\end{multline*}
\end{lemma}

\begin{proof}
Consider the discrete probability space $[n]$ with normalized counting measure, and let $\Sigma$ be a random variable that takes value $\sigma_{k,j}$ on point $j$, so that $\Sigma$ has distribution $\emp(\sigma_k)$.  Also, let $\Xi$ be the random variable that takes value $((Q_k^2)_{j,j})^{1/2}$ at point $j$.  Let $W$ be a normal random variable of variance $\eta$ independent of $(\Sigma,\Xi)$.  We first observe that $\Sigma + \Xi W$ has the distribution $\E \emp(\sigma_{k+1})$.  Indeed, for smooth functions $f$,
\[
\E \frac{1}{n} \sum_{j=1}^n f(\sigma_{k+1,j}) = \E \frac{1}{n} \sum_{j=1}^n f\left(\sigma_{k,j} + \eta^{1/2} (Q_k Z_k)_j\right).
\]
Since $\eta^{1/2}(Q_k Z_k)_j$ is a normal random variable of variance $\eta (Q_k^2)_{j,j}$, the expectation is the same if we replace $\eta^{1/2}(Q_k Z_k)_j$ with $((Q_k^2)_{j,j})^{1/2} W$.  Thus,
\[
\E \frac{1}{n} \sum_{j=1}^n f\left(\sigma_{k,j} + \eta^{1/2} (Q_k Z_k)_j\right) = \E \frac{1}{n} \sum_{j=1}^n f\left(\sigma_{k,j} + ((Q_k^2)_{j,j})^{1/2} W\right) = \E f(\Sigma + \Xi W).
\]

By choosing an appropriate larger probability space, we can arrange that $\Sigma$, $\Xi$, and $(W_t)_{0 \leq t \leq t_k}$ be in the same probability space (hence also $Y_t^\gamma$ is on the same probability space since $Y_t^\gamma$ is obtained by solving the SDE with Brownian motion $W_t$) such that $\Sigma$ and $Y_{t_{k}}^\gamma$ are optimally coupled, that is,
\[
\norm{\Sigma - Y_{t_k}^\gamma}_{L^2} = d_{W,2}\left(\emp(\sigma_k), \operatorname{dist}(Y^{\gamma}_{t_k})\right).
\]
Then take the rest of the Brownian motion $(W_t - W_{t_k})_{t \geq t_k}$ to be independent of $(\Sigma,\Xi,(W_t)_{0 \leq t \leq t_{k+1}}$).  Since $W_{t_{k+1}} - W_{t_k}$ has variance $\eta$, we can identify it with the $W$ from the previous paragraph.

Making the conditioning in $\E \left[\emp(\sigma_{k+1}) | A, \sigma_1, \dots, \sigma_k\right]$ notationally implicit for brevity, we want to estimate
\[
d_{W,2}\left(\E \emp(\sigma_{k+1}),\operatorname{dist}(Y_{t_{k+1}}^\gamma \mid Y_{t_{k}}^\gamma)\right) \leq \norm{(\Sigma + \Xi W) - Y_{t_{k+1}}^\gamma}_{L^2} = \norm{(\Sigma - Y_{t_k}^\gamma) + (\Xi W - Y_{t_{k+1}}^\gamma + Y_{t_k}^\gamma)}_{L^2}.
\]
Observe that the conditional expectation of $\Xi W$ given $(Y_{t_k}^\gamma,\Sigma)$ is zero because $W$ is independent of these variables.  Moreover, since $Y_{t_{k+1}}^\gamma - Y_{t_k}^\gamma = \int_{t_k}^{t_{k+1}} \sqrt{2} \beta v(t,Y_t^\gamma)\,dW_t$, this also has conditional expectation zero given $(Y_{t_k}^\gamma,\Sigma)$.  Thus, $\Xi W - Y_{t_{k+1}}^\gamma + Y_{t_k}^\gamma$ and $\Sigma - Y_{t_k}^\gamma$ are orthogonal in $L^2$ and so
\[
\norm{\Sigma + \Xi W - Y^{\gamma}_{t_{k+1}}}_{L^2}^2 = \norm{\Sigma - Y_{t_k}^\gamma}_{L^2}^2 + \norm{\Xi W - Y_{t_{k+1}}^\gamma + Y_{t_k}^\gamma}_{L^2}^2.
\]
Thus,
\begin{equation}
\label{eq:lemma-5.5-pythagorean}
d_{W,2}\left(\E \emp(\sigma_{k+1}), \operatorname{dist}(Y^{\gamma}_{t_{k+1}} \mid Y_{t_{k}}^\gamma)\right)^2 \leq d_{W,2}\left(\emp(\sigma_k),\operatorname{dist}(Y_{t_k}^\gamma)\right)^2 + \norm{\Xi W -  Y_{t_{k+1}}^\gamma + Y_{t_k}^\gamma}_{L^2}^2.
\end{equation}
Next, we estimate by triangle inequality
\begin{equation}
\label{eq:lemma-5.5-first-triangle}
\norm{\Xi W -  Y_{t_{k+1}}^\gamma + Y_{t_k}^\gamma}_{L^2} \leq \norm{\Xi W - \sqrt{2} \beta v_k(Y_{t_k}^\gamma) W}_{L^2} + \norm{Y_{t_{k+1}}^\gamma - Y_{t_k}^\gamma - \sqrt{2} \beta v_k(Y_{t_k}^\gamma) W}_{L^2}. 
\end{equation}
By \pref{lem: SDE tangent approximation}, the second term of \pref{eq:lemma-5.5-first-triangle} satisfies
\begin{equation}
\label{eq:lemma-5.5-first-triangle-second-term}
\norm{Y_{t_{k+1}}^\gamma - Y_{t_k}^\gamma - \sqrt{2} \beta v_k(Y_{t_k}^\gamma) W}_{L^2} \leq M_1 \beta^3 \eta.
\end{equation}
Meanwhile, by another triangle inequality, the first term satisfies
\begin{align}
\norm{\Xi W - \sqrt{2} \beta v_k(Y_{t_k}^\gamma) W}_{L^2} &= \eta^{1/2} \norm{\Xi - \sqrt{2} \beta v_k(Y_{t_k}^\gamma)}_{L^2} \notag \\
\label{eq:lemma-5.5-second-triangle}
&\leq \eta^{1/2} \norm{\Xi - \sqrt{2} \beta v_k( \Sigma)}_{L^2} + \eta^{1/2} \norm{\sqrt{2} \beta v_k(\Sigma) - \sqrt{2} \beta v_k(Y_{t_k}^\gamma)}_{L^2}.
\end{align}
Furthermore, \pref{lem: SDE multiplicative normalization error} and \pref{prop: diagonal matrix error and normalized magic} \pref{item:approx-diag-root}, along with the fact that $1+e^{5\beta^2} \le 2e^{5\beta^2}$, show that the first term of \pref{eq:lemma-5.5-first-triangle} satisfies
\begin{align*}
\norm{\Xi - \sqrt{2} \beta v_k(\Sigma)}_{L^2} &\leq \norm{E_{\mathcal{D}_n}[Q_k^2]^{1/2} - \sqrt{2} \beta D_k^{-1}}_2 + \sqrt{2} \beta \norm{D_k^{-1} - \diag(v_k(\sigma_{k}))}_2 \\
&\leq  M_2 \beta n^{-\delta/2} + 2 \sqrt{2} \beta d_{W,2}\left(\emp(\sigma_k),\operatorname{dist}(Y_{t_k}^\gamma)\right) + 2 \sqrt{2} \beta e^{5 \beta^2} \gamma.
\end{align*}
Moreover, since $v$ is $2$-Lipschitz in space by \pref{prop: Lambda estimates}, we get for the second term
\[
\norm{\sqrt{2} \beta v_k(\Sigma) - \sqrt{2} \beta v_{k}(Y^{\gamma}_{t_{k}})}_{L^2} \leq M_3 \beta d_{W,2}\left(\emp(\sigma_k),\operatorname{dist}(Y^{\gamma}_{t_k})\right).
\]
Thus, combining the last two inequalities with \pref{eq:lemma-5.5-second-triangle},
\[
\norm{\Xi W - \sqrt{2} \beta v_{k}(Y^{\gamma}_{t_k})W}_{L^2} \leq \eta^{1/2}\left[M_2 \beta n^{-\delta/2} + M_4 \beta d_{W,2}\left(\emp(\sigma_k),\operatorname{dist}(Y^{\gamma}_{t_k})\right) + M_5 \beta e^{5\beta^2}\gamma\right],
\]
which then combines with \pref{eq:lemma-5.5-first-triangle} and \pref{eq:lemma-5.5-first-triangle-second-term} to yield
\[
\norm{\Xi W -  Y^{\gamma}_{t_{k+1}} + Y^{\gamma}_{t_k}}_{L^2} \leq M_1 \beta^3 \eta + \eta^{1/2} \beta \left[ M_2 n^{-\delta/2} + M_4 d_{W,2}\left(\emp(\sigma_k),\operatorname{dist}(Y^{\gamma}_{t_k})\right) + M_5 e^{5\beta^2}\gamma \right].
\]
Finally, using the arithmetic-geometric mean inequality,
\[
\norm{\Xi W -  Y^{\gamma}_{t_{k+1}} + Y^{\gamma}_{t_{k}}}_{L^2}^2 \leq \eta \beta^2 \cdot \left[ M_6 \beta^4 \eta + M_7 n^{-\delta} +  M_8 d_{W,2}\left(\emp(\sigma_k),\operatorname{dist}(Y^{\gamma}_{t_k})\right)^2 + M_9 e^{10 \beta^2} \gamma^2 \right].
\]
This with \pref{eq:lemma-5.5-pythagorean} shows the desired conclusion with $C_1 = M_6$, $C_2 = M_7$, $C_3 = M_8$, $C_4 = M_9$.
\end{proof}

\subsection{Concentration for the empirical distribution} \label{subsec: SDE concentration}

Finally, we will estimate the difference between the empirical distribution of $\sigma_{k+1}$ and its expectation using concentration of measure.  The estimate that we obtain below is derived from concentration of measure alone, and is certainly not optimal.  Obtaining optimal estimates for the Wasserstein distances for empirical distributions is a challenging problem, and readers well-versed in probability theory may try to obtain better estimates using more refined tools.  Here it will be convenient to estimate the $L^1$-Wasserstein distance first, rather than directly attacking the $L^2$-Wasserstein distance.  Recall that the $L^1$-Wasserstein distance is
\[
d_{W,1}(\mu,\nu) = \inf \big\{ \norm{X - Y}_{L^1}: X \sim \mu, Y \sim \nu \big\} = \sup_{f \in \Lip_1(\R), f(0) = 0}\left|\int f\,d(\mu - \nu)\right|,
\]
where the last equality is by Monge-Kantorovich-Rubinstein duality; see \pref{cor:mkr-duality}.

We also use the following construction of a dense family of $1$-Lipschitz functions which is well known and easy to verify.

\begin{lemma} \label{lem: net of Lipschitz functions}
Fix $\ell,m \in \N$.  Let $\Lip_1([-2^\ell,2^\ell])$ be the set of $1$-Lipschitz functions on $[-2^\ell,2^\ell]$ that vanish at $0$.  Then there exist $2^{2^{\ell+ m+1}}$ functions in $\Lip_1([-2^\ell,2^\ell])$ that are $1/2^m$-dense with respect to the uniform norm.  Specifically take the Lipschitz functions that vanish at zero and are linear with slopes $\pm 1$ on each interval of the form $[(j-1)/2^m,j/2^m]$, for $j = -2^{\ell+m}+1,\dots, 2^{\ell+m}$.  The number of such functions is $2^{2^{\ell+m+1}}$ since they are uniquely described by the choice of slope $\pm 1$ on each of the $2\cdot 2^{\ell+m}$ intervals.
\end{lemma}

This allows us to estimate the $L^1$ Wasserstein distance using concentration if the measures are supported in a bounded interval.  For this purpose, we estimate probabilities of the maximum coordinate of $\sigma_k$ being large.

\begin{lemma} \label{lem: concentration for max coordinate}
Let $\alpha > 0$.  We have
\[
\max_j |(Q_k Z_k)_j| \leq n^\alpha
\]
with probability at least
\[
1 - C_1 n \exp(-C_2 n^{2\alpha - 4 \delta}).
\]
\end{lemma}

\begin{proof}
Recall from \pref{prop: diagonal matrix error and normalized magic} \pref{item:opnorm-bound} that $\norm{Q_k^2} \leq M_1 n^{4\delta}$.  Therefore, $(Q_k Z_k)_j$ is a normal random variable with mean zero and variance bounded by $M_1 n^{4\delta}$.  Thus,
\[
\Pr\left(\left|\left(Q_k Z_k\right)_j\right| \geq n^\alpha\right) \leq M_2 \exp(-M_3 n^{2 \alpha-4\delta}).
\]
The asserted estimate therefore follows from a union bound over the coordinates.
\end{proof}

\begin{lemma} \label{lem: basic Gaussian estimate}
Let $Z$ be a real Gaussian random variable with mean $a$ and variance $b$.  Then
\[
\mathbb{E}\left[\mathbbm{1}_{|Z - a| \geq c} |Z - a|^2\right] \leq \frac{2b}{\sqrt{2 \pi}}(b^{-1/2}c + b^{1/2}c^{-1}) e^{-b^{-1}c^2/2}.
\]
\end{lemma}

\begin{proof}
Without loss of generality, assume that $a = 0$.  Let $\tilde{Z} = b^{-1/2} Z$ which is a standard normal.  Then
\begin{align*}
\mathbb{E}\left[\mathbbm{1}_{|Z| \geq c} |Z|^2\right] &= b \mathbb{E}\left[\mathbbm{1}_{|\tilde{Z}| \geq b^{-1/2}c} |\tilde{Z}|^2\right] \\
&= \frac{2b}{\sqrt{2 \pi}} \int_{b^{-1/2}c}^\infty z^2 e^{-z^2/2}\,dz \\
&= \frac{2b}{\sqrt{2 \pi}} \left(\left[-ze^{-z^2/2}\right]_{b^{-1/2}c}^\infty + \int_{b^{-1/2}c}^\infty e^{-z^2/2}\,dz\right) \\
&\leq \frac{2b}{\sqrt{2 \pi}} \left[b^{-1/2}c e^{-b^{-1}c^2/2} + b^{1/2} c^{-1} \int_{b^{-1/2}c}^\infty ze^{-z^2/2}\,dz\right] \\
&= \frac{2b}{\sqrt{2 \pi}}(b^{-1/2}c + b^{1/2}c^{-1}) e^{-b^{-1}c^2/2}. \qedhere
\end{align*}
\end{proof}

\begin{lemma} \label{lem: SDE Wasserstein concentration}
Let $\delta \in (0,1/19]$ and let $\alpha \in (2\delta,1/7-4\delta/7)$.
Assume that the parameters of \pref{alg:hessian-ascent} satisfy $q^* \le 1$ with $\eta$ and $K$ held constant. Assume that $\sigma_k$ has been chosen (and so is deterministic for the purposes of this lemma) and assume that
\[
\max_j |\sigma_{k,j}| \leq k \eta^{1/2} n^\alpha.
\]
Then, conditioned on $A, Z_0, \dots, Z_{k-1}$, with probability at least
\[
1 - C_1 \eta^{-1} n  \exp(-C_2 n^{2 \alpha-4\delta}) - C_3 \exp \left( -C_4 \beta^2n^{1-4\delta-4 \alpha} + C_5\beta^{-1}\eta^{-1}n^{3\alpha} \right).
\]
in the Gaussian vector $Z_k$, we have
\[
\max_j |\sigma_{k+1,j}| \leq (k+1) \eta^{1/2} n^\alpha
\]
and
\[
d_{W,2}\left(\emp(\sigma_{k+1}), \E \left[\emp(\sigma_{k+1}) | A, \sigma_1, \dots, \sigma_k\right]\right) \leq C_7 n^{-\alpha/2}.
\]
\end{lemma}

\begin{remark}
For concreteness, one may take $\alpha = 1/9$ and then $n^{2\alpha - 4 \delta} = n^{2/9-4\delta} = n^{1-4 \delta - 7 \alpha}$.
\end{remark}

\begin{proof}
By \pref{lem: concentration for max coordinate}, $\max_j |\sigma_{k+1,j} - \sigma_{k,j}| \leq \eta^{1/2} n^\alpha$ with probability at least $1 - M_1 n \exp(-M_2 n^{2 \alpha-4\delta})$ in the Gaussian vector $Z_k$ conditioned on $A, Z_0, \dots, Z_{k-1}$.  This in particular, together with $K\eta \le q^* \le 1$, implies that
\begin{equation}
\label{eq:max-coordinate-event}
    \max_j |\sigma_{k+1,j}| \leq (k+1) \eta^{1/2} n^\alpha \leq K \eta^{1/2} n^{\alpha} \leq \eta^{-1/2} n^\alpha
\end{equation}
with probability at least $1 - M_1 Kn \exp(-M_2 n^{2 \alpha-4\delta})$.
Fix $\ell$ such that $2^{\ell-1} \leq 2 \eta^{-1/2} n^\alpha \leq 2^\ell$.  We truncate the expected empirical distribution as follows:  Let $\operatorname{proj}_{[-2^\ell,2^\ell]}$ be the truncation map
\[
\operatorname{proj}_{[-2^\ell,2^\ell]}(s) = \begin{cases}
    -2^\ell, & s \in (-\infty, -2^\ell] \\
    s, & s \in [-2^\ell, 2^\ell] \\
    2^\ell, & s \in [2^\ell,\infty).
\end{cases}
\]
and let $\tau_0$ be the pushforward of $\tau = \E \left[\emp(\sigma_{k+1}) | A, \sigma_1, \dots, \sigma_k\right]$ under $\operatorname{proj}_{[-2^\ell,2^\ell]}$.  Since $|\sigma_{k,j}| \leq k\eta^{1/2} n^{\alpha} \le \eta^{-1/2} n^\alpha$ by assumption, we have
\[
|\sigma_{k+1,j} - \sigma_{k,j}| \leq \eta^{-1/2} n^\alpha \implies |\sigma_{k+1,j}| \leq 2 \eta^{-1/2} n^\alpha \leq 2^\ell.
\]
and thus
\[
|\sigma_{k+1,j} - \proj_{[-2^\ell,2^\ell]}(\sigma_{k+1,j})| \leq \mathbbm{1}_{|\sigma_{k+1,j} - \sigma_{k,j}| \geq \eta^{-1/2} n^\alpha} |\sigma_{k+1,j} - \sigma_{k,j}|.
\]
Note that $\sigma_{k+1,j}$ is Gaussian with mean $\sigma_{k,j}$ and variance bounded by $M \eta n^{4 \delta}$ by \pref{prop: diagonal matrix error and normalized magic} \pref{item:opnorm-bound}.  Therefore, by Lemma \ref{lem: basic Gaussian estimate},
\begin{align*}
\mathbb{E}[\mathbbm{1}_{|\sigma_{k+1,j} - \sigma_{k,j}| \geq \eta^{-1/2} n^\alpha} (\sigma_{k+1,j} - \sigma_{k,j})^2]
&\leq \frac{2M \eta n^{4 \delta}}{\sqrt{2\pi}}(M^{-1/2} \eta^{-1} n^{\alpha-\delta} + M^{1/2} \eta n^{\delta-\alpha}) e^{-M^{-1} \eta^{-1} n^{-2\delta} \eta^{-1} n^{2\alpha}} \\
&\leq 2 M^{3/2} n^{\alpha+3\delta}e^{-M^{-1} \eta^{-2} n^{2\alpha-2\delta}}
.\end{align*}
Therefore,
\begin{align}
d_{W,2}(\tau,\tau_0)^2 &\leq \int_{|x| \geq \eta^{-1/2} n^\alpha} \left(x - \operatorname{proj}_{[-2^\ell,2^\ell]}(x)\right)^2\,d\tau(x) \notag\\
&\leq \frac{1}{n} \sum_{j=1}^n \mathbb{E}\left[|\sigma_{k+1,j} - \proj_{[-2^\ell,2^\ell]}(\sigma_{k+1,j})|^2\right] \notag\\
&\leq \frac{1}{n} \sum_{j=1}^n \mathbb{E}\left[ \mathbbm{1}_{|\sigma_{k+1,j} - \sigma_{k,j}| \geq \eta^{-1/2} n^\alpha} (\sigma_{k+1,j} - \sigma_{k,j})^2\right] \notag\\
&\leq 2M^{3/2} n^{\alpha+3\delta} e^{-M^{-1} \eta^{-2} n^{2\alpha-2\delta}} \notag\\
&\leq M' n^{-\alpha},\label{eq:wass-tau-tau0}
\end{align}
where we choose $M'$ such that $M' \ge 2M^{3/2}n^{2\alpha+3\delta}e^{-M^{-1}\eta^{-2}n^{2\alpha-2\delta}}$ by choosing $M' = 2M^{3/2}n_0e^{-M^{-1}n_0^{1/9}}$ for $n_0$ solving $0 = \frac{d}{dn_0}\left(n_0e^{-M^{-1}n_0^{1/9}}\right)$.

Now we proceed to estimate $d_{W,1}(\emp(\sigma_{k+1}), \tau_0)$.  Take $f \in \Lip_1([-2^\ell,2^\ell])$, and extend $f$ to an $1$-Lipschitz function on $\R$ by setting $f$ to be constant on $(-\infty,-2^\ell]$ and $[2^\ell,\infty)$, so that $f \circ \proj_{[-2^\ell,2^\ell]} = f$.  Consider
\[
F_f(Z_k) = \frac{1}{n} \sum_{j=1}^n f\left(\sigma_{k,j} + \eta^{1/2} (Q_k Z_k)_j \right)
\]
where $Z_k$ is a standard Gaussian vector
and recall from the proof of \pref{lem: SDE expectation estimate} that the distribution of $\sigma_{k,j}+\eta^{1/2}(Q_kZ_k)$ over the randomness of $j$ and $Z_k$ is equal to $\E \left[\emp(\sigma_{k+1}) | A, \sigma_1, \dots, \sigma_k\right]$, so $\E_{Z_k} [F_f(Z_k)] = \int f\,d\tau = \int (f \circ \proj_{[-2^\ell,2^\ell]})d\tau = \int f\,d\tau_0$.  Moreover, $F_f(Z_k)$ is an $\eta^{1/2} \norm{Q_k} /\sqrt{n}$-Lipschitz function of $Z_k$, since 
\begin{align*}
\frac{1}{n} \sum_{j=1}^n \left|f\left(\sigma_{k,j} + \eta^{1/2} (Q_k Z_k)_j \right) - f\left(\sigma_{k,j} + \eta^{1/2} (Q_k Z_k')_j \right)\right|
&\le \frac{1}{n} \eta^{1/2}\sum_{j=1}^n \left|  (Q_k (Z_k-Z_k'))_j \right| 
\\&= \frac{\eta^{1/2}}{n} |Q_k (Z_k-Z_k')|_1 \le \frac{\eta^{1/2}}{\sqrt{n}} \norm{Q_k}\,|Z_k-Z_k'|_2
\,.
\end{align*}
Now we apply the concentration estimate to the Lipschitz function $F_f(Z_k)$ above, while noting that $\norm{Q_k} n^{-1/2} \leq M_5 n^{2\delta-1/2}$ by \pref{prop: diagonal matrix error and normalized magic} \pref{item:opnorm-bound}.  By concentration of Lipschitz functions of Gaussians (i.e. \pref{lem: Herbst} with $n=1$), we have
\[
\left| F_f(Z_k) - \E_{Z_k} \left[F_f(Z_k)\right]\right| \leq \frac{1}{2^m}
\]
with probability at least
\[
1 - M_6 \exp \left( -M_7 n^{1-4\delta}2^{-2m}\eta^{-1} \right).
\]

Choose $m$ so that $2^{-m} \leq \beta \eta^{1/2} n^{-2 \alpha} \leq 2^{-m+1}$.  By \pref{lem: net of Lipschitz functions} and \pref{cor:mkr-duality}, we can estimate the $L^1$ Wasserstein distance between measures supported in $[-2^\ell,2^\ell]$ up to an error of size $1/2^m$ by testing only the $1$-Lipschitz functions $f$ in some set  $\calF$ satisfying $|\calF| = 2^{2^{\ell+m+1}}$.  We use a union bound for the probability of error for each of these Lipschitz functions $f$.
This yields
\begin{align*}
d_{W,1}(\emp(\sigma_{k+1}), \tau_0)
&= 
\sup_{\norm{f}_{\mathrm{Lip}}\le 1}\left|\int f(d\emp(\sigma_{k+1}) - d\tau_0)\right|
\\&\le
\frac{1}{2^m}
+ \sup_{f \in \calF}\left|\int f(d\emp(\sigma_{k+1}) - d\tau_0)\right|
\\&= \frac{1}{2^m}
+ \sup_{f \in \calF}\left| F_f(Z_k) - \E_{Z_k} \left[F_f(Z_k)\right]\right|
\\&\leq \frac{2}{2^m} \leq 2 \eta^{1/2} n^{-2\alpha}
\end{align*}
with probability at least
\begin{equation}
  \label{eq:prob-lipschitz-function-union-bound}  
1 - M_6 2^{2^{\ell+ m+1}} \exp \left( -M_7 n^{1-4\delta} 2^{-2m} \eta^{-1} \right),
\end{equation}
provided that $\max_j |\sigma_{k+1,j}| \leq \eta^{-1/2} n^\alpha \leq 2^\ell$ which by \pref{eq:max-coordinate-event} happens with probability at least
\begin{equation}
\label{eq:max-coordinate-event-prob}
1 - M_1 Kn \exp(-M_2 n^{2 \alpha-4\delta}).
\end{equation}
Hence, in this high-probability event, by H\"older's inequality,
\begin{equation*}
d_{W,2}(\emp(\sigma_{k+1}),\tau_0) \leq 2^{\ell/2} d_{W,1}(\emp(\sigma_{k+1}),\tau_0)^{1/2} \leq M_8 \eta^{-1/4} n^{\alpha/2} \eta^{1/4} n^{-\alpha} = M_8 n^{-\alpha/2},
\end{equation*}
which combines with \pref{eq:wass-tau-tau0} by triangle inequality to obtain
\begin{equation}
\label{eq:event-emp-close-to-exp}
d_{W,2}(\emp(\sigma_{k+1}),\tau) \le M_9 n^{-\alpha/2}.
\end{equation}

It remains to evaluate the probability \pref{eq:prob-lipschitz-function-union-bound} coming from concentration.  By construction, $2^{\ell} \leq 4 \eta^{-1/2} n^\alpha$ and $2^{m} \le 2\beta^{-1} \eta^{-1/2} n^{2 \alpha}$.  Thus,
\[
2^{2^{\ell+m+1}} \leq M_{9} \exp\left(M_{10}\beta^{-1}\eta^{-1}n^{3\alpha}\right).
\]
Hence,
\begin{align*}
M_6 2^{2^{\ell+ m+1}}&\exp \left( -M_7 n^{1-4\delta} 2^{-2m} \eta^{-1} \right) \\
&\leq M_6 M_{9} \exp\bigg(M_{10}\beta^{-1}\eta^{-1}n^{3\alpha} \bigg) \exp \left( -M_7 n^{1-4\delta} 2^{-2m}\eta^{-1} \right) \\
&\leq M_{11} \exp \left( -M_{12} \beta^2n^{1-4\delta-4 \alpha} +M_{10}\beta^{-1}\eta^{-1}n^{3\alpha} \right),
\end{align*}
which, when combined with \pref{eq:prob-lipschitz-function-union-bound} and \pref{eq:max-coordinate-event-prob}, yields the final probability of the bounds \pref{eq:max-coordinate-event} and \pref{eq:event-emp-close-to-exp}.
\end{proof}

\subsection{Conclusion of the convergence argument}

We are ready to finish proving \pref{thm:convergence-to-SDE}.  As preparation, we record a small computation for the inductive step.

\begin{corollary} \label{cor: main SDE inductive step}
Let $\delta \in (0,1/19]$, and let $\alpha \in (2\delta,1/7-4\delta/7)$.  Assume that the parameters $\eta$ and $n$ satisfy $n^{1/90} \eta \geq 1$.  Let $A$ satisfy the conclusions of \pref{thm:david-magic}.
Assume that $\max_j |\sigma_{k,j}| \leq k \eta^{1/2} n^\alpha$, that $2 d_{W,2}(\emp(\sigma_k), \operatorname{dist}(Y_{t_k}^\gamma)) + (1 + e^{5 \beta^2 t_k}) \gamma \leq \frac{1}{2 \sqrt{2} \beta}$, and moreover that
\[
d_{W,2}(\E \left[\emp(\sigma_{k+1}) | A, \sigma_1, \dots, \sigma_k\right],\operatorname{dist}(Y_{t_{k+1}}^\gamma)) \leq 1.
\]
Then, conditioned on $A, Z_0, \dots, Z_{k-1}$, with probability in the Gaussian vector $Z_k$ of at least
\[
1 - C_1 n\eta^{-1}\exp \left( -C_2 n^{2/9-4\delta} \right),
\]
we have both
\begin{multline*}
d_{W,2}\left(\emp(\sigma_{k+1}), \operatorname{dist}(Y^{\gamma}_{t_{k+1}})\right)^2 - d_{W,2}\left(\emp(\sigma_k), \operatorname{dist}(Y^{\gamma}_{t_k})\right)^2 \\
\leq \eta \cdot \left[ C_4 \beta^2 d_{W,2}\left(\emp(\sigma_k),\operatorname{dist}(Y^{\gamma}_{t_k})\right)^2 + C_5 \beta^6 \eta + C_6 \beta^2 n^{-\delta} + C_7 \beta^2 e^{10\beta^2} \gamma^2 + C_8 \eta^{-1} n^{-1/18}  \right].
\end{multline*}
and
\[\max_j |\sigma_{k+1,j}| \leq (k+1)\eta^{1/2} n^{\alpha}.\]
\end{corollary}

\begin{proof}
By triangle inequality and making the conditioning in $\E \left[\emp(\sigma_{k+1}) | A, \sigma_1, \dots, \sigma_k\right]$ notationally implicit for brevity,
\begin{align*}
d_{W,2}\left(\emp(\sigma_{k+1}), \operatorname{dist}(Y^{\gamma}_{t_{k+1}})\right)^2 &\leq \left( d_{W,2}\left(\E \emp(\sigma_{k+1}), \operatorname{dist}(Y^{\gamma}_{t_{k+1}})\right) + d_{W,2}\left(\emp(\sigma_{k+1}), \E \emp(\sigma_{k+1})\right)\right)^2 \\
&= d_{W,2}\left(\E \emp(\sigma_{k+1}), \operatorname{dist}(Y^{\gamma}_{t_{k+1}})\right)^2 \\
& \quad + 2 d_{W,2}\left(\E \emp(\sigma_{k+1}), \operatorname{dist}(Y^{\gamma}_{t_{k+1}})\right) d_{W,2}\left(\emp(\sigma_{k+1}), \E \emp(\sigma_{k+1})\right) \\
&\quad + d_{W,2}\left(\emp(\sigma_{k+1}), \E \emp(\sigma_{k+1})\right)^2.
\end{align*}
Assuming that $d_{W,2}(\E \emp(\sigma_{k+1}),\operatorname{dist}(Y^{\gamma}_{t_{k+1}})) \leq 1$, we can bound
\begin{multline*}
2 d_{W,2}\left(\E \emp(\sigma_{k+1}), \operatorname{dist}(Y^{\gamma}_{t_{k+1}})\right) d_{W,2}\left(\emp(\sigma_{k+1}), \E \emp(\sigma_{k+1})\right) + d_{W,2}\left(\emp(\sigma_{k+1}), \E \emp(\sigma_{k+1})\right)^2 \\
\leq M_1 n^{-\alpha/2} + M_2 n^{-\alpha},
\end{multline*}
using \pref{lem: SDE Wasserstein concentration} with $\alpha = 1/9$. Combining this with \pref{lem: SDE expectation estimate} using triangle inequality, we obtain the asserted estimate with probability
\[
1 - M_3\eta^{-1}n \exp \left( -M_4 n^{2/9-4\delta}\right) - M_5\exp\left(-M_6 n^{5/9-4\delta} + M_7 \beta^{-1} \eta^{-1} n^{3/9} \right).
\]
By our assumptions $\beta \geq 1$ and $\eta n^{1/90} \geq 1$, and hence $\beta^{-1} \eta^{-1} n^{3/9} \leq n^{31/90}$.  Since $-M_6 n^{5/9-4\delta} + M_7 n^{31/90} - M_8 \le -M_9n^{2/9-4\delta}$, we can absorb the $M_5\exp\left(-M_6 n^{5/9-4\delta} + M_7 n^{31/90} \right)$ term into the $M_3\eta^{-1}n\exp(-M_4 n^{2/9 - 4 \delta})$ term.
\end{proof}

Now we are ready to finish the argument.

\begin{proof}[Proof of \pref{thm:convergence-to-SDE}]
We induct over $k = 0, \dots, k'-1$.
Write $\zeta_{n,\beta,\eta,\gamma} := M_2 \beta^4 \eta + M_3 n^{-\delta} + M_4 e^{10 \beta^2} \gamma^2 + M_5 \beta^{-2}\eta^{-1} n^{-1/18}$ and $\Delta_k := d_{W,2}(\emp(\sigma_{k}), \operatorname{dist}(Y_{t_{k}}^\gamma))$.
The inductive hypothesis will be that with probability at least $1 - M_6 n(k-1)\eta^{-1} \exp \left( -M_7 n^{2/9-4\delta} \right)$,
\begin{equation}
\label{eq:inductive-delta}
\Delta_k^2 \leq \sum_{j=0}^{k-1} (1 + M_1 \eta \beta^2)^j \eta\beta^2 \cdot \zeta_{n,\beta,\eta,\gamma}
\end{equation}
and that the hypotheses of \pref{cor: main SDE inductive step} will hold for this value of $k$ with $\alpha = 1/9$.

  Note that $\Delta_0 = 0$ since $\sigma_0 = 0$ and $Y_0 = 0$, and furthermore the hypotheses of \pref{cor: main SDE inductive step} are vacuous at $k=0$.
  
For the inductive step, combining the probability from the inductive hypothesis with the probability from \pref{cor: main SDE inductive step} puts us in an event with probability at least $1 - M_6 nk\eta^{-1} \exp \left( -M_7 n^{2/9-4\delta} \right)$, so we can assume both bounds of \pref{cor: main SDE inductive step} hold.
The first bound says that
\[
\Delta_{k+1}^2 - \Delta_k^2 \leq \eta \cdot \left[ M_1 \beta^2 \Delta_k^2 + M_2 \beta^6 \eta + M_3 \beta^2 n^{-\delta} + M_4 \beta^2 e^{10 \beta^2} \gamma^2 + M_5 \eta^{-1} n^{-1/18} \right],
\]
and hence by \pref{eq:inductive-delta},
\[
\Delta_{k+1}^2 \leq (1 + M_1 \eta \beta^2) \Delta_k^2 + \eta\beta^2 \zeta_{n,\beta,\eta,\gamma}
\leq \sum_{j=0}^{k} (1 + M_1 \eta \beta^2)^j \eta\beta^2 \cdot \zeta_{n,\beta,\eta,\gamma}\,,
\]
as required by \pref{eq:inductive-delta} for the next step of the induction.

The second bound of \pref{cor: main SDE inductive step} gives the first hypothesis of \pref{cor: main SDE inductive step} for the next step.

For the second hypothesis, note that
\begin{align}
\Delta_{k+1}^2 &\leq \sum_{j=0}^{k} (1 + M_1 \eta \beta^2)^j \eta\beta^2 \cdot \zeta_{n,\beta,\eta,\gamma} \notag\\
&= \frac{(1 + M_1 \eta \beta^2)^{k+1} - 1}{M_1 \eta \beta^2} \eta\beta^2 \cdot \zeta_{n,\beta,\eta,\gamma} \notag\\
\label{eq:SDE-inductive-hypothesis-exponential}
&\leq M_1^{-1} \left(\exp( M_1 \beta^2 (k+1) \eta) - 1\right) \zeta_{n,\beta,\eta,\gamma}\,.
\end{align}
By the assumption \pref{eq: induction hypothesis for all thm} in the theorem statement, an application of the arithmetic-geometric mean inequality, upper-bounding $(k+1) \eta$ by 1, and upper-bounding $(1 + e^{5\beta^2})^2\gamma^2$ by $C_4^{-1}C_7 (\exp( C_4 \beta^2) - 1)e^{10\beta^2}\gamma^2$,
\begin{equation*}
\label{eq: induction hypothesis for all}
2\left( M_1^{-1} \left(\exp( M_1 \beta^2 (k+1) \eta) - 1\right) \zeta_{n,\beta,\eta,\gamma} \right)^{1/2} + (1 + e^{5\beta^2}) \gamma \leq \frac{1}{2 \sqrt{2} \beta}\,.
\end{equation*} 
The last two inequalities combine to give
\begin{equation}
\label{eq:final-wasserstein-2-bound-loose}
2d_{W,2}\left(\emp(\sigma_{k+1}), \operatorname{dist}(Y_{t_{k+1}}^\gamma)\right) + (1 + e^{5 \beta^2}) \gamma \leq \frac{1}{2 \sqrt{2} \beta},
\end{equation}
which implies the second hypothesis of \pref{cor: main SDE inductive step} for the next step.

Finally, observe that the same induction that leads to \pref{eq:final-wasserstein-2-bound-loose} also works when $\Delta_k$ is taken instead to be $d_{W,2}(\E \left[\emp(\sigma_{k+1}) | A, \sigma_1, \dots, \sigma_k\right], \operatorname{dist}(Y_{t_{k+1}}^\gamma))$, with the only difference being fewer applications of the triangle inequality and consequently fewer terms to upper-bound (alternatively, we may apply \pref{lem: SDE expectation estimate} with \pref{eq:final-wasserstein-2-bound-loose}), so, invoking $\beta \ge 1$, 
\[
d_{W,2}\left(\E \left[\emp(\sigma_{k+1}) | A, \sigma_1, \dots, \sigma_k\right], \operatorname{dist}(Y^{\gamma}_{t_{k+1}})\right) \leq \frac{1}{2 \sqrt{2} \beta} \leq 1,
\]
which is the final assumption needed for \pref{cor: main SDE inductive step}, finishing the proof of the inductive step.

By \pref{eq:SDE-inductive-hypothesis-exponential}, we obtain that for all $k \in \{0,\dots,k'\}$,
\begin{align*}
\Delta_k^2
&\leq M_1^{-1} \left(\exp( M_1 \beta^2 k' \eta) - 1\right) \zeta_{n,\beta,\eta,\gamma}
\end{align*}
with probability at least $1 - M_6 nk\eta^{-1} \exp \left( -M_7 n^{2/9-4\delta} \right)$, as required by the first conclusion of the theorem statement.
The second conclusion is immediately implied by \pref{eq:final-wasserstein-2-bound-loose}.
\end{proof}

\section{Energy Analysis}
\label{sec:energy-analysis}

In this section we finally obtain guarantees on the quality of the output of \pref{alg:hessian-ascent} by estimating the changes in our modified objective function at each step using Taylor expansion.  We will estimate the terms in the Taylor expansion using the bounds on the spectral properties of the Hessian~($\S$~\ref{sec:free-prob-for-hessian}), the convergence of the iterates in empirical distribution~($\S$~\ref{sec:convergence-to-ac}), and concentration inequalities.  The Taylor expansion is taken to the second degree in space (with third-order remainder) due to the Brownian-motion-like nature of the process and the need to use the spectral properties of the Hessian.

\begin{theorem}
Assume \pref{ass:sk-frsb}.  Let the Gaussian matrix $A$ be chosen from the high probability event in \pref{thm:david-magic} \pref{item:4-approx-diag}.  Consider the objective function
\[
    \obj(t,\sigma) = \beta \angles{\sigma, A \sigma} - \sum_{i=1}^n\tilde{\Lambda}_{\gamma}(t,\sigma_i) - \beta^2 n \int_t^1 s F_\mu(s)\,ds\,,
\]
and fix $\beta = \frac{10}{\eps}$, $\gamma = e^{-C_1 \beta^2}$ and $\eta = \gamma^{8}$ and $n \geq \eta^{-90}$ for some sufficiently large constant $C$.  Let $z \in \{-1,1\}^n$ be the final solution output by~\pref{alg:hessian-ascent} on input $A$ with access to $\mu_\beta$ and $\Phi$ (and its derivatives). Then, conditioned on $A$, with probability at least $1 - C_1 \exp(-C_2 n^{1/3 - 4\delta} + C_3 (\log \eta^{-1})^2)$,
\[
    \frac{1}{n}H(z) \ge \frac{\calP_{\beta}}{\beta} - \frac{\eps}{5} - O(\eps^2) - O(n^{-\alpha})\, , 
\]
where $\alpha = \min\left(\frac{\delta}{4}, \frac{1}{24}\right)$.
\end{theorem}

Here we use the following notation:
\begin{align*}
t_k &= k \eta \\
Q_k &= Q(t_k,\sigma_k) \\
\Delta \sigma_k &= \sigma_{k+1} - \sigma_k = \eta^{1/2} Q_k Z_k.
\end{align*}
and
\[
v(t,\sigma) = \frac{1}{\partial_{y,y} \tilde{\Lambda}_{\gamma}(t,\sigma)}.
\]
For a function $f(t,y)$ (for instance $\partial_y \tilde{\Lambda}_{\gamma}$ or $\partial_{y,y} \tilde{\Lambda}_{\gamma}$), we write $f(t,\cdot)$ for the coordinate-wise application of $f$ to a vector in space, i.e.
\[
f(t,\sigma_k) = (f(t,\sigma_{k,j}))_{j=1}^n\,.
\]

\subsection{Overview of Taylor expansion}

We will express the total change in the objective function as a telescoping sum of the changes at each step.  Thus, most of the section will be devoted to estimating the increment $\obj(t_{k+1},\sigma_{k+1}) - \obj(t_k,\sigma_k)$.  Here we will describe the terms that arise from the Taylor expansion, then in the following sections we will estimate each one of them, and in the conclusion we will sum up the steps and deduce that the value of $\angles{\sigma, A \sigma}$ is close to the energy functional that describes the theoretical maximum, which is computed in \cite[Section 3]{montanari2021optimization}.

To evaluate the increment, let us first separate the space update and the time update:
\[
\operatorname{obj}(t_{k+1}, \sigma_{k+1}) - \operatorname{obj}(t_k,\sigma_k) = [\operatorname{obj}(t_k,\sigma_{k+1}) - \operatorname{obj}(t_k,\sigma_k)] + [\operatorname{obj}(t_{k+1},\sigma_{k+1}) - \operatorname{obj}(t_k,\sigma_{k+1})] 
\]
The time update is further broken down as follows, using the integral form of Taylor expansion:
\begin{align*}
& \quad \operatorname{obj}(t_{k+1},\sigma_{k+1}) - \operatorname{obj}(t_k, \sigma_{k+1}) \\
&= -\sum_{j=1}^n \int_{t_k}^{t_{k+1}} \partial_t \tilde{\Lambda}_{\gamma}(t,\sigma_{k+1,j})\,dt + n \beta^2 \int_{t_k}^{t_{k+1}} F_\mu(t) t \,dt \\
&=_{\text{\pref{prop:primal-pde-lipschitz}}} -\beta^2 \sum_{j=1}^n \int_{t_k}^{t_{k+1}} \left( \frac{1}{\partial_{y,y} \tilde{\Lambda}_{\gamma}(t,\sigma_{k+1,j})} - \gamma + F_{\mu}(t)\left((\sigma_{k+1})_j - \gamma \partial_y \tilde{\Lambda}_{\gamma}(t,(\sigma_{k+1,j}))\right)^2 \right) \,dt \\
&\qquad \qquad \qquad \qquad \qquad + n \beta^2 \int_{t_k}^{t_{k+1}} F_\mu(t) t \,dt \\
&= -\beta^2 \sum_{j=1}^n \int_{t_k}^{t_{k+1}} v(t,\sigma_{k+1,j})\,dt + n \beta^2 \gamma \eta - \beta^2 \int_{t_k}^{t_{k+1}} F_\mu(t) \left( \sum_{j=1}^n (\id - \gamma \partial_y \tilde{\Lambda}_{\gamma}(t,\cdot))(\sigma_{k+1,j})^2 - nt \right)\,dt
\end{align*}
The space update can be broken down as follows:
\begin{align*}
& \quad    \operatorname{obj}(t_k, \sigma_{k+1}) - \operatorname{obj}(t_k,\sigma_k) \\
&= \beta \left( \angles{\sigma_{k+1}, A_{\sym} \sigma_{k+1}} - \angles{\sigma_k, A_{\sym} \sigma_k} \right) - \sum_{j=1}^n \left( \tilde{\Lambda}_{\gamma}(t_k,\sigma_{k+1,j}) - \tilde{\Lambda}_{\gamma}(t_k,\sigma_{k,j}) \right).
\end{align*}
Now we write
\[
\angles{\sigma_{k+1}, A_{\sym} \sigma_{k+1}} - \angles{\sigma_k, A_{\sym} \sigma_k} = 2 \angles{\Delta \sigma_k, A_{\sym} \sigma_k} + \angles{\Delta \sigma_k, A_{\sym} \Delta \sigma_k}.
\]
Furthermore, by the Taylor expansion with Lagrange remainder, there is some $\xi_{k,j}$ between $\sigma_{k,j}$ and $\sigma_{k+1,j}$ such that
\[
\tilde{\Lambda}_{\gamma}(t_k,\sigma_{k+1,j}) - \tilde{\Lambda}_{\gamma}(t_k,\sigma_{k,j}) = \partial_y \tilde{\Lambda}_{\gamma}(t_k,\sigma_{k,j}) \Delta \sigma_{k,j} + \frac{1}{2} \partial_{y,y} \tilde{\Lambda}_{\gamma}(t_k,\sigma_{k,j}) (\Delta \sigma_{k,j})^2 + \frac{1}{6} \partial_{y,y,y} \tilde{\Lambda}_{\gamma}(t_k, \xi_{k,j}) (\Delta \sigma_{k,j})^3.
\]

Finally, we combine all these terms.  For reasons that will be apparent later, we will add and subtract $\beta^2 \eta \sum_j v(t_k,\sigma_{k,j})$.  Thus,
\begin{align}
\operatorname{obj}(t_{k+1}, \sigma_{k+1}) &- \operatorname{obj}(t_k,\sigma_k) \\
=&\,\angles{2 \beta A_{\sym} \sigma_k - \partial_y \tilde{\Lambda}_{\gamma}(t_k,\sigma_k), \Delta \sigma_k} \label{eq: Taylor gradient term} \\
&+ \frac{1}{2} \angles{\Delta \sigma_k,  (2 \beta A_{\sym} - \diag(\partial_{y,y} \tilde{\Lambda}_{\gamma}(t_k,\sigma_{k,j}))) \Delta \sigma_k} - \beta^2 \eta \sum_{j=1}^n v(t_k,\sigma_{k,j}) \label{eq: Taylor Hessian term} \\
&+ \frac{1}{6} \sum_{j=1}^n \partial_{y,y,y} \tilde{\Lambda}_{\gamma}(t_k, \xi_{k,j}) (\Delta \sigma_{k,j})^3 \label{eq: Taylor third derivative term} \\
&- \beta^2 \int_{t_k}^{t_{k+1}} F_\mu(t) \left( \sum_{j=1}^n \left(\id - \gamma \partial_y \tilde{\Lambda}_{\gamma}(t,\cdot)\right)(\sigma_{k+1,j})^2 - nt \right)\,dt \label{eq: Taylor time term} \\
& + \beta^2 \sum_{j=1}^n \eta v(t_k,\sigma_{k,j}) - \beta^2 \sum_{j=1}^n \int_{t_k}^{t_{k+1}} v(t,\sigma_{k+1,j})\,dt \label{eq: Taylor telescoping error} \\
&+ n\beta^2 \gamma \eta \label{eq: Taylor gamma error}\,.
\end{align}
In the subsequent sections, we will show that each of these terms is small with high probability over the Gaussian vectors $Z_k$, where ``small'' means bounded by $\eta n$ times something which vanishes in the limit as $n \to \infty$, then $\eta \to 0$, then $\gamma \to 0$, then $\beta \to \infty$.

\begin{itemize}
    \item The \textbf{gradient term} \eqref{eq: Taylor gradient term} will have expectation zero in $Z_k$ and the fluctuations will be controlled by concentration of measure arguments in conjunction with regularity properties of $\tilde{\Lambda}_{\gamma}$.
    \item The \textbf{Hessian term} \eqref{eq: Taylor Hessian term} will be estimated using the approximate eigenvector condition for the covariance matrix that we chose, together with concentration arguments on the norms of a single update vector.
    \item The \textbf{third-order remainder term} \eqref{eq: Taylor third derivative term} will be estimated directly using concentration arguments that apply to the $\ell^3$ norm of the update $\Delta \sigma_k$.
    \item The \textbf{time term} \eqref{eq: Taylor time term} will be estimated using convergence in Wasserstein distance to the SDE as well as the relationship between $Y_t^\gamma$ and $Y_t$.
    \item The remaining error term \eqref{eq: Taylor telescoping error} will be handled by time continuity estimates for $v$ and by telescoping, and \eqref{eq: Taylor gamma error} does not require any further comment.
\end{itemize}
Summing up the estimates for the increments $\obj(t_{k+1},\sigma_{k+1}) - \obj(t_k,\sigma_k)$, we arrive at a lower bound for $\obj(t_K,\sigma_K)$, which then translates into a lower bound for the value of the Hamiltonian $H(\sigma_K)$ at the last iterate.  To analyze the potential function $V(t_K,\sigma_K)$ as well as the telescoping term \eqref{eq: Taylor telescoping error}, we use convergence of the empirical distributions of $\sigma_k$ to the distribution of $Y_t$ established in \S\ref{sec:convergence-to-ac}.  Under the form of fRSB that we have assumed, the energy in the large $n$ and large $\beta$ limit matches the true optimum, using standard arguments already presented in~\cite[Section 3]{montanari2021optimization}.  In other words, the convergence results and fRSB enable us to show that the energy gain produced by the algorithm matches that predicted by the Parisi solution.  We conclude by analyzing the effect of rounding and making appropriate choices for all the parameters.

\subsection{The gradient term}

Here we estimate the gradient term \eqref{eq: Taylor gradient term} using concentration of measure.

\begin{lemma}[Bounded gradient of modified objective under primal process]\label{lem:gradient-gaussian-chaos} \mbox{} \\
Let $\sigma_k$ be generated as in \pref{alg:hessian-ascent} and \pref{eq: sigma step}. Assume that the vectors $Z_0$, \dots, $Z_{k-1}$ from \pref{eq: sigma step} are in the high probability event from \pref{thm:convergence-to-SDE} and that the matrix $A$ satisfies the conclusions of \pref{thm:david-magic} \pref{item:4-approx-diag}.
Then, conditioned on $A, Z_0, \dots, Z_{k-1}$, with probability at least $1 - C_5 \exp(-C_6 n^{1/3 - 4 \delta})$ in the Gaussian vector $Z_k$, we have
     \[
     \left| \angles{2 \beta A_{\sym} \sigma_k - \partial_y \tilde{\Lambda}_{\gamma}(t_k,\sigma_k), \Delta \sigma_k} \right| \leq  C_3 \left( 3 \beta + \frac{1}{\gamma} \right) \eta^{1/2} n^{2/3}.
     \]
\end{lemma}

\begin{proof}
Recall $\Delta \sigma_k = \eta^{1/2} Q_k Z_k$.  Note
\[
\E [\angles{2 \beta A_{\sym} \sigma_k - \partial_y \tilde{\Lambda}_{\gamma}(t_k,\sigma_k), \eta^{1/2} Q_k Z_k} \mid A, Z_0, \dots, Z_{k-1}] = 0\,,
\]
and the quantity in the expectation is a Lipschitz function of $Z_k$ with Lipschitz norm bounded by
\[
\left( 2 \beta \norm{A_{\sym}} + \max_{i \in [n]}|\partial_y \tilde{\Lambda}_{\gamma}(t_k,(\sigma_k)_i)| \right)|\sigma_k|_2 \eta^{1/2} \norm{Q_k}.
\]
Recall $2 \norm{A_{\sym}} \leq 3$ since we are in the high probability event from \pref{thm:david-magic} \pref{item:4-approx-diag}. Furthermore, $\norm{Q_k} \leq M_1 n^{2\delta}$, which follows from~\pref{prop: Lambda estimates} and~\pref{prop: diagonal matrix error and normalized magic}.  Lastly, it is apparent from \pref{eq: Lambda gamma def 2} that $\Lambda(t,y)$ is an even function of $y$, so $\partial_y \Lambda(t,y)$ is an odd function, and this immediately yields that $\partial_y \tilde{\Lambda}_{\gamma}(t,0) = 0$. 
Consequently, since $\partial_{y,y} \tilde{\Lambda}_{\gamma} \leq 1/\gamma$ by \pref{prop: Lambda estimates} \pref{item:second-deriv}, we have
\[
|\partial_y \tilde{\Lambda}_{\gamma}(t,y)| \leq \frac{1}{\gamma} |y|.
\]

Hence, the Lipschitz norm can be bounded by
\[
M_1 \left( 3 \beta + \frac{1}{\gamma} \right) |\sigma_k|_2 \eta^{1/2} n^{2\delta}.
\]
Moreover, using~\pref{cor: L2 norm of gamma SDE weight},
\begin{align*}
n^{-1/2} |\sigma_k|_2 = \left( \frac{1}{n}\sum_{j=1}^n \sigma_{k,j}^2\right)^{1/2}\!\!  &{}\leq d_{W,2}(\emp(\sigma_k),\dist(Y_{t_k}^\gamma)) + \norm{Y_{t_k}^\gamma}_{L^2}
\\&{}\leq d_{W,2}(\emp(\sigma_k),\dist(Y_{t_k}^\gamma)) + t_k^{1/2} (1 + \sqrt{2} \beta (1 + e^{5 \beta^2 t_k}) \gamma).
\end{align*}
Therefore, the Lipschitz norm can be bounded by
\[
\left( 3 \beta + \frac{1}{\gamma} \right) \left( M_2 + M_3\, d_{W,2}(\emp(\sigma_k),\dist(Y_{t_k}^\gamma)) + M_4 \beta e^{5 \beta^2} \gamma \right) \eta^{1/2} n^{1/2 + 2\delta}.
\]
Since we are in the high probability event from \pref{thm:convergence-to-SDE}, we can assume that
\[
\left( M_2 + M_3\, d_{W,2}(\emp(\sigma_k),\dist(Y_{t_k}^\gamma)) + M_4 \beta e^{5 \beta^2} \gamma \right) \leq M_5.
\]
By the Gaussian concentration inequality for Lipschitz functions~\cite[Theorem 5.1.3]{vershynin2018high}, we obtain the asserted error bound; here note that $(n^{2/3})^2 / (n^{1/2 + 2 \delta})^2 = n^{1/3 - 4 \delta}$.
\end{proof}

\subsection{The Hessian term}

In order to control the Hessian term \eqref{eq: Taylor Hessian term}, we will first compare the quadratic expression
\[
\angles{\Delta \sigma_k, (2 \beta A_{\sym} - \diag(\partial_{y,y} \tilde{\Lambda}_{\gamma}(t_k,\sigma_{k,j}))) \Delta \sigma_k}
\]
with its expectation by a concentration inequality due to Hanson and Wright.  Then to control the expectation, we will use the approximate eigenvector condition from \pref{prop: diagonal matrix error and normalized magic} \pref{item:approx-eigenvec} which shows that the range of $Q(t_k,\sigma_k)^2$ is approximately in the kernel of $2 \beta A_{\sym} - D(t_k,\sigma_k) - 2 \beta^2 \tr(D(t_k,\sigma_k)^{-1})$.  However, since $D(t_k,\sigma_k)$ has a multiplicative normalization compared to $\diag(\partial_{y,y} \tilde{\Lambda}_{\gamma}(t_k,\sigma_{k,j}))$, we also need to control a few other approximation errors.

\begin{lemma}[Concentration of the Hessian term] \label{lem: Hessian concentration}
Suppose that the matrix $A$ satisfies the conclusions of \pref{thm:david-magic} \pref{item:4-approx-diag}.  Suppose that the parameters $\beta$, $K$, $\eta$, $n$ satisfy $n^{1/90} \eta \geq 1$ and \eqref{eq: induction hypothesis for all thm} and that the vectors $Z_0$, \dots, $Z_{k-1}$ are chosen in the high probability event in \pref{thm:convergence-to-SDE}.  Then conditioned on $A, Z_0, \dots, Z_{k-1}$, with probability at least $1 - \exp(-n^{1/3-4\delta})$ in the Gaussian vector $Z_k$, we have
\begin{multline*}
\left| \angles{\Delta \sigma_k, (2 \beta A_{\sym} - \diag(\partial_{y,y} \tilde{\Lambda}_{\gamma}(t_k,\sigma_{k,j}))) \Delta \sigma_k} - n \eta \tr_n\big[Q_k\big(2 \beta A_{\sym} - \diag(\partial_{y,y} \tilde{\Lambda}_{\gamma}(t_k,\sigma_{k,j}))\big) Q_k\big] \right| \\
\leq C \left( 3 \beta + \frac{1}{\gamma} \right) \beta^2 \eta n^{2/3}.
\end{multline*}
\end{lemma}

\begin{proof}
Recall that
\[
\angles{\Delta \sigma_k, \left(2 \beta A_{\sym} - \diag(\partial_{y,y} \tilde{\Lambda}_{\gamma}(t_k,\sigma_{k,j}))\right) \Delta \sigma_k} = \eta \angles{Z_k, Q_k \big(2 \beta A_{\sym} - \diag(\partial_{y,y} \tilde{\Lambda}_{\gamma}(t_k,\sigma_{k,j}))\big) Q_k Z_k}.
\]
Since $Z_k$ is a standard Gaussian vector,
\[
\E \angles{\Delta \sigma_k, \left(2 \beta A_{\sym} - \diag(\partial_{y,y} \tilde{\Lambda}_{\gamma}(t_k,\sigma_{k,j}))\right) \Delta \sigma_k} = n \eta \tr_n\left[Q_k \left(2 \beta A_{\sym} - \diag(\partial_{y,y} \tilde{\Lambda}_{\gamma}(t_k,\sigma_{k,j}))\right) Q_k\right].
\]
For convenience, let
\[
R_k := \eta Q_k \left(2 \beta A_{\sym} - \diag(\partial_{y,y} \tilde{\Lambda}_{\gamma}(t_k,\sigma_{k,j}))\right) Q_k.
\]
By the Hanson-Wright inequality,
\[
\Pr\left(|\angles{Z_k, R_k Z_k} - \E \angles{Z_k, R_k Z_k}| > \zeta \right) \leq 2 \exp\left(-\min\left(\frac{\zeta^2}{16 n\norm{R_k}_{2}^2}, \frac{\zeta}{4 \norm{R_k}}\right) \right).
\]
Here we compute $\norm{R_k}_{2} \le \norm{R_k}$ and
\begin{align*}
\norm{R_k} &= \eta \norm{Q_k \left(2 \beta A_{\sym} - \diag(\partial_{y,y} \tilde{\Lambda}_{\gamma}(t_k,\sigma_{k,j})) \right)Q_k} \\
&\leq \eta \norm{2 \beta A_{\sym} - \diag(\partial_{y,y} \tilde{\Lambda}_{\gamma}(t_k,\sigma_{k,j}))} \norm{Q_k^2} \\
&\leq \eta \left( 3 \beta + \frac{1}{\gamma} \right) M_1 n^{4\delta},
\end{align*}
where we applied \pref{prop: diagonal matrix error and normalized magic} \pref{item:opnorm-bound} and the facts that $\norm{2 A_{\sym}} \leq 3$ and $\partial_{y,y} \tilde{\Lambda}_{\gamma} \leq 1/\gamma$. 
Now we substitute
\[
\zeta = 4\eta \left( 3 \beta + \frac{1}{\gamma} \right) M_1 n^{2/3},
\]
so that
\[
\min\left(\frac{\zeta^2}{16n \norm{R_k}_{2}^2}, \frac{\zeta}{4 \norm{R_k}}\right) = \min \left( \frac{ n^{4/3}}{n^{1+4\delta}}, \frac{n^{2/3}}{n^{4 \delta}} \right) \geq n^{1/3-4\delta}. \qedhere
\]
\end{proof}

\begin{lemma}[Approximation errors in the Hessian term] \label{lem: Hessian expectation}
Suppose that the matrix $A$ satisfies the conclusions of \pref{thm:david-magic} \pref{item:4-approx-diag}.  Suppose that the parameters $\beta$, $K$, $\eta$, $n$ satisfy $n^{1/90} \eta \geq 1$ and \eqref{eq: induction hypothesis for all thm} and that the vectors $Z_0$, \dots, $Z_{k-1}$ are chosen in the high probability event in \pref{thm:convergence-to-SDE}.  Then
\begin{multline*}
\left| \tr_n[Q_k(2 \beta A_{\sym} - \diag(\partial_{y,y} \tilde{\Lambda}_{\gamma}(t_k,\sigma_{k,j}))) Q_k] - \frac{2 \beta^2}{n} \sum_{j=1}^n v(t_k,\sigma_{k,j}) \right| \\
\leq C_1 \beta^3 n^{-\delta} + C_2\beta^2(\beta+\gamma^{-1}) n^{-1/4+2\delta} + C_3 \beta^4 \left( 2 d_{W,2}(\emp(\sigma_k), \dist(Y_{t_k}^\gamma)) + (1 + e^{5 \beta^2}) \gamma \right).
\end{multline*}
\end{lemma}

\begin{proof}
Recall from \pref{sec:convergence-to-ac} that
\[
D_k = \sqrt{2} \beta \norm{\diag(v(t_k,\sigma_{k,j}))}_2 \diag(\partial_{y,y} \tilde{\Lambda}_{\gamma}(t_k,\sigma_{k,j}))
\]
Therefore, we can write
\begin{align*}
2 \beta A_{\sym} &- \diag(\partial_{y,y} \tilde{\Lambda}_{\gamma}(t_k,\sigma_{k,j})) - 2 \beta^2 \tr_n(\diag(v(t_k,\sigma_k))) \\
&= \frac{1}{\sqrt{2} \beta \norm{\diag(v(t_k,\sigma_{k,j}))}_2} \left( 2 \beta A_{\sym} - D_k - 2 \beta^2 \tr_n(D_k^{-1}) \right) \\
& \quad + \left( 1 - \frac{1}{\sqrt{2} \beta \norm{\diag(v(t_k,\sigma_{k,j}))}_2} \right) 2 \beta A_{\sym} \\
&\quad + \left(1 - \frac{1}{2 \beta^2 \norm{\diag(v(t_k,\sigma_{k,j}))}_2^2} \right) 2 \beta^2 \tr_n(\diag(v(t_k,\sigma_k)))
\end{align*}
Thus, by the triangle inequality,
\begin{align}
\bigg \lVert \big(2 \beta A_{\sym} &- \diag(\partial_{y,y} \tilde{\Lambda}_{\gamma}(t_k,\sigma_{k,j})) - 2 \beta^2 \tr_n(\diag(v(t_k,\sigma_k)))\big) Q_k \bigg \rVert_2 \label{eq: estimate for Hessian error bound} \\
&\le \frac{1}{\sqrt{2} \beta \norm{\diag(v(t_k,\sigma_{k,j}))}_2} \norm{ (2 \beta A_{\sym} - D_k - 2 \beta^2 \tr(D_k^{-1})) Q_k }_2 \label{eq: Hessian error bound term 1} \\
& \quad + \left| 1 - \frac{1}{\sqrt{2} \beta \norm{\diag(v(t_k,\sigma_{k,j}))}_2} \right| 2 \beta \norm{A_{\sym} Q_k}_2  \label{eq: Hessian error bound term 2} \\
&\quad + \left| \sqrt{2} \beta \norm{\diag(v(t_k,\sigma_{k,j}))}_2 - \frac{1}{\sqrt{2} \beta \norm{\diag(v(t_k,\sigma_{k,j}))}_2} \right| \sqrt{2} \beta \norm{Q_k}_2, \label{eq: Hessian error bound term 3}
\end{align}
where we used that $\tr(\diag(v(t_k,\sigma_{k,j}))) \leq \norm{\diag(v(t_k,\sigma_{k,j}))}_2$.  To estimate \eqref{eq: Hessian error bound term 1}, recall from \pref{prop: diagonal matrix error and normalized magic} \pref{item:approx-eigenvec} that
\[
\norm{ (2 \beta A_{\sym} - D_k - 2 \beta^2 \tr(D_k^{-1})) Q_k }_2 \leq M_1 \beta^2 n^{-\delta} + M_0(\beta+\gamma^{-1})n^{-1/4+2\delta}.
\]
To handle \eqref{eq: Hessian error bound term 2} and \eqref{eq: Hessian error bound term 3}, we recall from \pref{prop: diagonal matrix error and normalized magic} \pref{item:approx-trace} that
\[
\norm{Q_k}_2 \leq M_2 \beta,
\]
and in addition, $\norm{2 A_{\sym}} \leq 3$.  In order to bound \eqref{eq: Hessian error bound term 2} and \eqref{eq: Hessian error bound term 3}, we also need to estimate $\sqrt{2} \beta \norm{\diag(v(t_k,\sigma_{k,j}))}_2 - 1$; recall that in the proof of \pref{thm:convergence-to-SDE}, and in particular in \pref{prop: diagonal matrix error and normalized magic}, we arranged that
\[
\sqrt{2} \beta \norm{\diag(v(t_k,\sigma_{k,j}))}_2 \geq 1/2
\]
and
\begin{equation} \label{eq: another estimate for Hessian error}
\left| \sqrt{2} \beta \norm{\diag(v(t_k,\sigma_{k,j}))}_2 - 1 \right| \leq \sqrt{2} \beta \left( 2 d_{W,2}(\emp(\sigma_k), \dist(Y_{t_k}^\gamma)) + (1 + e^{5 \beta^2}) \gamma \right).
\end{equation}
We therefore have that
\begin{equation} \label{eq: another estimate for Hessian error 2}
\left| 1 - \frac{1}{(\sqrt{2} \beta \norm{\diag(v(t_k,\sigma_{k,j}))}_2)} \right| \leq 2 \cdot \sqrt{2} \beta \left( 2 d_{W,2}(\emp(\sigma_k), \dist(Y_{t_k}^\gamma)) + (1 + e^{5 \beta^2}) \gamma \right),
\end{equation}
which we will substitute into \eqref{eq: Hessian error bound term 2}.  Similarly, to estimate \eqref{eq: Hessian error bound term 3}, we use that
\[
\left| \sqrt{2} \beta \norm{\diag(v(t_k,\sigma_{k,j}))}_2 - \frac{1}{\sqrt{2} \beta \norm{\diag(v(t_k,\sigma_{k,j}))}_2} \right| \leq \left| \sqrt{2} \beta \norm{\diag(v(t_k,\sigma_{k,j}))}_2 - 1 \right| + \left| 1 - \frac{1}{\sqrt{2} \beta \norm{\diag(v(t_k,\sigma_{k,j}))}_2} \right|,
\]
which can then be estimated by \eqref{eq: another estimate for Hessian error} and \eqref{eq: another estimate for Hessian error 2}.  Substituting all these estimates into \eqref{eq: Hessian error bound term 1} - \eqref{eq: Hessian error bound term 3}, we obtain the following bound for \eqref{eq: estimate for Hessian error bound}:
\begin{align*}
    \bigg \lVert \big(2 \beta A_{\sym} - \diag(\partial_{y,y} \tilde{\Lambda}_{\gamma}(t_k,\sigma_{k,j})) - 2 \beta^2 \tr_n(\diag(v(t_k,\sigma_k)))\big) Q_k \bigg \rVert_2 &\leq M_3 \beta^2 n^{-\delta} + M_3\beta(\beta+\gamma^{-1}) n^{-1/4+2\delta} +\\
    &M_4 \beta^3 \left( 2 d_{W,2}(\emp(\sigma_k), \dist(Y_{t_k}^\gamma)) + (1 + e^{5 \beta^2}) \gamma \right).    
\end{align*}
This in turn implies that
\begin{align*}
\tr_n\big[Q_k\big(2 \beta A_{\sym} - \diag(\partial_{y,y} \tilde{\Lambda}_{\gamma}(t_k,\sigma_{k,j})) - &2 \beta^2 \tr_n(\diag(v(t_k,\sigma_k)))\big) Q_k\big] \leq M_5\beta^3 n^{-\delta} +  M_5\beta^2(\beta+\gamma^{-1}) n^{-1/4+2\delta} + \\
&\qquad\qquad\qquad M_6 \beta^4 \left( 2 d_{W,2}(\emp(\sigma_k), \dist(Y_{t_k}^\gamma)) + (1 + e^{5 \beta^2}) \gamma \right).
\end{align*}
Finally, we write
\begin{align*}
    &\tr_n[Q_k(2 \beta A_{\sym} - \diag(\partial_{y,y} \tilde{\Lambda}_{\gamma}(t_k,\sigma_{k,j}))) Q_k] - 2\beta^2\tr_n[\diag(v(t_k,\sigma_{k,j}))] = \\
    &\tr_n\big[Q_k\big(2 \beta A_{\sym} - \diag(\partial_{y,y} \tilde{\Lambda}_{\gamma}(t_k,\sigma_{k,j})) - 2 \beta^2 \tr_n(\diag(v(t_k,\sigma_k)))\big) Q_k\big] + 2\beta^2(\tr_n[Q_k^2] - 1) \tr_n[\diag(v(t_k,\sigma_{k,j}))],
\end{align*}
and recall from \pref{prop: diagonal matrix error and normalized magic} \pref{item:approx-trace} that $|\tr_n[Q_k^2] - 1| \leq M_7 \beta^2 n^{-\delta}$.  Since we already have the error term $M_5 \beta^3 n^{-\delta}$, this new error can be combined with the previous one up to changing constants.
\end{proof}

By combining \pref{lem: Hessian concentration} and \pref{lem: Hessian expectation}, we obtain the following result.

\begin{lemma}[Bound for the Hessian term]\label{lem:bounded-hessian-fluctuations}
Suppose that $Z_0$, \dots, $Z_{k-1}$ are chosen in the high probability event from \pref{thm:convergence-to-SDE}, which occurs with probability at least $1 - n\eta^{-2}C_1 \exp(-C_2 n^{2/9 - 4 \delta})$, and that the matrix $A$ satisfies the conclusions of \pref{thm:david-magic} \pref{item:4-approx-diag}, which happens with probability $1 - C_3 \exp(-C_4 n)$.  Then, conditioned on $A, Z_0, \dots, Z_{k-1}$, with probability $1 - 2 \exp(-n^{1/3-4\delta})$ in the Gaussian vector $Z_k$, we have
\begin{multline*}
\left| \frac{1}{2} \angles{\Delta \sigma_k, \big(2 \beta A_{\sym} - \diag(\partial_{y,y} \tilde{\Lambda}_{\gamma}(t_k,\sigma_{k,j}))\big) \Delta \sigma_k} - \beta^2 \eta \sum_{j=1}^n  v(t_k,\sigma_{k,j}) \right| \leq \eta n \cdot \biggl( C_1(3 \beta + \gamma^{-1}) \beta^2 n^{-1/3}\\
 + C_2 \beta^3 n^{-\delta} +  C_3\beta^2(\beta+\gamma^{-1}) n^{-1/4+2\delta} + C_4 \beta^4 d_{W,2}(\emp(\sigma_k),\dist(Y_{t_k}^\gamma)) + C_5 e^{5 \beta^2} \beta^4 \gamma \biggr).
\end{multline*}
\end{lemma}

\subsection{The remainder term}

This section proves the following bound for the third-order remainder.

\begin{lemma}[Bounded third-order Taylor remainder] \label{lem: Taylor remainder term}
Assume that $Z_0$, \dots, $Z_{k-1}$ are in the high probability event from \pref{thm:convergence-to-SDE} and that the matrix $A$ satisfies the conclusions of \pref{thm:david-magic} \pref{item:4-approx-diag}.  Then, conditioned on $A, Z_0, \dots, Z_{k-1}$, with probability at least $1 - C_1 \exp(-C_2 n^{1/3-4\delta})$, 
we have
\[
\sum_{j=1}^n |\partial_{y,y,y} \tilde{\Lambda}_{\gamma}(t_k, \xi_{k,j})| |(\Delta \sigma_k)_j|^3 \leq \frac{C_3}{\gamma^2} \beta^{3/2} \eta^{3/2} n.
\]
\end{lemma}

Because $|\partial_{y,y,y} \tilde{\Lambda}_{\gamma}| \leq 2 / \gamma^2$, this lemma follows immediately once we establish the following bounds for the $3$-norm of the updates.

\begin{lemma}[Bounded $3$-norm of the updates] \label{lem:bounded-three-norm-update} \mbox{} \\
Suppose that the matrix $A$ satisfies the conclusions of \pref{thm:david-magic} \pref{item:4-approx-diag}.  Suppose that the parameters $\beta$, $K$, $\eta$, $n$ satisfy $n^{1/90} \eta \geq 1$ and \eqref{eq: induction hypothesis for all thm} and that the vectors $Z_0$, \dots, $Z_{k-1}$ are chosen in the high probability event in \pref{thm:convergence-to-SDE}.  Then, conditioned on $A, Z_0, \dots, Z_{k-1}$, with probability at least $1 - C_1 \exp(-C_2 n^{2/3-4\delta})$ in the Gaussian vector $Z_k$, we have
\[
\sum_{i=1}^n |(\Delta \sigma_k)_i|^3 \leq C_3 \beta^{3/2} \eta^{3/2} n.
\]
\end{lemma}

\begin{proof}
Recall that $|x|_r \le |x|_p$ for every $1 \le p\le r$. Therefore,
\[
    |x|_3 = \left( \sum_{j=1}^n |x_j|^3 \right)^{1/3} \leq \left( \sum_{j=1}^n |x_j|^2 \right)^{1/2} = |x|_2\,.
\]
Consequently, $f(\Delta \sigma_k) = \left( \sum_{j=1}^n |(\Delta \sigma_k)_j|^3 \right)^{1/3}$ is a $1$-Lipschitz function of $\Delta \sigma_k$.  Moreover, conditioned on $A$ and $Z_0$, \dots, $Z_{k-1}$, the variable $\Delta \sigma_k$ is Gaussian with covariance $\eta Q_k^2$.
By the concentration inequality for Lipschitz functions of Gaussian random variables~\cite[Theorem 5.1.3]{vershynin2018high},
\begin{equation}
\label{eq:3-norm-remainder}
\Pr\left(\left|f(\Delta \sigma_k) - \E\left[f(\Delta \sigma_k)\right]\right| \geq \eta^{1/2} n^{1/3}\right) \leq \exp\left(-\frac{n^{2/3}}{2 \norm{Q_k^2}}\right) \leq \exp(-M_1 n^{2/3-4\delta}).
\end{equation}
As $y = z^{1/3}$ is concave for $z \in \R_{\ge 0}$, Jensen's inequality implies
\begin{align*}
\E f(\Delta \sigma_k) 
&= \E \left( \sum_{j=1}^n |(\Delta \sigma_k)_j|^3 \right)^{1/3} \\
&\leq \left( \E \sum_{j=1}^n |(\Delta \sigma_k)_j|^3 \right)^{1/3}\,.
\end{align*}
Since $(\Delta \sigma_k)_i \sim \calN\left(0, \eta(Q_k^2)_{i,i}\right)$,
\[
\E |\Delta \sigma_i|^3 = M_2 (Q_k^2)_{i,i}^{3/2}\eta^{3/2}
\]
for some constant $M_2$.  Thus,
\begin{align*}
\E\left[f(\Delta \sigma_k)\right] &\leq M^{1/3}_2 \Tr_n(E_{\mathcal{D}_n}[\eta Q_k^2]^{3/2})^{1/3}\\
&= M_2^{1/3} n^{1/3} \eta^{1/2} \tr_n(E_{\mathcal{D}_n}[Q_k^2]^{3/2})^{1/3} \\
&\leq M_2^{1/3} n^{1/3} \eta^{1/2} \tr_n(E_{\mathcal{D}_n}[Q_k^2]^2)^{1/4} \\
&\leq M_3 n^{1/3} \eta^{1/2} \beta^{1/2},
\end{align*}
where the last inequality follows from \pref{prop: diagonal matrix error and normalized magic} \pref{item:hs-diagonal} and the assumed application of \pref{thm:convergence-to-SDE} to $Z_0, \dots, Z_{k-1}$.  Therefore, by \pref{eq:3-norm-remainder},
\[
\Pr\left(f(\Delta \sigma_k) \geq M_3 n^{1/3} \beta^{1/2}\eta^{1/2}\right) \leq \exp(-M_1 n^{2/3 - 4 \delta}).
\]
It follows that
\[
\sum_{i=1}^n |\Delta \sigma_i|^3 \leq M_4 n \beta^{3/2}\eta^{3/2}\,.
\]
with probability at least
\[
1 - \exp(-M_1 n^{2/3 - 4 \delta}).
\]
in the Gaussian vector $Z_k$, conditioning on the vectors $Z_0$, \dots, $Z_{k-1}$ and assuming that they fall under the high probability event in~\pref{thm:convergence-to-SDE}. 
\end{proof}

\subsection{The time terms}

We will now estimate \eqref{eq: Taylor time term} by controlling how close the regularization-corrected self-overlap $\frac{1}{n} \sum_{j=1}^n (\sigma_{k+1,j} - \gamma\partial_y \tilde{\Lambda}_{\gamma}(t, \sigma_{k+1,j}))^2$ is to $t$ through the convergence to the primal Auffinger-Chen SDE \pref{thm:convergence-to-SDE}.

\begin{lemma}[Bounding the $2$-norm under the $\gamma$-convolved SDE]\label{lem:time-sde-norm} \mbox{} \\
Assume that the hypotheses and conclusions of~\pref{thm:convergence-to-SDE} are satisfied and that $\beta \eta^{1/2} < 1$.  Then for some constants $C_j$, we have
\[
\left| \frac{1}{n} \sum_{j=1}^n (\sigma_{k+1,j} - \gamma\partial_y \tilde{\Lambda}_{\gamma}(t, \sigma_{k+1,j}))^2 - t \right| \leq C_1 d_{W,2}(\emp(\sigma_{k+1}), \dist(Y_{t_{k+1}}^\gamma)) + C_2 e^{5 \beta^2} \gamma + C_3 \beta^3 \eta^{1/2}\,,
\]
where $t \in [t_k,t_{k+1}]$.
\end{lemma}

\begin{proof}
Let $\Sigma$ be a random variable with distribution $\emp(\sigma_{k+1})$ and let $Y_t^\gamma$ be the stochastic process defined earlier.  Consider an optimal coupling of $\Sigma$ and $Y_{t_{k+1}}^\gamma$ on some probability space.  Recall that $\id - \gamma \partial_y \tilde{\Lambda}_{\gamma}(t,\cdot)$ is $1$-Lipschitz.  Therefore,
\begin{align*}
    \norm{[\Sigma - \gamma \partial_y \tilde{\Lambda}_{\gamma}(t,\Sigma)] - [Y_{t_{k+1}}^\gamma - \gamma \partial_y \tilde{\Lambda}_{\gamma}(t, Y_{t_{k+1}}^\gamma)]}_{L^2}
    &\leq \norm{\Sigma - Y_{t_{k+1}}^\gamma}_{L^2} \\
    &= d_{W,2}(\operatorname{emp}(\sigma_{k+1}), \dist(Y_{t_{k+1}}^\gamma)).
\end{align*}
Meanwhile, by~\pref{lem:sde-closeness},
\[
    \norm{[Y_{t_{k+1}}^\gamma - \gamma \partial_y \tilde{\Lambda}_{\gamma}(t, Y_{t_{k+1}}^\gamma)] - Y_{t_{k+1}}}_{L^2} \leq 5^{-1/2} e^{5 \beta^2} \gamma.
\]
Furthermore, from the proof of~\pref{lem: SDE tangent approximation} in the case that $\gamma = 0$,
\[
    \norm{Y_{t_{k+1}} - Y_t}_{L^2} \leq 2\beta\sqrt{\eta}\,.
\]
Finally, for any $t \in [0,q^*_\beta)$,~\cite[Lemma 3.4]{montanari2021optimization} implies that $\norm{Y_t}_{L^2} = t^{1/2}$.  Thus, combining these estimates with the triangle inequality yields,
\begin{align*}
\left| \left( \frac{1}{n} \sum_{j=1}^n (\sigma_{k+1,j} - \gamma\partial_y \tilde{\Lambda}_{\gamma}(t, \sigma_{k+1,j}))^2 \right)^{1/2} - t^{1/2} \right| &\leq \norm{\Sigma - Y_t}_{L^2} \\
&\leq d_{W,2}(\operatorname{emp}(\sigma_{k+1}), \dist(Y_{t_{k+1}}^\gamma)) + 5^{-1/2} e^{5 \beta^2} \gamma + 2\beta \eta^{1/2}.
\end{align*}
Now writing temporarily $s = \frac{1}{n} \sum_{j=1}^n (\sigma_{k+1,j} - \gamma\partial_y \tilde{\Lambda}_{\gamma}(t, \sigma_{k+1,j}))^2$, we have
\begin{align*}
|s - t| &= |s^{1/2} - t^{1/2}|(s^{1/2} + t^{1/2}) \\
&\leq |s^{1/2} - t^{1/2}|(t^{1/2} + |s^{1/2} - t^{1/2}| + t^{1/2}) \\
&\leq |s^{1/2} - t^{1/2}|(2 + |s^{1/2} - t^{1/2}|)
\end{align*}
since $t \leq 1$.  Furthermore, if the assumptions of the theorem are satisfied, then $d_{W,2}(\operatorname{emp}(\sigma_{k+1}), \dist(Y_{t_{k+1}}^\gamma)) + 5^{-1/2} e^{5 \beta^2} \gamma + 2 \beta \eta^{1/2}$ is bounded by a constant $M$, so we can simply estimate this by $|s - t| \leq M |s^{1/2} - t^{1/2}|$.
\end{proof}

We thus obtain the following estimate for the final distortion term.

\begin{lemma}[Distortion of drift term in time error]\label{lem:time-distortion}\mbox{} \\
Assume the hypotheses and conclusions of~\pref{thm:convergence-to-SDE}.  Then
\begin{align*}
&\left| \beta^2\! \int_{t_k}^{t_{k+1}}\! F_\mu(t) \!\left( \sum_{j=1}^n ((\id - \gamma \partial_y \tilde{\Lambda}_{\gamma}(t,\cdot))(\sigma_{k+1,j}))^2 - nt \!\right)\!dt \right| 
\\&\qquad\qquad\qquad{}\leq n \beta^2 \eta \left[ C_1 d_{W,2}(\operatorname{emp}(\sigma_{k+1}), \dist(Y_{t_{k+1}}^\gamma)) + C_2 e^{5 \beta^2} \gamma + 2\beta \eta^{1/2} \right].
\end{align*}
\end{lemma}

\begin{proof}
Note that $F_\mu(t) \leq 1$ and apply the previous lemma to estimate $|\frac{1}{n} \sum_{j=1}^n ((\id - \gamma \partial_y \tilde{\Lambda}_{\gamma}(t,\cdot))(\sigma_{k+1,j}))^2 - t|$.
\end{proof}

We finish this section with an estimate to help with \eqref{eq: Taylor telescoping error}.  Note that because we use telescoping, we include the summation over $k$ in the statement here.

\begin{lemma}[Bound for telescoping term] \label{lem: telescoping term estimate}
\[
\left| \sum_{k=0}^{K-1} \left( \beta^2 \sum_{j=1}^n \eta v(t_k,(\sigma_k)_j) - \beta^2 \sum_{j=1}^n \int_{t_k}^{t_{k+1}} v(t,(\sigma_{k+1})_j)\,dt \right) \right| \leq 9 \beta^4 \eta n.
\]
\end{lemma}

\begin{proof}
We write
\[
\eta v(t_k,\sigma_k) - \int_{t_k}^{t_{k+1}} v(t,\sigma_{k+1})\,dt = \left( \eta v(t_k,\sigma_k) - \eta v(t_{k+1},\sigma_{k+1}) \right) + \left( \eta v(t_{k+1},\sigma_{k+1}) - \int_{t_k}^{t_{k+1}} v(t,\sigma_{k+1})\,dt \right)
\]
The first term telescopes when we sum over $k$, that is,
\[
\left| \sum_{k=0}^{K-1} (v(t_k,\sigma_k) - v(t_{k+1},\sigma_{k+1})) \right| = |v(0,0) - v(t_K, \sigma_K)| \leq 1
\]
since $\gamma \leq v \leq 1 + \gamma$ by 
\pref{prop: Lambda estimates} \pref{item:second-deriv}.  After multiplying by $\beta^2 \eta$ and summing over $j \in [n]$, the final contribution is $\beta^2 \eta n \leq \beta^4 \eta n$.

To estimate the second term, recall that by \pref{prop: Lambda estimates} (4),
\[
|v(t_{k+1},\sigma_{k+1}) - v(t,\sigma_{k+1})| \leq 18 \beta^2 (t_{k+1} - t).
\]
Integrating from $t_k$ to $t_{k+1}$ yields
\[
\left| \eta v(t_{k+1},\sigma_{k+1}) - \int_{t_k}^{t_{k+1}} v(t,\sigma_{k+1})\,dt \right| \leq 9 \beta^2 \eta^2.
\]
After summing over $j$ and multiplying by $\beta^2$, we obtain
\[
\left| \beta^2 \sum_{j=1}^n \eta v(t_{k+1},(\sigma_{k+1})_j) - \beta^2 \sum_{j=1}^n \int_{t_k}^{t_{k+1}} v(t,(\sigma_{k+1})_j)\,dt \right| \leq 9 \beta^4 \eta^2 n.
\]
Finally, summing over $k$ gives a final contribution of $9 \beta^4 (K \eta) \eta n \leq 9 \beta^4 \eta n$.
\end{proof}

\subsection{Conclusion of the energy argument}

Now we can put together all the estimates we have made in the individual terms in the Taylor expansion and sum over $k$.

\begin{theorem}[Estimates for change in the objective function] \label{thm:taylor-bound}
Suppose \pref{ass:sk-frsb}.  Suppose that the matrix $A$ satisfies the conclusions of \pref{thm:david-magic} \pref{item:4-approx-diag}.  Suppose that the parameters $\beta$, $K$, $\eta$, $n$ satisfy $n^{1/90} \eta \geq 1$ and \eqref{eq: induction hypothesis for all thm} and that the vectors $Z_0$, \dots, $Z_{K-1}$ are chosen in the high probability event in \pref{thm:convergence-to-SDE}.  Then with probability at least $1 - C_1n\eta^{-2} \exp(-C_2 n^{2/9 - 4\delta})$ in the Gaussian vectors $Z_0$, \dots, $Z_{K-1}$, the conclusions of \pref{thm:convergence-to-SDE} hold and in addition
\[
\operatorname{obj}(t_K, \sigma_K) - \operatorname{obj}(t_0,\sigma_0)
\geq -n \cdot e^{C_4 \beta^2} \left( C_5 \gamma^{-1} \eta^{-1/2} n^{-1/36} + C_6 \gamma^{-2} \eta^{1/2} + C_7 n^{-\delta/2} + C_8 \gamma \right).
\]
\end{theorem}

\begin{proof}
We take the Taylor expansion in \eqref{eq: Taylor gradient term} - \eqref{eq: Taylor gamma error} and apply the bounds for each individual terms that we have proved, namely:
\begin{itemize}
    \item \eqref{eq: Taylor gradient term} is estimated by \pref{lem:gradient-gaussian-chaos},
    \item \eqref{eq: Taylor Hessian term} is estimated by \pref{lem:bounded-hessian-fluctuations},
    \item \eqref{eq: Taylor third derivative term} is estimated by \pref{lem: Taylor remainder term},
    \item \eqref{eq: Taylor time term} is estimated by \pref{lem:time-distortion}.
\end{itemize}
Hence, assuming we are in the high probability events described in Theorem \ref{thm:convergence-to-SDE} and in each of those lemmas, we have
\begin{align}
\operatorname{obj}(t_{k+1}, \sigma_{k+1}) &- \operatorname{obj}(t_k,\sigma_k) \\
\geq &-\,n\eta \cdot M_1 \left( 3 \beta + \gamma^{-1} \right) \eta^{-1/2} n^{-1/3} \label{eq: Taylor gradient term done} \\
&- n\eta \cdot \biggl( M_2(3 \beta + \gamma^{-1}) \beta^2 n^{-1/3} + M_3 \beta^3 n^{-\delta} + M_4 \beta^4 d_{W,2}(\emp(\sigma_k),\dist(Y_{t_k}^\gamma)) + M_5 e^{5 \beta^2} \beta^4 \gamma \biggr)  \label{eq: Taylor Hessian term done} \\
&- n\eta \cdot M_2\beta^2(\beta + \gamma^{-1})n^{-1/4+2\delta} \nonumber\\
&- n\eta \cdot M_6 \beta^{3/2} \gamma^{-2}\eta^{1/2} \label{eq: Taylor third derivative term done} \\
&- n \eta \cdot \left[ M_7 \beta^2 d_{W,2}(\emp(\sigma_{k+1}), \dist(Y_{t_{k+1}}^\gamma)) + M_8 \beta^2 e^{5 \beta^2} \gamma +  M_9 \beta^3 \eta^{1/2} \right]  \label{eq: Taylor time term done} \\
& + \beta^2 \sum_{j=1}^n \eta v(t_k,\sigma_{k,j}) - \beta^2 \sum_{j=1}^n \int_{t_k}^{t_{k+1}} v(t,\sigma_{k+1,j})\,dt \label{eq: Taylor telescoping error done} \\
&- n\eta \cdot \beta^2 \gamma \label{eq: Taylor gamma error done}\,,
\end{align}
where we have left \eqref{eq: Taylor telescoping error done} alone in order to apply telescoping later.   We combine the remaining errors \eqref{eq: Taylor gradient term done} - \eqref{eq: Taylor time term done} and \eqref{eq: Taylor gamma error} and throw away terms that are dominated by other terms in the expansion, and we thus obtain:
\[
\begin{aligned}
n \eta \cdot \biggl( M_{10} \left( 3 \beta + \gamma^{-1} \right) \eta^{-1/2} (n^{-1/3}+n^{-1/4+2\delta})& + M_{11} \beta^3 \gamma^{-2} \eta^{1/2} + M_{12} \beta^3 n^{-\delta}
\\ &+ M_{13} \beta^2 d_{W,2}(\emp(\sigma_k),\dist(Y_{t_k}^\gamma)) + M_{13} \beta^4 e^{5 \beta^2} \gamma \biggr).
\end{aligned}
\]
Then we recall from \pref{thm:convergence-to-SDE} that
\begin{align*}
d_{W,2}(\emp(\sigma_k),\dist(Y_{t_k}^\gamma)) &\leq \exp( M_{14} \beta^2) \left[ M_{15} \beta^4 \eta + M_{16} n^{-\delta} + M_{17} e^{10 \beta^2} \gamma^2 + M_{18} \beta^{-2} \eta^{-1} n^{-1/18} \right]^{1/2} \\
&\leq \exp( M_{14} \beta^2) \left( M_{19} \beta^2 \eta^{1/2} + M_{20} n^{-\delta/2} + M_{21} e^{5 \beta^2} \gamma + M_{21} \beta^{-1} \eta^{-1/2} n^{-1/36} \right).
\end{align*}
Hence, the errors from \eqref{eq: Taylor gradient term done} - \eqref{eq: Taylor time term done} and \eqref{eq: Taylor gamma error} can be bounded by
\begin{equation} \label{eq: most Taylor terms}
n \eta \cdot e^{M_{22} \beta^2} \left( M_{23} \gamma^{-1} \eta^{-1/2} n^{-1/36} + M_{24} \gamma^{-2} \eta^{1/2} + M_{25} n^{-\delta/2} + M_{26} \gamma \right).
\end{equation}
Finally, we sum the error for $k = 0$, \dots, $K-1$.  Hence, the $n \eta$ in front of \eqref{eq: most Taylor terms} becomes $n K \eta \leq n$.  Meanwhile, the summation over $k$ of \eqref{eq: Taylor telescoping error done} is estimated by \pref{lem: telescoping term estimate} as smaller than $9 \beta^4 \eta n$.  Since we already have terms of the form $n \exp(M \beta^2) \gamma^{-2} \eta^{1/2}$, this additional error can be absorbed by changing the constants.  Therefore,
\[
\operatorname{obj}(t_K, \sigma_K) - \operatorname{obj}(t_0,\sigma_0)
\geq -n \cdot e^{M_{22} \beta^2} \left( M_{23} \gamma^{-1} \eta^{-1/2} n^{-1/36} + M_{27} \gamma^{-2} \eta^{1/2} + M_{25} n^{-\delta/2} + M_{26} \gamma \right).
\]

It remains to estimate the probability of failure.  We take a union bound over the errors in each step that arise both from the proof of  \pref{thm:convergence-to-SDE} and from the arguments in this section.  In \pref{lem:gradient-gaussian-chaos}, \pref{lem:bounded-hessian-fluctuations}, \pref{lem: Taylor remainder term}, and \pref{lem:time-distortion}, the probabilities are all bounded by $M_{27} \exp(-M_{28} n^{1/3 - 4 \delta})$, assuming that $Z_0$, \dots, $Z_{k-1}$ are chosen already from the high probability events in \pref{thm:convergence-to-SDE}.
We take a union bound over the steps, where the number of terms is bounded by $\eta^{-1}$, for a total probability of $M_{27}\eta^{-1}\exp(-M_{28} n^{1/3 - 4 \delta})$.
Meanwhile, the probabilities of error from \pref{thm:convergence-to-SDE} are bounded by $M_{29} n\eta^{-2}\exp(-M_{30} n^{2/9-4\delta})$, and so can combine these error estimates together at the cost of changing the constants.
\end{proof}

Combining~\pref{thm:taylor-bound}, which shows that the per step deviation of the objective function is sufficiently small, with high probability, with the uniform estimates for the moduli of continuity of $\tilde{\Lambda}_{\gamma}$ provided in \S\ref{sec:primal-parisi-pde} allows one to conclude that, under fRSB, the desired energy increments come arbitrarily close to the value given by the Parisi formula. The last step is to use the fact that the norm of the iterates is sufficiently large, and that they are close to the cube to guarantee that the rounding scheme outlined in~\pref{prop:rounding-energy} doesn't lose too much energy.

\begin{theorem}[Final energy estimate for truncated $\sigma_K$]\label{thm:final-energy}
    Fix $\eps > 0$. Assume \pref{ass:sk-frsb}, assume $n^{1/90} \eta \geq 1$ and \eqref{eq: induction hypothesis for all thm}, and assume we are in the high probability events of \pref{thm:david-magic}, \pref{thm:convergence-to-SDE}, \pref{thm:taylor-bound}.  Let $\sigma_K$ be the vector at the last step of the algorithm and let $\tilde{\sigma}$ be the truncation $\tilde{\sigma}_j = \sgn(\sigma_{K,j}) \min(1, |\sigma_{K,j}|)$. Finally, write
    \[
        \mathcal{E}(\beta) = 2 \beta \int_0^{q_\beta^*} \int_t^1 F_\mu(s)\,ds\,.
    \]
    Then
    \begin{align*}
    \frac{1}{n} H(\tilde{\sigma}) &\geq \mathcal{E}(\beta)
    -\frac{1}{\beta} \exp( C_2 \beta^2) \left[ C_5 \eta^{-1/2} n^{-1/36} + C_6 \eta^{1/2} + C_7 n^{-\delta/2} \right]^{8\beta^2 / (1 + 8 \beta^2)} \\
    & \quad\quad\quad - \frac{1}{\beta} \exp(C_8 \beta^2) [C_9 \gamma^{-1} \eta^{-1/2} n^{-1/36} + C_{10} \gamma^{-2} \eta^{1/2} + C_{11} \gamma^{4 \beta^2/(1+4\beta^2)}].
    \end{align*}
\end{theorem}

\begin{proof}
~

\subsubsection*{Unwinding the objective function.}
By the definition of $\obj$,
\[
\beta \angles{\sigma_K, A_{\sym} \sigma_K} = \obj(q_\beta^*,\sigma_K) - \obj(0,0) + \sum_{j=1}^n \tilde{\Lambda}_{\gamma}(q_\beta^*,\sigma_{K,j}) - n \Lambda(0,0) - n \beta^2 \int_0^{q_\beta^*} t F_\mu(t)\,dt.
\]
Thus, substituting $\calE(\beta)$ in \pref{cor: SDE energy estimate}, we get
\[
\beta \angles{\sigma_K, A_{\sym} \sigma_K} = \beta \mathcal{E}(\beta)n + \obj(q_\beta^*,\sigma_K) - \obj(0,0) + \sum_{j=1}^n \tilde{\Lambda}_{\gamma}(q_\beta^*,\sigma_{K,j}) - n \E \Lambda(q_\beta^*,Y_{q_\beta^*}).
\]
Recall that $t_0 = 0$, $\sigma_0 = 0$, and $t_K = q_\beta^*$, and hence \pref{thm:taylor-bound} gives us a lower bound for $\obj(q_\beta^*,\sigma_K) - \obj(0,0)$.  It thus remains to estimate the approximation error $\sum_{j=1}^n \tilde{\Lambda}_{\gamma}(q_\beta^*,\sigma_{K,j}) - n \E \Lambda(q_\beta^*,Y_{q_\beta^*})$ for the entropy function, and then to perform a truncation to obtain a vector $\tilde{\sigma} \in \{-1,1\}^n$ from $\sigma_K$.

\subsubsection*{Approximation error for entropy.}
We now use the SDE and regularity analysis from \S\ref{sec:primal-parisi-pde} to estimate $\sum_{j=1}^n \tilde{\Lambda}_{\gamma}(q_\beta^*, (\sigma_K)_j)$.  Let $\Sigma$ be a random variable with distribution $\emp(\sigma_K)$, so that
\[
\frac{1}{n} \sum_{j=1}^n \tilde{\Lambda}_{\gamma}(q_\beta^*, (\sigma_K)_j) = \E \tilde{\Lambda}_{\gamma}(q_\beta^*, \Sigma).
\]
Assume that $\Sigma$ is on the same probability space as the processes $Y_t$ and $Y_t^\gamma$ such that $\Sigma$ and $Y_{q_\beta^*}$ are optimally coupled, that is, $\norm{\Sigma - Y_{q_\beta^*}}_{L^2} = d_{W,2}(\emp(\sigma_K),\dist(Y_{q_\beta^*}))$.  Let $f: \R \to [-1,1]$ be the truncation $f(y) = \sgn(y) \min(1, |y|)$.  Since $\tilde{\Lambda}_{\gamma}(t,y)$ is an increasing function of $|y|$, we have
\[
\tilde{\Lambda}_{\gamma}(q_\beta^*, \Sigma) \geq \tilde{\Lambda}_{\gamma}(q_\beta^*, f(\Sigma)) \geq \Lambda(q_\beta^*, f(\Sigma)) - (1 + 4 \beta^2) (2 \beta^2 \gamma)^{4 \beta^2/(1 + 4 \beta^2)},
\]
where the second estimate follows from~\pref{prop: Lambda gamma versus Lambda}.  Furthermore, by \pref{lem: Lambda modulus of continuity},
\begin{align*}
\E |\Lambda(q_\beta^*, f(\Sigma)) - \Lambda(q_\beta^*, Y_{q_\beta^*})| &\leq (1 + 8 \beta^2) \E \left| \frac{f(\Sigma) - Y_{q_\beta^*}}{2} \right|^{8\beta^2 / (1 + 8 \beta^2)} \\
&\leq (1 + 8 \beta^2) \left( \frac{\norm{f(\Sigma) - Y_{q_\beta^*}}_{L^1}}{2} \right)^{8\beta^2 / (1 + 8 \beta^2)},
\end{align*}
where the last estimate follows from H\"older's inequality.  Finally, note that since $Y_{q_\beta^*} \in [-1,1]$,
    \begin{align*}
    \norm{f(\Sigma) - Y_{q_\beta^*}}_{L^1} &\leq \norm{\Sigma - Y_{q_\beta^*}}_{L^1} \leq \norm{\Sigma - Y_{q_\beta^*}}_{L^2} \\
    &= d_{W,2}(\emp(\sigma_K),\dist(Y_{q_\beta^*})) \\
    &\leq d_{W,2}(\emp(\sigma_K), \dist(Y_{q_\beta^*}^\gamma)) + \norm{Y_{q_\beta^*}^\gamma - Y_{q_\beta^*}}_{L^2} \\
    &\leq_{\text{\pref{lem:sde-closeness}}} d_{W,2}(\emp(\sigma_K), \dist(Y_{q_\beta^*}^\gamma)) + \sqrt{2} e^{5 \beta^2} \gamma.
    \end{align*}
    Thus,
    \begin{multline} \label{eq: entropy approximation error}
    \sum_{j=1}^n \tilde{\Lambda}_{\gamma}(q_\beta^*, (\sigma_K)_j) - n \E \Lambda(q_\beta^*, Y_{q^*_\beta}) \\ \geq -n \left( (1 + 8\beta^2) \left( d_{W,2}(\emp(\sigma_K), \dist(Y_{q_\beta^*}^\gamma)) + \sqrt{2} e^{5 \beta^2} \gamma \right)^{8\beta^2 / (1 + 8 \beta^2)} + (1 + 4 \beta^2) (2 \beta^2 \gamma)^{4 \beta^2/(1 + 4 \beta^2)} \right).
    \end{multline}

    \subsubsection*{Truncation error.}  Recall we set $\tilde{\sigma}_j = f(\sigma_{K,j})$ where $f(y) = \sgn(y) \min(|y|,1)$.  We will now estimate $|\sigma_K - \tilde{\sigma}|_2$.  Note that
    \[
    \frac{1}{n} |\sigma_K - \tilde{\sigma}|_2^2 = \E |\Sigma - f(\Sigma)|^2,
    \]
    where $\Sigma$ is a random variable with distribution $\emp(\sigma_K)$ as above.  Since $Y_{q_\beta^*} \in [-1,1]$, we have $|\Sigma - f(\Sigma)| \leq |\Sigma - Y_{q_\beta^*}|$, so that, by \pref{lem:sde-closeness}, 
    \begin{align*}
    n^{-1/2} |\sigma_K - \tilde{\sigma}|_2 &= \norm{\Sigma - f(\Sigma)}_{L^2} \\  
    &\leq \norm{\Sigma - Y_{q_\beta^*}}_{L^2} \\
    &\leq d_{W,2}(\emp(\sigma_K), \dist(Y_{q_\beta^*})) \\
    &\leq d_{W,2}(\emp(\sigma_K), \dist(Y_{q_\beta^*}^\gamma)) + \sqrt{2} e^{5 \beta^2} \gamma.
    \end{align*}
    Since $\norm{A_{\sym}} \leq 3$ in our given sample, we conclude that
    \begin{align*}
    \angles{\tilde{\sigma}, A_{\sym} \tilde{\sigma}} &\geq \angles{\sigma_K, A_{\sym} \sigma_K} + 2 \angles{\tilde{\sigma}, A_{\sym} (\tilde{\sigma} - \sigma_K)} - \angles{(\tilde{\sigma} - \sigma_K), A_{\sym}(\tilde{\sigma} - \sigma_K)} \\
    &\geq \angles{\sigma_K, A_{\sym} \sigma_K} - 6 |\tilde{\sigma}|_2 |\tilde{\sigma} - \sigma_K|_2 - 3 |\tilde{\sigma} - \sigma_K|_2^2 \\
    &\geq \angles{\sigma_K, A_{\sym} \sigma_K} - 6n \norm{\Sigma - f(\Sigma)}_{L^2} - 3n \norm{\Sigma - f(\Sigma)}_{L^2}^2.
    \end{align*}
    In the high probability event of \pref{thm:convergence-to-SDE}, $d_{W,2}(\emp(\sigma_K), \dist(Y_{q_\beta^*}^\gamma)) + \sqrt{2} e^{5 \beta^2} \gamma$ will be bounded by a constant, and hence $\norm{\Sigma - f(\Sigma)}_{L^2}^2 \leq M_1 \norm{\Sigma - f(\Sigma)}_{L^2}$.  Thus,
    \begin{equation} \label{eq: rounding approximation error}
    \angles{\tilde{\sigma}, A_{\sym} \tilde{\sigma}} \geq \angles{\sigma_K, A_{\sym} \sigma_K} - nM_2 \left(d_{W,2}(\emp(\sigma_K), \dist(Y_{q_\beta^*}^\gamma)) + \sqrt{2} e^{5 \beta^2} \gamma \right).
    \end{equation}
    
    \subsubsection*{Putting the estimates together.}
    We write
    \begin{align*}
    \beta \angles{\tilde{\sigma}, A_{\sym} \tilde{\sigma}} - \beta \mathcal{E}(\beta)n &= \beta \angles{\tilde{\sigma}, A_{\sym} \tilde{\sigma}} - \beta \angles{\sigma_K, A_{\sym} \sigma_K} \\
    &\quad + \sum_{j=1}^n \tilde{\Lambda}_{\gamma}(q_\beta^*, \sigma_{K,j}) - n \E[\Lambda(q_\beta^*,Y_{q_\beta^*})] \\
    &\quad +\obj(q_\beta^*,\sigma_K) - \obj(0,0).
    \end{align*}
    On the right-hand side, the first line is estimated by \eqref{eq: rounding approximation error}, the second line is estimated by \eqref{eq: entropy approximation error}, and the third line is estimated by \pref{thm:taylor-bound}, which yields
    \begin{align*}
    \beta \angles{\tilde{\sigma}, A_{\sym} \tilde{\sigma}} - \beta \mathcal{E}(\beta)n &\geq -n M_2 \beta \left( d_{W,2}(\emp(\sigma_K), \dist(Y_{q_\beta^*}^\gamma)) + \sqrt{2} e^{5 \beta^2} \gamma \right) \\
    &\quad -n M_3 \beta^2 \left( \left( d_{W,2}(\emp(\sigma_K), \dist(Y_{q_\beta^*}^\gamma)) + \sqrt{2} e^{5 \beta^2} \gamma \right)^{8\beta^2 / (1 + 8 \beta^2)} + (2 \beta^2 \gamma)^{4 \beta^2/(1 + 4 \beta^2)} \right) \\
    &\quad -n \cdot e^{C_4 \beta^2} \left( M_4 \gamma^{-1} \eta^{-1/2} n^{-1/36} + M_5 \gamma^{-2} \eta^{1/2} + M_6 n^{-\delta/2} + M_7 \gamma \right),
    \end{align*}
    where we have estimated $(1 + \beta^2)^{8\beta^2 / (1 + 8 \beta^2)}$ by a constant times $\beta^2$.  To combine these estimates conveniently, note that the $d_{W,2}(\emp(\sigma_K),\dist(Y_{q_\beta^*}^{\gamma}))$ and $\gamma$ terms on the first line can be absorbed into the estimate on the third line from \pref{thm:taylor-bound}, in the same way as we did in the proof of that theorem.  For the second line, we write
    \[
    \left( d_{W,2}(\emp(\sigma_K), \dist(Y_{q_\beta^*}^\gamma)) + \sqrt{2} e^{5 \beta^2} \gamma \right)^{8\beta^2 / (1 + 8 \beta^2)} \leq M_8 d_{W,2}(\emp(\sigma_K), \dist(Y_{q_\beta^*}^\gamma))^{8\beta^2 / (1 + 8 \beta^2)} + M_9 (e^{5 \beta^2} \gamma)^{8\beta^2 / (1 + 8 \beta^2)}.
    \]
    Since $\gamma \leq 1$, and since $4 \beta^2 / (1 + 4 \beta^2) \leq 8\beta^2/(1 + 8 \beta^2)$, we can thus write
    \[
    M_9 (e^{5 \beta^2} \gamma)^{8\beta^2 / (1 + 8 \beta^2)} + (2 \beta^2 \gamma)^{4 \beta^2/(1 + 4 \beta^2)} \leq e^{M_{10} \beta^2} \gamma^{4 \beta^2/(1 + 4 \beta^2)}.
    \]
    We finally combine this with the $\gamma$ terms on the third line, and altogether the error is bounded by
    \begin{multline*}
    n \biggl( M_{11} d_{W,2}(\emp(\sigma_K),\dist(Y_{q_\beta^*}^\gamma))^{8\beta^2/(1 + 8 \beta^2)} \\ + e^{C_4 \beta^2} \left( M_4 \gamma^{-1} \eta^{-1/2} n^{-1/36} + M_5 \gamma^{-2} \eta^{1/2} + M_6 n^{-\delta/2} + M_7 \gamma^{4 \beta^2/(1+4\beta^2)} \right) \biggr).
    \end{multline*}
    Then we estimate $d_{W,2}(\emp(\sigma_K),\dist(Y_{q_\beta^*}^\gamma))$ by \pref{thm:convergence-to-SDE} and combine terms by similar reasoning as did before.  This yields the estimate asserted in the theorem after dividing by $\beta$.
\end{proof}

It remains to choose appropriate values of $\beta$, $\gamma$, $\eta$, and $n$ to the make the final error smaller than $\eps$.

\begin{corollary}[Quantitative lower bound on energy approximation]\label{cor:energy-lower-bound} \mbox{} \\
    Suppose \pref{ass:sk-frsb}.  Let $\sigma_K$ be the final iterate of~\pref{alg:hessian-ascent}, with $\tilde{\sigma} = \mathsf{trunc}\left(\sigma_K\right)$ as defined in~\pref{thm:final-energy}. Furthermore, assume we are in the high-probability events of \pref{thm:david-magic}, \pref{thm:convergence-to-SDE}, \pref{thm:taylor-bound}, and \pref{prop:rounding-energy}. Then, with $\calP_{\beta}$ as in \eqref{eq:parisi-formula-optimization},
    \[
        \frac{1}{n}H(\sigma^*) \ge \frac{\calP_\beta}{\beta} - \frac{\eps}{5} - O(\eps^2) - O_{\eps}(n^{-\alpha})\,, 
    \]
    provided that $\alpha = \min\left(\delta/4, 1/24\right)$, $\beta = \frac{10}{\eps}$, $\eta = e^{-C\beta^2}$, $\gamma = \eta^{1/8} = e^{-C\beta^2/8}$, and $n\ge \eta^{-90}$ for some sufficiently large absolute constant $C > 0$, where $\sigma^* = \mathsf{round}(\tilde{\sigma})$ with the rounding procedure described in \pref{prop:rounding-energy}.

Consequently,
\[
H(\sigma^*)\ge
\left(1-\frac{\eps}{2}-O(\eps^2)-o_n(1)\right)
\max_{\sigma\in\{\pm1\}^n}H(\sigma)
\]
with probability $1-\exp(-n^{\Omega(1)})$ as $n\to\infty$.
\end{corollary}

\begin{proof}
    We now invoke~\pref{prop:rounding-energy} in conjunction with a precise set of parameterizations for $\beta$, $\eta$ and $\gamma$ as functions of $\eps \in (0,1/2)$ that give the desired approximation to the energy, with high probability under the choice of input and the algorithm's internal randomness.   Direct computation shows that for large enough $C$, our assumptions on $\beta$, $\eta$, $\gamma$, and $n$ imply the hypothesis \eqref{eq: induction hypothesis for all thm} in Theorem \ref{thm:convergence-to-SDE}, so that we can apply Theorem \ref{thm:taylor-bound}.
    
    By~\pref{thm:final-energy} with our choices of $\beta$, $\eta$, and $\gamma$,
    \[
        \frac{1}{n}H(\tilde{\sigma}) \ge \calE(\beta) - O(\eps^{2}) - O_{\eps}(n^{-\alpha}).
    \]
    Set $\sigma^* = \mathsf{round}(\tilde{\sigma})$ with $\mathsf{round}(\cdot)$ defined as in \pref{prop:rounding-energy}. Then
    \[
        \frac{1}{n}H(\sigma^*) \ge_{\text{w.h.p. as in \pref{prop:rounding-energy}}} \frac{1}{n}H(\tilde{\sigma}) - \frac{O(1)}{n^\alpha}\,.
    \]
    
    Putting the previous two estimates together with $\alpha = \min\left(\delta/4,1/24\right)$ gives
    \begin{align*}
        \frac{1}{n}H(\sigma^*) \ge \mathcal{E}(\beta) - \frac{O(1)}{n^\alpha} - O(\eps^{2}) - O_{\eps}(n^{-\alpha})\,.
    \end{align*}
    Using the elementary log-sum-exp bound 
\[\frac{\calP_{\beta}}{\beta} \le \lim_{n \to \infty}\E\frac{1}{n}\max_{\sigma\in\{\pm1\}^n}H(\sigma) + \frac{\log 2}{\beta}\]
together with \cite[Lemma 3.6 \& 3.7]{montanari2021optimization}, a simple integration-by-parts 
    in conjunction with~\cite[Proof of Theorem 2]{montanari2021optimization} (after adjusting for our normalization of the Hamiltonian) and $\beta = \frac{10}{\eps}$ so that $(2 \log 2 + 1/2)\eps/10 < \eps/5$ imply that, under fRSB,
    \[
         \mathcal{E}(\beta) = 2\beta\left(\int_0^{q^*_\beta}\left(\int_t^1 F_\mu(s)ds\right)dt\right) \ge \frac{\calP_\beta}{\beta} - \frac{\eps}{5}.
    \]
    Combining this with the previous inequality yields our first conclusion:
    \[
        \frac{1}{n}H(\sigma^*) \ge \frac{\calP_\beta}{\beta} - \frac{\eps}{5} - O(\eps^2) - O_{\eps}(n^{-\alpha}).
    \]

Then with the elementary log-sum-exp bound 
\[\frac{\calP_{\beta}}{\beta} \ge \lim_{n \to \infty}\E\frac{1}{n}\max_{\sigma\in\{\pm1\}^n}H(\sigma) = \lim_{\beta \to \infty} \frac{\calP_{\beta}}{\beta}.\]
    combined with~\pref{cor:ground-state-sk-parisi} and the concentration of the ground state energy~\cite[Theorem 5.1.3]{vershynin2018high}, we ultimately obtain with high probability
    \[
        \frac{1}{n}H(\sigma^*) \ge \frac{1}{n}\max_{\sigma \in \{-1,1\}^n} H(\sigma) - \frac{\eps}{5} - O(\eps^2) - o_n(1).
    \]
    By using the simple bound $\lim_{n\to\infty}\E\frac{1}{n}\max_{\sigma \in \{-1,1\}^n} H(\sigma) \ge 1/2$ from ~\cite[Proof of Theorem 2]{montanari2021optimization} (again, after adjusting for our normalization of the Hamiltonian), we obtain the multiplicative form
\[
H(\sigma^*)\ge
\left(1-\frac{\eps}{2}-O(\eps^2)-o_n(1)\right)
\max_{\sigma\in\{\pm1\}^n}H(\sigma).
    \qedhere
\]
\end{proof}

\section{Discussion \& Open Problems}\label{sec:discussion-open-problems}
We present a few natural directions for further inquiry given the successful initiation of Subag's algorithmic program to the domain of the hypercube.
 
\subsection{Extending to higher-degree polynomials}
It is natural to apply the PHA algorithmic framework on higher-degree random polynomials over the hypercube, such as mixed $p$-spin models. A mixed $p$-spin model is a non-homogeneous random degree-$p$ polynomial defined as,
\[
    H_p(\sigma) := \sum_{k \ge 2}^{p}\frac{\gamma_k}{n^{(k-1)/2}}\Iprod{A^{(k)}, \sigma^{\ot k}}\, ,
\]
where $\{A^{(k)}\}_{k \in [p]}$ is a family of independent order-$k$ Gaussian tensors and $\{\gamma_k \ge 0\}_{k \in[p]}$ are non-negative constants that weight the model. The goal is to compute,
\[
    \max_{\sigma \in \{\pm 1\}^n} H_p(\sigma)\, .
\]
This maximum value is captured, almost surely, by a modest generalization of the Parisi formula for the SK model~\cite{talagrand2010mean, talagrand2011mean, panchenko2013sherrington}. There is an existing AMP algorithm~\cite{alaoui2020optimization} that achieves an approximation ratio given by a natural \emph{relaxation} of the Parisi formula. We suspect that some more technical work extending the analysis in \S\ref{sec:convergence-to-ac} and \S\ref{sec:energy-analysis} to work with a potential function for the more general model~\cite{chen2018generalized} can be combined with estimates of spectral convergence for the Hessian of $H_p$ provided by Subag~\cite[Lemma 3]{subag2021following} to straightforwardly generalize the PHA algorithm.    



\subsection{Relaxed Parisi formula and high-entropy step sum-of-squares hierarchy}\label{sec:parisi-formula-and-sos}

High-entropy step (HES) distributions are defined as Euler-Maruyama discretizations of a family of sufficiently regular It\^o processes without drift. 
In a recent result, it was demonstrated that if $H^{\mathrm{sp}}(\sigma)$ is a spherical spin glass Hamiltonian in the fRSB regime, then the maximum expected value of $H^{\mathrm{sp}}(\sigma)$ attained by any HES process over $\sigma$ can be certified by a low-degree SoS proof (up to constant factors)~\cite[Theorem 1.4]{sandhu2024sum}.

The constant factors the certificates are off by arise due to imprecise estimates on the higher-order derivatives of the Hamiltonian. An approach is suggested by the authors to strengthen the bound \cite[\S 1.1, Pg 8]{sandhu2024sum}. Roughly, the approach would involve removing a wasteful $\ell^{\infty}$ bound to account for the Fourier mass restricted to fixed levels when bounding the nuclear norm of the cumulants of the HES distribution \cite[Lemma 5.26]{sandhu2024sum}. The feasibility argument of \cite{sandhu2024sum} critically uses the fact that Subag's Hessian ascent \emph{can} be randomized over a HES distribution, though, the challenging part is demonstrating that \emph{no} HES process can do better.

Note that the PHA algorithm is quite similar to a HES distribution (over the cube), as it is generated by a specific discretized It\^o process. To instantiate a HES SoS hierarchy over the hypercube~\cite[Open Question 1.8]{sandhu2024sum} it is crucial to have a HES distribution to demonstrate \emph{feasibility} for the same problem over the hypercube. The PHA algorithm provides this HES distribution \emph{conditioned on} there existing a low-degree matrix polynomial in the prior iterates that approximates $Q^2(t_k,\sigma_k)$ for every $k \in [K]$. The systematic control over the maximum of the bulk spectrum and Frobenius norm of $Q^2(t_k,\sigma_k)$ provided by~\pref{thm:david-magic} suggests that this is likely the case with a large constant choice of degree and an appropriate choice of $\delta$ to curtail the operator norm~(see~\cite[$\S$~2.3, HES SoS hierarchy]{sandhu2024sum}).

Additionally, the fact that the objective function's fluctuations can, when expressed via Taylor expansion-type arguments, be certified via concentration of measure and regularity arguments over the high-entropy process induced by the PHA algorithm is yet more evidence that the technique of proofs in~\cite[$\S$~6.2]{sandhu2024sum} can possibly be adapted to work on the cube. Therefore, provided that a low-degree matrix polynomial approximation for $Q^2(t_k,\sigma_k)$ exists, it is likely possible to instantiate the SoS HES hierarchy on the hypercube to provide low-degree SoS certificates for the \emph{relaxed} Parisi formula by combining the techniques in~\cite[\S 5 \& \S 6]{sandhu2024sum} with combinatorial techniques from free probability theory~\cite{nica2006lectures}.

\appendix

\section{It{\^o} Calculus}

In the energy analysis of the algorithm, convergence to a primal version of the Auffinger-Chen representation (\pref{thm:auffinger-chen}) plays a crucial role. The representation is a stochastic restatement of the Parisi Variational-Principle (\pref{thm:parisi-formula}) in terms of functions of a \emph{drifted} Brownian motion. We briefly state elementary results in It{\^o} calculus that we utilize; for a detailed review see~\cite{oksendal2013stochastic}. 

It{\^o} processes are closed under twice-differentiable functions.
\begin{definition}[It{\^o} formula]\label{def:ito-formula}
    Let $\{X_t\}_{t \in \R_{\ge 0}}$ be an It{\^o} process that satisfies the SDE,
    \[
        dX_t = f(t, X_t)dt + g(t, X_t)dW_t\, .
    \]
    Further, suppose that $h \in C^{1,2}(\R_{\ge 0},\R)$. Then,
    \[
        dh(t,X_t) = \left(\partial_t h(t,X_t) + f(t,X_t)\partial_x h(t,X_t) + \frac{g(t,X_t)^2}{2}\partial_{x,x}h(t,X_t)\right)dt + g(t,X_t) \partial_x h(t,X_t) dW_t\,.
    \]
\end{definition}

An important fact is that the set of self-driven Brownian motions follow a Plancherel-like identity known as the It{\^o} isometry.
\begin{lemma}[It{\^o} isometry]
    Given an It{\^o} process $\{X_t\}_{t \in \R_{\ge 0}}$, the It{\^o} integral forms an isometry of normed vector spaces induced by the functional inner product over the vector space of square-integrable functions. Consequently,
    \[
        \E\left[\left(\int_0^t X_s dW_s\right)^2\right] = \E\left[\int_0^t X^2_s ds\right]\,.
    \] 
\end{lemma}

\section{Wasserstein Metrics} \label{sec: Wasserstein distance}
Denote by $\operatorname{Lip}_b(M)$ the set of bounded Lipschitz functions over a metric space $(M,d)$. Let,
\[
    \norm{f}_{\operatorname{Lip}} := \sup_{x \neq y}\frac{\|f(x) - f(y)\|}{d(x,y)}\,.
\]
Below we review some elementary statements about Wasserstein metrics from the theory of optimal transport; see \cite{ambrosio2021lectures} for a detailed treatment.

\begin{definition}[Wasserstein metric~{\cite[Definition 8.1 \& Remark 8.2]{ambrosio2021lectures}}]\label{def:wasserstein-distance} \mbox{} \\
    Fix $p \in [1,\infty]$. Given a metric space $(M,d)$ with probability measures $\mu$ and $\nu$ on $M$ with $p$-finite moments, the Wasserstein $p$ metric is given as,
    \[
        d_{W,p}(\mu,\nu) = \inf_{\pi \in \Gamma(\mu,\nu)}\left(\E_{(x,y) \sim \pi} d(x,y)^p\right)^{1/p}\,,
    \]
    where $\Gamma(\mu,\nu)$ denotes the set of all probability measures $\pi$ on $M \times M$ that satisfy
    \[
         \pi(S \times M) = \mu(S)\,,\text{ and, } \pi(M \times S) = \nu(S)\, ,
    \]
    for every measurable set $S$.
\end{definition}

The \emph{Kantorovich-Rubinstein duality} gives a dual characterization of $d_{W,1}$ in terms of a variational problem over Lipschitz pushforwards of the underlying measures.

\begin{theorem}[Kantorovich-Rubinstein Duality, special case of~{\cite[Theorem 3.1]{ambrosio2021lectures}}]\label{thm:kr-duality}\mbox{} \\
    Let $(M,d)$ be a metric space, and $\mu$ and $\nu$ be measures over $M$ with finite first moment. Then,
    \[
        d_{W,1}\left(\mu,\nu\right) = \inf_{\pi \in \Gamma(\mu,\nu)}\left(\E_{(x,y) \sim \pi} d(x,y)\right) = \sup_{(f,g)\,\in I_d}\left(\E_{x\sim\mu}\left[f(x)\right] + \E_{y\sim\nu}\left[g(y)\right]\right)\,,
    \]
    where,
    \[
        I_d := \left\{(f,g)\,\in \operatorname{Lip}(M) \times \operatorname{Lip}(M)\mid f(x) + g(y) \le d(x,y)\right\}\,.
    \]
\end{theorem}

By fixing a base point $x_0 \in M$ and choosing $k(x) = \inf_{y \in M}\left(d(x,y) - g(y)\right) - \inf_{z \in M}\left(d(x_0,z) - g(z)\right)$ we see that finite first moments imply that $k\in L^1(\mu)\cap L^1(\nu)$ and that it satisfies $k(x_0) = 0$.
Hence, an easy corollary is that $(f,g)$ can be replaced by $(k,-k)$ without decreasing the dual value, and so the supremum can be taken over all $1$-Lipschitz functions over $(M,d)$ that evaluate to 0 at the base point $x_0$.

\begin{corollary}[Monge-Kantorovich-Rubinstein duality (simplified)]\label{cor:mkr-duality}\mbox{} \\
    Let $\mu$ and $\nu$ be as in \pref{thm:kr-duality}. Then,
    \[
        d_{W,1}\left(\mu,\nu\right) = \sup_{f(x_0) = 0,\, \norm{f}_{\operatorname{Lip}} \leq 1} \left| \int f\,d\mu - \int f\,d\nu \right|
    \]
\end{corollary}

We will also use the comparison between Wasserstein distances, especially the $L^1$ and $L^2$ Wasserstein distances.

\begin{lemma}[Comparison of Wasserstein distances] \label{lem: Wasserstein comparison}
If $p \leq q$, then $d_{W,p}(\mu,\nu) \leq d_{W,q}(\mu,\nu)$.  On the other hand, suppose that $\mu$ and $\nu$ are supported in a set $S$ of diameter at most $K$.  Then $d_{W,q}(\mu,\nu) \leq K^{1-p/q} d_{W,p}(\mu,\nu)^{p/q}$.
\end{lemma}

\begin{proof}
Fix a transport plan $\pi$ between $\mu$ and $\nu$.  Then we have by H{\"o}lder's inequality that
\[
\left( \mathbb{E}_{(x,y) \sim \pi} d(x,y)^p \right)^{1/p} \leq \left( \mathbb{E}_{(x,y) \sim \pi} d(x,y)^q \right)^{1/q}.
\]
Hence, taking the infimum over $\pi$ proves the first claim.  Similarly, if $\mu$ and $\nu$ are supported in a set of diameter $K$, then
\[
\mathbb{E}_{(x,y) \sim \pi} d(x,y)^q \leq K^{q-p} \mathbb{E}_{(x,y) \sim \pi} d(x,y)^p.
\]
Taking the $1/q$ power and taking the infimum over $\pi$ completes the proof.
\end{proof}

\section{Ruelle Probability Cascades}\label{sec:rpc}

 To aid the reader in translating the Ruelle Probability Cascades into an expression for the solution $\Phi$ of the Parisi PDE, we briefly recall a critical part of the construction of the Ruelle Probability Cascades, as well as the Hopf-Cole transformation which provides an explicit formula for $\Phi$ when the measure $\mu$ is atomic; see \cite[\S 2.1, \S 2.2 \& \S 2.3]{panchenko2013sherrington} and \cite[\S 13.1 \& \S 14.2]{talagrand2011mean} for a detailed overview.  We focus on expressing the Parisi solution at a fixed time $t_0$ when $\mu$ is a finitely supported measure.  The measure $\mu$ here is not necessarily the Parisi minimizer, but simply an arbitrary input probability measure for the Parisi differential equation.  Since the Parisi equation is solved backwards in time, the solution at time $t_0$ only depends on $\mu|_{(t_0,1]}$, and hence we set up the following notation.
 
 \begin{notation}[Atomic measure $\mu$]\label{not:atomic-parisi-measure}\mbox{} \\
     Let $0 \leq t_0 < \dots < t_r = 1$, and fix a finitely supported probability measure $\mu$ with $\supp(\mu) \subseteq [0,t_0] \cup \{t_1,\dots,t_r\}$.  Let $\zeta_j = \mu([0,t_j])$, and write $\mu|_{(t_0,1]}$ as $\sum_{j=1}^r (\zeta_j - \zeta_{j-1}) \delta_{t_j}$. 
 \end{notation}

 Below, we detail a part of the construction of the RPCs that assigns a random measure to the leaves of a certain $\infty$-ary tree.

\begin{definition}[$\operatorname{RPC}(\mu)$~{\cite[\S 2]{panchenko2013sherrington}}]\label{def:rpc}\mbox{} \\
    Let $\mu$ be an atomic measure as defined in \pref{not:atomic-parisi-measure} to parameterize the RPC tree. Fix $\mathcal{A} = \bigsqcup_{j=0}^r \N^j$ to represent an $\infty$-ary tree of depth $r+1$ rooted at $\emptyset$, with every node $\pi \in \N^j$ uniquely specified by a path $p(\pi) = (\emptyset, \alpha_1, (\alpha_1,\alpha_2),\dots,(\alpha_1,\dots,\alpha_j))$. Now, to every vertex $\pi \in \N^j$ for $0 \le j < r$ associate an independent Poisson point process with mean measure $\rho_j(dx) = \zeta_jx^{-1-\zeta_j}dx$ and arrange all arriving points in decreasing order $u_{\pi 1},u_{\pi 2},\dots$ to associate with the children of $\pi$. Finally, define a random measure on the leaves of the tree $\{\alpha\}_{\alpha \in \N^r}$ as follows,
    \[
        v_\alpha = \frac{\prod_{\beta \in p(\alpha)}u_\beta}{\sum_{\gamma \in \N^r}\prod_{\beta \in p(\gamma)}u_\beta}\,.
    \]
\end{definition}

Standard arguments dictate that the probability mass $v_\alpha$ assigned to every leaf $\alpha \in \N^r$ is almost-surely finite (see \cite[Lemma 2.4]{panchenko2013sherrington}). Asymptotically, when the parameter $\mu$ is taken to be the Parisi measure $\mu_\beta$ from \pref{ass:sk-frsb}, the random measure on the leaves of $\operatorname{RPC}(\mu_{\beta})$ approaches the Gibbs measure up to orthogonal transformation. 
Each vertex in the RPC tree has a geometric location constructed via a weighted linear combination over a complete orthonormal basis for the underlying separable Hilbert space, with the weighting decided by the points $\{t_j\}_{j=0}^r$ in the support of the atomic measure $\mu$. However, this does not affect the proof of the following lemmata, and is therefore omitted; see \cite[\S 2.3]{panchenko2013sherrington} for more details. 

\begin{lemma}[Hopf-Cole transformation]\label{lem:hc-transform}\mbox{} \\
 Let $\mu$ be a finitely supported measure and use \pref{not:atomic-parisi-measure}. Let $\Phi$ be the solution to the Parisi PDE associated to $\mu$.  Let $Z_1$, \dots, $Z_r$ be independent normal random variables with $E[Z_{j+1}] = 0$ and $\Var(Z_{j+1}) = 2 \beta^2 (t_{j+1} - t_j)$ for every $0 \le j \le r-1$.  Then
\[
\Phi(t_j,x) = \begin{cases} \frac{1}{\zeta_j} \log \E \exp(\zeta_j \Phi(t_{j+1},x+Z_{j+1}) ), & \zeta_j > 0 \\
\mathbb{E} \Phi(t_{j+1},x+Z_{j+1}), & \zeta_j = 0 \end{cases}
\]
\end{lemma}

\begin{proof}
On the interval $[t_j, t_{j+1}]$, we want
\[
\partial_t \Phi(t,x) = -\beta^2 (\partial_{x,x} \Phi(t,x) + \zeta_j \partial_x \Phi(t,x)^2).
\]
Hence, in the case where $\zeta_j > 0$,
\begin{align*}
\partial_t \exp(  \zeta_j \Phi(t,x) ) &= \zeta_j \partial_t \Phi(t,x) \exp(\zeta_j \Phi(t,x)) \\
&= -\zeta_j \beta^2 (\partial_{x,x} \Phi(t,x) + \zeta_j \partial_x \Phi(t,x)^2) \exp(\zeta_j \Phi(t,x))  \\ 
&= -\beta^2 [\zeta_j \partial_{x,x} \Phi(t,x) + \zeta_j^2 \partial_x \Phi(t,x)^2]\exp\left(\zeta_j\Phi(t,x)\right) \\
&= -\beta^2 \partial_{x,x} [\exp(\zeta_j \Phi(t,x))].
\end{align*}
Recall from \pref{prop: Phi derivative bound} that $|\partial_x \Phi(t,x)| \leq 1$ so $\Phi$ grows linearly, and it is not hard to see that the solution to the heat equation is unique for a function that grows at most exponentially.  Therefore, for $t \in [t_j,t_{j+1}]$, we have
\[
\exp(  \zeta_j \Phi(t,x) ) = \frac{1}{\sqrt{4\pi \beta^2 (t_{j+1} - t)}} \int_{\R} \exp(\zeta_j \Phi(t_{j+1},x+z)) \exp\left(-\frac{z^2}{4 \beta^2(t_{j+1} - t)} \right) \,dz.
\]
Hence, if $Z_{j+1}$ is a normal random variable of mean zero and variance $2 \beta^2(t_{j+1} - t_j)$, then we have
\[
\exp( \zeta_j \Phi(t_j,x) ) = \E[\exp(\zeta_j \Phi(t_{j+1},x+Z_{j+1}))], 
\]
which is equivalent to the asserted formula.  In the case when $\zeta_j = 0$, the formula reduces to the standard probabilistic formula for solutions to the heat equation.
\end{proof}

Below, we briefly sketch how the Ruelle Probability Cascades can be used to construct a solution to provide a specific representation to the Hopf-Cole transformed expression for $\Phi(t_0,x)$.

\begin{lemma}[RPC based representation of $\Phi(t_0,x)$] \label{lem: RPC model} \mbox{} \\
Let $\mu$ be a finitely supported measure on $[0,1]$, let $t_0 \in [0,1)$ and use \pref{not:atomic-parisi-measure}.  Let $\mathcal{A} = \bigsqcup_{j = 0}^{r} \N^j$.  Then there exist nonnegative random variables $(v_\alpha)_{\alpha \in \N^r}$ such that $\sum_{\alpha \in \N^r} v_\alpha = 1$, and there exist Gaussian random variables $(Z_\alpha)_{\alpha \in \N^r}$ with $\E Z_\alpha = 0$ and $\Var(Z_\alpha) = 2 \beta^2(1 - t_0)$ such that for all $x$,
\[
\Phi(t_0,x) = \E \log \sum_{\alpha \in \N^r} 2 v_\alpha \cosh(x + Z_\alpha).
\]
\end{lemma}

\begin{proof}
First consider the case where $\mu([0,t_0]) > 0$ so that all the $\zeta_j$'s are strictly positive.  Let $Z_1$, \dots, $Z_r$ be independent normal random variables with $\E[Z_{j+1}] = 0$ and $\Var(Z_{j+1}) = 2 \beta^2 (t_{j+1} - t_j)$.  For compatibility with the notation of \cite{panchenko2013sherrington}, recall that we can express $Z_j$ as a function of a uniform random variable $\omega_j$ on $[0,1]$ so that $\omega_1$, \dots, $\omega_r$ are independent.  Let $X_j$ be the random variable
\[
X_j(\omega_1,\dots,\omega_j) = \Phi(t_j, x + Z_1(\omega_1) + \dots + Z_j(\omega_j)),
\]
so that
\[
X_j(\omega_1,\dots,\omega_j) = \frac{1}{\zeta_j} \log \E \left[ \exp(\zeta_j X_{j+1}(\omega_1,\dots,\omega_{j+1})) \mid \omega_1, \dots, \omega_j \right]
\]
when $\zeta_j > 0$ and $X_j(\omega_1,\dots,\omega_j) =\mathbb{E} [X_{j+1}(\omega_1,\dots,\omega_{j+1}) \mid \omega_1, \dots, \omega_j]$ if $\zeta_j = 0$.
Now consider independent uniform random variables $\omega_\alpha$ in $[0,1]$ indexed by strings $\alpha = (n_1,\dots,n_j)$ of natural numbers, i.e.\ the index set is $\mathcal{A} = \bigsqcup_{j=0}^\infty \N^j$.  For $\alpha = (n_1,\dots,n_r) \in \N^r$, write
\[
\Omega_\alpha = (\omega_{(n_1)}, \omega_{(n_1,n_2)},\dots,\omega_{(n_1,\dots,n_r)}).
\]
Let $(v_\alpha)_{\alpha \in \mathcal{A}}$ be the random weights constructed using the RPCs (\pref{def:rpc}). In particular, $\sum_{\alpha \in \N^r} v_\alpha = 1$.  Then \cite[Theorem 2.9]{panchenko2013sherrington} shows that
\[
X_0 = \E \log \sum_{\alpha \in \N^r} v_\alpha \exp X_r(\Omega_\alpha).
\]
In other words,
\[
\Phi(t_0,x) = \E \log \sum_{(n_1,\dots,n_r) \in \N^r} v_{(n_1,\dots,n_r)} \exp\left( \Phi(1,x+Z_1(\omega_{n_1}) + \dots Z_r(\omega_{(n_1,\dots,n_r)})) \right).
\]
Writing
\[
Z_\alpha(\Omega_\alpha) = Z_1(\omega_{n_1}) + \dots + Z_r(\omega_{(n_1,\dots,n_r)}), 
\]
we see that $Z_\alpha$ is normal with mean zero and variance $\sum_{j=0}^{r-1} 2\beta^2 (t_{j+1} - t_j) = 2\beta^2 (1 - t_0)$.  Finally, note $\exp(\Phi(1,\cdot)) = 2 \cosh(\cdot)$, which proves the asserted formula.

In the case where $\mu([0,t_0]) = 0$, note that $t_1$ is the infimum of the support of $\mu$.  Then by the preceding argument, we can express
\[
\Phi(t_1,x) = \E \log \sum_{\alpha \in \N^r} 2 v_\alpha \cosh(x + Z_\alpha),
\]
where $Z_\alpha$ is a Gaussian random variable of variance $2 \beta^2(1 - t_1)$.  Recall $\Phi$ satisfies $\partial_t \Phi = -\beta^2 \partial_{x,x} \Phi$ on $[t_0,t_1]$.  Let $Z$ be a normal random variable of mean zero and variance $2 \beta^2 (t_1 - t_0)$ independent of all the other random variables.  Then
\begin{align*}
\Phi(t_0,x) &= \E \Phi(t_1,x+Z) \\
&= \E \log \sum_{\alpha \in \N^r} 2 v_\alpha \cosh(x + Z + Z_\alpha).
\end{align*}
Hence, we have the asserted formula at $t_0$ using $\tilde{Z}_\alpha = Z + Z_\alpha$ which is normal of mean zero and variance $2 \beta^2(1 - t_1) + 2 \beta^2(t_1 - t_0) = 2 \beta^2(1 - t_0)$. 
\end{proof}

\newpage

\addtocontents{toc}{\protect\setcounter{tocdepth}{-1}}

\section*{Acknowledgements}

\addtocontents{toc}{\protect\setcounter{tocdepth}{1}}

DJ \& JS thank the union of postdocs and academic researchers at University of California for introducing them to each other.

JS \& JSS thank Chris Jones for a stimulating discussion on the algorithmic applications of Brownian motion. 
JS \& JSS are grateful to Todd Kemp for an instructive and helpful discussion in Summer-2023, and DJ is grateful to Kemp for his advice as postdoc supervisor in 2020-2023.
JSS is indebted to Brice Huang for a particularly inspiring \& informative discussion at Harvard University in Fall-2022 which motivated the formulation of a Hessian ascent framework on the hypercube, and for comments on a draft of this paper.

The initial development of this work was conducted in Summer-2023 jointly at the Halıcıoğlu Data Science Institute and the department of mathematics at the University of California San Diego during a set of particularly fruitful discussions between all the authors.

JSS is extremely grateful to Alexandra Kolla for her advice and support as postdoc supervisor. 
JSS did this work in part as a visiting scholar at the Flatiron Institute hosted by Prof. SueYeon Chung. 

The authors are grateful to the anonymous referees for their careful review and detailed comments.

\addtocontents{toc}{\protect\setcounter{tocdepth}{-1}}

\section*{Funding}

\addtocontents{toc}{\protect\setcounter{tocdepth}{1}}

DJ acknowledges funding from the National Science Foundation (US), grant DMS-2002826; the National Sciences and Engineering Research Council (Canada), grant RGPIN-2017-05650; Denmark's Independent Research Fund, grant 1026-00371B; and the Horizon EU Marie Sk{\l}odowska-Curie Action FREEINFOGEOM, grant 101209517.  JS acknowledges support from National Science Foundation (US) awards \#2217058 and \#2112665.
JSS acknowledges partial support from Defense Advanced Research Projects Agency ONISQ program award HR001120C0068 during the early stages of this work.

\addtocontents{toc}{\protect\setcounter{tocdepth}{-1}}

\section*{Competing interests}

\addtocontents{toc}{\protect\setcounter{tocdepth}{1}}

The authors have no relevant financial or non-financial interests to disclose.

\addtocontents{toc}{\protect\setcounter{tocdepth}{-1}}

\section*{Data availability}

\addtocontents{toc}{\protect\setcounter{tocdepth}{1}}

No datasets were generated or analysed during the current study.

\addtocontents{toc}{\protect\setcounter{tocdepth}{-1}}

\bibliographystyle{sn-mathphys-num}
\bibliography{cmp-revision}

\addtocontents{toc}{\protect\setcounter{tocdepth}{1}}

\newpage

\end{document}